\documentclass[12pt]{amsart}
\usepackage{amsfonts,amsrefs,latexsym,amsmath, amssymb}
\usepackage{mathrsfs,MnSymbol}
\usepackage{url,color}
\usepackage{upgreek}
\usepackage{fancyhdr}
\usepackage{hyperref}
\usepackage{times}
\usepackage{scalefnt}
\usepackage{url}

\newtheorem{proposition}{Proposition}[section]
\newtheorem{lemma}[proposition]{Lemma}
\newtheorem{corollary}[proposition]{Corollary}
\newtheorem{theorem}{Theorem}[section]

\theoremstyle{definition}
\newtheorem{definition}{Definition}[section]
\newtheorem{remark}{Remark}[section]

\numberwithin{equation}{section}

\newcommand{\leftexp}[2]{{\vphantom{#2}}^{#1}{#2}}

\newcommand{\speed}{c_s}
\newcommand{\SecondFund}{k}
\newcommand{\Riemann}{Riem}
\newcommand{\Ric}{Ric}
\newcommand{\Sch}{Sch}
\newcommand{\gfour}{\mathbf{g}}
\newcommand{\hfour}{\mathbf{h}}
\newcommand{\Chfour}{\mathbf{\Gamma}}
\newcommand{\kfour}{\mathbf{k}}
\newcommand{\Ricfour}{\mathbf{Ric}}
\newcommand{\Riemfour}{\mathbf{Riem}}
\newcommand{\Rfour}{\mathbf{R}}
\newcommand{\Tfour}{\mathbf{T}}
\newcommand{\Sfour}{\mathbf{S}}
\newcommand{\Schfour}{\mathbf{Sch}}
\newcommand{\ufour}{\mathbf{u}}
\newcommand{\Nml}{\hat{\mathbf{N}}}
\newcommand{\longNml}{\mathbf{N}}
\newcommand{\Dfour}{\mathbf{D}}
\newcommand{\Wfour}{\mathbf{W}}

\newcommand{\newg}{G}
\newcommand{\newsec}{K}
\newcommand{\freenewsec}{\hat{K}}
\newcommand{\newp}{P}
\newcommand{\adjustednewp}{\hat{P}}
\newcommand{\newu}{U}

\newcommand{\newlap}{\Psi}
\newcommand{\dlap}{\Theta}

\newcommand{\gdet}{\mbox{det}g}
\newcommand{\gbigdet}{\mbox{det}\newg}

\newcommand{\metricenergy}[1]{\mathscr{E}_{(Metric);#1}}
\newcommand{\fluidenergy}[1]{\mathscr{E}_{(Fluid);#1}}
\newcommand{\totalenergy}[2]{\mathscr{E}_{(Total);#1;#2}}

\newcommand{\highnorm}[1]{\mathscr{H}_{(Frame-Total);#1}}
\newcommand{\lowkinnorm}[1]{\mathscr{C}_{(Frame-Kinetic);#1}}
\newcommand{\lowpotnorm}[1]{\mathscr{C}_{(Frame-Potential);#1}}

\newcommand{\smallparameter}{\theta}

\newcommand{\Eucfour}{\mathbf{E}}
\newcommand{\Euc}{E}
\newcommand{\ID}{I}

\topmargin - .23 in

\begin{document}

%\scalefont{1.10}

\title{Stable Big Bang Formation in Near-FLRW Solutions to the Einstein-Scalar Field and Einstein-Stiff Fluid Systems}
%The Nonlinear Past-Stability of the Isotropic FLRW Big Bang Solution to the Einstein-scalar field and Einstein-stiff fluid systems
\author{Igor Rodnianski$^{*}$
\and 
Jared Speck$^{**}$}
%\date{\today}
%\email{jspeck@math.mit.edu}

\thanks{$^*$Princeton University, Department of Mathematics, Fine Hall, Washington Road, Princeton, NJ 08544-1000, USA. \texttt{irod@math.princeton.edu}}

\thanks{$^{**}$Massachusetts Institute of Technology, Department of Mathematics, 77 Massachusetts Ave, Room 2-163, Cambridge, MA 02139-4307, USA. \texttt{jspeck@math.mit.edu}}

\thanks{$^{*}$ IR gratefully acknowledges support from NSF grant \# DMS-1001500.}	

\thanks{$^{**}$ JS gratefully acknowledges support from NSF grant \# DMS-1162211 and
from a Solomon Buchsbaum grant administered by the Massachusetts Institute of Technology.
}	

\begin{abstract}
	We prove a stable singularity formation result, without symmetry assumptions, 
	for solutions to the Einstein-scalar field and Einstein-stiff fluid systems. 
	Our results apply to small perturbations of the spatially flat Friedmann-Lema\^{\i}tre-Robertson-Walker 
	(FLRW) solution with topology $(0,\infty) \times \mathbb{T}^3.$  The FLRW solution models 
	a spatially uniform scalar-field/stiff fluid evolving in a spacetime 
	that expands towards the future
	and that has a ``Big Bang'' singularity at $\lbrace 0 \rbrace \times \mathbb{T}^3,$
	where its curvature blows up.  
	We place ``initial'' data on a Cauchy hypersurface $\Sigma_1'$ that 
	are close, as measured by a Sobolev norm,
	to the FLRW data induced on $\lbrace 1 \rbrace \times \mathbb{T}^3.$
	We then study the asymptotic behavior of the perturbed solution in the \emph{collapsing} direction
	and prove that its basic qualitative and quantitative features 
	closely resemble those of the FLRW solution.
	%We show that various qualitative and quantitative features of the FLRW solution are also realized in 
	%the perturbed solutions. 
	In particular, we construct constant mean curvature-transported spatial coordinates for the perturbed solution
	covering $(t,x) \in (0,1] \times \mathbb{T}^3$ and show that it also has a Big Bang
	at $\lbrace 0 \rbrace \times \mathbb{T}^3,$ where its curvature blows up.
	The blow-up confirms Penrose's Strong Cosmic Censorship hypothesis for the ``past-half'' of near-FLRW solutions.
	Furthermore, we show that the Einstein equations are dominated by kinetic 
	(time derivative) terms that induce approximately monotonic behavior near the Big Bang, and consequently,
	various time-rescaled components of the solution converge to functions
	of $x$ as $t \downarrow 0.$
	
	The most difficult aspect of the proof is showing that
	the solution exists for $(t,x) \in (0,1] \times \mathbb{T}^3,$
	and to this end, we derive a hierarchy of energy estimates
	that are allowed to mildly blow-up as $t \downarrow 0.$ 
	To close these estimates, 
	it is essential that we are able to rule out more singular energy blow-up,
	which is in turn tied to the most important ingredient in our analysis:
	an $L^2-$type \emph{energy approximate monotonicity inequality} that holds 
	for near-FLRW solutions.
	%To close the energy hierarchy, we must also
	%derive estimates for various terms and their low order derivatives
	%that are less singular in $t$ compared to the energy estimates.
  In the companion article 
	\cite{iRjS2014a}, we used the approximate monotonicity
	to prove a stability result for solutions to linearized versions of the equations.
	The present article shows that the linear stability result
	can be upgraded to control the nonlinear terms.
	
\bigskip

\bigskip


\end{abstract}

\maketitle

\centerline{\today}

%\setcounter{tocdepth}{1}
%\pagenumbering{roman} 
%\tableofcontents 
%\newpage 
%\pagenumbering{arabic}

\section{Introduction} \label{S:INTRO}
In this article, we prove a stable singularity formation result for 
a class of solutions to the Einstein-scalar field and Einstein-stiff fluid systems. 
\begin{remark} \label{R:SCALARFIELD}
	It is well-known (see, for example, \cite{dC2007}) that the scalar field matter model can be viewed as a special subcase of the stiff 
	fluid model under the assumptions that the scalar field's gradient is timelike, that the
	fluid is irrotational, and that the metric dual of the fluid four velocity is exact up to normalization. 
	Thus, we can restrict our attention to the Einstein-stiff fluid system;
	all of the results that we derive for this system immediately imply corresponding results for the Einstein-scalar field system. 
\end{remark}
\noindent We derive our main results 
by exploiting a new form of approximate $L^2$-type monotonicity that holds 
for solutions close to the well-known Friedmann-Lema\^{\i}tre-Robertson-Walker (FLRW) solution
on $(0,\infty) \times \mathbb{T}^3$ (see \eqref{E:BACKGROUNDSOLUTION}). 
The FLRW solution expands towards the future and contains 
a Big Bang singularity along its past boundary $\lbrace 0 \rbrace \times \mathbb{T}^3,$
where its curvature blows up. In the companion article \cite{iRjS2014a}, we derived approximate monotonicity
and stability results for linearized versions of the related but simpler (see Remark~\ref{R:SCALARFIELD}) 
Einstein-scalar field system, where we placed linearized data along $\lbrace 1 \rbrace \times \mathbb{T}^3$
and studied the solutions in the singular direction $t \downarrow 0.$
More precisely, in \cite{iRjS2014a}, we studied 
solutions to a large family of linear equations formed by 
linearizing the Einstein-scalar field equations around the well-known
family of Kasner solutions (see \eqref{E:KASNER}), which are explicit
and depend only on $t.$ Exceptional cases aside, 
the Kasner solutions also have Big Bang singularities at $\lbrace 0 \rbrace \times \mathbb{T}^3,$
and in fact, the FLRW solution is the unique spatially isotropic example.
However, in \cite{iRjS2014a}, we obtained compellingly strong linear stability statements only for 
near-FLRW Kasner backgrounds. Roughly, in the present article, we show that those stability results 
are strong enough to control the nonlinear terms in the near-FLRW regime, 
which allows us to prove that the FLRW Big Bang singularity is nonlinearly stable.
It is natural to ask whether or not our results can be extended to 
show the stability of the Big Bang singularity for far-from-FLRW Kasner solutions.
We do not take a stance on what to expect. An essential first step towards an answer 
would be to carry out a detailed linearized stability analysis for from-from-FLRW Kasner backgrounds,
thus going beyond the regime addressed in detail in \cite{iRjS2014a}.

Our work here can be viewed as addressing various aspects of the following two related questions:

\begin{center}
	\begin{enumerate} 
		\item Which initial data have incomplete maximal globally hyperbolic developments?
		\item How do the incomplete solutions behave near their boundary? In particular, 
		does a true singularity, such as curvature blow-up, occur at the boundary? 
	\end{enumerate}
\end{center}

Our main results, which we summarize in Sect.~\ref{SS:MAINSUMMARY}, provide a detailed answer to the above questions for 
near-FLRW solutions.
%a class of cosmological\footnote{By ``cosmological spacetime,'' we mean that $\mathbf{M}$ has compact Cauchy hypersurfaces and that the Ricci curvature of $\gfour_{\mu \nu}$ verifies $\Ricfour_{\alpha \beta}\mathbf{X}^{\alpha}\mathbf{X}^{\beta} \geq 0$ for all timelike vectors $\mathbf{X}^{\mu}.$ The latter condition is always verified by the solutions that we consider
%because Einstein's equations imply that $\Ricfour_{\mu \nu} = \Tfour_{\mu \nu} - \frac{1}{2} \gfour_{\mu \nu} (\gfour^{-1})^{\alpha \beta} \Tfour_{\alpha \beta} \Tfour_{\mu \nu}$ 
%and because the energy-momentum tensor $\Tfour_{\mu \nu}$ of a stiff fluid verifies the strong energy condition.}
%spacetime-matter\footnote{By ``spacetime,'' we mean a four-dimensional time-orientable manifold $\mathbf{M}$ together with a Lorentzian %metric $\gfour_{\mu \nu}$ of signature $(-,+,+,+).$ } 
%solutions. 
Specifically, we provide a complete description of the past dynamics of solutions launched by 
data given on a spacelike hypersurface $\Sigma_1'$ 
that are $H^8-$close to the FLRW data induced on $\lbrace 1 \rbrace \times \mathbb{T}^3.$
%an \emph{open} set of initial data given on a spacelike hypersurface $\Sigma_1'.$ 
%We assume that the data are $H^8-$close to the FLRW data, 
We prove that the perturbed solutions ``develop'' a ``Big Bang'' singularity to the past of $\Sigma_1'.$
For simplicity, we restrict our attention to the case $\Sigma_1' = \mathbb{T}^3 := [-\pi,\pi]^3$
(with the ends identified). 

\subsection{The Einstein-stiff fluid equations and context}
The Einstein-stiff fluid equations are a special case of the 
the Euler-Einstein equations, which are often used in cosmology.
The role of the fluid, which is the most common matter model used in cosmology, is to model the average matter-energy content of the entire universe. 

\subsubsection{The equations}
\label{SSS:THEEQUATIONS}
Relative to an arbitrary coordinate system, the Einstein-stiff fluid equations are
\begin{subequations}
\begin{align}
	\Ricfour_{\mu \nu} - \frac{1}{2}\Rfour \gfour_{\mu \nu} & = \Tfour_{\mu \nu}, && (\mu,\nu = 0,1,2,3), \label{E:EINSTEININTRO} \\
	\Dfour_{\mu} \Tfour^{\mu \nu} & = 0, && (\nu = 0,1,2,3), & \label{E:DTIS0INTRO}
\end{align}
\end{subequations}
where $\Ricfour_{\mu \nu}$ denotes the Ricci tensor of $\gfour_{\mu \nu},$ $\Rfour$ denotes the scalar curvature of 
$\gfour_{\mu \nu},$ $\Dfour_{\mu}$ denotes the Levi-Civita connection of $\gfour_{\mu \nu},$ and 
$\Tfour_{\mu \nu} := \gfour_{\mu \alpha} \gfour_{\nu \beta} \Tfour^{\alpha \beta}$ denotes 
the metric dual of the energy-momentum tensor $\Tfour^{\mu \nu}$ of a perfect fluid:
\begin{align} \label{E:TFOURFLUID}
	\Tfour^{\mu \nu} & = (\rho + p) \ufour^{\mu} \ufour^{\nu} + p(\gfour^{-1})^{\mu \nu}.
\end{align}
In \eqref{E:TFOURFLUID}, $\rho$ denotes the fluid's proper energy density, $p$ denotes the pressure, and
$\ufour^{\mu}$ denotes the four-velocity, which is a future-directed vectorfield subject to the normalization condition
\begin{align}
	\gfour_{\alpha \beta} \ufour^{\alpha} \ufour^{\beta} & = - 1. \label{E:UNORMALIZED} 
\end{align}

The system \eqref{E:EINSTEININTRO}-\eqref{E:UNORMALIZED} is not closed because there are not enough fluid equations. In cosmology, the equations are often closed (see \cite[Ch. 5]{rW1984}) by assuming the \emph{equation of state} 
\begin{align} \label{E:EOS}
	p = \speed^2 \rho,
\end{align} 
where $0 \leq \speed \leq 1$ is a constant known as the \emph{speed of sound}. 
The \emph{stiff fluid} is the case $\speed = 1.$ This equation of state is unique among all those of
the form \eqref{E:EOS}. In particular, our main results show that its special properties regularize near-FLRW solutions in a neighborhood of their Big Bang singularities. By ``regularize,'' we mean that the singularity formation processes are controlled, stable, and approximately monotonic. See Sect.~\ref{SSS:OTHERMATTER} for additional details regarding the special properties of the stiff fluid.

\subsubsection{Connection to the Hawking-Penrose theorems and the Strong Cosmic Censorship hypothesis}
\label{SSS:HAWKINGPENROSE}
The famous cosmological singularity theorems of Penrose and Hawking 
(see e.g. \cite{sH1967}, \cite{rP1965}, and the discussion in \cite{sHrP1970}) 
guarantee that a large class of globally hyperbolic\footnote{Globally hyperbolic spacetimes are those
containing a Cauchy hypersurface.} spacetimes are necessarily ``singular'' in a certain weak sense. 
In fact, these theorems are able to answer question (1) posed above in a general setting. 
For example, Hawking's theorem (see \cite[Theorem 9.5.1]{rW1984}) guarantees that under certain assumptions verified in our problem of interest, all past-directed 
timelike geodesics are incomplete.\footnote{
The assumptions needed in Hawking's theorem are \textbf{i)} the matter model verifies the \emph{strong energy condition}, which is  $\left(\Tfour_{\mu \nu} - \frac{1}{2} (\gfour^{-1})^{\alpha \beta} \Tfour_{\alpha \beta} \Tfour_{\mu \nu} \right) \mathbf{X}^{\mu} \mathbf{X}^{\nu} \geq 0$ whenever $\mathbf{X}$ is timelike; \textbf{ii)} $\SecondFund_{\ a}^a < - C$ everywhere on the initial Cauchy hypersurface, where $\SecondFund$ is the second fundamental form of $\Sigma_t$ 
(see \eqref{E:SECONDFUNDDEF})
and $C > 0$ is a constant. More precisely, Hawking's theorem shows that no past-directed timelike curve can have length greater than $\frac{3}{C}.$ In our main results, we derive a slightly sharper version of this estimate that is tailored to the solutions addressed in this article (see inequality \eqref{E:LENGTHEST}).
}
However, the theorems do not reveal the nature of the incompleteness. In particular, they allow for both of the following possibilities: \textbf{i)} the spacetime is inextendible across the region where the geodesics terminate due to some kind of blow-up of an invariant quantity 
(this is what happens, for example, in the FLRW solution \eqref{E:BACKGROUNDSOLUTION}); or \textbf{ii)} the spacetime can be extended as a solution in a ``regular'' (but perhaps non-unique) fashion across the region where the geodesics terminate (this can happen when a Cauchy horizon develops, as in the case of the well-known Taub family of solutions to the Einstein-vacuum equations;
see \cite{jGjP2009}, \cite{aR2005b}, \cite{pCjI1993}). 
To the best of our knowledge, there are no prior results that are able to distinguish between these two scenarios for 
an open set of data (without symmetry assumptions). That is, there are no prior results that answer question (2) posed above. 
Our main results fully confirm the scenario \textbf{i)} for near-FLRW data due to curvature blow-up
at\footnote{As we describe below, the relevant time coordinate $t$ for the perturbed solution
has level sets with mean curvature $- (1/3) t^{-1}.$} 
$\lbrace t = 0 \rbrace.$
This shows in particular that the solutions cannot be extended beyond $\lbrace t = 0 \rbrace$
and yields a proof of Penrose's \emph{Strong Cosmic Censorship hypothesis}\footnote{
Roughly, this conjecture asserts that the maximal globally hyperbolic development of the data
should be inextendible as a ``regular" Lorentzian manifold.} 
\cite{rP1969}
in the ``collapsing half'' of the near-FLRW spacetimes.
The overall strategy of our proof is to provide an exhaustive picture of the collapsing half of the maximal globally hyperbolic development\footnote{We explain the
maximal globally hyperbolic development in more detail in Section \ref{SS:IVP}.} 
of the data including a description of the asymptotic behavior of the solution near its past boundary. 
%The assumption that the matter is a stiff fluid plays an essential role in parts of our analysis.

\subsubsection{The main ideas behind the proof and the role of the approximate monotonicity}
\label{SSS:NEEDFORMONOTONICITY}
The main ideas behind our analysis are \textbf{i)} exploiting the approximate 
$L^2-$type monotonicity inequality 
mentioned above. When we derive energy estimates, the approximate 
monotonicity manifests
as the availability of some negative definite quadratic spacetime integrals that encourage decay towards the past and
that allow us to absorb various error integrals. The approximate\footnote{The monotonicity is approximate
in the sense that it yields negative definite spacetime integrals only for some of the solution variables.}  
monotonicity holds in certain regimes and seems to require the presence of certain kinds of matter such as the stiff fluid. 
Our derivation of the approximate monotonicity is based on a hierarchy of special structures 
in the Einstein-stiff fluid equations that are visible in \emph{constant mean curvature (CMC)-transported spatial coordinates}. 
More precisely, the key portion of the approximate monotonicity  
is available in part because of a 
favorably signed (in a subtle sense) linear term in the fluid equations that leads to control over the lapse 
(the lapse $n$ is defined in Sect.~\ref{SS:CHOICEOFVARIABLES}).  
We provide an overview of
the monotonicity in Sect.~\ref{SSS:KEYMONOTONICITYINTRO}, and
we provide the full details in Prop.~\ref{P:FUNDAMENTALENERGYINEQUALITY}.
In Sect.~\ref{SS:MONOTONICITYEXPLANATION}, we shed light on the
direction of the approximate monotonicity by recalling
an alternate one-parameter family of parabolic gauges for the lapse variable,
first introduced in \cite{iRjS2014a}.
For certain parameter choices, the approximate monotonicity is also visible in these gauges 
\emph{but the corresponding PDEs are locally well-posed only in the past direction.} 
The CMC gauge, which we use in the vast majority of our analysis, 
is formally a singular limit of the parabolic family. 
To the best of our knowledge, the idea of using parabolic gauges 
for the Einstein equations first appeared in the numerical relativity literature
in \cite{jBgDeS1996}, where the authors suggested that such gauges might
be suitable for studying the long-time evolution of solutions.
Readers may also consult \cite{cGjMG2006} for a discussion of local 
well-posedness for the Einstein equations under various gauge conditions involving a parabolic equation.
\textbf{ii)} We commute the equations with spatial derivatives and derive sharp pointwise and $L^2$ bounds for the inhomogeneous terms that arise. This analysis is delicate because many terms become singular at the Big Bang, 
and this can cause our $L^2-$type energies to blow-up. 
To close our estimates, \emph{we must prove that the energy blow-up rate is mild in the approach towards the Big Bang.} 
This is where we need the approximate monotonicity: 
without the negative definite integrals that it provides, 
we would be unable to absorb certain key error integrals.
Consequently, we would be able to show only that the
energies are bounded from above by the size of the data times 
$t^{-C}$ for a large constant $C,$
and such blow-up is too severe to allow us to close our bootstrap argument.
We must also derive the precise blow-up rates for the inhomogeneous terms
that cannot be absorbed into the negative definite spacetime integrals.
Some of the blow-up rates are dangerous yet
optimal in the sense that they are saturated
by the Kasner solutions \eqref{E:KASNER}. By ``dangerous,'' we mean that
the rates can cause mild energy blow-up and that
\emph{a slightly worse inhomogeneous term blow-up rate could lead to rapid energy growth, which would
completely destroy our proof.}
To bound the inhomogeneous terms, we establish an intricate hierarchy of decay/growth rates that 
distinguishes between the behavior of lower-order and higher-order derivatives near the Big Bang. %In Sect.~\ref{SS:MODELPROBLEM}, we discuss a simple model problem in order to illustrate the main aspects of this analysis.

%In the companion work {\color{red} INSERT}, we derived the approximate monotonicity
%and corresponding stability results for linearized versions of the Einstein-scalar field system.
%Specifically, we derived the strongest version of our linear stability results 
%when we linearized the equations around Kasner solutions that are
%sufficiently close to the Friedmann-Lema\^{\i}tre-Robertson-Walker (FLRW) solution;
%The analog solutions to the Einstein-stiff fluid equations are respectively \eqref{E:KASNER} and \eqref{E:BACKGROUNDSOLUTION}.
%Roughly, our goal in the present is show that the estimates of {\color{red} INSERT} 
%are sufficient to control the nonlinear terms in the region $(0,1] \times \mathbb{T}^3$
%when the ``initial" data $\lbrace t = 1 \rbrace$ are near that of the FLRW solution.
 
\subsection{Details on the FLRW and Kasner solutions} \label{SS:SOLUTIONSWESTUDY}
The FLRW solution to \eqref{E:EINSTEININTRO}-\eqref{E:DTIS0INTRO} when $\speed = 1$ is
\begin{align} \label{E:BACKGROUNDSOLUTION}
	\widetilde{\gfour} & = - dt^2 + t^{2/3} \sum_{i=1}^3 (dx^i)^2, \qquad \widetilde{p} = \frac{1}{3} t^{-2}, \qquad \widetilde{\ufour}^{\mu} = \delta_0^{\mu}, \qquad (t,x) \in (0,\infty) \times \mathbb{T}^3,
\end{align}
where $\delta_{\mu}^{\nu}$ $(0 \leq \mu, \nu \leq 3)$ is the standard Kronecker delta. It is a spatially flat member of a family of dynamic spatially homogeneous solutions that were first discovered\footnote{
The solutions in \cite{aF1922en} had 
$\speed=0$ and positive spatial curvature (the spacetime manifolds that Friedmann considered were diffeomorphic to $\mathbb{R} \times \mathbb{S}^3$), 
but many other values of $\speed$ and other choices of spatial curvature have since been considered.} 
by Friedmann 
in 1922 \cite{aF1922en}. 
The family continues to strongly influence the way we think about the possible structure of our universe. These solutions model a spatially homogeneous, isotropic spacetime that contains a perfect fluid, and many members of the family (including \eqref{E:BACKGROUNDSOLUTION}) have a 
Big Bang singularity at $t = 0.$ It was Lema\^{\i}tre who first emphasized that many members of the family, if extrapolated backwards in time (from, say, $t=1$), predict the presence of a Big Bang \cite{gL1931}. The solutions were investigated in more detail in the 1930's by Robertson and Walker, and hence the family is often referred as the FLRW family. 

The FLRW solution \eqref{E:BACKGROUNDSOLUTION} is also a member of the generalized Kasner family\footnote{The Kasner family is often defined to contain only vacuum solutions.} (the vacuum Kasner solutions were discovered in 1921 \cite{eK1921}), a class of Bianchi type I solutions\footnote{For an overview of the Bianchi I class and other symmetry classes that we mention later,
readers may consult \cite{pCgGdP2010}.} 
that are spatially homogeneous but generally anisotropic. When the matter model is a perfect fluid verifying $p = \speed^2 \rho,$ the (generalized) Kasner solutions 
to the Euler-Einstein equations \eqref{E:EINSTEININTRO}-\eqref{E:UNORMALIZED} are
\begin{align} \label{E:KASNER}
	\gfour_{KAS} & = - dt^2 + \sum_{i=1}^3 t^{2q_i} (dx^i)^2, \qquad p = \bar{p} t^{-2}, \qquad \ufour^{\mu} = \delta_0^{\mu},
		\qquad (t,x) \in (0,\infty) \times \mathbb{T}^3,
\end{align}
where $\delta_{\nu}^{\mu}$ ($\mu, \nu = 0,1,2,3$) is the standard Kronecker delta,
the constants $q_i$ are called the \emph{Kasner exponents}, and $\bar{p} \geq 0$ is a constant denoting the pressure
at $t=1.$ 
The Kasner exponents are constrained by the relations
 \begin{subequations} 
\begin{align} 
	\sum_{i = 1}^3 q_i & = 1, 
	\label{E:KASNERTRACECONDITION} 
	\\
	\sum_{i=1}^3 q_i^2 & = 1 - 2 \speed^{-2} \bar{p}.
	\label{E:KASNERHAMILTONIANCONSTRAINT}
\end{align}
\end{subequations}
\eqref{E:KASNERTRACECONDITION} corresponds to a choice of time coordinate such that
$\SecondFund_{\ a}^a(t,x) \equiv - t^{-1},$ while \eqref{E:KASNERHAMILTONIANCONSTRAINT} is a consequence of 
$\SecondFund_{\ a}^a(t,x) \equiv - t^{-1}$ plus the
Hamiltonian constraint equation \eqref{E:GAUSSINTRO}, which is discussed below. 
Here and throughout, $\SecondFund_{\ j}^i$ denotes the (mixed) second fundamental form of the constant-time
hypersurface $\Sigma_t$ (see \eqref{E:SECONDFUNDDEF}).
It is important to note that it is not possible to have all three $q_i > 0$ in the absence of matter. The solution \eqref{E:BACKGROUNDSOLUTION} is a special case of \eqref{E:KASNER}, \eqref{E:KASNERTRACECONDITION}-\eqref{E:KASNERHAMILTONIANCONSTRAINT} in which 
$\speed = 1,$ $\bar{p} = 1/3,$ and $q_i = 1/3$ for $i=1,2,3.$ It is straightforward to compute that for
any Kasner solution \eqref{E:KASNER} with exponents verifying \eqref{E:KASNERTRACECONDITION} and \eqref{E:KASNERHAMILTONIANCONSTRAINT},
the corresponding Kretschmann scalar 
$|\Riemfour|_{\gfour}^2 := \Riemfour^{\alpha \beta \kappa \lambda} \Riemfour_{\alpha \beta \kappa \lambda}$ 
verifies
\begin{align}
	|\Riemfour_{KAS}|_{\gfour_{KAS}}^2 & = 4 t^{-4}
		\left\lbrace
			\sum_{i=1}^3 q_i^4
			+ \sum_{1 \leq i < j \leq 3} q_i^2 q_j^2
			+ \sum_{i=1}^3 q_i^2
			- 2 \sum_{i=1}^3 q_i^3
		\right\rbrace 
		\\
	& \geq 4 t^{-4} \sum_{1 \leq i < j \leq 3} q_i^2 q_j^2,
		\notag
\end{align}
where $\Riemfour_{KAS}$ denotes the Riemann tensor of $\gfour_{KAS}.$ In particular, unless two of the $q_i$ are equal to $0,$
it follows that $|\Riemfour_{KAS}|_{\gfour_{KAS}}^2$ blows up at the rate $t^{-4}$ as $t \downarrow 0.$ Hence, for
such Kasner solutions, $\lbrace t = 0 \rbrace$ is a true singularity 
beyond which one cannot regularly extend the solution.

The vacuum Kasner solutions (i.e., $\bar{p} = 0$) are famous for the controversial alleged role that they play in approximating highly oscillatory solutions to Einstein's equations near a singularity. Despite the controversy, it is known that the vacuum Kasner solutions are in fact intimately connected to the dynamics of certain \emph{spatially homogeneous} solutions to the Euler-Einstein system
with Bianchi IX symmetry \cite{hR2001}. Under the assumptions of spatial homogeneity and Bianchi IX symmetry, the Euler-Einstein equations reduce to a coupled system of ODEs whose solutions exhibit highly complicated, oscillatory behavior. Specifically, in \cite{hR2001}, Ringstr\"{o}m showed (among many other things) that under the equation of state 
\eqref{E:EOS} with $0 \leq \speed < 1,$ the limit points of generic (i.e., non-Taub) Bianchi IX solutions in the approach towards the singularity must be either \emph{vacuum} Bianchi type I (i.e., vacuum Kasner), vacuum Bianchi type $\mbox{VII}_0$, or
vacuum Bianchi type II.
In particular, Ringstr\"{o}m's work showed that a sub-stiff fluid has a negligible effect on the solutions near the singularity. Furthermore, he showed that almost all such solutions are oscillatory in the sense that there are at least three distinct limit points. Ringstr\"{o}m's work (see also \cite{hR2000}) established for the first time a rigorous picture of the global behavior of Misner's ``mixmaster" solutions\footnote{This is the name Misner gave to the spatially homogeneous vacuum solutions belonging to the symmetry class Bianchi IX.}  
\cite{cM1969}. 

%One of the basic tenets of the BKL conjectures is that \emph{generic} cosmological solutions to Einstein's equations should be locally (in space) well-approximated near their singularities by a sequence of spatially homogeneous metric solutions belonging to a certain Bianchi class. There is debate over over which Bianchi classes should be of the greatest relevance, but many authors have suggested that mixmaster-type behavior should be prevalent (see, for example \cite{mHcU2009}). 

In \emph{contrast} to the oscillatory behavior described above, the solutions that we study in this article behave in 
an approximately monotonic fashion near the singularity. 
As we have mentioned, the presence of the stiff fluid plays an essential role in driving the regular behavior near the singularity. Some of the heuristics connected to the mollifying effect induced by certain kinds of matter have been known for many years. Specifically, Belinskii and Khalatnikov \cite{vBiK1972} gave heuristic arguments suggesting that a scalar field should have a regularizing, oscillatory-suppressing effect on solutions near the singularity. In a later article \cite{jB1978}, Barrow argued that fluids verifying the equation of state $p = \speed^2 \rho$ should induce a similar effect if and only if $\speed = 1;$ he referred to the 
mollifying effect of a stiff fluid as \emph{quiescent cosmology}. Thus, our main result validates the quiescent cosmological picture for spacetimes launched by near-FLRW data. A similar effect was observed in the ODE solutions studied by Ringstr\"{o}m \cite{hR2001}, but only for the equation of state $p = \rho.$ In three spatial dimensions, the monotonicity is roughly associated with the following two conditions: \textbf{i)} all three $q_i > 0;$ $\textbf{ii)}$ the presence of a stiff fluid (with $p > 0$ everywhere), which, as mentioned above, has properties that distinguish it from all other perfect fluids. However, our main theorem only covers the case in which all three $q_i$ are near $1/3.$ In further contrast to the sub-stiff case, the stiff fluid plays an \emph{essential} role (albeit somewhat indirectly) in controlling the dynamics of the solution near the singularity.

\subsection{Initial value problem}
\label{SS:IVP}
%In the prior discussion, we tacitly assumed that the Einstein equations are evolution equations that admit an initial value problem formulation. 
It has been known since the fundamental result of Choquet-Bruhat \cite{CB1952} that the Einstein equations are effectively hyperbolic and can be decomposed into a system of constraint and evolution equations. However, because of their diffeomorphism invariance, the hyperbolic character of the equations becomes apparent only after one makes a choice of a gauge. 
As we have mentioned, we derive our main results in CMC-transported spatial coordinates,
which has been used in various settings
such as the derivation of breakdown criteria for solutions \cite{sKiR2010}, \cite{aS2011}, \cite{qW2012}. 
In this gauge, the Einstein-stiff fluid system comprises constraint equations, wave-like equations for the spatial metric components, first-order hyperbolic equations for the fluid variable components, and an elliptic equation for the lapse function $n.$ We present our choice of
solution variables and the corresponding PDEs that they verify in Sect.~\ref{S:EEINCMC}
(it is convenient to work with renormalized solution variables). 
We normalize the time coordinate $t$ such that data are placed at $t = 1$ 
and the mean curvature of the constant-time hypersurface $\Sigma_t$ is $-\frac{1}{3}t^{-1}.$ That is,
$\SecondFund_{\ a}^a(t,x) \equiv - t^{-1},$ where $\SecondFund_{\ j}^i$ is the (mixed) second fundamental form of $\Sigma_t$
(see \eqref{E:SECONDFUNDDEF}). 
%Our main goal in this article is to study the behavior of solutions as $t \downarrow 0.$

Some spacetimes do not contain any CMC slices \cite{rB1988a}. To show that
near-FLRW spacetimes always contain one, 
we briefly use an alternate gauge for the Einstein equations. Specifically, we use a harmonic map gauge, which is closely connected to the well-known harmonic coordinate gauge (which is also known as the \emph{wave coordinate} gauge). In short, we first sketch a proof of local well-posedness for the Einstein-stiff fluid system in harmonic map gauge for near-FLRW data; this is a standard result - see Prop.~\ref{P:HARMONIC}. We then combine the estimates of the proposition with a modified version of a theorem of Bartnik to deduce the existence of a CMC hypersurface equipped with near-FLRW fields verifying the Einstein constraints; see Prop.~\ref{P:CMCEXISTS}
and Corollary \ref{C:NEARFLRWFIELDSONCMC}. Once we have produced a CMC slice, we immediately switch to CMC-transported spatial coordinates, which we use for all remaining analysis.

Before summarizing our main results and providing further context, we state some basic facts regarding the initial value problem for the Euler-Einstein equations. Our data consist of the three manifold $\Sigma_1' = \mathbb{T}^3$ 
equipped with the tensorfields 
$(\mathring{g}_{ij},\mathring{\SecondFund}_{ij},\mathring{p},\mathring{u}^i).$
Here, $\mathring{g}_{ij}$ is a Riemannian metric,
$\mathring{\SecondFund}_{ij}$ is a symmetric two-tensor, 
$\mathring{p}$ is a function, and
$\mathring{u}^i$ is a vectorfield. If 
$\Sigma_1'$ is a CMC hypersurface with $\mathring{\SecondFund}_{\ a}^a(x) \equiv -1,$ then we write 
``$\Sigma_1$'' instead of ``$\Sigma_1'.$'' A solution launched by the data consists of a spacetime $(\mathbf{M}, \gfour_{\mu \nu})$ 
and an embedding $\Sigma_1' \overset{\iota}{\hookrightarrow}\mathbf{M}$ such that 
$\iota(\Sigma_1')$ is a Cauchy hypersurface in $(\mathbf{M},\gfour_{\mu \nu}),$ and the following additional fields on $\mathbf{M}$: 
a function $p$ and a vectorfield $\ufour^{\mu}$ normalized by
$\gfour_{\alpha \beta}\ufour^{\alpha} \ufour^{\beta} = -1.$ The spacetime fields must verify the equations 
\eqref{E:EINSTEININTRO}-\eqref{E:EOS}. Furthermore, they must be such that $\iota^* \gfour = \mathring{g},$ $\iota^* \kfour = \mathring{\SecondFund},$ $\iota^* p = \mathring{p},$ $\iota^* \ufour = \mathring{u},$ where $\kfour$ is the second fundamental form of $\iota(\Sigma_1')$ 
and $\iota^*$ denotes pullback by $\iota.$ The relation $\iota^* \ufour = \mathring{u}$ denotes a slight abuse of notation; we mean that the one-form that is $\mathring{g}-$dual to $\mathring{u}$ is the pullback by $\iota$ of the one-form that is $\gfour-$dual to $\ufour.$ 
Throughout the article, we will often suppress the embedding and identify $\Sigma_1'$ with $\iota(\Sigma_1').$ 

It is well-known (see Lemma~\ref{L:CONSTRAINTS}) that the initial data are constrained by the \emph{Gauss} and \emph{Codazzi} equations, which take the following form for fluids verifying $p = \speed^2 \rho:$
\begin{subequations}
\begin{align}
	\mathring{R} 
	- \mathring{\SecondFund}_{\ b}^a \mathring{\SecondFund}_{\ a}^b
	+ (\mathring{\SecondFund}_{\ a}^a)^2 
	& = 2 \Tfour(\Nml,\Nml)|_{\Sigma_1'} = 
	2(1 + \speed^{-2}) \mathring{p} (1 + \mathring{u}_a \mathring{u}^a) 
	- 2 \mathring{p}, 
	\label{E:GAUSSINTRO} \\
	\nabla_a \mathring{\SecondFund}_{\ j}^a - \nabla_j \mathring{\SecondFund}_{\ a}^a  
		& = - \Tfour(\Nml,\frac{\partial}{\partial x^j})|_{\Sigma_1'}
		= (1 + \speed^{-2}) \mathring{p}(1 + \mathring{u}_a \mathring{u}^a)^{1/2} \mathring{u}_j. 
		\label{E:CODAZZIINTRO}
\end{align}
\end{subequations}
Above, $\Nml^{\mu}$ denotes the future-directed unit normal to $\Sigma_1'$ in $(\mathbf{M},\gfour_{\mu \nu}),$
$\Tfour(\Nml,\Nml) : = \Tfour_{\alpha \beta} \Nml^{\alpha} \Nml^{\beta},$
$\nabla_i$ denotes the Levi-Civita connection of $\mathring{g}_{ij},$ and $\mathring{R}$ denotes 
the scalar curvature of $\mathring{g}_{ij}.$ The above equations are sometimes referred to as
the \emph{Hamiltonian and momentum constraints}.

A well-known result of Choquet-Bruhat and Geroch \cite{cBgR1969} states that all sufficiently regular data
verifying the constraints launch a unique maximal solution to the Euler-Einstein equations 
\eqref{E:EINSTEININTRO}-\eqref{E:EOS}. This solution is called the \emph{maximal globally hyperbolic development of the data}. More precisely, this maximal solution is unique up to isometry in the class of \emph{globally hyperbolic spacetimes}, which are spacetimes containing a Cauchy hypersurface. 
The work \cite{cBgR1969} is an abstract existence result that does not provide any quantitative
information. Our main results provide, in the case $\speed = 1,$ quantitative information for the ``past-half'' of the maximal globally hyperbolic development of data that are sufficiently close to the FLRW data.

\subsection{Summary of the main results} \label{SS:MAINSUMMARY}
We now summarize our main results. 
See Prop.~\ref{P:CMCEXISTS} and Theorem~\ref{T:BIGBANG} for more detailed statements.
%including additional information about the behavior of the perturbed solutions.

\begin{changemargin}{.25in}{.25in} 
\textbf{Summary of the Main Stable Big Bang Formation Theorem.} 
	The FLRW solution \eqref{E:BACKGROUNDSOLUTION} is a past-globally stable
	singular solution to the Einstein-stiff fluid system \eqref{E:EINSTEININTRO}-\eqref{E:UNORMALIZED}, $p = \rho.$
	More precisely, if the perturbed ``initial'' data 
	$\big(\mathring{g}, \mathring{\SecondFund}, \mathring{p}, \mathring{u} \big)$
	on the manifold $\Sigma_1' = \mathbb{T}^3$ verify the constraints \eqref{E:GAUSSINTRO}-\eqref{E:CODAZZIINTRO}
	and are $\epsilon^2$ close in $H^N$ ($N \geq 8,$ and $\partial_i \mathring{g}_{jk} \in H^N$) 
	to the FLRW data (at time $1$), then for sufficiently small $\epsilon,$
	the maximal globally hyperbolic development of the perturbed data contains a spacelike hypersurface $\Sigma_1$ that is near $\Sigma_1'$ 
	and that has constant 
	mean curvature equal to $-1/3.$	Furthermore, the past of $\Sigma_1$ is foliated by a family of spacelike hypersurfaces 
	$\Sigma_t,$ $t \in (0,1],$ 
	upon which the CMC condition $\SecondFund_{\ a}^a(t,x) \equiv - t^{-1}$ holds. 
	Relative to the time coordinate $t,$ the 
	perturbed solution exists on the manifold-with-boundary $(0,1] \times \mathbb{T}^3$ and remains close 
	to the FLRW solution. Specifically, there exist a collection of 
	transported spatial coordinates\footnote{Technically, the coordinate functions $(x^1,x^2,x^3)$ themselves 
	cannot be globally defined on $\mathbb{T}^3,$ 
	but the vectorfields $\partial_i = \frac{\partial}{\partial x^i}$
	and their dual one-forms $dx^i$ can globally defined on $\mathbb{T}^3$
	in a smooth fashion.}
	$(x^1,x^2,x^3)$ and a large constant $c > 0$ such that
	relative to these coordinates,
	the components of the spacetime metric $\gfour = - n^2 dt^2 + g_{ab} dx^a dx^b,$ 
	the spatial volume form factor $\sqrt{\gdet},$ 
	the components $\SecondFund_{\ j}^i = - \frac{1}{2} n^{-1} g^{ia} \partial_t g_{aj}$
	of the mixed second fundamental form of $\Sigma_t,$
	the pressure $p,$ the components $u^i$ of the $\gfour-$orthogonal projection of
	the four-velocity $\ufour$ onto $\Sigma_t,$ 
	the $\Sigma_t-$normal component $\gfour(\ufour, \Nml) = - (1 + g_{ab} u^a u^b)^{1/2},$
	the spacetime Riemann curvature tensor $\Riemfour,$
	and the spacetime Weyl curvature tensor $\Wfour$
	verify the following convergence estimates as $t \downarrow 0,$ where 
	%$\epsilon^2$ is the size of the deviation of the perturbed data from the FLRW data
	%and 
	$i,j = 1,2,3:$
	\begin{subequations}
	\begin{align}
	\left\| n - 1 \right\|_{C^0} 
		& \lesssim \epsilon t^{4/3 - c \sqrt{\epsilon}},
		\label{E:INTROLAPSEMLIMIT} 
		\\	
	\left\| t^{-1} \sqrt{\gdet} - \upsilon_{Bang} \right \|_{C^0} 
		& \lesssim \epsilon t^{4/3 - c \sqrt{\epsilon}},
		\label{E:INTROVOLFORMLIMIT} 
		\\
	\left\| g_{ia} \left[\mbox{exp} \left(2 \ln t \newsec_{Bang} \right) \right]_{\ j}^a
		- M_{ij}^{Bang} \right\|_{C^0} & \lesssim \epsilon t^{4/3 - c \sqrt{\epsilon}}, 
		 \label{E:INTROLIMITINGMETRICBEHAVIOR} 
		 \\
	\left\| t \SecondFund_{\ j}^i - (\newsec_{Bang})_{\ j}^i \right \|_{C^0}  
		& \lesssim \epsilon t^{4/3 - c \sqrt{\epsilon}}, \label{E:INTROKLIMIT} \\	
	\left\| t^2 p - \newp_{Bang} \right\|_{C^0} & \lesssim \epsilon t^{4/3 - c \sqrt{\epsilon}},
		\label{E:INTROPLIMIT} 
		\\
	\left\| u^i \right \|_{C^0} & \lesssim \epsilon t^{1/3 - c \sqrt{\epsilon}},
		\label{E:INTROULIMIT} 
		\\
	\left \| \gfour(\ufour, \Nml) + 1  \right \|_{C^0} & \lesssim \epsilon t^{4/3 - c \sqrt{\epsilon}},
		\label{E:INTROUNORMALLIMIT} 
		\\
	\left \|t^4 |\Riemfour|_{\gfour}^2 - F_{Bang} \right \|_{C^0} & \lesssim \epsilon t^{4/3 - c \sqrt{\epsilon}}, 
		\label{E:INTROSPACETIMEKRETSCHMANNBLOWUP} 
		\\
		F_{Bang} & := \Big\lbrace 2 (\newsec_{\ b}^a \newsec_{\ a}^b)^2 + 4 \newsec_{\ b}^a \newsec_{\ a}^b
		+ 2 \newsec_{\ b}^a \newsec_{\ c}^b \newsec_{\ d}^c \newsec_{\ a}^d
		+ 8 \newsec_{\ b}^a \newsec_{\ c}^b \newsec_{\ a}^c \Big\rbrace|_{\newsec = \newsec_{Bang}}, 
			\\
		\left\|t^4 |\Wfour|_{\gfour}^2 - (F_{Bang} - \frac{20}{3} \newp_{Bang}^2) \right\|_{C^0} 
		& \lesssim \epsilon t^{2/3 - c \sqrt{\epsilon}}. \label{E:INTROSPACETIMEWEYLBLOWUPINTRO}
	\end{align}
	\end{subequations}
	
	Above, $\upsilon_{Bang},$ $M^{Bang},$ $\newsec_{Bang},$
	$\newp_{Bang},$ $F_{Bang},$
	and $(F_{Bang} - \frac{20}{3} \newp_{Bang}^2)$ are limiting fields on $\mathbb{T}^3$ that are $C \epsilon$
	close in the $C^0$ norm to the corresponding FLRW fields, which are 
	respectively $1,$ $\Euc,$ $- \frac{1}{3} \ID,$ $\frac{1}{3},$ $\frac{20}{27},$ and $0.$
	Here, $\Euc_{ij} = \mbox{diag}(1,1,1)$ denotes the standard Euclidean metric on $\mathbb{T}^3$ 
	and $\ID_{\ j}^i = \mbox{diag}(1,1,1)$ denotes the identity. Furthermore,
	$\mbox{exp} \left(2 \ln t \newsec_{Bang} \right)$ denotes the standard
	matrix exponential of $2 \ln t \newsec_{Bang},$ where $\newsec_{Bang}$ is viewed as
	a $3 \times 3$ matrix with components $(\newsec_{Bang})_{\ j}^i.$
	The FLRW solution has the form \eqref{E:BACKGROUNDSOLUTION} relative to the 
	coordinates $(t,x^1,x^2,x^3).$ In addition, the limiting fields verify the following relations:
	\begin{subequations}
	\begin{align} 
	(\newsec_{Bang})_{\ a}^a & = - 1, 
		\label{E:INTROLIMITINGKTRACE} \\
		2 \newp_{Bang} + (\newsec_{Bang})_{\ b}^a (\newsec_{Bang})_{\ a}^b & = 1.
		\label{E:INTROLIMITINGFIELDCONSTRAINT}
	\end{align}
	\end{subequations}
	
	The top order Sobolev norm $\highnorm{N}$ (see Def.~\ref{D:NORMS}), which measures the deviation of the
	perturbed renormalized solution variables' \emph{components} from the 
	corresponding renormalized FLRW solution variables' \emph{components}, 
	verifies the following bound for $t \in (0,1]:$
	\begin{align} \label{E:HIGHNORMBLOWUPINTRO}
		\highnorm{N}(t) \leq \epsilon t^{- c \sqrt{\epsilon}}.
	\end{align}
	
	Furthermore, the spacetime $\big((0,1] \times \mathbb{T}^3, \gfour \big)$ is past-timelike geodesically incomplete
	and inextendible beyond $t = 0.$ As $t \downarrow 0,$ the $3-$volume of the constant-time hypersurfaces $\Sigma_t$ 
	collapses to $0,$ the pressure $p$ blows-up like $t^{-2},$ and the spacetime Kretschmann scalar $|\Riemfour|_{\gfour}^2$ 
	blows-up like $t^{-4}.$ Finally, the curvature singularity at $t=0$ is dominated by
	the Ricci components of the Riemann curvature tensor (the ratio $\frac{|\Wfour|_{\gfour}^2}{|\Riemfour|_{\gfour}^2}$
	remains order $\epsilon$ throughout the evolution).
	
\end{changemargin}

\hfill $\qed$

\begin{remark} \label{R:NODECAYTOFLRW}
	Note that the perturbed solutions do not generally converge to the FLRW solution
	as $t \downarrow 0,$ as is shown by near-FLRW isotropic Kasner solutions.
\end{remark}

\begin{remark} \label{R:INFINITEDIMENSIONAL}
	The estimates \eqref{E:INTROVOLFORMLIMIT}-\eqref{E:INTROPLIMIT} allow for
	an \emph{infinite dimensional family} of possible ``end states'' at the Big Bang $\lbrace t = 0 \rbrace.$
	For example, the limiting field $\newp_{Bang}$ 
	can in principle be any member of an \emph{open} set of functions.
\end{remark}

\begin{remark} \label{R:VTD}
	The proof of the convergence results
	is based on showing that the lower-order time derivative terms dominate the lower-order spatial
	derivative terms for $t$ near $0.$ That is, for $t$ near $0,$
	the Einstein equations are well-approximated by
	truncated equations formed by discarding spatial derivatives.
	Similar behavior had been observed for the Einstein-vacuum
	equations in the polarized Gowdy\footnote{
	The Gowdy solutions are a special subclass of the $\mathbb{T}^2-$symmetric solutions; they are 
	characterized by the vanishing of the \emph{twist constants} 
	$(\gfour^{-1})^{\mu \mu'} \pmb{\epsilon}_{\alpha \beta \mu \nu} \mathbf{X}^{\alpha} \mathbf{Y}^{\beta} \Dfour_{\mu'} \mathbf{X}^{\nu}$
	and 
	$(\gfour^{-1})^{\mu \mu'} \pmb{\epsilon}_{\alpha \beta \mu \nu} \mathbf{X}^{\alpha} \mathbf{Y}^{\beta} \Dfour_{\mu'} \mathbf{Y}^{\nu},$
	where $\pmb{\epsilon}$ is the volume form of $\gfour$ and $\mathbf{X}$ and $\mathbf{Y}$ are the Killing fields corresponding to the two 	symmetries. The 
	polarized Gowdy solutions are defined to be those Gowdy solutions such that $\gfour_{\alpha \beta}\mathbf{X}^{\alpha} 	\mathbf{Y}^{\beta} = 0.$
	} 
	class by Isenberg and Moncrief \cite{jIvM1990}.
	In general relativity, 
	going back to the work \cite{dEeLrS1972},
	the truncated equations are
	sometimes called the \emph{velocity term dominated} (VTD) equations.
\end{remark}

\subsection{Additional connections to previous work}
\label{SS:MORECONNECTIONS}
We now discuss some additional connections between 
the present work and prior results.

\subsubsection{Global solutions without symmetry assumptions} \label{SSS:GLOBALSOLUTIONSINGR}
There are only a modest number of prior results in which solutions to the Einstein equations
corresponding to an open set of data (without symmetry assumptions) have been understood in full 
detail. By ``full detail,'' we mean that the basic qualitative and quantitative features of
the data's maximal globally hyperbolic development (or at least a past or future half of them) 
have been exposed. The existing examples can roughly be grouped into two classes. The first
class was birthed by Christodoulou-Klainerman's groundbreaking proof of the stability of 
Minkowski spacetime \cite{dCsK1993} as a solution to the Einstein-vacuum equations. We remark that 
the corresponding near-Minkowski spacetime manifolds are diffeomorphic to $\mathbb{R}^4$ and are therefore not cosmological. 
The main theorem roughly states that a class of  asymptotically flat near-vacuum-Euclidean initial data sets launch maximal globally hyperbolic developments that are geodesically complete and that look in the large like the Minkowski spacetime. 
A second proof, which relied on wave coordinates and allowed for the presence of a (small) scalar field, 
was given by Lindblad-Rodnianski in \cite{hLiR2005,hLiR2010}. Various extensions of these results can be 
found in \cite{lBnZ2009,jL2009,jS2010b}. The analysis in all of these proofs is 
based on the \emph{dispersive} decay properties of solutions to wave-like equations on $\mathbb{R}^{1+3},$ 
together with careful studies of the special structure of the nonlinearities present in the 
gauges that were employed.

The second class concerns the future stability of a class of cosmological solutions to Einstein's equations when a positive
cosmological constant $\Lambda$ (or alternatively, a suitable matter model that generates a similar effect) is added to the equations. 
Specifically, one includes the additional term $\Lambda \gfour_{\mu \nu}$ on the left-hand side of \eqref{E:EINSTEININTRO}.
This additional term creates accelerated expansion in certain solutions, which can in turn stabilize them.
The second class was brought into existence by Friedrich's work on the stability of the de Sitter spacetime \cite{hF1986a}, which 
is a solution to the Einstein-vacuum equations with $\Lambda > 0.$
Related future-stability results for solutions featuring various matter models can be found in 
\cite{hF1991,cLjK2013,hR2008,hR2009,iRjS2013,jS2012,jS2013}. The 
analysis in \cite{hF1986a,hF1991,cLjK2013} is based on applications of the conformal method, which was developed by Friedrich. This method is discussed in a bit more detail in Sect.~\ref{SSS:WEYLCURVATUREHYP}. In contrast, 
in \cite{hR2008,hR2009,iRjS2013,jS2012,jS2013} the analysis is based on the \emph{dissipative} nature of the kinds wave equations that are generated by expanding spacetimes in a well-chosen harmonic-type coordinate system. That is, the spacetime expansion 
and the term $\Lambda \gfour_{\mu \nu}$ lead to the presence of friction-like terms in the PDEs, and the friction induces monotonicity in the solutions. The analysis behind our main results has more features in common with this framework than with the analysis of the 
stability of Minkowski spacetime or the conformal method, though there are some key novel features in the present work.
In particular, we stress that our main results are not based on dispersive effects, but rather on estimates related to
monotonicity and time integrability.

\subsubsection{Global solutions based on a different form of monotonicity} \label{SSS:GLOBALMONOTONICITY}
Fisher and Moncrief discovered a form of monotonicity, quite different than 
the one in this article, which holds for Einstein-vacuum solutions in some regimes.
More precisely, they constructed a reduced Hamiltonian description of the 
Einstein-vacuum equations \cite{aFvM1994,aFvM2000a,aFvM2000b,aFvM2001,aFvM1997,aFvM2002}
that applied to a class of solutions foliated by CMC hypersurfaces $\Sigma_t.$
The Hamiltonian is the volume functional 
of the $\Sigma_t,$ and they showed 
that it is a monotonic quantity for solutions to the reduced equations.
Andersson and Moncrief used this monotonicity 
to prove a global stability result \cite{lAvM2004}
(see also \cite{mR2009}) that does not fit neatly into either of  
the two classes described in Section \ref{SSS:GLOBALSOLUTIONSINGR}. 
Specifically, they proved a global stability
result in the expanding direction for a compactified version of \emph{vacuum} FLRW-type solutions 
whose spatial slices are hyperboloidal (i.e., they have constant negative sectional curvature).
They showed that the perturbed spacetimes are future geodesically complete and, in $3$ spatial
dimensions, that they decay towards the background solution. 
In addition to using the monotonicity of Fisher-Moncrief, their
proof also relied, in the case of $3$ spatial dimensions, on Mostow's rigidity theorem.

Andersson and Moncrief performed their analysis in CMC-spatial harmonic coordinates. 
They imposed the spatial harmonic coordinate condition \cite{lAvM2003} to ``reduce'' 
the Ricci tensor $R_{ij}$ of the first fundamental form
$g$ (of $\Sigma_t$) to an elliptic operator acting on $g.$ 
That is, in spatial harmonic coordinates, 
$R_{ij} = - \frac{1}{2} g^{ab} \partial_a \partial_b g_{ij} + f_{ij}(g,\partial g).$
The spatial harmonic coordinate condition, though it may have advantages in certain contexts,
introduces additional complications into the analysis. The complications arise from the necessity of including a non-zero 
$\Sigma_t-$tangent
``shift vector'' $X^i$ in the spacetime metric $\gfour:$ $\gfour = - n^2 dt^2 + g_{ab}(dx^a + X^a dt)(dx^b + X^b dt).$ To 
enforce the spatial harmonic coordinate condition, the components
$X^i$ must verify a system of elliptic PDEs that are coupled to the other solution variables.

As we discuss in Theorem~\ref{T:LOCAL}, the spatial harmonic coordinate condition is not necessary 
for proving a local well-posedness result; one can instead use transported spatial coordinates. 
In transported spatial coordinates, the additional terms appearing in the expression for the $R_{ij}$ are of the form $\frac{1}{2}(\partial_i \Gamma_j + \partial_j \Gamma_i),$ where $\Gamma_i$ is a contracted Christoffel symbol of the $3-$metric $g.$
In the main energy identity that one encounters during the derivation of a priori $L^2$ estimates for $\SecondFund_{\ j}^i$ and 
$\partial_i g_{jk},$ additional terms are generated by the contracted Christoffel symbols and are roughly of the form 
$\int_{\Sigma_t} g^{ij} \SecondFund_{\ i}^a \partial_a \Gamma_j \, dx.$ Since $\Gamma_i = \Gamma_i(g,\partial g),$
this spatial integral appears to depend on too many derivatives of $g$ to close a local well-posedness argument. However, after integration by parts, this integral can be replaced with $- \int_{\Sigma_t} g^{ij} (\partial_a \SecondFund_{\ i}^a) \Gamma_j \, dx$ plus lower-order terms, and the constraint equation \eqref{E:CODAZZIINTRO} allows us to replace $\partial_a \SecondFund_{\ i}^a$ with lower-order terms. \emph{The energy estimates for the nonlinear system therefore close.} 
Remarkably, as we described in \cite{iRjS2014a},
we have not seen this observation made in the literature. However, a complete proof of local well-posedness requires that one derive estimates for a linearized version of the equations. Linearization may destroy some of the structure of the system, which may invalidate the energy estimate procedure just described. In Theorem~\ref{T:LOCAL}, we recall a standard way of circumventing this difficulty while still using transported spatial coordinates.

When deriving energy estimates for the Einstein-stiff fluid system, we use a differential analog of the integration by parts argument described in the previous paragraph. Specifically, we account for this integration by parts by including the terms $- t^{1/3} \big[1 + t^{4/3} \newlap\big] (\newg^{-1})^{ij} \dot{\Gamma}_a \dot{\freenewsec}_{\ i}^a$ and $-t^{1/3} \big[1 + t^{4/3} \newlap\big] (\newg^{-1})^{ia} \dot{\Gamma}_a \dot{\freenewsec}_{\ i}^j$ on the right-hand side of the definition \eqref{E:METRICCURRENTJ} for $\dot{\mathbf{J}}_{(metric)}^j.$ See Sect.~\ref{SSS:KEYMONOTONICITYINTRO} for a detailed discussion of the role played by the spacetime vectorfields $\dot{\mathbf{J}}_{(metric)}^{\mu}$ in our proof of stable singularity formation. Given these observations, it would be interesting to see if the future stability results of \cite{lAvM2004} can be proved directly in CMC-transported spatial coordinates.

\subsubsection{Prescribed asymptotics and Fuchsian methods near the singularity} \label{SSS:PRESCRIBEDASYMPTOTICS}
It is generally difficult to obtain a detailed picture of the asymptotics of a solution
launched by Cauchy data, especially near a singularity. 
%In fact, based on our current PDE techniques,
%there is no hope of being able to carry out such an analysis in anything resembling a fully satisfactory
%degree of generality. 
An alternative approach is to \emph{prescribe} the asymptotic behavior 
and to then try to \emph{construct} solutions that have the given asymptotics. This 
can be viewed as a form of ``putting data on the singularity.'' 
In the interest of obtaining a picture of the behavior of ``general'' solutions,
it is desirable to show that one can carry out such a procedure for a family of prescribed asymptotics
that depend on the ``maximum number'' of degrees of freedom in the Einstein initial data. 
However, even if one can achieve the maximum number, 
one should be careful in interpreting the results: it may be that the map
from the space of asymptotics to the space of solutions is highly degenerate; \emph{what appears to be a 
``general class of solutions'' from the point of view of function counting could in principle fail 
to be a large class from other more physically relevant points of view.} In particular, it could
happen that the ``general class of solutions'' that one constructs in this fashion is,
for example, nowhere dense (relative to a reasonable topology on the function spaces).
%{\color{red} INSERT}

In \cite{lAaR2001}, Andersson and Rendall carried out a prescribed asymptotics-type construction for
solutions to the Einstein-scalar field and the Einstein-stiff fluid systems. They constructed a family of solutions that
are well-approximated by solutions to a VTD (see Remark \ref{R:VTD}) system,
and one can view the VTD solutions as the prescribed asymptotics.
They formed the VTD system by simply discarding all spatial derivative terms in the Einstein-matter equations.
The solutions in \cite{lAaR2001} 
have a Big Bang-type singularity, and a neighborhood of the singularity 
can be covered by Gaussian coordinates 
such that the singularity is synchronized at $\lbrace t = 0 \rbrace.$ 
The family depends on the same number of free functions as do the data for the general space of solutions, and no symmetry assumptions were made. However, the construction only produced solutions that are spatially analytic. 
From the physical point of view (and in particular from the point of view of finite speed of propagation), the analyticity restriction 
is undesirable, for analytic functions on a connected domain are completely determined by their behavior at a single point. 
In fact, the results derived in this article were partially inspired by \cite{lAaR2001}, where they were stated as open problems. The basic idea of the proof in \cite{lAaR2001} was to first construct a large family of spatially analytic solutions to the VTD system, and to then expand a spatially analytic
solution to the Einstein-matter equation as a VTD solution plus error terms. The error terms were shown to verify a system of Fuchsian PDEs, a general theory of which (based on the earlier ideas of \cite{mBcG1977}) has been developed by Kichesnassamy 
(see e.g. \cite{sKi1996a}, \cite{sKi1996b}, \cite{sKi1996c}). Roughly speaking, a Fuchsian PDE is one of the form
\begin{align} \label{E:FUCHSIAN}
	t \partial_t u + A(x)u = F(t,x,u,\partial_x u).
\end{align}
Above, $u$ is the array of unknowns and $A(x)$ is a matrix-valued function that has to verify certain technical conditions. In \cite{lAaR2001}, Andersson-Rendall used a slight extension of a Fuchsian existence theorem proved in \cite{sKaR1998} to conclude that the Fuchsian system verified by the error terms has a unique solution that is analytic in $x$ and that tends to $0$ as $t \downarrow 0.$ In particular, they showed that the solution to Einstein's equations converges to a solution of the VTD system (which does not depend on spatial derivatives).
We also remark that prior to \cite{lAaR2001}, similar results had been derived in the absence of matter under various symmetry assumptions \cite{sKjI1999}, \cite{sKaR1998}, including a result of Rendall that did not require the assumption of spatial analyticity \cite{aR2000b}. An alternative proof of the latter result invoking the use of second order Fuchsian techniques has recently been provided in \cite{fBpL2010c}. The 
results of \cite{lAaR2001} were extended to higher dimensions and other matter models in \cite{tDmHrAmW2002}.
Related results have been obtained in \cite{kA2000a}, \cite{yCBjIvM2004}, \cite{fS2002}.

\subsubsection{Weyl curvature hypothesis and isotropic singularities} \label{SSS:WEYLCURVATUREHYP}
Another approach to studying Big Bang-type singularities involves devising a 
formulation of Einstein's equations that allows one to solve a Cauchy problem
with data given on the singular hypersurface $\lbrace t = 0 \rbrace$ itself; see e.g. 
\cite{kApT1999a}, \cite{cCpN1998}, \cite{rN1993a}, \cite{rN1993b}, 
\cite{pT1990}, \cite{pT1991}, \cite{pT2002}. 
The spacetimes launched by this procedure are said to contain a \emph{conformal singularity} or
an \emph{isotropic singularity}. This approach
is motivated by Penrose's \emph{Weyl curvature hypothesis} \cite{rP1979}, which posits that in cosmological solutions,
the Weyl tensor should tend to zero as the initial singularity is approached. 
%One could also contemplate replacing this hypothesis with the following weakened conjecture: the Weyl tensor remains small in some sense as the initial singularity is approached. Roughly speaking, the ideas motivating Penrose's version of the Weyl curvature hypothesis were: \textbf{i)} there should be an entropy associated with gravitational fields; \textbf{ii)} the universe should have been in a low entropy state near the initial singularity; and \textbf{iii)} gravitational entropy should be associated with the Weyl tensor. Our main result shows that a quite weak version of the Weyl curvature hypothesis
%holds for the spacetimes that we consider: 
In the present work, we show that for near-FLRW solutions, 
the blow-up of the spacetime Riemann curvature is dominated by the Ricci
components, with the Weyl components making only a tiny correction. More precisely, 
even though $|\Ricfour|_{\gfour}^2$ and $|\Wfour|_{\gfour}^2$
are both allowed to diverge as $t \downarrow 0,$ we show that
\begin{align} \label{E:WEYLSMALLINTRO}
	\sup_{t \in (0,1]} 
	\left|\frac{|\Wfour|_{\gfour}^2}
		{|\Ricfour|_{\gfour}^2}\right|
		 \lesssim \epsilon,
\end{align}
where $\epsilon^2$ is the size of the deviation of the perturbed initial data from the FLRW data. 
%The ratio \eqref{E:WEYLSMALLINTRO} was suggested in \cite{jAjW1984} and \cite{sGjW1985} as an alternative to Penrose's measure of gravitational entropy. 
We note that the estimate \eqref{E:WEYLSMALLINTRO} is almost saturated by 
near-FLRW Kasner solutions (which verify \eqref{E:WEYLSMALLINTRO} with $\epsilon$ replaced by $\epsilon^2$).
%We remark that Gr{\o}n and Hervik have argued \cite{oGsH2001} in favor of replacing
%the ratio on the left-hand side of \eqref{E:WEYLSMALLINTRO}
%with the ratio $\sqrt{\gdet} \left|\frac{|\Wfour|_{\gfour}^2}{|\Ricfour|_{\gfour}^2}\right|^{1/2};$ 
%they argue that this ratio is better able to mimic the desired behavior of gravitational entropy. 
%Here, $\sqrt{\gdet}$ denotes the volume form factor of a constant time hypersurface $\Sigma_t.$ Now the estimates 
%of our main theorem imply that for the near-FLRW solutions under consideration, we have
%\begin{align} \label{E:WEYLSMALLINTROALT}
%	\sqrt{\gdet} \left|\frac{|\Wfour|_{\gfour}^2}
%		{|\Ricfour|_{\gfour}^2} \right|^{1/2}
%		 \lesssim \epsilon^{1/2} t.
%\end{align}
%Hence, this quantity indeed converges to $0$ as $t \downarrow 0.$

The established framework for studying isotropic singularities is essentially an extension of Friedrich's 
\emph{conformal method} (see e.g. \cite{hF1986}, \cite{hF2002}). Roughly speaking, this 
corresponds to studying a rescaled metric 
\begin{align}
	\gfour = \Omega^2 \hat{\gfour},
\end{align}
where $\gfour$ is the physical spacetime metric of interest, $\Omega$ is a conformal scaling factor, and
$\hat{\gfour}$ is the rescaled ``unphysical'' spacetime metric; analogous rescalings are carried out for the other field variables.
In the context of this article, one may roughly think of $\lbrace \Omega = 0 \rbrace$ as corresponding to a Big Bang singularity.
The main point is that even though $\gfour$ degenerates along $\lbrace \Omega = 0 \rbrace,$ it may be possible that
$\hat{\gfour}$ remains a regular Lorentzian metric. For example, through the change of variables $\tau = \frac{3}{2} t^{2/3},$ the FLRW metric \eqref{E:BACKGROUNDSOLUTION} can be seen to be conformally equivalent to the standard Minkowski metric on $(0, \infty) \times \mathbb{T}^3:$
\begin{align}
	\widetilde{\gfour} =  \frac{2}{3} \tau \overbrace{\big\lbrace -d \tau^2 + \sum_{i=1}^3 (dx^i)^2 \big\rbrace}^{\mbox{regular when $\tau = 0$}}.
\end{align}
Thus, if Einstein's equations are reformulated in terms of
$\Omega$ and $\hat{\gfour}$ (plus some gauge conditions that select a choice of $\Omega$), one may hope to prove that
$\hat{\gfour}$ remains a regular Lorentzian metric through $\lbrace \Omega = 0 \rbrace.$ When
$\hat{\gfour}$ does remain regular, one can deduce sharp information
about the behavior of $\gfour$ up to the set $\lbrace \Omega = 0 \rbrace.$  
The conformal method has proven to be fruitful for studying matter models such as Maxwell fields, Yang-Mills fields, and perfect fluids with the equation of state $p = (1/3) \rho.$ An important common feature of these matter models,
which seems to be necessary for applying the conformal method, is that they have trace-free energy-momentum tensors. For these matter models, when a positive cosmological constant is included in the Einstein equations, the conformal method 
has been applied to derive global future stability results \cite{hF1991}, \cite{cLjK2013} for a class of rapidly expanding 
``de-Sitter-like'' half-spacetimes. 

In contrast, for the kinds of half-spacetimes we are considering, the conformal method does not seem to allow one to deduce a 
true stability result. On the one hand, the aforementioned works \cite{rN1993a}, \cite{rN1993b}, \cite{pT1990}, \cite{pT1991}, \cite{pT2002} show that one can locally solve the ``singular Cauchy problem'' for the rescaled variables 
$\hat{\gfour}$ etc. with suitable ``conformal initial data" given at the singular hypersurface; the solution is obtained by applying the Fuchsian techniques described above. On the other hand,
there are fewer degrees of freedom in the conformal data compared to the full Einstein data (see, for example, the discussion in 
\cite[Section 6.1]{aR2005b}). Thus, the map from the conformal data to the space of full solutions is far from onto. In fact, even the near-FLRW Kasner solutions \eqref{E:KASNER} do not exhibit a nice conformal structure near their singularities.

\subsection{Choice of coordinates and field variables for analyzing the Einstein-stiff fluid system} \label{SS:CHOICEOFVARIABLES}
As we mentioned above, our approximate monotonicity inequalities are visible
relative to CMC-transported spatial coordinates, which we describe in detail in Sect.~\ref{S:EEINCMC}.
In these coordinates, the spacetime metric $\gfour$ is decomposed into the lapse function $n$ and the Riemannian $3-$metric $g$ on
$\Sigma_t$ as follows:
\begin{align}
	\gfour & = - n^2 dt^2 + g_{ab} dx^a dx^b. \label{E:GFOURDECOMPINTRO} 
\end{align}
Above, $\lbrace dx^i \rbrace_{i=1}^3$ are the one-forms corresponding to local
coordinates on $\mathbb{T}^3$ and $t$ is a time function on $\mathbf{M}.$ The coordinates are constructed such that
$\SecondFund_{\ a}^a \equiv -t^{-1}$ 
along $\Sigma_t := \lbrace (s,x) \in (0,1] \times \mathbb{T}^3 \ | \ s = t \rbrace,$ where 
$\SecondFund_{\ j}^i = - \frac{1}{2} n^{-1} g^{ia} \partial_t g_{aj}$ is the mixed second fundamental form of $\Sigma_t.$
In order to enforce this condition, the lapse must verify the following elliptic PDE (see \eqref{E:LAPSE})
for stiff fluid matter:
\begin{align} \label{E:LAPSEINTRO}
	g^{ab}\nabla_a \nabla_b (n - 1) 
		& = (n - 1) \Big\lbrace R + (\SecondFund_{\ a}^a)^2 - 2p u_a u^a \Big\rbrace 
			+ R - 2p u_a u^a,
\end{align}
where $\nabla$ is the Levi-Civita connection of $g$ and
$R$ is its scalar curvature. The fluid velocity is decomposed as
\begin{align} \label{E:UFOURDECOMPINTRO}
	\ufour = (1 + u_a u^a)^{1/2} \Nml + u^a \partial_a,
\end{align}
where $\Nml = n^{-1} \partial_t$ is the future-directed normal to $\Sigma_t.$ 
The factor $(1 + u_a u^a)^{1/2}$ is a consequence of the normalization condition \eqref{E:UNORMALIZED}. 

Our most important gauge choice is the CMC time coordinate and the corresponding lapse PDE \eqref{E:LAPSEINTRO}:
since Einstein's equations are fundamentally hyperbolic, \emph{the only conceivable way of synchronizing the singularity across spatial slices 
%(and thereby circumventing the property of finite speed of propagation) 
is to construct a time coordinate by invoking a gauge that involves an infinite speed of propagation}, such as the elliptic PDE \eqref{E:LAPSEINTRO}. The important issue of constructing an ``initial'' hypersurface of constant mean curvature $-1/3$ is addressed in Prop.~\ref{P:CMCEXISTS}. Alternatively, 
singularity synchronization could be achieved with the help of the parabolic lapse gauges described in Sect.~\ref{SS:MONOTONICITYEXPLANATION},
although we do not use these gauges in the present article.

%We now explain our choice of field variables, which we have selected to make the analysis of the 
%nonlinearities as transparent as possible. Relative to the coordinates of \eqref{E:BACKGROUNDSOLUTION}, 
%the components of the FLRW solution variables decay/grow like powers of $t.$ 
To analyze perturbed solutions, it is convenient to 
introduce the following renormalized solution variables, 
which factor out the $t-$behavior of the FLRW solution.
%and vanish at it. 
%Hence, much of our analysis is stated in terms of the following renormalized variables.

\begin{definition}[\textbf{Renormalized solution variables}] \label{D:RESCALEDVAR}
We define ($i,j,k = 1,2,3$)
\begin{subequations}
\begin{align}
	\newg_{ij} &:= t^{-2/3} g_{ij}, & (\newg^{-1})^{ij} & = t^{2/3} g^{ij}, & \sqrt{\gbigdet} & = t^{-1} \sqrt{\gdet}, \\
	\upgamma_{j \ k}^{\ i} & := g^{ai}\partial_j g_{ak} = (\newg^{-1})^{ai}\partial_j \newg_{ak}, &&&& \label{E:LITTLEGAMMA} \\	
	\freenewsec_{\ j}^i &:= t \SecondFund_{\ j}^i + \frac{1}{3} \ID_{\ j}^i, &&&& \\
	\newu^i &:= t^{-1/3} u^i, & \adjustednewp & := t^2 p - \frac{1}{3}, && \\
	\newlap &:= t^{-4/3} (n - 1), &\dlap_i & :=  t^{-2/3} \partial_i n. &&
\end{align}
\end{subequations}
Above, $\ID_{\ j}^i = \mbox{diag}(1,1,1)$ denotes the identity transformation.
\end{definition}

%\begin{remark}
%	Note that the quantities
%	$\upgamma_{j \ k}^{\ i}$ can roughly be viewed (from the point of view of structure and number of derivatives) as connection 
%	coefficients 	of $g_{ij}$ and are therefore not coordinate invariant. 
%\end{remark}

\begin{remark}
	We will never implicitly lower and raise
	indices with the renormalized metric $\newg$ and its inverse $\newg^{-1};$ 
	we will explicitly indicate all factors of $\newg$ or $\newg^{-1}.$
\end{remark}

Note that the CMC condition $\SecondFund_{\ a}^a(t,x) \equiv - t^{-1}$ is equivalent to $\freenewsec_{\ a}^a = 0.$
All of the variables except for $\sqrt{\gbigdet}$ and $\upgamma$ are simply time-rescaled/shifted 
versions of the original solution variables.
We introduced the volume form variable $\sqrt{\gbigdet}$ because it satisfies an evolution equation that
has a very favorable structure compared to the evolution equations verified by the $\newg_{ij}.$  
As we will see, the quantities $n - 1$ and $\partial_i n$ have different asymptotic (in time) properties as $t \downarrow 0.$ 
This difference is important, and thus 
we have chosen to replace the lapse with two variables 
$\newlap$ and $\dlap_i.$ We introduced the variable $\upgamma_{j \ k}^{\ i}$
as a new unknown in place of 
$\partial_i g_{jk}.$ As we will see, the evolution equation verified by
$\upgamma_{j \ k}^{\ i}$ has a favorable structure. 
Even though $\upgamma_{j \ k}^{\ i}$ roughly has the structure of a connection coefficient and
is therefore not an invariant quantity, for the purposes of analysis, 
we choose to view $\upgamma_{j \ k}^{\ i}$ as a tensorfield that 
happens to have the form $\upgamma_{j \ k}^{\ i} = (\newg^{-1})^{ai}\partial_j \newg_{ak}$ 
relative to our CMC-transported coordinates.

For the FLRW solution \eqref{E:BACKGROUNDSOLUTION}, 
we have $\newg_{ij} = \Euc_{ij},$ $(\newg^{-1})^{ij} = (\Euc^{-1})^{ij},$ $\sqrt{\gbigdet} = 1,$ 
$\freenewsec_{\ j}^i = 0,$ $\newu^i = 0,$ $\adjustednewp = 0,$ $\newlap = 0,$ and $\dlap_i = 0,$ where $\Euc_{ij} = \mbox{diag}(1,1,1)$ denotes the standard Euclidean metric on $\mathbb{T}^3.$

We now provide a brief preview of the behavior of the perturbed solutions: 
for perturbed data,
the renormalized variable components $\freenewsec_{\ j}^i,$ $\sqrt{\gbigdet},$ and $\adjustednewp$
remain uniformly close to the corresponding FLRW components, 
while the remaining variable components are allowed to deviate at worst like 
$\mbox{small amplitude} \times t^{- c \sqrt{\epsilon}}$ as $t \downarrow 0.$ 
%This deviation can be viewed as an analog of the $\epsilon |\ln t|$ term in inequality \eqref{E:APRIORILINFINITYBOUND} of the model %problem.
Similar estimates hold for the lower-order derivatives of the solution variables.

%\subsection{The rescaled equations and their nonlinearities}
% A fundamental aspect of our proof is the provision of a detailed analysis of the nonlinearities that appear in the equations. In order to take full advantage of the structure of the equations, when deriving the constraint and evolution equations, we make one further adjustment to the variables $\newsec$ and $\newp$ from above.
%Specifically, we derive equations for the renormalized quantities
%$\big[\sqrt{\gbigdet} (\newsec_{\ j}^i + \frac{1}{3} \ID_{\ j}^i) \big]$
%and $\big[(\sqrt{\gbigdet})^2 (\newp - \frac{1}{3})\big];$ see
%\eqref{E:MOMENTUMCONSTRAINT}, \eqref{E:MOMENTUMCONSTRAINTRAISED},
%\eqref{E:RESCALEDKEVOLUTION}, and \eqref{E:RESCALEDPEVOLUTION}. These quantities 
%are $0$ for the FLRW solution and furthermore, the factors $\sqrt{\gbigdet}$ and $(\sqrt{\gbigdet})^2$ 
%lead to some important cancellation in the evolution equations, thanks to the
%identity $\partial_t \ln \sqrt{\gbigdet} = t^{1/3} \newlap$ [see equation \eqref{E:LOGVOLFORMEVOLUTION}].

\subsection{A summary of the analysis}
We now summarize the main ideas behind the proof of our main results.

\subsubsection{The top level picture}
\label{SSS:TOPLEVELPICTURE}
The proof is based on a long bootstrap argument that yields a priori estimates for the total solution energies
\begin{align} \label{E:TOTALENERGYINTRO}
	\totalenergy{\smallparameter_*}{M}^2 & := \smallparameter_* \metricenergy{M}^2 + \fluidenergy{M}^2,
\end{align}
which shows in particular that they remain finite for $t \in (0,1].$ From the a priori estimates and a 
standard continuation principle (see Theorem~\ref{T:LOCAL}), 
we conclude that the perturbed solution exists on the slab $(0,1] \times \mathbb{T}^3.$
We stress that \emph{deriving a priori estimates for the $\totalenergy{\smallparameter_*}{M}$ 
is the main step in the proof of stable singularity formation.}
Above, $\metricenergy{M}$ is a metric energy, $\fluidenergy{M}$ is a fluid energy, $\smallparameter_* > 0$
is a small positive constant that we choose in Sect.~\ref{S:FUNDAMENATLENERGYINEQUALITIES}, and $0 \leq M \leq N.$
Here and throughout most of the article, $N \geq 8$ denotes an integer representing 
the number of derivatives we need to close our estimates.

The details of the energies \eqref{E:TOTALENERGYINTRO} 
(see Def.~\ref{D:METRICANDFLUIDENERGIES}) do not concern us at the present; 
they are order $M$
$(0 \leq M \leq N)$ Sobolev-type energies that naturally arise from integration by parts identities, 
and $\metricenergy{M}$ and $\fluidenergy{M}$ respectively control the derivatives of the metric and the fluid. 
In order to derive a priori estimates for the energies, we make bootstrap assumptions for three solution-controlling norms:
\[
	\highnorm{N}(t), 
	\lowkinnorm{N-3}(t), 
	\mbox{and \ }
	\lowpotnorm{N-4}(t).
\] 
The norms control the distance of the renormalized solution variables of Def.~\ref{D:RESCALEDVAR}
from their FLRW background values \emph{with various $t-$weights hiding in the definitions of the norms.}
See Sect.~\ref{S:NORMSANDCURRENTS} for their definitions.

The norms $\highnorm{M}(t)$ are built out of Sobolev norms 
of the \emph{components} of the renormalized solution variables relative to the transported 
spatial coordinate frame. We stress that the 
$\highnorm{M}(t)$ are distinct from the $\totalenergy{\smallparameter_*}{M}(t).$ 
We introduce the norms $\highnorm{M}(t)$ because our energies $\totalenergy{\smallparameter_*}{M}(t)$
are quasilinear and their control over the frame components can
degenerate as the renormalized metric $\newg(t,x)$ degenerates. 
Hence, for the purposes of analysis and Sobolev embedding, it is convenient to work with 
norms $\highnorm{M}(t),$ whose coerciveness is solution-independent. 
The main reason that we choose to measure the size of the solution variables' frame components is:
our derivation of strong estimates (see Sect.~\ref{SSS:STRONGESTIMATESINTRO}) for the lower-order 
derivatives is based on an analysis of frame components.
Another reason is that our proof of the existence of the limiting profiles
$M_{ij}^{Bang}(x),$ $(\newsec_{Bang})_{\ j}^i(x),$ etc.
(see Sect.~\ref{SS:MAINSUMMARY} and Theorem~\ref{T:BIGBANG})
is also based on an analysis of frame components.
We connect the coerciveness of 
the $\totalenergy{\smallparameter_*}{M}$ to the coerciveness of 
the $\highnorm{M}$ in Sect.~\ref{S:COMPARISON}. 

The norms $\lowkinnorm{N-3}(t)$ and $\lowpotnorm{N-4}(t)$ are built out of $C^M-$type norms of the 
\emph{components} of various renormalized solution variables relative to the transported spatial coordinate frame. 
The norm $\lowkinnorm{N-3}(t)$ controls the ``kinetic'' variables and 
the norm $\lowpotnorm{N-4}(t)$ controls the ``potential'' variables 
(see Sect.~\ref{SS:KINETICTERMSDOMINATE} for additional discussion concerning the kinetic and potential variables).
We introduce $\lowkinnorm{N-3}(t)$ and $\lowpotnorm{N-4}(t)$ in order to take advantage of the strong estimates 
verified by the solution's lower-order derivatives. 
More precisely, the norms $\lowkinnorm{N-3}(t)$ and $\lowpotnorm{N-4}(t)$
control the renormalized solution variables' components with stronger $t-$weights
than the $t-$weights afforded by the Sobolev norm $\highnorm{N}.$
This leads to better control over the lower-order derivatives compared to the higher-order derivatives,
a fact that plays a key role in the proof of our main stable singularity 
formation theorem. We discuss this issue in more detail in Sect.~\ref{SSS:STRONGESTIMATESINTRO}.

To derive a priori estimates for  the
$\totalenergy{\smallparameter_*}{M},$ we also need to measure
the pointwise $|\cdot|_{\newg}$ norms 
(see Def.~\ref{D:POINTWISENORMS})
of various tensorfields. The 
reason is that the energies $\totalenergy{\smallparameter_*}{M},$
which are the quantities that we will be able to estimate via integration by parts, 
control square integrals of $|\cdot|_{\newg}.$ In particular, we use the  
norms $|\cdot|_{\newg}$ during our proof of Prop.~\ref{P:POINTWISEESTIMATES}. 
In this proposition, we derive pointwise bounds for the $|\cdot|_{\newg}$ norms of the
inhomogeneous terms appearing in the $\partial_{\vec{I}}-$commuted equations. 
This is a crucially important preliminary step in our derivation of a priori estimates for 
$\totalenergy{\smallparameter_*}{M}$ because square integrals of the $|\cdot|_{\newg}$ norms of the inhomogeneous terms 
drive the evolution of the $\totalenergy{\smallparameter_*}{M}.$

We now state the norm bootstrap assumptions that we use in our proof of stable singularity formation. 
We assume that on a time interval $(T,1]$ of existence, the following norm bounds hold:
\begin{subequations}
\begin{align} 
	\highnorm{N}(t) & \leq \epsilon t^{-\upsigma}, && t \in (T,1], 
		\label{E:HIGHBOOTINTRO} \\
	\lowkinnorm{N-3}(t) & \leq 1, && t \in (T,1], 
		\label{E:KINBOOTINTRO} \\
	\lowpotnorm{N-4}(t) & \leq t^{-\upsigma}, && t \in (T,1]. \label{E:POTBOOTINTRO}
\end{align}
\end{subequations}
Above, $\epsilon$ and $\upsigma$ are small positive numbers whose smallness we adjust throughout our analysis.
Our main task is to show how to derive strict improvements of \eqref{E:HIGHBOOTINTRO}-\eqref{E:POTBOOTINTRO}
under a near-FLRW assumption on the data (given at $t = 1$).

The main step in deriving the improvements is to obtain a coupled system of 
integral inequalities for the $\totalenergy{\smallparameter_*}{M},$ a task that we accomplish in
in Prop.~\ref{P:ENERGYINTEGRALINEQUALITIES}. In simplified form, the inequalities read\footnote{
One slight technical difficulty, which we do not discuss in detail here, is that the energies 
$\totalenergy{\smallparameter_*}{M}^2$ do not directly control the 
quantities $\sum_{1 \leq |\vec{I}| \leq M} \| | \partial_{\vec{I}} \newg |_{\newg} \|_{L^2}^2 
+ \| | \partial_{\vec{I}} \newg^{-1} |_{\newg} \|_{L^2}^2.$
To control these quantities, we derive another hierarchy of inequalities (see Prop.~\ref{P:SOBFORG})
that is coupled to the hierarchy \eqref{E:ENERGYINTEGRALINEQUALITIESINTRO}.
}
\begin{align} \label{E:ENERGYINTEGRALINEQUALITIESINTRO}
		\totalenergy{\smallparameter_*}{M}^2(t) & \leq C_N \overbrace{\highnorm{M}^2(1)}^{\mbox{``data''}} \\
		& + c_N \epsilon \int_{s=t}^{s=1} s^{-1} \totalenergy{\smallparameter_*}{M}^2(s) \, ds
			\underbrace{+ C_N \epsilon \int_{s=t}^{s=1} s^{-1 - c_N \sqrt{\epsilon}} \totalenergy{\smallparameter_*}{M-1}^2(s) \, ds}_{
			\mbox{absent if $M = 0$}}
			+ \cdots. \notag
\end{align}
Above, $c_N$ and $C_N$ are positive constants and $\epsilon$ is the small positive number featured in
the bootstrap assumptions \eqref{E:HIGHBOOTINTRO}-\eqref{E:POTBOOTINTRO}. 
In Corollary \ref{C:ENERGYINTEGRALINEQUALITIES}, we analyze the hierarchy and derive 
the main a priori estimates by a Gronwall argument.
Let us restate some of the estimates that appear in the proof of the corollary: 
if $\totalenergy{\smallparameter_*}{M}(1) \lesssim \epsilon^2$ for $0 \leq M \leq N$
and $\epsilon$ is sufficiently small,
then the following inequalities hold on any slab $(T,1] \times \mathbb{T}^3$ of existence:
\begin{align} \label{E:MAINTOTALENERGYESTIMATEMINRO}
		\totalenergy{\smallparameter_*}{M}^2(t) & \leq C \epsilon^4 t^{- c_M \sqrt{\epsilon}}.
\end{align}
With the help of the comparison estimates of Prop.~\ref{P:COMPARISON} and a few additional estimates, 
inequality \eqref{E:MAINTOTALENERGYESTIMATEMINRO} allows us to deduce an improvement of the main Sobolev norm bootstrap assumption \eqref{E:HIGHBOOTINTRO} when the data are sufficiently small. As we have mentioned, 
inequality \eqref{E:MAINTOTALENERGYESTIMATEMINRO} allows for mild energy blow-up
as $t \downarrow 0.$ Because of the presence of the two integrals on the right-hand side of \eqref{E:ENERGYINTEGRALINEQUALITIESINTRO}, 
our methods force us to accept this loss. 
%However, we note that for generalized Kasner solutions that are within $\epsilon$ of the FLRW solution 
%\eqref{E:BACKGROUNDSOLUTION} at $t=1,$ we have $\totalenergy{\smallparameter_*}{M}(t) = C \epsilon$ for all $t.$ Thus, it might be possible to devise an alternative argument that improves \eqref{E:MAINTOTALENERGYESTIMATEMINRO}.

We now discuss some important aspects of the structure of inequality \eqref{E:ENERGYINTEGRALINEQUALITIESINTRO}. 
The factor $\epsilon s^{-1}$ in the integral
$C_N \epsilon \int_{s=t}^{s=1} s^{-1} \totalenergy{\smallparameter_*}{M}^2(s) \, ds$ is the most delicate term.
It arises from a family of \emph{borderline cubic error integrals} 
such as the one \eqref{E:BORDERLINECUBICTERM},
which leads to the presence of an error integral of the form
\begin{align} \label{E:SAMPLEBORDERLINEERRORINTEGRAL}
		\int_{s=t}^{s=1} 
			s^{-1} 
			\| |\freenewsec|_{\newg} \|_{L^{\infty}}
			\totalenergy{\smallparameter_*}{M}^2(s)
		\, ds.
\end{align}
To bound the integral \eqref{E:SAMPLEBORDERLINEERRORINTEGRAL} by
$\leq C_N \epsilon \int_{s=t}^{s=1} s^{-1} \totalenergy{\smallparameter_*}{M}^2(s) \, ds,$
we derive the following key estimate (for components):
\begin{align} \label{E:SECONDFUNDIMPROVEDINTRO}
	\Big|\freenewsec_{\ j}^i \Big| \lesssim \epsilon.
\end{align}
Note that \eqref{E:SECONDFUNDIMPROVEDINTRO} yields
$|\freenewsec|_{\newg}^2 = \freenewsec_{\ b}^a \freenewsec_{\ a}^b \lesssim \epsilon^2$ as desired.
A slightly worse bound in \eqref{E:SECONDFUNDIMPROVEDINTRO},
such as $\epsilon s^{- c \epsilon},$ would 
lead to the integral
$C_N \epsilon \int_{s=t}^{s=1} s^{-1 - c \epsilon} \totalenergy{\smallparameter_*}{M}^2(s) \, ds$
on the right-hand side of \eqref{E:ENERGYINTEGRALINEQUALITIESINTRO}. This integral would
completely destroy the validity of the estimate
\eqref{E:MAINTOTALENERGYESTIMATEMINRO} and invalidate all of our main results. 
The bound \eqref{E:SECONDFUNDIMPROVEDINTRO} is available because
the lower-order derivatives of the solution obey strong $C^M$ estimates,
as we describe in Sect.~\ref{SSS:STRONGESTIMATESINTRO}.
Similar remarks apply to the integrals
$C_N \epsilon \int_{s=t}^{s=1} s^{-1 - c_N \sqrt{\epsilon}} \totalenergy{\smallparameter_*}{M-1}^2(s) \, ds.$
Here, we have conceded a slightly worse loss of $\epsilon s^{-1 - c_N \sqrt{\epsilon}},$ but this concession
is allowed because the integral depends on a lower-order energy (which, during an inductive bootstrap argument, would already have been suitably bounded). 
%To derive the estimate \eqref{E:MAINTOTALENERGYESTIMATEMINRO}, it is important that the 
%size of the additional loss of $s^{- c_N \sqrt{\epsilon}}$ goes to $0$ as the size of the data goes to $0.$ 
Again, the availability of the non-fatal factor $\epsilon s^{-1 - c_N \sqrt{\epsilon}}$ is a consequence of the strong
estimates obeyed by the lower-order derivatives. 

%We remark that the inequalities \eqref{E:MAINTOTALENERGYESTIMATEMINRO}
%are roughly optimal in the following sense: for a near-isotropic Kasner solution $\gfour_{KAS} = - dt^2 + \sum_{i=1}^3 t^{2q_j} (dx^i)^2,$
%with $|q_j - 1/3| \lesssim \epsilon,$ we have 

%\begin{align} \label{E:KASNERENERGYLOSS}
%	\totalenergy{\smallparameter_*}{0}(t) & \approx \epsilon t^{- c \epsilon}, & t \in & (0,1].
%\end{align}
%Inequality \eqref{E:KASNERENERGYLOSS} shows that there is no hope of significantly improving 
%inequality \eqref{E:MAINTOTALENERGYESTIMATEMINRO} using our methods.

\subsubsection{Strong $C^M$ estimates for the lower-order derivatives} \label{SSS:STRONGESTIMATESINTRO}
We now outline our derivation of strong $C^M$ estimates for the lower-order derivatives of the solution.
By ``strong,'' we mean that they are stronger than the estimates afforded by the bootstrap assumptions and 
Sobolev embedding. Note that in Sect.~\ref{SSS:TOPLEVELPICTURE}, we explained why
these strong estimates are an essential ingredient in our derivation of the a priori energy estimates.
We derive the strong estimates by taking advantage of the special structure of the Einstein equations
in CMC-transported spatial coordinates. They incur a loss in derivatives
because to derive them, we fix the spatial point $x$ and treat evolution equations as ODEs 
\emph{with small sources that depend on the higher-order derivatives}.
We now discuss the main ideas behind deriving the strong estimates; see
Prop.~\ref{P:STRONGPOINTWISE} for the details.
We first discuss the components $\freenewsec_{\ j}^i.$ The Sobolev norm bootstrap assumption \eqref{E:HIGHBOOTINTRO} guarantees 
only the inadequate bound $\Big|\freenewsec_{\ j}^i \Big| \lesssim \epsilon t^{- \upsigma}.$
Based on the presence of the error integral \eqref{E:SAMPLEBORDERLINEERRORINTEGRAL}, this bound
would lead to the following error integral on the right-hand side of \eqref{E:ENERGYINTEGRALINEQUALITIESINTRO}:
$c_N \epsilon \int_{s=t}^{s=1} s^{-1 - \upsigma} \totalenergy{\smallparameter_*}{M}^2(s) \, ds.$
As we described in Sect.~\ref{SSS:TOPLEVELPICTURE}, such an error integral
is damaging and would destroy the viability of our proof.
We now sketch a proof of how to derive the critically important improved estimate
$\Big|\freenewsec_{\ j}^i \Big| \lesssim \epsilon$ stated in \eqref{E:SECONDFUNDIMPROVEDINTRO}.
To derive this bound, one can insert the bootstrap assumptions 
\eqref{E:HIGHBOOTINTRO}-\eqref{E:POTBOOTINTRO}
into the evolution equation \eqref{E:RESCALEDKEVOLUTION} and deduce
\begin{align} \label{E:PARTIALTIMPROVEDKINTRO}
	\Big|\partial_t \freenewsec_{\ j}^i\Big| & \lesssim \epsilon t^{-1/3 - Z \upsigma},
\end{align}
where $Z$ is a positive integer that depends only on the number of terms in the products on the right-hand side
of \eqref{E:RESCALEDKEVOLUTION}. More precisely, to derive \eqref{E:PARTIALTIMPROVEDKINTRO}, we roughly use the following strategy. We 
isolate $\partial_t \freenewsec_{\ j}^i$ using the evolution equation \eqref{E:RESCALEDKEVOLUTION}
and put all remaining terms on the right-hand side.
We bound these remaining terms by using the bootstrap assumptions \eqref{E:HIGHBOOTINTRO}-\eqref{E:POTBOOTINTRO} + Sobolev embedding. The bootstrap assumptions imply that in a given product, all renormalized field variable frame components except for at most one can be bounded in $C^0$ by $t^{-\upsigma};$ this is where we use the lower-order bootstrap assumptions 
\eqref{E:KINBOOTINTRO}-\eqref{E:POTBOOTINTRO}. 
As a consequence of the Sobolev norm bootstrap assumption \eqref{E:HIGHBOOTINTRO}, the
possible exceptional term in the product is bounded by at worst $\epsilon t^{-2/3 - \upsigma}.$ 
This additional factor $t^{-2/3}$ comes from the ``potential terms'' $\upgamma$ and $\newu$ appearing
(implicitly) in the norm bootstrap assumption \eqref{E:HIGHBOOTINTRO}; these terms appear on the left-hand side of \eqref{E:HIGHBOOTINTRO} weighted with a factor $t^{2/3}.$ The factors of $t^{2/3}$ emerge naturally from the structure of our energy estimates, which we discuss in Sect.~\ref{SSS:KEYMONOTONICITYINTRO} (see in particular the powers of $t$
in equations \eqref{E:METRICJCOERCIVEINTRO}- \eqref{E:FLUIDJCOERCIVEINTRO}). 
Using this strategy, we can more or less directly infer the power of 
$\epsilon t^{-1/3 - Z \upsigma}$ on the right-hand side of \eqref{E:PARTIALTIMPROVEDKINTRO} by examining the evolution equation \eqref{E:RESCALEDKEVOLUTION}, where powers of $t$ explicitly appear. Roughly speaking, the ``worst'' power of $t$ explicitly appearing in equation \eqref{E:RESCALEDKEVOLUTION} is $t^{1/3},$ while the above argument has just shown that the bootstrap assumptions allow us to estimate any product of renormalized field variable frame components in $C^0$ by at worst $\epsilon t^{-2/3 - Z \upsigma}.$ Thus, in total, the worst source term for 
$\partial_t \freenewsec_{\ j}^i$ can be bounded by $\epsilon t^{1/3} t^{-2/3 - Z \upsigma},$ which yields \eqref{E:PARTIALTIMPROVEDKINTRO}. We now integrate \eqref{E:PARTIALTIMPROVEDKINTRO}
from $t$ to $1,$ use the integrability of $t^{-1/3 - Z \upsigma}$ over the interval $(0,1)$ (for sufficiently small $\upsigma$),
and use the small-data assumption $\Big|\freenewsec_{\ j}^i\Big|(1,x) \lesssim \epsilon,$ to deduce the desired strong estimate 
\eqref{E:SECONDFUNDIMPROVEDINTRO}.

The remaining strong estimates for the non-lapse variables in Prop.~\ref{P:STRONGPOINTWISE} 
are equally important and can be derived with a similar strategy,
but we must prove them in a viable order; the proofs of some of the strong estimates require the availability 
other ones that must be proved independently. Hence, the proofs reveal some effective partial dynamic decoupling of the 
lower-order derivatives of some of the solution variables.
As a last step in the proof of the proposition, we 
use the maximum principal
to derive the estimates for the lapse variables; 
see Sect.~\ref{SSS:LAPSEINTRO}.

\subsection{The renormalized lapse variables - governed by two elliptic PDEs} \label{SSS:LAPSEINTRO}
We now highlight some interesting and important issues that arise in our analysis of the renormalized lapse variables
$\newlap$ and $\dlap_i$ (see Def.~\ref{D:RESCALEDVAR}). 
These issues are relevant for the derivation of strong estimates for the lower-order
lapse derivatives and also for the derivation of our main energy estimates. The main point is that the elliptic equation verified by $\newlap$ can be expressed in two different ways; see equations \eqref{E:LAPSERESCALEDELLIPTIC} and \eqref{E:LAPSELOWERDERIVATIVES}.
The equivalence (for lower-order derivatives) of these two expressions is a consequence of the Hamiltonian constraint equation \eqref{E:RINTERMSOFKPANDU}, which connects the scalar curvature $R$ of $g$ to $\freenewsec$ and the fluid variables. 

The second lapse equation \eqref{E:LAPSELOWERDERIVATIVES} involves inhomogeneous terms that are of a favorable size,
but that depend on \emph{one derivative} of $\upgamma$ (because $R$ depends on one derivative of $\upgamma$). The first lapse equation 
\eqref{E:LAPSERESCALEDELLIPTIC} involves homogeneous terms that are of a dangerous size $\sim t^{- 4/3}.$ However, these terms have an extra degree of differentiability compared to the inhomogeneous terms appearing in the second lapse equation. In view of these remarks, we adopt the following basic strategy: whenever we want to derive estimates for the lower-order derivatives of the lapse variables, we use the second lapse equation; the source terms are much smaller in this equation. However, in order to close the top order energy estimates, we are forbidden from using the favorable second equation because of its dependence on one derivative of $\upgamma;$ there would be too many derivatives on $\upgamma$ to close the estimates. Thus, for these top order derivatives, we are forced to use the first equation. We are thus compelled to prove estimates that are just good enough to allow us to control the top order derivatives of the lapse despite 
the presence of the dangerous source terms.

These estimates reflect a tension that is enforced by the Hamiltonian constraint \eqref{E:RINTERMSOFKPANDU}. On the one hand, if one
expresses $R$ in terms of $g$ and its spatial derivatives, then our strong estimates show that 
$R$ is an order $\epsilon t^{-(2/3_+)}$ term. On the other hand, the kinetic terms (which are defined in Sect.~\ref{SS:KINETICTERMSDOMINATE}) 
in \eqref{E:RINTERMSOFKPANDU} suggest that $R$ ``wants to be'' an order $\epsilon t^{-(2_+)}$ term. For the lower-order spatial derivatives of $R,$ our strong estimates guarantee that the former estimates hold, which means that there must be severe cancellation among
the kinetic terms in \eqref{E:RINTERMSOFKPANDU}. For the top order spatial derivatives of $R,$ the 
expression in terms of $g$ and its spatial derivatives is not in $L^2,$ and 
thus the kinetic terms in \eqref{E:RINTERMSOFKPANDU} dictate the top order $L^2$ behavior of the lapse.
A major aspect of our analysis is that we constantly have to battle this kind of tension.

\subsection{Dominance of the kinetic terms and VTD behavior} \label{SS:KINETICTERMSDOMINATE}
It is convenient to group the solution variables into two classes, namely the ``kinetic terms'' and the ``potential terms.''
By kinetic terms, we mean the variables $\freenewsec_{\ j}^i$ and $\adjustednewp,$ while by potential terms, we mean 
$\upgamma_{j \ k}^{\ i}$ and $\newu^i.$ Moreover, since the variable $\sqrt{\gbigdet}$ has analytic properties 
in common with $\freenewsec_{\ j}^i$ and $\adjustednewp,$ we will also refer to it as a kinetic variable. Similarly,
since $\newg_{ij},$ $(\newg^{-1})^{ij},$ $\newlap,$ and $\dlap_i$ have analytic properties in common with $\upgamma_{j \ k}^{\ i}$ and $\newu^i,$
we will also refer to these variables as potential variables. Our analysis 
(more precisely, the strong $C^M$ estimates discussed in Sect.~\ref{SSS:STRONGESTIMATESINTRO}) roughly shows that in many cases, 
the kinetic terms are the dominant influence in the Einstein-stiff fluid equations near the singularity. 
As an example, we discuss the evolution equation \eqref{E:PARTIALTKCMC}, which we express as
\begin{align} \label{E:KINVSPOTEXAMPLE}
	\partial_t (\SecondFund_{\ j}^i) = - \frac{1}{t} \SecondFund_{\ j}^i 
		+ \mbox{linear-in-potential terms}.
\end{align}
The strong estimates of Sect.~\ref{S:STRONGESTIMATES} 
can be used to show that the potential terms on the right-hand side of \eqref{E:KINVSPOTEXAMPLE}
are $\lesssim \epsilon t^{-2/3 - c \sqrt{\epsilon}}$ while the kinetic term
$- \frac{1}{t} \SecondFund_{\ j}^i$ is of the much larger order $t^{-2}.$

Moreover, consider (for example) inequality \eqref{E:INTROKLIMIT},  
which shows that the time-rescaled field
$t \SecondFund_{\ j}^i$ converges to a time-rescaled-near-FLRW field
$(\newsec_{Bang})_{\ j}^i(x)$ as $t \downarrow 0.$ 
Another way to view this convergence is as follows: 
$\SecondFund_{\ j}^i(t,x)$ is asymptotic to a field
$t^{-1}(\newsec_{Bang})_{\ j}^i(x)$ that verifies an $x-$parameterized ODE. This $x-$parameterized ODE is 
obtained by simply throwing away the potential terms from equation \eqref{E:KINVSPOTEXAMPLE}. 
This is an example of the VTD behavior mentioned in Remark \ref{R:VTD}

\subsubsection{The main energy integral inequality and the key approximate monotonicity inequality} \label{SSS:KEYMONOTONICITYINTRO}
We now discuss the derivation of the integral inequality \eqref{E:FUNDAMENTALENERGYINEQUALITY} for
the total energies $\totalenergy{\smallparameter_*}{M}^2,$ which is
one of the main ingredients in our derivation of the energy hierarchy \eqref{E:ENERGYINTEGRALINEQUALITIESINTRO}. 
The derivation of \eqref{E:FUNDAMENTALENERGYINEQUALITY} is essentially an elaborate integration by parts inequality that takes into
account the following key structures: \textbf{i)} the availability of the strong $C^M$ estimates for the lower-order derivatives; \textbf{ii)} the partially decoupled nature of the metric and fluid energy estimates; 
and most importantly, \textbf{iii)} \textbf{the unexpected availability of two additional favorably signed spacetime integrals that control the lapse variables.} Because of its importance, we quickly summarize here (and describe in more detail below)
the main idea behind point \textbf{iii)}. The key observation is that in deriving
energy estimates for the $\partial_{\vec{I}}$ commuted fluid quantities, we encounter the following error 
integral (see \eqref{E:KEYTERMINTRO}):
\begin{align} \label{E:FIRSTKEYTERMINTRO}
	- \frac{2}{3} t^{1/3} \int_{\mathbb{T}^3} \left(\partial_{\vec{I}} \adjustednewp \right)
		\left(\partial_{\vec{I}} \newlap \right) \, dx.
\end{align}
Amazingly, we are able to derive identities showing that \eqref{E:FIRSTKEYTERMINTRO} has a good sign 
towards the past and in fact yields signed integrals that provide control of $t-$weighted versions of
the lapse quantities
$\| \partial_{\vec{I}} \newlap \|_{L^2}$
and $\| |\partial_{\vec{I}} \dlap|_{\newg} \|_{L^2}.$
Moreover, as we shall see, we need the signed integrals to absorb other 
error integrals that would otherwise spoil our estimates.

Because our integration by parts arguments are computationally involved, we have chosen to organize the calculations using the framework of \emph{energy currents}. These currents allow us to derive $L^2-$type identities for the derivatives $\partial_{\vec{I}}$ 
of the renormalized field variables, where $\vec{I}$ is a spatial derivative multi-index. We introduce the abbreviated notation  
$\dot{\freenewsec}_{\ j}^i := \partial_{\vec{I}} \freenewsec_{\ j}^i,$ 
$\dot{\upgamma}_{j \ k}^{\ i} := \partial_{\vec{I}} \upgamma_{j \ k}^{\ i},$  
$\dot{\newlap} := \partial_{\vec{I}} \newlap,$ 
$\dot{\dlap}_i := \partial_{\vec{I}} \dlap_i,$
$\dot{\adjustednewp} := \partial_{\vec{I}} \adjustednewp,$
$\dot{\newu}^i := \partial_{\vec{I}} \newu^i.$
The quantities $\dot{\freenewsec},$ etc. are known as \emph{variations}. 
Note that the notation $\cdot$ in the variations has nothing to do with time derivatives.
The energy currents are spacetime vectorfields 
$\dot{\mathbf{J}}_{(Metric)}^{\mu}[(\dot{\freenewsec},\dot{\upgamma}) , (\dot{\freenewsec},\dot{\upgamma})]$ and
$\dot{\mathbf{J}}_{(Fluid)}^{\mu}[(\dot{\adjustednewp},\dot{\newu}) , (\dot{\adjustednewp},\dot{\newu})]$ 
that depend quadratically on their arguments $[\cdot,\cdot],$ on the solution variables $\newg_{ij},$ $\newp,$ and $\newu^i,$
and in the case of $\dot{\mathbf{J}}_{(Metric)}^j,$ also on $\dot{\dlap}.$
We give their precise definitions in Def.~\ref{D:METRICCURRENT} and Def.~\ref{D:FLUIDEOV}. Roughly speaking, these 
currents exist and have useful properties because the evolution equations are hyperbolic. 
A general framework addressing the availability 
of and properties of energy currents was provided by Christodoulou \cite{dC2000} (see also \cite{dC2007} regarding the existence/use of currents for the relativistic Euler equations in Eulerian variables). However, instead of using the general framework, 
we have simply derived the currents by hand in this article. 
 
The most important analytic properties of the currents are the following:
\begin{enumerate}
	\item Under our bootstrap assumptions, the following positivity properties are verified by the quadratic forms
		$\dot{\mathbf{J}}_{(Metric)}^0[\cdot, \cdot]$ and $\dot{\mathbf{J}}_{(Fluid)}^0[\cdot, \cdot]:$
		\begin{subequations}
		\begin{align}
			\dot{\mathbf{J}}_{(Metric)}^0[(\dot{\freenewsec},\dot{\upgamma}) , (\dot{\freenewsec},\dot{\upgamma})] 
				& =  |\dot{\freenewsec}|_{\newg}^2 + \frac{1}{4} t^{4/3} |\dot{\upgamma}|_{\newg}^2, 
			\label{E:METRICJCOERCIVEINTRO} \\
			\dot{\mathbf{J}}_{(Fluid)}^0[(\dot{\adjustednewp},\dot{\newu}) , (\dot{\adjustednewp},\dot{\newu})] 
				& \approx \frac{1}{2} \dot{\adjustednewp}^2 
				+ \frac{2}{9} t^{4/3} |\dot{\newu}|_{\newg}^2. \label{E:FLUIDJCOERCIVEINTRO}
	\end{align}
	\end{subequations}
	\item By using the Einstein-stiff fluid equations for substitution,  
	$\partial_{\mu} \big(\dot{\mathbf{J}}_{(Metric)}^{\mu}[(\dot{\freenewsec},\dot{\upgamma}) , (\dot{\freenewsec},\dot{\upgamma})] \big)$ 
	and \\
	$\partial_{\mu} \big(\dot{\mathbf{J}}_{(Fluid)}^{\mu}[(\dot{\adjustednewp},\dot{\newu}) , (\dot{\adjustednewp},\dot{\newu})] \big)$
	can be expressed in terms of inhomogeneous terms that depend on the variations themselves, 
	\emph{but not on} their spacetime derivatives
	$\partial_{\nu} (\dot{\freenewsec},\dot{\upgamma}),$ $\partial_{\nu} (\dot{\adjustednewp},\dot{\newu}),$
	or $\partial_{\nu} (\dot{\newlap},\dot{\dlap}).$
	\end{enumerate}
	We provide precise expressions for 
	$\partial_{\mu} \big(\dot{\mathbf{J}}_{(Metric)}^{\mu}[(\dot{\freenewsec},\dot{\upgamma}) , (\dot{\freenewsec},\dot{\upgamma})] \big)$ 
	and
	$\partial_{\mu} \big(\dot{\mathbf{J}}_{(Fluid)}^{\mu}[(\dot{\adjustednewp},\dot{\newu}) , (\dot{\adjustednewp},\dot{\newu})] \big)$ 
	in \eqref{E:DIVMETRICJ} and \eqref{E:DIVFLUIDJ}. These two equations are simply differential versions of integration by parts 
	identities. The energies $\metricenergy{M}$ and $\fluidenergy{M}$ appearing in \eqref{E:TOTALENERGYINTRO}
	are constructed by integrating the $0$ (i.e., normal) component of the
	currents over the hypersurfaces $\Sigma_t:$
	\begin{subequations}
	\begin{align}
		\metricenergy{M}^2(t) 
		&:= \sum_{|\vec{I}| \leq M}\int_{\Sigma_t}
		\dot{\mathbf{J}}_{(Metric)}^0[\partial_{\vec{I}} (\freenewsec, \upgamma), \partial_{\vec{I}} (\freenewsec, \upgamma)]  \, dx, 
		\label{E:METRICENERGYINTRO} \\
	\fluidenergy{M}^2(t) 
		&:= \sum_{|\vec{I}| \leq M}\int_{\Sigma_t}
		\dot{\mathbf{J}}_{(Fluid)}^0[\partial_{\vec{I}} (\adjustednewp, \newu), \partial_{\vec{I}} (\adjustednewp, \newu)]  \, dx.
		\label{E:FLUIDENERGYINTRO}
\end{align}
\end{subequations}
By applying the divergence theorem, we will obtain \emph{separate} a priori integral inequalities
for the energies $\metricenergy{M}(t)$ and $\fluidenergy{M}(t);$ see \eqref{E:METRICENERGYMINTEGRALINEQUALITY} and 
\eqref{E:FLUIDENERGYMINTEGRALINEQUALITY}. We remark that the inhomogeneous terms that arise in commuting the equations with $\partial_{\vec{I}}$ 
(for $|\vec{I}| \leq M$) are present on the right-hand sides of these inequalities. The starting points for 
the a priori estimates are the following identities, valid for $t \in (0,1]:$
\begin{subequations}
\begin{align}
	\metricenergy{M}^2(t) & = \metricenergy{M}^2(1) 
		- \sum_{|\vec{I}| \leq M} \int_{s = t}^1 
		\partial_{\mu} \big(\dot{\mathbf{J}}_{(Metric)}^{\mu}
			[\partial_{\vec{I}} (\freenewsec, \upgamma), \partial_{\vec{I}} (\freenewsec, \upgamma)] \big) \, dx, 
		\label{E:METRICDIVINTRO} \\
	\fluidenergy{M}^2(t) & = \fluidenergy{M}^2(1) 
		- \sum_{|\vec{I}| \leq M} \int_{s = t}^1 
		\partial_{\mu} \big(\dot{\mathbf{J}}_{(Fluid)}^{\mu}
			[\partial_{\vec{I}} (\adjustednewp, \newu), \partial_{\vec{I}} (\adjustednewp, \newu)] \big) \, dx,
		\label{E:FLUIDDIVINTRO}
\end{align}
\end{subequations}
which follow from the divergence theorem. The challenge is to use the equations \eqref{E:DIVMETRICJ} and \eqref{E:DIVFLUIDJ}, the strong estimates for the lower-order derivatives, and the structure of the inhomogeneous terms in the 
Einstein-stiff fluid equations to estimate the right-hand sides of 
\eqref{E:METRICDIVINTRO}-\eqref{E:FLUIDDIVINTRO} back in terms of $	\metricenergy{M}$ and $\fluidenergy{M};$
this is exactly what we accomplish in the energy inequality hierarchies \eqref{E:ENERGYINTEGRALINEQUALITIESINTRO}.

We now discuss the delicate issues that arise in combining the metric and fluid energy integral inequalities into the 
fundamental total energy integral inequality \eqref{E:FUNDAMENTALENERGYINEQUALITY}. 
We first discuss the metric energy integral inequality \eqref{E:METRICENERGYMINTEGRALINEQUALITY}. This is 
a mostly standard energy integral inequality that arises from carefully analyzing the terms in the divergence 
identity \eqref{E:DIVMETRICJ} for 
$\partial_{\mu} \big(\dot{\mathbf{J}}_{(Metric)}^{\mu}[(\dot{\freenewsec},\dot{\upgamma}) , (\dot{\freenewsec},\dot{\upgamma})] \big).$
The positive spacetime integral
on the left-hand side of \eqref{E:METRICENERGYMINTEGRALINEQUALITY}, which provides control of
$|\dot{\upgamma}|_{\newg}^2,$ arises when the time derivative $\partial_t$ hits the $t^{4/3}$ factor in the 
product $\frac{1}{4} t^{4/3} (\newg^{-1})^{ab} (\newg^{-1})^{ef} \newg_{ij} \dot{\upgamma}_{e \ a}^{\ i} \dot{\upgamma}_{f \ b}^{\ j}$
from the expression \eqref{E:METRICCURRENT0} for $\dot{\mathbf{J}}_{(Metric)}^0[(\dot{\freenewsec},\dot{\upgamma}) , (\dot{\freenewsec},\dot{\upgamma})] .$ There are cross terms of the form $|\dot{\upgamma}|_{\newg}|\dot{\newu}|_{\newg}$ and $|\dot{\upgamma}|_{\newg}|\dot{\dlap}|_{\newg}$ that arise in the derivation of \eqref{E:METRICENERGYMINTEGRALINEQUALITY}, but these terms can be respectively bounded by $C^{-1}|\dot{\upgamma}|_{\newg}^2 + C|\dot{\newu}|_{\newg}^2$ and $C^{-1}|\dot{\upgamma}|_{\newg}^2 + C|\dot{\dlap}|_{\newg}^2.$ For large enough $C,$ the $C^{-1}|\dot{\upgamma}|_{\newg}^2$ terms can be absorbed into the positive spacetime integral on the left-hand side of \eqref{E:METRICENERGYMINTEGRALINEQUALITY}, while the other two terms generate quadratic fluid and lapse terms \emph{that appear with dangerous large constants}; we will soon address how we handle these large constants.

The derivation of the fluid energy integral inequality \eqref{E:FLUIDENERGYMINTEGRALINEQUALITY} is much more delicate
and is at the heart of our derivation of the approximate $L^2$ monotonicity inequality.
We first discuss the more standard features of the inequality.
The positive spacetime integral involving $|\partial_{\vec{I}} \newu|_{\newg}^2$ $=|\dot{\newu}|_{\newg}^2$
on the left-hand side of \eqref{E:FLUIDENERGYMINTEGRALINEQUALITY}
arises when the time derivative $\partial_t$ hits the $t^{4/3}$ factor in the product 
$2 t^{4/3} \big[\adjustednewp + \frac{1}{3} \big]^2 \newg_{ef} \dot{\newu}^e \dot{\newu}^f$ 
from the expression \eqref{E:FLUIDCURRENT0} for $\dot{\mathbf{J}}_{(Fluid)}^0[(\dot{\adjustednewp},\dot{\newu}) , (\dot{\adjustednewp},\dot{\newu})] .$ 
This is completely analogous to our analysis of $\partial_{\mu} \big(\dot{\mathbf{J}}_{(Metric)}^{\mu}[(\dot{\freenewsec},\dot{\upgamma}) , (\dot{\freenewsec},\dot{\upgamma})] \big).$ The subtle feature is the origin
of the positive spacetime integrals on the left-hand side of \eqref{E:FLUIDENERGYMINTEGRALINEQUALITY} 
involving the lapse variables. 
These spacetime integrals arise from a subtle analysis of the spacetime integral corresponding to the 
$- \frac{2}{3} t^{1/3} \dot{\adjustednewp} \dot{\newlap}$ term on the right-hand side of the expression \eqref{E:DIVFLUIDJ}
for $\partial_{\mu} \big(\dot{\mathbf{J}}_{(Fluid)}^{\mu}[(\dot{\adjustednewp},\dot{\newu}) , (\dot{\adjustednewp},\dot{\newu})] \big).$
The spatial integral corresponding to this term is 
\begin{align} \label{E:KEYTERMINTRO}
	- \frac{2}{3} t^{1/3} \int_{\mathbb{T}^3} \left(\partial_{\vec{I}} \adjustednewp \right)
		\left(\partial_{\vec{I}} \newlap \right) \, dx.
\end{align}
\emph{The miracle is the following:} we can use version 1 of the commuted lapse equation, namely \eqref{AE:LAPSEICOMMUTED}, to replace
the term $\frac{-2}{3} \partial_{\vec{I}} \adjustednewp$ with
$\frac{-1}{3} \mathcal{L} \partial_{\vec{I}} \newlap,$ plus some error terms (the signs and the size of the constant factors is crucially important). 
Here, $\mathcal{L}$ is the negative definite elliptic operator defined in \eqref{E:LDEF}; we have $\mathcal{L} = t^{4/3} (\newg^{-1})^{ab} \partial_a \partial_b - 1 + \mbox{error terms}.$ We remark that this step requires the combined use of some of the special structure of the Einstein equations in our gauge, for in deriving the commuted lapse equation \eqref{AE:LAPSEICOMMUTED}, 
we used the Hamiltonian constraint and the constant mean curvature condition. That is, the fact that we can replace $\frac{-2}{3} \partial_{\vec{I}} \adjustednewp$ with $\frac{-1}{3} \mathcal{L} \partial_{\vec{I}} \newlap$ is a consequence of the original lapse equation \eqref{E:LAPSEINTRO}, the constant mean curvature condition, and the Hamiltonian constraint \eqref{E:RINTERMSOFKPANDU}. We then integrate by parts in \eqref{E:KEYTERMINTRO} (after making the replacement), which, up to some additional error terms, generates two \emph{negative} spacetime integrals on the \emph{right}-hand side of \eqref{E:FLUIDDIVINTRO}. One of the integrals controls $|\partial_{\vec{I}} \newlap|^2,$ while the other controls $|\partial_{\vec{I}} \dlap|_{\newg}^2.$ These integrals are multiplied by certain constants that
are of great importance. We bring these two integrals over to the left-hand side of inequality \eqref{E:FUNDAMENTALENERGYINEQUALITY}, which 
in total results in the presence of the three spacetime integrals involving \emph{positive} constants. The reason that the size of the constants is important is the presence of the cross term $- 4 t \big[\adjustednewp + \frac{1}{3} \big]^2 \dot{\newu}^a \dot{\dlap}_a \approx
- \frac{4}{9} t \dot{\newu}^a \dot{\dlap}_a$ on the right-hand side of the expression \eqref{E:DIVFLUIDJ} 
for $\partial_{\mu} \big(\dot{\mathbf{J}}_{(Fluid)}^{\mu}[(\dot{\adjustednewp},\dot{\newu}) , (\dot{\adjustednewp},\dot{\newu})] \big).$
It turns out that the constants
available are large enough such that the quadratic integral corresponding to this term can be completely soaked up into the positive integrals on the left-hand side of \eqref{E:FUNDAMENTALENERGYINEQUALITY} with a bit of room to spare. 
Furthermore, there are \emph{no integrals on the right-hand side of the fluid energy integral inequality \eqref{E:FLUIDENERGYMINTEGRALINEQUALITY} involving the quadratic top order term $|\partial_{\vec{I}} \upgamma|_{\newg}^2.$} 
This structure will play an absolutely essential role when we combine the metric and fluid energies.

In order to combine the metric and fluid energy integral inequalities \eqref{E:METRICENERGYMINTEGRALINEQUALITY} and 
\eqref{E:FLUIDENERGYMINTEGRALINEQUALITY} into the main energy integral inequality \eqref{E:FUNDAMENTALENERGYINEQUALITY}, we simply
add a small positive multiple $\smallparameter$ of the 
metric energy inequality \eqref{E:METRICENERGYMINTEGRALINEQUALITY} to the fluid energy inequality \eqref{E:FUNDAMENTALENERGYINEQUALITY}.
If $\smallparameter = \smallparameter_*$ is sufficiently small, then the integrals corresponding to the dangerous terms
$C\smallparameter |\dot{\newu}|_{\newg}^2$ and $C \smallparameter |\dot{\dlap}|_{\newg}^2$ present on the \emph{right-hand} side of the metric estimate can be soaked up into the \emph{left-hand} side of the fluid estimate. In this manner, we have therefore eliminated all of the unfavorably signed top order pure quadratic terms with large constants. This is the content of the main energy integral inequality \eqref{E:METRICENERGYMINTEGRALINEQUALITY}; this is the aforementioned 
``approximate monotonicity'' inequality. 

Note that we have not yet discussed the following key issue connected to the derivation of the energy hierarchy \eqref{E:ENERGYINTEGRALINEQUALITIESINTRO}:
that of bounding the spacetime integrals appearing (implicitly, in the divergence of the currents) 
on the right-hand sides of \eqref{E:METRICDIVINTRO}-\eqref{E:FLUIDDIVINTRO} that arise from the inhomogeneous terms in the $\partial_{\vec{I}}-$commuted equations. More precisely, to close our estimates,
we have to bound these integrals in terms of the energies $\totalenergy{\smallparameter_*}{M}(t).$ 
In Sect.~\ref{SS:INHOMOGENEOUSTERMINTEGRALS}, we provide a brief overview of this analysis.

The favorably signed lapse spacetime integrals described in this section somewhat remind of 
Guo's work \cite{yG1998}, in which he proved small-data global existence
for irrotational solutions to the Euler-Poisson system in $3$ spatial dimensions. 
Guo's result is far from obvious, for without the coupling to the Poisson equation,
small-data irrotational Euler solutions in $3$ spatial dimensions  
can blow-up in finite time \cite{dC2007}, \cite{tS1985}. 
Roughly, the reasons that small-data blow-up occurs
in solutions to the irrotational Euler wave equation 
are that the nonlinearities do not verify the null condition
and that solutions to the corresponding
linearized wave equation decay at the non-integrable rate $(1 + t)^{-1}.$
The main idea behind Guo's proof
was his observation that 
linearizing the wave equation
verified by the velocity $u$  
in the irrotational Euler-Poisson equations
leads to a \emph{favorable linear-in-$u$ term}. That is, the linearized equation 
is a Klein-Gordon equation with a decay-producing mass term. Since 
solutions to Klein-Gordon equations in $3$ spatial dimensions decay at the integrable rate $(1 + t)^{-3/2}$
(see \cite{wS1989}), Guo was able to exploit this property
to prove his small-data global existence result.
The effect generated by the favorable linear-in-$u$ term is in rough analogy 
with the availability of the coercive lapse integrals discussed above.

\subsubsection{The inhomogeneous term integrals} \label{SS:INHOMOGENEOUSTERMINTEGRALS}
In order to derive the energy inequality hierarchy \eqref{E:ENERGYINTEGRALINEQUALITIESINTRO}, we 
have to bound the spacetime integrals appearing (implicitly, in the divergence of the currents) 
on the right-hand sides of \eqref{E:METRICDIVINTRO}-\eqref{E:FLUIDDIVINTRO}
in terms of the energies $\totalenergy{\smallparameter_*}{M}(t).$ 
Most of these integrals are 
generated by the many inhomogeneous terms that appear on the right-hand side of
the $\partial_{\vec{I}}-$commuted equations. We carry out this analysis Sect.~\ref{S:POINTWISEINHOMOGENEOUS} 
and Sect.~\ref{S:LAPSEANDINHOMEINTEGRALESTIMATES}. 
The crux of it is our derivation of pointwise bounds for the $|\cdot|_{\newg}$ norms of the inhomogeneous terms;
see Prop.~\ref{P:POINTWISEESTIMATES}. Following this, we can easily
estimate the spacetime integrals by squaring the pointwise bounds and integrating.
To facilitate the analysis, we divide many of the inhomogeneous terms into two classes: 
``junk terms,'' whose integrals are easy to estimate, and ``borderline'' terms, whose integrals have to be treated with care; see e.g. 
our labeling of the inhomogeneous terms in 
equations \eqref{AE:METRICGAMMACOMMUTED}-\eqref{AE:SECONDFUNDCOMMUTED}.

The borderline terms generate some of the spacetime integrals in \eqref{E:ENERGYINTEGRALINEQUALITIESINTRO} involving 
the dangerous factors $C_N \epsilon s^{-1}$ and $C_N \epsilon s^{-1 - c_N \sqrt{\epsilon}}.$ The 
challenge is to show that the dangerous factors are not worse than this.
The main idea behind our analysis 
is to bound products $\sum_{|\vec{I}_1| + |\vec{I}_2| \leq |\vec{I}|} 
|\partial_{\vec{I}_1} v |_{\newg}  |\partial_{\vec{I}_2} w |_{\newg}$ by using the 
strong $C_{\newg}^M$ estimates of Prop.~\ref{P:STRONGPOINTWISE}
to control the term with the least number of derivatives on it in $L^{\infty}.$ In the cases
where $|\vec{I}_1| = 0$ or $|\vec{I}_2| = 0,$ the corresponding spacetime integral is principal order in terms
of its place in the hierarchy. It is therefore especially important that we have the best possible
$L^{\infty}$ estimates in these cases, and estimates such as \eqref{E:SECONDFUNDUPGRADEPOINTWISEGNORM} 
with $M=0$ (which is implied by \eqref{E:SECONDFUNDIMPROVEDINTRO}) play a distinguished role.

\subsubsection{Other matter models} \label{SSS:OTHERMATTER}
We now explain what distinguishes the stiff fluid equation of state 
from others of the form $p = \speed^2 \rho.$ The main point is that for a general equation of state,
the elliptic lapse PDE contains a pure kinetic term proportional to $p - \rho,$
that is, it is of the form $g^{ab} \nabla_a \nabla_b (n - 1) = \frac{3}{2} (p - \rho) + \cdots;$
see \eqref{E:LAPSE}. Clearly this term vanishes only in the case $\speed = 1.$
If present, this term would dominate the behavior of the lapse 
and preclude our ability to derive the strong estimates for it at the lower orders 
(as described in Sect.~\ref{SSS:LAPSEINTRO}).
Our entire proof would therefore break down.
Similar remarks hold for the evolution equation verified by the components $\SecondFund_{\ j}^i.$
It would be interesting to characterize those matter models for which
the relevant pure kinetic terms are absent. 
For such matter models, it may be possible to prove a theorem analogous to
our main stable singularity formation theorem.

\subsection{Approximate monotonicity via parabolic lapse gauges} \label{SS:MONOTONICITYEXPLANATION}
In \cite{iRjS2014a}, we constructed a one-parameter family of gauges in which
the approximate monotonicity is also visible. The gauge condition is given by 
replacing the CMC condition
$\SecondFund_{\ a}^a(t,x) = - t^{-1}$ with
\begin{align} \label{E:PARABOLICGAUGE}
	\uplambda^{-1}(n-1) 
	= t \SecondFund_{\ a}^a + 1,
\end{align}
where $\uplambda \neq 0$ is a real number. In the case of the scalar field
matter model, we derived the approximate monotonicity
whenever $\uplambda \geq 3.$ In addition to imposing 
\eqref{E:PARABOLICGAUGE}, we also used transported spatial coordinates and decomposed  
$\gfour = - n^2 dt^2 + g_{ab} dx^a dx^b$ as in \eqref{E:GFOURDECOMPINTRO}.
Note that the case $\uplambda = \infty$
formally corresponds to the CMC condition.
Under the gauge \eqref{E:PARABOLICGAUGE} with transported spatial coordinate, 
the Einstein-stiff fluid equations
look much as they do in CMC case.
The most significant change is that the elliptic CMC equation \eqref{E:LAPSEINTRO}
is replaced with the following parabolic lapse equation,
which for $\uplambda > 0$ is locally well posed only the past direction:
\begin{align} \label{E:PARABOLICLAPSEEQN}
	\uplambda^{-1} \frac{1}{t} \partial_t (n-1)
	+
	g^{ab}\nabla_a \nabla_b (n - 1) 
	& = (n - 1) 
				\left\lbrace 
					1 - \uplambda^{-1} 
					+ R 
					- 2p u_a u^a 
				\right\rbrace 
				\\
	& \ \ 
			+  \uplambda^{-1}(\uplambda - 2) \frac{1}{t^2} (n-1)^2
			+  \uplambda^{-2} \frac{1}{t^2} (n-1)^3
			+ R 
			- 2p u_a u^a.
		\notag
\end{align}
Based on the linearized stability results of \cite{iRjS2014a}, we expect that the main results 
of the present article could also be derived in the parabolic lapse gauge;
we do not pursue this issue in detail here. 

\begin{remark}
	An advantage of the parabolic gauges is that one does not have to construct 
	a CMC hypersurface.
\end{remark}

%Based on these observations, it is reasonable to speculate that, roughly speaking, one could add small amounts of other matter into the Einstein-stiff fluid system without destroying our stability result - \emph{as long as the additional matter models still lead to the vanishing of the corresponding kinetic terms in the lapse equation and the evolution equation for $\SecondFund_{\ j}^i.$}

%\subsection{The fluid is not the problem}
%{\color{red}}

\subsection{Paper outline} 

\begin{itemize}
	\item In Sect.~\ref{S:NOTATION}, we introduce some notation and conventions that we use throughout the article.
	\item In Sect.~\ref{S:EEINCMC}, we state the Einstein-stiff fluid equations relative to 
		CMC-transported spatial coordinates.
	\item In Sect.~\ref{S:RESCALEDVARIABLES}, we introduce renormalized solution variables.
		We then state the PDEs that are verified by the renormalized variables. The system is
		equivalent to the system from Sect.~\ref{S:EEINCMC}.
	%\item In Sect.~\ref{S:COMMUTEDEQUATIONS}, we state the general form of the PDEs verified by
	%	the $\partial_{\vec{I}}-$differentiated renormalized variables.
	\item In Sect.~\ref{S:NORMSANDCURRENTS}, we introduce the norms that we use to study the renormalized
		solution variables. We also introduce the equations of variation, which is another name for the equations
		verified by the $\partial_{\vec{I}}-$differentiated variables. Finally, we introduce the metric and fluid energy current vectorfields.
		In Sect.~\ref{S:FUNDAMENATLENERGYINEQUALITIES}, we use these currents in the divergence theorem to derive our
		fundamental integration by parts integral inequalities for the renormalized solution variables and their spatial derivatives.
	\item In Sect.~\ref{S:BOOTSTRAPANDENERGYDEF}, we introduce the norm bootstrap assumptions that we use in the proof of our main 	
		stable singularity formation theorem. We then define the metric and fluid energies that we use in service of our analysis of the 
		norms.
		The currents of Sect.~\ref{S:NORMSANDCURRENTS} play a fundamental role in the definitions of the energies.
		Finally, we deduce the coercive properties of the energies.
	%\item In Sect.~\ref{S:VOLUMEFORMESTIMATES}, we use the bootstrap assumptions to derive estimates for the renormalized volume form.
	\item In Sect.~\ref{S:STRONGESTIMATES}, we use the bootstrap assumptions to derive strong $C^M$ estimates for the lower-order
		derivatives of the renormalized solution variables.
	\item In Sect.~\ref{S:LAPSEKEYLINEAR}, we provide a preliminary $L^2$ analysis of the lapse variables. 
		In particular, we prove a key proposition that shows that a certain quadratic term, which arises in the 
		divergence identity for the fluid energy current, leads to $L^2-$control over the renormalized lapse variables.
	\item In Sect.~\ref{S:FUNDAMENATLENERGYINEQUALITIES}, we derive our
		fundamental integration-by-parts-based energy integral inequalities for the solution. These estimates provide
		preliminary integral inequalities for the energies. However, the inhomogeneous term integrals, 
		which arise from the inhomogeneous terms in the $\partial_{\vec{I}}-$commuted equations, 
		are not estimated in this section.
	\item In Sect.~\ref{S:POINTWISEINHOMOGENEOUS}, we use the strong estimates of Sect.~\ref{S:STRONGESTIMATES}
		to derive suitable pointwise bounds for the inhomogeneous terms appearing in the 
		$\partial_{\vec{I}}-$commuted equations.
	\item In Sect.~\ref{S:LAPSEANDINHOMEINTEGRALESTIMATES}, we use the estimates of Sect.~\ref{S:LAPSEKEYLINEAR}
		and Sect.~\ref{S:POINTWISEINHOMOGENEOUS} to bound the $L^2$ norms of the renormalized
		lapse variables and the inhomogeneous terms by the energies.
	\item In Sect.~\ref{S:COMPARISON}, we compare the coerciveness of the solution Sobolev norms to the coerciveness of the energies.
	\item In Sect.~\ref{S:FUNDAMENTALAPRIORI}, we combine the estimates of the previous sections
		in order to derive our fundamental a priori estimates for the norms of the renormalized variables.
	\item In Sect.~\ref{S:LWPANDCMC}, we discuss local well-posedness and continuation criteria for the Einstein equations.
		We also show the existence of a CMC hypersurface in the spacetimes under consideration.
	\item In Sect.~\ref{S:STABLESINGULARITY}, we prove our main theorem showing past-stable Big Bang formation
		in spacetimes launched by near-FLRW data.
	\item In Appendix \ref{A:METRICANDCURVATUREID}, we provide some basic metric and curvature relations.
	\item In Appendix \ref{A:CMCDERIVATION}, we provide some additional details concerning the equations that
		we stated in Sect.~\ref{S:EEINCMC}.
	\item In Appendix \ref{A:RESCALEDEQNS}, we provide some additional details concerning the equations 
		for the renormalized variables that we stated in Sect.~\ref{S:RESCALEDVARIABLES}.
	\item In Appendix \ref{A:COMMUTED}, we provide the precise form of the
		$\partial_{\vec{I}}-$commuted equations.
\end{itemize}

\section{Notation and Conventions} \label{S:NOTATION}

In this section, we summarize some notation and conventions that we use throughout the article.

\subsection{Foliations}

The spacetime manifolds $\mathbf{M}$ of interest will be equipped with a time function $t$
that partitions certain regions $\mathbf{V} \subset \mathbf{M}$ 
into spacelike hypersurfaces of constant time: 
$\mathbf{V} = (T,1] \times \mathbb{T}^3 = \cup_{t \in (T,1]} \Sigma_t.$ The
$\Sigma_t$ are often CMC hypersurfaces. The level sets of $t$ 
are denoted by $\Sigma_t:$
\begin{align}
	\Sigma_t:= \lbrace (s,x) \in \mathbf{V} \ | \ s = t \rbrace.
\end{align}

\subsection{Metrics}
Most (but not all) of the article concerns spacetime metrics $\gfour$ of the form $\gfour = - n^2 dt^2 + g_{ab} dx^a dx^b.$
$n(t,x)$ is the lapse function, and $g_{ij}(t,x)$ is a Riemannian metric on $\Sigma_t.$

\subsection{Indices and determinants} \label{SS:INDICES}
Greek ``spacetime'' indices $\alpha, \beta, \cdots$ take on the values $0,1,2,3,$ while Latin ``spatial'' indices $a,b,\cdots$ 
take on the values $1,2,3.$ Repeated indices are summed over (from $0$ to $3$ if they are Greek, and from $1$ to $3$ if they are Latin). Spatial indices are lowered and raised with the Riemannian $3-$metric $g_{ij}$ and its inverse $g^{ij}.$ We never implicitly lower and raise
indices with the renormalized metric $\newg$ defined in Def.~\ref{D:RESCALEDVAR}; we always explicitly indicate the factors of
$\newg$ and $\newg^{-1}$ whenever the renormalized metric is involved in lowering or raising.

We use the notation 
\begin{align}
	\gdet 
\end{align}
to denote the determinant of the $3 \times 3$ matrix $g_{ij}.$

\subsection{Spacetime tensorfields and $\Sigma_t-$tangent tensorfields}

We denote spacetime tensorfields $\Tfour_{\nu_1 \cdots \nu_n}^{\ \ \ \ \ \ \ \ \mu_1 \cdots \mu_m}$ in bold font. 
We denote the $\gfour-$orthogonal projection of $\Tfour_{\nu_1 \cdots \nu_n}^{\ \ \ \ \ \ \ \ \mu_1 \cdots \mu_m}$ 
onto the constant-time hypersurfaces $\Sigma_t$ in non-bold font: 
$T_{b_1 \cdots b_n}^{\ \ \ \ \ \ \ \ a_1 \cdots a_m}.$ 
We also denote general $\Sigma_t-$tangent tensorfields in non-bold font.

\subsection{Coordinate systems and differential operators} \label{SS:COORDINATES}
We often work in a fixed standard local coordinate system $(x^1,x^2,x^3)$ on $\mathbb{T}^3.$ The vectorfields 
$\partial_j := \frac{\partial}{\partial x^{j}}$ are globally well-defined even though the coordinates themselves are not. 
Hence, in a slight abuse of notation, we use $\lbrace \partial_1, \partial_2, \partial_3 \rbrace$ to denote the globally defined vectorfield frame. The corresponding co-frame is denoted by $\lbrace dx^1, dx^2, dx^3 \rbrace.$ 
The spatial coordinates can be extended to a local coordinate system $(x^0,x^1,x^2,x^3)$ on manifolds-with-boundary of the form 
$(T,1] \times \mathbb{T}^3,$ and we often write $t$ instead of $x^0.$ The precise manner in which this is carried out is explained
at the beginning of Sect.~\ref{S:EEINCMC}. The corresponding vectorfield frame on $(T,1] \times \mathbb{T}^3$ 
is $\lbrace \partial_0, \partial_1, \partial_2, \partial_3 \rbrace,$ and the corresponding dual frame is
$\lbrace dx^0, dx^1, dx^2, dx^3 \rbrace.$  Relative to this frame, the FLRW metric $\widetilde{\gfour}$ is of the form 
\eqref{E:BACKGROUNDSOLUTION}. The symbol $\partial_{\mu}$ denotes the frame derivative $\frac{\partial}{\partial x^{\mu}},$ and we often write $\partial_t$ instead of $\partial_0$ and $dt$ instead of $dx^0.$ Most of our estimates and equations are 
stated relative to the frame
$\big\lbrace \partial_{\mu} \big\rbrace_{\mu = 0,1,2,3}$ and the dual frame $\big\lbrace dx^{\mu} \big\rbrace_{\mu = 0,1,2,3}.$

We use the notation $\partial f$ to denote the \emph{spatial coordinate} gradient of the function $f.$ Similarly,
if $\dlap$ is a $\Sigma_t-$ tangent one-form, then $\partial \dlap$ denotes the $\Sigma_t-$tangent type $\binom{0}{2}$
tensorfield with components $\partial_i \dlap_j$ relative to the frame described above.

If $\vec{I} = (n_1,n_2,n_3)$ is a triple of non-negative integers, then we define the \emph{spatial} multi-index coordinate differential operator $\partial_{\vec{I}}$ by $\partial_{\vec{I}} := \partial_1^{n_1} \partial_2^{n_2} \partial_3^{n_3}.$ The notation 
$|\vec{I}| := n_1 + n_2 + n_3$ denotes the order of $\vec{I}.$ 

Throughout, $\Dfour$ denotes the Levi-Civita connection of $\gfour.$ We write
\begin{align} \label{E:SPACETIMECOV}
	\Dfour_{\nu} \Tfour_{\nu_1 \cdots \nu_n}^{\ \ \ \ \ \ \ \mu_1 \cdots \mu_m} = 
		\partial_{\nu} \Tfour_{\nu_1 \cdots \nu_n}^{\ \ \ \ \ \ \ \mu_1 \cdots \mu_m} + 
		\sum_{r=1}^m \Chfour_{\nu \ \alpha}^{\ \mu_r} \Tfour_{\nu_1 \cdots \nu_n}^{\ \ \ \ \ \ \ \mu_1 \cdots \mu_{r-1} \alpha \mu_{r+1} 
			\cdots \mu_m} - 
		\sum_{r=1}^n \Chfour_{\nu \ \nu_{r}}^{\ \alpha} 
		\Tfour_{\nu_1 \cdots \nu_{r-1} \alpha \nu_{r+1} \cdots \nu_n}^{\ \ \ \ \ \ \ \ \ \ \ \ \ \ \ \ \ \ \ \ \ \ \mu_1 \cdots \mu_m}
\end{align} 
to denote a component of the covariant derivative of a tensorfield $\Tfour$ 
(with components $\Tfour_{\nu_1 \cdots \nu_n}^{\ \ \ \ \ \ \ \ \mu_1 \cdots \mu_m}$) defined on
$\mathbf{M}.$
The Christoffel symbols of $\gfour,$ which we denote by $\Chfour_{\mu \ \nu}^{\ \alpha},$ are defined in 
\eqref{E:FOURCHRISTOFFEL}.

We use similar notation to denote the covariant derivative of a $\Sigma_t-$tangent tensorfield $T$ 
(with components $T_{b_1 \cdots b_n}^{\ \ \ \ \ \ \ a_1 \cdots a_m}$) with respect to the Levi-Civita connection $\nabla$ 
of the Riemannian metric $g.$ The Christoffel symbols of $g,$ which we denote by $\Gamma_{j \ k}^{\ i},$
are defined in \eqref{E:THREECHRISTOFFEL}.

\subsection{Commutators and Lie brackets} \label{SS:COMMUTATORS}
Given two operators $A$ and $B,$
\begin{align} 
	\left[A , B\right]
\end{align}
denotes the operator commutator $A B - B A.$

If $\mathbf{X}$ and $\mathbf{Y}$ are two vectorfields, then
\begin{align}
	\mathcal{L}_{\mathbf{X}}\mathbf{Y} = [\mathbf{X}, \mathbf{Y}] 
\end{align}
denotes the Lie derivative of $\mathbf{Y}$ with respect to $\mathbf{X}.$
Relative to an arbitrary coordinate system,
\begin{align} \label{E:LIEBRACKETCOORDINATES}
	[\mathbf{X}, \mathbf{Y}]^{\mu} = \mathbf{X}^{\alpha} \partial_{\alpha} \mathbf{Y}^{\mu} 
		- \mathbf{Y}^{\alpha} \partial_{\alpha} \mathbf{X}^{\mu}.
\end{align}

\subsection{$L^2$ norms} \label{SS:L2NORMS}

All of our Sobolev norms are built out of the (spatial) $L^2$ norms of scalar quantities 
(which may be the components of a tensorfield). If $f$ 
is a function defined on the hypersurface $\Sigma_t,$ then we define the standard $L^2$ norm $\big\| f \big\|_{L^2}$ over $\Sigma_t$ as follows:
\begin{align} \label{E:SOBOLEVNORMDF}
	\big\| f \big\|_{L^2} = \big\| f \big\|_{L^2}(t) := \bigg(\int_{\mathbb{T}^3} \big| f(t,x^1,x^2,x^3) \big|^2 dx \bigg)^{1/2}.
\end{align}
Above, the notation $``\int_{\mathbb{T}^3} f \, dx"$ denotes the integral of $f$ over $\mathbb{T}^3$ with respect to the measure corresponding to the volume form of the \emph{standard Euclidean metric $\Euc$} on $\mathbb{T}^3,$ which has the components
$\Euc_{ij} = \mbox{diag}(1,1,1)$ relative to the frame defined in Sect.~\ref{SS:COORDINATES}.

\subsection{Constants} \label{SS:RUNNINGCONSTANTS}
We use $C,c,$ etc. to denote positive numerical constants that are free to vary from line to line. We allow $C,c$ to depend on $N,$
but we always choose $C,$ $c$ so that they are independent of all functions that are sufficiently close to 
the FLRW solution. We sometimes use notation such as $C_N$ when we want to explicitly indicate that $C$ depends on $N.$ 
We use symbols such as $C_*,$ $c_*$ etc., to denote constants that play a distinguished role in the analysis. If $A$ and $B$ are two quantities, then we often write 
\begin{align}
	A \lesssim B
\end{align}
whenever there exists a constant $C > 0$ such that $A \leq C B.$ 
Furthermore, if $A \lesssim B$ and $B \lesssim A,$ then we 
often write
\begin{align}
	A \approx B.
\end{align}

\section{Constant Mean Curvature-Transported Spatial Coordinates} \label{S:EEINCMC}

In this section, we provide the Einstein-stiff fluid equations relative to CMC-transported spatial coordinates. 
For additional details concerning the derivation of the equations starting from the system
\eqref{E:EINSTEININTRO}-\eqref{E:EOS} (with $\speed = 1$), see Appendix \ref{A:CMCDERIVATION}. 

Before stating the equations, we first define the variables that play a role in the standard CMC-transported spatial coordinates formulation. We begin by discussing the spatial coordinates. We assume that $(\mathbf{M},\gfour)$ is a cosmological spacetime 
containing a region $\mathbf{V}$ that is foliated by spacelike hypersurfaces $\Sigma_t,$ where $t \in (T,1]$ is a time function. In this article, $\Sigma_t = \mathbb{T}^3,$ i.e., $\mathbf{V} = (T,1] \times \mathbb{T}^3.$ 
We will soon impose the condition that the $\Sigma_t$ are CMC hypersurfaces. The existence of CMC hypersurfaces in the spacetimes of interest is guaranteed by Prop.~\ref{P:CMCEXISTS}. Let $(x^1,x^2,x^3)$ denote local coordinates on
a neighborhood $\mathcal{U} \subset \Sigma_1.$ We can extend these spatial coordinates to a
spacetime coordinate system $(t,x^1,x^2,x^3)$ on $(T,1] \times \mathcal{U} \subset \mathbf{V}$ by requiring that
$- \Nml x^i = 0$ for $i = 1,2,3.$ Here, $\Nml$ is the future-directed normal to $\Sigma_t,$ and we are slightly
abusing notation by using the symbol ``$x^i$'' to denote both the transported coordinate function and also its 
restriction to $\Sigma_1.$ This provides us with a coordinate system on $(T,1] \times \mathcal{U} \subset \mathbf{V}.$
On $(T,1] \times \mathcal{U},$ $\gfour$ can be uniquely decomposed into a lapse function $n$ and a Riemannian $3-$metric $g$ 
on $\Sigma_t$ as follows:
\begin{subequations}
\begin{align}
	\gfour & = - n^2 dt \otimes dt + g_{ab} dx^a \otimes dx^b, \label{E:GFOURDECOMP} \\
	\gfour^{-1} & = - n^{-2} \partial_t \otimes \partial_t 
		+ g^{ab} \partial_a \otimes \partial_b. \label{E:GINVERSEFOURDECOMP}
\end{align}	
\end{subequations}
Above, $g^{ij}$ denotes the inverse of $g_{ij}.$ The future-directed normal to $\Sigma_t$ is 
\begin{align}
	\Nml & = n^{-1} \partial_t.
\end{align}
We denote the Levi-Civita connection of $\gfour$ by $\Dfour$ and that of $g$ by $\nabla.$

We assume that there is a stiff fluid present in $\mathbf{V}.$ The fluid's four-velocity can be decomposed as
\begin{align} \label{E:UNORMAL}
	\ufour = (1 + u_a u^a)^{1/2} \Nml + u^a \partial_a,
\end{align}
where the factor $(1 + u_a u^a)^{1/2}$ enforces the normalization condition \eqref{E:UNORMALIZED}. The energy-momentum tensor
\eqref{E:TFOURFLUID} of the stiff fluid can be decomposed (with the indices ``downstairs'') as
\begin{align}
	\Tfour = \Tfour(\Nml,\Nml) \Nml_{\flat} \otimes \Nml_{\flat} 
		- \Tfour(\Nml,\partial_a) \left(\Nml_{\flat} \otimes dx^a + dx^a \otimes \Nml_{\flat} \right) 
		+ T_{ab} dx^a \otimes dx^b,
\end{align}
where $(\Nml_{\flat})_{\mu} := \gfour_{\mu \alpha} \Nml^{\alpha}$ is the metric dual of $\Nml,$
\begin{subequations}
\begin{align}
	\Tfour(\Nml,\Nml) & = p + 2p (1 + u_a u^a), 
		\label{E:TFOURNORMAL} \\
	\Tfour(\Nml,\partial_i) & = - 2p (1 + u_a u^a)^{1/2} u_i, 
		\label{E:TFOURNORMALSPATIAL} \\
	T_{ij} & = 2p u_i u_j + p g_{ij}. \label{E:TFOURSPATIAL}
\end{align}
\end{subequations}

The second fundamental form $\SecondFund$ of $\Sigma_t$ is defined by requiring that following relation 
holds for all vectorfields $X,Y$ tangent to $\Sigma_t:$
\begin{align} \label{E:SECONDFUNDDEF}
	\gfour(\Dfour_X \Nml, Y) = - \SecondFund(X,Y).
\end{align}
It is a standard fact that $\SecondFund$ is symmetric:
\begin{align}
	 \SecondFund(X,Y) =  \SecondFund(Y,X).
\end{align}
For such $X,Y,$ the action of the spacetime connection $\Dfour$ can be decomposed into
the action of $\nabla$ and $\SecondFund$ as follows:
\begin{align} \label{E:DDECOMP}
	\Dfour_X Y = \nabla_X Y - \SecondFund(X,Y)\Nml.
\end{align}

\begin{remark} \label{R:ALWAYSMIXED}
	When analyzing the components of $\SecondFund,$ \textbf{we will always assume that it is written in mixed form 
	as $\SecondFund_{\ j}^i$ with the first index upstairs and the second one downstairs.} This convention is absolutely essential for 
	some of our analysis; in the problem of interest to us, the evolution and constraint equations verified by the components 
	$\SecondFund_{\ j}^i$ have a more favorable structure than the corresponding equations verified by $\SecondFund_{ij}.$
\end{remark}

Throughout the vast majority of our analysis, we normalize the CMC hypersurfaces $\Sigma_t$ as follows:
\begin{align} \label{E:CMCCONDITION}
	\SecondFund_{\ a}^a & = - \frac{1}{t},&& t \in (0,1].
\end{align}
In order for \eqref{E:CMCCONDITION} to hold, the lapse has to verify the elliptic equation \eqref{E:LAPSE}.

A slightly inconvenient fact is the following: $\mathbb{T}^3$ cannot be covered by a single coordinate chart.
One coordinate chart can cover all but a measure $0$ set, but it takes several to cover all of $\mathbb{T}^3.$ 
Hence, we need more than one coordinate chart to carry out our analysis. This motivates the following definition.

\begin{definition}[\textbf{Standard atlas, standard charts, standard coordinates, and Euclidean metric}] \label{D:STANDARDATLAS}
We fix a finite collection of charts $\lbrace (\mathcal{U}_m, (x_m^1,x_m^2,x_m^3)) \rbrace_{m=1}^{M}$ that cover
$\mathbb{T}^3.$ We require that the transition maps $x_{m_1} \circ x_{m_2}^{-1}$ are \emph{translations} defined on subsets of 
$\mathbb{R}^3.$ We refer to this collection as \emph{the standard atlas} on $\mathbb{T}^3$ with \emph{standard charts} and
\emph{standard coordinates}. 

We then fix the Euclidean metric $\Euc$ on $\mathbb{T}^3.$ We can view $\Euc$ either as 
tensorfield inherent to $\mathbb{T}^3$ or as a spacetime tensorfield
that is defined along the hypersurface $\Sigma_1 \simeq \mathbb{T}^3$ and that 
has vanishing components in the direction $\Nml.$
From the former point of view, $\Euc$ has components 
$\Euc_{ij} = \mbox{diag}(1,1,1)$ relative to any of the standard charts. 
\end{definition}

Using the above construction for transporting coordinates and the standard atlas, we can construct 
\emph{transported spatial coordinates} that, together with $t,$ 
cover the spacetime region $\mathbf{V} = (T,1] \times \mathbb{T}^3.$ 
Each spacetime chart is of the form $((T,1] \times \mathcal{U}_m, (t,x_m^1,x_m^2,x_m^3)).$ We refer to this collection as the \emph{the standard atlas} on $(T,1] \times \mathbb{T}^3$ with \emph{standard charts}. 

The Euclidean metric can be extended to each leaf $\Sigma_t,$ $t \in (T,1],$ by requiring that $\underline{\mathcal{L}}_{-\Nml} \Euc = 0$
and that any contraction of $\Euc$ with $\Nml$ vanishes. Here, $\underline{\mathcal{L}}_{-\Nml} \Euc$ denotes the restriction of the spacetime tensorfield $\mathcal{L}_{-\Nml} \Euc$ to $\Sigma_t.$ Given any tensorfield $T$ tangent to the $\Sigma_t,$ \emph{the components of $T$ relative to the coordinate frames on the overlapping regions $\lbrace t \rbrace \times \mathcal{U}_{m_1} \cap \lbrace t \rbrace \times \mathcal{U}_{m_2}$ are independent of the chart}. Hence, one can morally carry out the analysis as if there were only a single chart.

We never need to directly refer to the spatial coordinates, but rather only the components of tensorfields relative to the
coordinate frames and also coordinate partial derivatives of these components. Hence, there
is an alternative way to think about the above construction that is useful for our ensuing analysis. 
One can imagine that we have fixed a \emph{globally defined smooth holonomic frame field} $\lbrace e_{(1)}, e_{(2)}, e_{(3)} \rbrace$ on $\Sigma_1$ that is orthonormal with respect to the Euclidean metric. The corresponding dual frame field, which we denote by $\lbrace \theta^{(1)}, \theta^{(2)}, \theta^{(3)} \rbrace,$ verifies $\theta^{(i)}\big(e_{(j)} \big) = \delta_{j}^i,$ where $\delta_{j}^i$ is the Kronecker delta. Relative to \emph{any} of the standard charts on $\Sigma_1,$ we have $e_{(i)} = \partial_i,$ $\theta^{(i)} = dx^i.$ 
The frame $\lbrace e_{(1)}, e_{(2)}, e_{(3)} \rbrace,$ can be extended to each leaf $\Sigma_t,$ $t \in (T,1],$ by requiring that
$\mathcal{L}_{-\mathbf{N}} e_{(i)} = 0,$ $(i=1,2,3).$ Here, $\mathbf{N}$ is a renormalized version of $\Nml$ such that
$\mathbf{N} t = 1$ (i.e., $\mathbf{N} = \frac{\partial}{\partial t}$ relative to any local transported coordinate system).
Similarly, the dual frame $\lbrace \theta^{(1)}, \theta^{(2)}, \theta^{(3)} \rbrace$ can be extended to each leaf $\Sigma_t,$ $t \in (T,1],$ by requiring that $\underline{\mathcal{L}}_{-\mathbf{N}} \theta^{(i)} = 0,$ $(i=1,2,3).$

Any type $\binom{s}{r}$ tensorfield $T$ that is tangent to the $\Sigma_t$ can be decomposed as 
$T = T_{m_1 \cdots m_r}^{\ \ \ \ \ \ \ \ n_1 \cdots n_s} e_{(n_1)} \otimes \cdots \otimes e_{(n_s)} \otimes \theta^{(m_1)} \otimes \cdots \otimes \theta^{(m_r)}.$ The components $T_{m_1 \cdots m_r}^{\ \ \ \ \ \ \ \ n_1 \cdots n_s}(t,x)$ are globally defined functions on
$\Sigma_t.$ We can then view the Einstein-stiff fluid equations in CMC-transported spatial coordinates, which are presented just below, as PDES in the components of various tensorfields. Furthermore, we can view our estimates as estimates of these components and their 
frame derivatives $e_{(a)} T_{m_1 \cdots m_r}^{\ \ \ \ \ \ \ \ n_1 \cdots n_s} 
:= e_{(a)}^b \partial_b T_{m_1 \cdots m_r}^{\ \ \ \ \ \ \ \ n_1 \cdots n_s}.$

Having established the above conventions, we now state the equations.

\begin{proposition}[\textbf{The Einstein-stiff fluid equations in CMC-transported spatial coordinates}]
In CMC-transported spatial coordinates normalized by $\SecondFund_{\ a}^a = - \frac{1}{t},$ 
the Einstein-stiff fluid system consists of the following equations.

The \textbf{constraint equations} verified by $g_{ij}, \SecondFund_{\ j}^i, p,$ and $u^i$ are:
\begin{subequations}
\begin{align}
		R - \SecondFund_{\ b}^a \SecondFund_{\ a}^b + \underbrace{(\SecondFund_{\ a}^a)^2}_{t^{-2}} 
		& = \overbrace{2 p + 4p u_a u^a}^{2 \Tfour(\Nml,\Nml)}, \label{E:HAMILTONIAN} \\
		\nabla_a \SecondFund_{\ i}^a - \underbrace{\nabla_i \SecondFund_{\ a}^a}_0 & = 
		\underbrace{2p (1 + u_a u^a)^{1/2} u_i}_{- \Tfour(\Nml,\partial_i)}. \label{E:MOMENTUM}
\end{align}
\end{subequations}

The \textbf{evolution equations} verified by $g_{ij}$ and $\SecondFund_{\ j}^i$ are:
\begin{subequations}
\begin{align}
	\partial_t g_{ij} & = - 2 n g_{ia}\SecondFund_{\ j}^a, \label{E:PARTIALTGCMC} \\
	\partial_t (\SecondFund_{\ j}^i) & = - g^{ia} \nabla_a \nabla_j n
		+ n \Big\lbrace R_{\ j}^i + \underbrace{\SecondFund_{\ a}^a}_{-t^{-1}} \SecondFund_{\ j}^i 
			\underbrace{- 2p u^i u_j}_{- T_{\ j}^i + (1/2)\ID_{\ j}^i \Tfour} 
			\Big\rbrace,  \label{E:PARTIALTKCMC}
\end{align}
\end{subequations}
where $R$ denotes the scalar curvature of $g_{ij},$ $R_{\ j}^i$ denotes the Ricci curvature of $g_{ij}$
(a precise expression is given in Lemma~\ref{L:RICCIDECOMP}),
and $\ID_{\ j}^i = \mbox{diag}(1,1,1)$ denotes the identity transformation.

%{\color{red} INSERT}

The \textbf{stiff fluid equations} (i.e., the Euler equations with $p = \rho$) are:
\begin{subequations}
\begin{align}
	(1 + u_a u^a)^{1/2} \partial_t p + n u^a \nabla_a p
		& + 2p \Big\lbrace (1 + u_a u^a)^{-1/2} u_b \partial_t u^b + n \nabla_a u^a \Big\rbrace  
			\label{E:EULERPCMC} \\
		& = 2p \Big\lbrace -\frac{n}{t} (1 + u_a u^a)^{1/2}
			+ n (1 + u_a u^a)^{-1/2} \SecondFund_{\ f}^e u_e u^f
			- u^a \nabla_a n \Big\rbrace \notag, \\
		2p \Big\lbrace (1 + u_a u^a)^{1/2} \partial_t u^j + n u^a \nabla_a u^j \Big\rbrace
		& + (1 + u_a u^a)^{1/2} u^j \partial_t p + n(g^{ja} + u^j u^a)\nabla_a p 
			\label{E:EULERUCMC} \\
		& = 4 n p (1 + u_a u^a)^{1/2} \SecondFund_{\ b}^j  u^b - 2p (1 + u_a u^a) g^{jb} \nabla_b n.
		\notag
\end{align}
\end{subequations}

The \textbf{lapse equation} (for a general perfect fluid\footnote{We state the lapse equation in the case of a general
perfect fluid because we referred to this equation earlier in the article. However, our main results only
apply in the case of the stiff fluid equation of state $p = \rho.$}) is:
\begin{align} \label{E:LAPSE}
	g^{ab} \nabla_a \nabla_b (n - 1) 
		& = (n - 1) \Big\lbrace R + \underbrace{(\SecondFund_{\ a}^a)^2}_{t^{-2}} - (\rho + p)u_a u^a + \frac{3}{2} (p - \rho) 
		\Big\rbrace \\
	& \ \ + R - (\rho + p)u_a u^a + \frac{3}{2} (p - \rho)
			+ \underbrace{(\SecondFund_{\ a}^a)^2 - \partial_t (\SecondFund_{\ a}^a)}_{0}. \notag
\end{align}

\end{proposition}

\hfill $\qed$

\section{The Equations Verified by the Renormalized Variables} \label{S:RESCALEDVARIABLES}
In this section, we reformulate the Einstein-stiff fluid CMC-transported spatial coordinate equations in terms of 
the renormalized variables of Def.~\ref{D:RESCALEDVAR}.
We decompose the resulting equations into main terms, borderline error terms that must be handled with care,
and junk error terms that are easy to control.
In Appendix \ref{A:RESCALEDEQNS}, we provide
a more detailed derivation of the equations;
here, we only state them.
The two main merits of working with the renormalized variables are
\textbf{i)} they make the time dependence of the FLRW background solution explicit and thus help
us to identity order $1$ and approximately order $1$ quantities in the study of perturbations; 
\textbf{ii)} they yield equations with a favorable structure. 

\subsection{Constraint equations for the renormalized variables}

\begin{proposition}[\textbf{The renormalized constraints}] \label{P:CCONSTRAINTSREFORM}
In terms of the renormalized variables of Def.~\ref{D:RESCALEDVAR}, the constraint equations
\eqref{E:HAMILTONIAN}-\eqref{E:MOMENTUM} can be expressed as follows:
\begin{subequations}
\begin{align}
	R & = 2 t^{-2} \adjustednewp 
		 + t^{-2} \leftexp{(Border)}{\mathfrak{H}} + t^{-2/3} \leftexp{(Junk)}{\mathfrak{H}}, 
		\label{E:RINTERMSOFKPANDU} \\
	\partial_a \freenewsec_{\ i}^a & = 
			\frac{2}{3} \newg_{ia} \newu^a  
			+ \leftexp{(Border)}{\mathfrak{M}}_i +  t^{4/3} \leftexp{(Junk)}{\mathfrak{M}}_i,
			\label{E:MOMENTUMCONSTRAINT}
\end{align}
where $R$ denotes the Ricci curvature of the non-rescaled metric $g,$
and the error terms $\leftexp{(Border)}{\mathfrak{H}},$ $\leftexp{(Junk)}{\mathfrak{H}},$
$\leftexp{(Border)}{\mathfrak{M}}_i,$ and $\leftexp{(Junk)}{\mathfrak{M}}_i$ 
are defined in \eqref{AE:HAMILTONIANERRORBORDER}, \eqref{AE:HAMILTONIANERRORJUNK},
\eqref{AE:MOMENTUMERRORBORDERLINE}, and \eqref{AE:MOMENTUMERRORJUNK}.

Furthermore, the following alternative version of \eqref{E:MOMENTUMCONSTRAINT} holds:
\begin{align} 
	(\newg^{-1})^{ia} \partial_a \freenewsec_{\ i}^j 
	& = \frac{2}{3} \newu^j 
		+ \leftexp{(Border)}{\widetilde{\mathfrak{M}}}^j + t^{4/3} \leftexp{(Junk)}{\widetilde{\mathfrak{M}}}^j, 
		\label{E:MOMENTUMCONSTRAINTRAISED}
\end{align}
\end{subequations}
where the error terms $\leftexp{(Border)}{\widetilde{\mathfrak{M}}}^j$ and $\leftexp{(Junk)}{\widetilde{\mathfrak{M}}}^j$ are defined in \eqref{AE:MOMENTUMALTERRORBORDERLINE} and \eqref{AE:MOMENTUMALTERRORJUNK}.

\end{proposition}

\hfill $\qed$

\subsection{The elliptic equations verified by the renormalized lapse variables}

The following negative-definite linear elliptic operator plays a fundamental role in our analysis of 
the lapse.

\begin{definition} [\textbf{The elliptic operator $\mathcal{L}$}] \label{D:LAPSEOP}
\begin{subequations}
\begin{align} \label{E:LDEF}
	\mathcal{L} & := t^{4/3} (\newg^{-1})^{ab} \partial_a \partial_b 
		- (1 + f), \\
	f & := 2 \adjustednewp 
			+ \freenewsec_{\ b}^a \freenewsec_{\ a}^b
		+ 2 t^{4/3} \adjustednewp \newg_{ab} \newu^a \newu^b
		+ \frac{2}{3} t^{4/3} \newg_{ab} \newu^a \newu^b. 
		\label{E:ELLIPTICOPERATORJUNKTERM}
\end{align}
\end{subequations}

Alternatively, with the help of equations \eqref{E:RICCIDECOMP} and
\eqref{AE:RINTERMSOFKPANDU}, $\mathcal{L}$ can be expressed as follows:
\begin{subequations}
\begin{align} \label{E:LWITHFTILDEDEF}
	\mathcal{L} & = t^{4/3} (\newg^{-1})^{ab} \partial_a \partial_b 
		- (1 + \widetilde{f}), \\
	\widetilde{f} & :=  
			- \frac{1}{2} t^{4/3} (\newg^{-1})^{ef} \partial_e \upgamma_{f \ a}^{\ a}  
			+ t^{4/3} (\newg^{-1})^{ab} \partial_a \Gamma_b 
			+ t^{4/3}  \leftexp{(Ricci)}{\triangle}_{\ a}^a 
			\label{E:TILDEELLIPTICOPERATORJUNKTERM} 
			\\
			& \ \ - 2 t^{4/3} \adjustednewp \newg_{ab} \newu^a \newu^b
		 - \frac{2}{3} t^{4/3} \newg_{ab} \newu^a \newu^b, \notag
\end{align}
\end{subequations}
where the error term $\leftexp{(Ricci)}{\triangle}_{\ j}^i$ is defined in \eqref{E:RICCIERROR}.

\end{definition}

\begin{proposition} [\textbf{The equations verified by the renormalized lapse variable}] \label{P:LAPSEALT}
Assume the stiff fluid equation of state $p = \rho.$
In terms of the renormalized variables of Def.~\ref{D:RESCALEDVAR}, the lapse equation \eqref{E:LAPSE}
can be expressed in the following two forms:
\begin{subequations}
\begin{align} \label{E:LAPSERESCALEDELLIPTIC}
	\mathcal{L} \newlap 
	& = 2 t^{- 4/3} \adjustednewp 
			+ t^{- 4/3} \leftexp{(Border)}{\mathfrak{N}} 
			+ \leftexp{(Junk)}{\mathfrak{N}}, 
		\\
	\mathcal{L} \newlap & = \leftexp{(Border)}{\widetilde{\mathfrak{N}}} + t^{2/3} \leftexp{(Junk)}{\widetilde{\mathfrak{N}}}, 
	\label{E:LAPSELOWERDERIVATIVES}
\end{align}
\end{subequations}
where the error terms $\leftexp{(Border)}{\mathfrak{N}},$ $\leftexp{(Junk)}{\mathfrak{N}},$ $\leftexp{(Border)}{\widetilde{\mathfrak{N}}},$
and $\leftexp{(Junk)}{\widetilde{\mathfrak{N}}}$ are defined in \eqref{AE:LAPSENBORDERLINE}, \eqref{AE:LAPSENJUNK},
\eqref{AE:LAPSENTILDEBORDERLINE}, and \eqref{AE:LAPSENTILDEJUNK}.
\end{proposition}

\hfill $\qed$

\subsection{Evolution equations for the renormalized variables}

\subsubsection{Evolution equation for the renormalized volume form factor}

\begin{lemma} [\textbf{Evolution equation for} $\sqrt{\gbigdet}$] \label{L:RESCALEDVOLFORMEVOLUTION}
The renormalized volume form factor $\sqrt{\gbigdet}$ 
of Def.~\ref{D:RESCALEDVAR}
verifies the following evolution equation:
\begin{align} \label{E:LOGVOLFORMEVOLUTION}
	\partial_t \ln \sqrt{\gbigdet} & = t^{1/3} \newlap.
\end{align}

\end{lemma}

\hfill $\qed$

\subsubsection{Evolution equations for the renormalized metric variables}

\begin{proposition} [\textbf{The renormalized metric evolution equations}] \label{P:METRICEVOLUTIONREFORMULATED}
The renormalized metric $\newg_{ij}$ and its inverse $(\newg^{-1})^{ij}$ 
of Def.~\ref{D:RESCALEDVAR}
verify the following evolution equations:
\begin{subequations}
\begin{align} \label{E:GEVOLUTION}
	\partial_t \newg_{ij}
	& = - 2t^{-1} \newg_{ia} \freenewsec_{\ j}^a
		+ t^{1/3} \leftexp{(Junk)}{\mathfrak{G}}_{ij},
		\\
 \partial_t (\newg^{-1})^{ij} & = 2 t^{-1} (\newg^{-1})^{ia} \freenewsec_{\ a}^j 
 		+ t^{1/3} \leftexp{(Junk)}{\widetilde{\mathfrak{G}}}^{ij},
		\label{E:GINVERSEEVOLUTION}
\end{align}
\end{subequations}
where the error terms $\leftexp{(Junk)}{\mathfrak{G}}_{ij}$ and $\leftexp{(Junk)}{\widetilde{\mathfrak{G}}}^{ij}$ are 
defined in \eqref{E:METRICEVOLUTIONJUNK}-\eqref{E:INVERSEMETRICEVOLUTIONJUNK}.

Furthermore, the quantities $\upgamma_{e \ i}^{\ b}$ and $\freenewsec_{\ j}^i$
verify the following evolution equations:
\begin{subequations}
\begin{align}
	\partial_t \upgamma_{e \ i}^{\ b} 
	& = - 2 t^{-1} \big[1 + t^{4/3} \newlap\big] \partial_e \freenewsec_{\ i}^b 
		+ \frac{2}{3}t^{-1/3} \dlap_e \ID_{\ i}^b 	
		\label{E:LITTLEGAMMAEVOLUTIONDETGFORM} \\
	& \ \ + t^{-1} \leftexp{(Border)}{\mathfrak{g}}_{e \ i}^{\ b}
		+ t^{1/3} \leftexp{(Junk)}{\mathfrak{g}}_{e \ i}^{\ b}, \notag \\
	\partial_t \freenewsec_{\ j}^i 
	& = - \frac{1}{2} t^{1/3} 
		\big[1 + t^{4/3} \newlap\big] (\newg^{-1})^{ef} \partial_e \upgamma_{f \ j}^{\ i}
			\label{E:RESCALEDKEVOLUTION} \\
	& \ \ + \frac{1}{2} t^{1/3} \big[1 + t^{4/3} \newlap\big] 
		 (\newg^{-1})^{ia} \partial_a \Big( \underbrace{\newg_{jb} (\newg^{-1})^{ef} \upgamma_{e \ f}^{\ b} 
				- \frac{1}{2} \upgamma_{j \ b}^{\ b}}_{\Gamma_j} \Big)  \notag \\
	& \ \  + \frac{1}{2} t^{1/3} \big[1 + t^{4/3} \newlap\big]  
		(\newg^{-1})^{ia}\partial_j \Big( \underbrace{\newg_{ab} (\newg^{-1})^{ef} \upgamma_{e \ f}^{\ b} 
				- \frac{1}{2} \upgamma_{a \ b}^{\ b}}_{\Gamma_a} \Big) \notag \\
	& \ \ - t (\newg^{-1})^{ia} \partial_a \dlap_j 
		+ \frac{1}{3} t^{1/3} \newlap \ID_{\ j}^i
		+ t^{1/3} \leftexp{(Junk)}{\mathfrak{K}}_{\ j}^i, \notag
\end{align}
\end{subequations}
where the error terms $\leftexp{(Border)}{\mathfrak{g}}_{e \ i}^{\ b},$
$\leftexp{(Junk)}{\mathfrak{g}}_{e \ i}^{\ b},$ 
and $\leftexp{(Junk)}{\mathfrak{K}}_{\ j}^i$ are 
defined in \eqref{AE:GAMMAEVOLUTIONERROR}, 
\eqref{AE:JUNKGAMMAEVOLUTIONERROR},
and \eqref{AE:KEVOLUTIONERROR}. Above, 
$\Gamma_j$ denotes a contracted Christoffel symbol of the renormalized metric $\newg_{ij}.$

\end{proposition}

\hfill $\qed$

%In particular, 

%\begin{align}
%	\partial_t \upgamma_{e \ a}^{\ a}
%		& = 2 t^{-1/3} \dlap_e. \label{E:SPATIALDERIVATIVESOFVOLUMEFORMEVOLUTION} 
%\end{align}

\subsubsection{Evolution equations for the renormalized stiff fluid variables}

\begin{proposition}  [\textbf{The renormalized stiff fluid evolution equations}] \label{P:EULERALT}
In terms of the renormalized variables of Def.~\ref{D:RESCALEDVAR}, the stiff fluid equations 
(i.e., the Euler equations with $\speed = 1$) can be decomposed as follows:
\begin{subequations}
\begin{align}
	\partial_t \adjustednewp 
		& + 2 t^{1/3} \big[1 + t^{4/3} \newlap\big]\big[1 + t^{4/3} \newg_{ab} \newu^a \newu^b\big]^{1/2} 
			\big[\adjustednewp + \frac{1}{3} \big] 
			\partial_c \newu^c \label{E:RESCALEDPEVOLUTION} \\
		& - 2 t^{5/3} \frac{\big[1 + t^{4/3} \newlap\big] \big[\adjustednewp + \frac{1}{3} \big]}{\big[1 + t^{4/3} \newg_{ab} \newu^a 
			\newu^b\big]^{1/2}} \newg_{ef} \newu^e \newu^c \partial_c \newu^f	\notag \\
		& = - \frac{2}{3} t^{1/3} \newlap 
			+ t^{1/3} \leftexp{(Junk)}{\mathfrak{P}}, \notag 
\end{align}
\begin{align}
	\partial_t \newu^j 
		& - t^{1/3} \big[1 + t^{4/3} \newlap\big] \big[1 + t^{4/3} \newg_{ab} \newu^a \newu^b\big]^{1/2} \newu^j \partial_c \newu^c 
			\label{E:RESCALEDUEVOLUTION} \\
		& + t^{5/3} \frac{\big[1 + t^{4/3} \newlap\big]}{\big[1 + t^{4/3} \newg_{ab} \newu^a \newu^b\big]^{1/2}} 
			\newg_{ef} \newu^j \newu^e \newu^c \partial_c \newu^f \notag \\
	& + t^{1/3} \frac{\big[1 + t^{4/3} \newlap\big] \newu^c \partial_c \newu^j}{\big[1 + t^{4/3} \newg_{ab} \newu^a \newu^b\big]^{1/2}} 
		\notag \\
	& + t^{-1} \frac{\big[1 + t^{4/3} \newlap\big] \left\lbrace (\newg^{-1})^{jc} + t^{4/3} \newu^j \newu^c \right\rbrace \partial_c 
			\adjustednewp}{2\big[1 + t^{4/3} \newg_{ab} \newu^a \newu^b\big]^{1/2} 
			\big[\adjustednewp + \frac{1}{3} \big]} \notag \\
	& = - t^{-1/3} (\newg^{-1})^{j a} \dlap_a 
		+ t^{-1} \leftexp{(Border)}{\mathfrak{U}}^j + t^{1/3} \leftexp{(Junk)}{\mathfrak{U}}^j, \notag 
\end{align}
\end{subequations}
where the error terms $\leftexp{(Junk)}{\mathfrak{P}},$ $\leftexp{(Border)}{\mathfrak{U}}^j,$ and $\leftexp{(Junk)}{\mathfrak{U}}^j$ are are defined in 
\eqref{AE:PERROR}, \eqref{AE:UJBORDERLINE}, and \eqref{AE:UJJUNK}.

\end{proposition}
\hfill $\qed$

\subsection{The commuted renormalized equations}
\label{SS:COMMUTEDRENORMALIZEDBRIEF}
To complete our analysis, we must commute the equations of Sect.~\ref{S:RESCALEDVARIABLES} 
with the differential operators $\partial_{\vec{I}};$ 
see the notation defined in Sect.~\ref{SS:COORDINATES}. 
The commuted equations are straightforward to derive but lengthy to state.
Hence, to avoid impeding the flow of the paper, we have relegated 
this material to Appendix \ref{A:COMMUTED}.

\section{Norms, Equations of Variation, and Energy Currents} \label{S:NORMSANDCURRENTS}
In this section, we introduce the $C^M$ norms and Sobolev norms that
play a fundamental role in our analysis of the renormalized solution variables.
Next, we introduce the equations of variation, which are the PDEs verified by the 
$\partial_{\vec{I}}-$differentiated quantities.
We then define the metric and fluid energy currents, which are vectorfields that depend
quadratically on the $\partial_{\vec{I}}-$differentiated quantities. The currents will be used in
Sect.~\ref{S:FUNDAMENATLENERGYINEQUALITIES} via the divergence theorem to derive energy integral inequalities 
for the $\partial_{\vec{I}}-$differentiated quantities. 
Finally, for use in Sect.~\ref{S:FUNDAMENATLENERGYINEQUALITIES}, 
given a solution to the equations of variation, we compute the divergences of the corresponding energy currents.

\subsection{Norms}  \label{SS:NORMS}

We will derive strong $C^M$ estimates for the lower-order derivatives of the solution variables by analyzing 
their components relative to the transported spatial coordinate frame. However, in order to close our energy estimates, 
we will also need to measure the size of tensors by using geometric norms corresponding to the metric $\newg.$ 
We therefore introduce the following definitions.

\begin{definition}  [\textbf{Pointwise norms}] \label{D:POINTWISENORMS}
	Let $T$ be a $\Sigma_t-$tangent tensor with components $T_{b_1 \cdots b_n}^{\ \ \ \ \ \ \ \ a_1 \cdots a_m}$
	relative to our standard atlas (see Def.~\ref{D:STANDARDATLAS}) on $\mathbb{T}^3.$ 
	Then $|T|_{Frame}$ denotes a norm of the \emph{components} of $T$ relative to the standard atlas:
	\begin{subequations}
	\begin{align}
		|T|_{Frame}^2 :=  \sum_{a_1 = 1}^3 \cdots \sum_{a_m = 1}^3 \sum_{b_1 = 1}^3 \sum_{b_n = 1}^3  
			\left|T_{b_1 \cdots b_n}^{\ \ \ \ \ \ \ \ a_1 \cdots a_m} \right|^2.
	\end{align}
	
	$|T|_{\newg}$ denotes the $\newg-$norm of $T,$ where $\newg$ is the renormalized spatial metric:
	\begin{align} \label{E:NEWGNORM}
		|T|_{\newg}^2 :=  
		  \newg_{a_1 a_1'} \cdots \newg_{a_m a_m'}
		  (\newg^{-1})^{b_1 b_1'} \cdots (\newg^{-1})^{b_n b_n'} 
			T_{b_1 \cdots b_n}^{\ \ \ \ \ \ \ \ a_1 \cdots a_m}
			T_{b_1' \cdots b_n'}^{\ \ \ \ \ \ \ \ a_1' \cdots a_m'}.
	\end{align}
	\end{subequations}
	
	In proving Theorem~\ref{T:BIGBANG}, we also use the norm 
	$|\cdot|_g,$ which is defined like \eqref{E:NEWGNORM} but with $g$ in place of $\newg.$
		
	%Similarly, if $\Tfour$ is a spacetime tensor with components
	%$\Tfour_{\alpha_1 \cdots \alpha_m}^{\ \ \ \ \ \beta_1 \cdots \beta_n},$ then we define
	%the Euclidean norm of $\Tfour$ by
	
	%	\begin{align} \label{E:TSPACETIMEEUCLIDNORM}
	%	|\Tfour|_{\Eucfour}^2 := \Eucfour_{\beta_1 \beta_1'} \cdots \Eucfour_{\beta_n \beta_n'} (\Eucfour^{-1})^{\alpha_1 \alpha_1'} \cdots 
	%		(\Eucfour^{-1})^{\alpha_m \alpha_m'} \Tfour_{\alpha_1 \cdots \alpha_m}^{\ \ \ \ \ \ \ \ \ \beta_1 \cdots \beta_n} 
	%		\Tfour_{\alpha_1' \cdots \alpha_m'}^{\ \ \ \ \ \ \ \ \ \beta_1' \cdots \beta_n'},
	%\end{align}
	%where $\Eucfour_{\mu \nu}:= \mbox{diag}(1,1,1,1)$ is the standard Euclidean metric on $(0,\infty) \times \mathbb{T}^3.$
	
	Similarly, if $\Tfour$ is a spacetime tensor with components
	$\Tfour_{\beta_1 \cdots \beta_n}^{\ \ \ \ \ \ \ \ \ \alpha_1 \cdots \alpha_m},$ then we define the $\gfour$ ``norm'' of $\Tfour$ by
	\begin{align} \label{E:TSPACETIMEGFOURNORM}
		|\Tfour|_{\gfour}^2 :=  
			\gfour_{\alpha_1 \alpha_1'} \cdots \gfour_{\alpha_m \alpha_m'} 
			(\gfour^{-1})^{\beta_1 \beta_1'} \cdots (\gfour^{-1})^{\beta_n \beta_n'}
			\Tfour_{\beta_1 \cdots \beta_n}^{\ \ \ \ \ \ \ \ \ \alpha_1 \cdots \alpha_m} 
			\Tfour_{\beta_1' \cdots \beta_n'}^{\ \ \ \ \ \ \ \ \ \alpha_1' \cdots \alpha_m'}.
	\end{align}
	Note that the quantity \eqref{E:TSPACETIMEGFOURNORM} can be non-positive since $\gfour$ is Lorentzian.
	 
\end{definition}

Our main bootstrap assumptions will concern Sobolev norms of the \emph{components} of the solution variables 
and also $C^M-$type norms of the lower-order derivatives of their \emph{components}. Our derivation of energy estimates
will involve similar norms but with geometric $\newg-$type norms in place of the frame component norms.
We will make use of the following norms.

\begin{definition} [\textbf{$H^M$ and $C^M$ norms}]  \label{D:CMNORMS}
	Let $T$ be a $\Sigma_t-$tangent tensor with components $T_{b_1 \cdots b_n}^{\ \ \ \ \ \ \ \ a_1 \cdots a_m}$
	relative to our standard atlas (see Def.~\ref{D:STANDARDATLAS}) on $\mathbb{T}^3.$ Let $\newg$ be the renormalized spatial metric.
	We define
\begin{subequations}
\begin{align}
	\| T \|_{H_{Frame}^M} = \| T \|_{H_{Frame}^M}(t) & := 
		\sum_{|\vec{I}| \leq M} \left\| \left|\partial_{\vec{I}} T (t,\cdot) \right|_{Frame} \right \|_{L^2}, \\
	\| T \|_{H_{\newg}^M} = \| T \|_{H_{\newg}^M}(t) & := 
		\sum_{|\vec{I}| \leq M} \left\| \left|\partial_{\vec{I}} T (t,\cdot) \right|_{\newg} \right \|_{L^2}, \\
	\| T \|_{C_{Frame}^M} = \| T \|_{C_{Frame}^M}(t) & := 
		\sum_{|\vec{I}| \leq M} \sup_{x \in \mathbb{T}^3} \left|\partial_{\vec{I}} T(t,x) \right|_{Frame}, \\
	\| T \|_{C_{\newg}^M} = \| T \|_{C_{\newg}^M}(t) & := \sum_{|\vec{I}| \leq M} \sup_{x \in \mathbb{T}^3} \left| \partial_{\vec{I}} 
		T(t,x) \right|_{\newg}, 
		\label{E:CGMFRAMENORMDEF}
\end{align}
\end{subequations}
where $\partial_{\vec{I}} T$ is defined to be the tensorfield with components
\begin{align} \label{E:PARTIALIT}
	(\partial_{\vec{I}} T)_{b_1 \cdots b_n}^{\ \ \ \ \ \ \ \ a_1 \cdots a_m} 
	:= \partial_{\vec{I}} (T_{b_1 \cdots b_n}^{\ \ \ \ \ \ \ \ a_1 \cdots a_m}).
\end{align}
\end{definition}

\begin{remark} \label{R:INDICESLOCATION}
Note that $\partial_{\vec{I}}$ does not commute with the lowering and raising of indices of $T$ with
$g$ and $g^{-1}.$ Thus, in \eqref{E:PARTIALIT}, the location of the indices in the tensors $T$ of interest 
is understood to have been established by Def.~\ref{D:RESCALEDVAR}.
\end{remark}

\begin{remark}
	For scalar-valued tensors $T,$ there is no point in writing the subscript on the norms and we simply write, 
	for example, $|T|,$ $\| T \|_{C^M},$ or $\| T \|_{H^M}.$
\end{remark}

\begin{definition}[\textbf{Norms for the solution}] \label{D:NORMS}

The specific norms featured in our near-FLRW bootstrap assumptions are defined as follows:
	\begin{subequations}
\begin{align}
	\highnorm{M}(t) & := 
		\underbrace{\sum_{1 \leq |\vec{I}| \leq M} \left\| \left| \partial_{\vec{I}} \newg \right|_{Frame} \right\|_{L^2}}_{\mbox{absent when $M = 0$}} 
		+ \underbrace{\sum_{1 \leq |\vec{I}| \leq M} \left\| \left| \partial_{\vec{I}} \newg^{-1} \right|_{Frame} \right\|_{L^2}}_{\mbox{absent when $M = 0$}}  
		\label{E:HIGHNORM} \\ 
		& \ \ + \left\| \freenewsec \right\|_{H_{Frame}^M} 
			+ t^{2/3} \| \upgamma  \|_{H_{Frame}^M} 
			\notag  \\
		& \ \ + t^{2/3} \| \newlap  \|_{H^{M-1}}
			+ \| \dlap  \|_{H_{Frame}^{M-2}}
			+ t^{4/3} \| \newlap  \|_{H^M}
			+ t^{2/3} \| \dlap  \|_{H_{Frame}^{M-1}}
			\notag \\
		& \ \ + t^2 \| \newlap \|_{H^{M+1}} 
				+ t^{4/3} \| \dlap  \|_{H_{Frame}^M} 
				+ t^{8/3} \| \newlap \|_{H^{M+2}} 
				+ t^2 \| \dlap \|_{H_{Frame}^{M+1}}
			\notag \\
		& \ \ + \left\| \adjustednewp  \right\|_{H^M}
			+ t^{2/3} \left\| \newu  \right\|_{H_{Frame}^M},
		 \notag \\
	\lowkinnorm{M}(t) & := \left\| \freenewsec \right\|_{C_{Frame}^M}  
		+ \left\| \adjustednewp  \right\|_{C^M}, 
			\label{E:LOWKINNORM} \\
	\lowpotnorm{M}(t) & := \left\| \newg - \Euc \right\|_{C_{Frame}^{M+1}} 
		+ \left\| \newg^{-1} - \Euc^{-1} \right\|_{C_{Frame}^{M+1}} 
		\label{E:LOWPOTNORM} \\
		& + \ \ \| \upgamma \|_{C_{Frame}^M} 
			\notag \\
		& \ \ + \| \newlap  \|_{C_{Frame}^{M-1}} 
			+ t^{-2/3} \| \dlap  \|_{C_{Frame}^{M-2}} 
			 \notag \\
		& \ \ + \left\| \newu  \right\|_{C_{Frame}^M}.
			\notag  
\end{align}
\end{subequations}
We remark that the Euclidean metric $\Euc$ appearing on the right-hand side of \eqref{E:LOWPOTNORM}
was constructed in Sect.~\ref{S:EEINCMC}.
It is understood that Sobolev norms are omitted from the above formulas when their order is negative.

\end{definition}

\subsection{The metric equations of variation and the metric energy currents}

Our fundamental integration by parts identities for the $\partial_{\vec{I}}-$differentiated quantities are computationally involved. To facilitate the presentation, it is convenient to introduce the following notation for the differentiated quantities,
also known as \emph{variations}.
\begin{definition}[\textbf{Variations}]
\label{D:VARIATIONS}
\begin{subequations}
\begin{align}
	\dot{\freenewsec}_{\ j}^i & := \partial_{\vec{I}} \freenewsec_{\ j}^i, 
	&&
	\dot{\upgamma}_{j \ k}^{\ i} := \partial_{\vec{I}} \upgamma_{j \ k}^{\ i}, 
		\label{E:METVARTN} \\
	\dot{\Gamma}_j & := \leftexp{(\vec{I})}{\Gamma_j} = \newg_{jb} (\newg^{-1})^{ef} \partial_{\vec{I}} \upgamma_{e \ f}^{\ b} 
		- \frac{1}{2} \partial_{\vec{I}} \upgamma_{j \ b}^{\ b}, &&
		\\
	\dot{\newlap} & := \partial_{\vec{I}} \newlap, 
	&& \dot{\dlap}_i := \partial_{\vec{I}} \dlap_i, 
		\\
	\dot{\adjustednewp} & :=  \partial_{\vec{I}} \adjustednewp, 
	&& \dot{\newu}^j := \partial_{\vec{I}} \newu^j. \label{E:PANDUVARTN}
\end{align}
\end{subequations}
Note that the notation $\cdot$ is unrelated to time derivatives.
\end{definition}

We also introduce notation for some of the inhomogeneous terms in the commuted equations.

\begin{definition}[\textbf{Shorthand notation for some inhomogeneous terms}]
\label{D:INHOMSHORTHAND}
We use the following notation for the inhomogeneous terms in equations
\eqref{AE:MOMENTUMCONSTRAINTCOMMUTED}, \eqref{AE:RAISEDMOMENTUMCONSTRAINTCOMMUTED}, 
\eqref{AE:METRICGAMMACOMMUTED}, \eqref{AE:SECONDFUNDCOMMUTED},
\eqref{AE:PARTIALTPCOMMUTED}, and \eqref{AE:PARTIALTUJCOMMUTED}:
\begin{subequations}
\begin{align}
	\dot{\mathfrak{M}}_i & := \leftexp{(\vec{I});(Border)}{\mathfrak{M}}_i 
	 	+ t^{4/3} \leftexp{(\vec{I});(Junk)}{\mathfrak{M}}_i, 
	 	\label{E:SHORTHANDDOTMATHFRAKM} \\
	\dot{\widetilde{\mathfrak{M}}}^j & := \leftexp{(\vec{I});(Border)}{\widetilde{\mathfrak{M}}}^j
		+ t^{4/3} \leftexp{(\vec{I});(Junk)}{\widetilde{\mathfrak{M}}}^j, 
			\\
	\dot{\mathfrak{g}}_{j \ k}^{\ i} &:= t^{-1} \leftexp{(\vec{I});(Border)}{\mathfrak{g}_{j \ k}^{\ i}}
		+ t^{1/3} \leftexp{(\vec{I});(Junk)}{\mathfrak{g}_{j \ k}^{\ i}}, \\
	\dot{\mathfrak{K}}_{\ j}^i & := t^{1/3} \leftexp{(\vec{I});(Junk)}{\mathfrak{K}}_{\ j}^i, 
	\label{E:SHORTHANDDOTMATHFRAKK}
		\\
	\dot{\mathfrak{P}} &:= t^{1/3} \leftexp{(\vec{I});(Junk)}{\mathfrak{P}}, \\
	\dot{\mathfrak{U}}^j & :=	t^{-1} \leftexp{(\vec{I});(Border)}{\mathfrak{U}}^j 
	+ t^{1/3} \leftexp{(\vec{I});(Junk)}{\mathfrak{U}}^j.
\end{align}
\end{subequations}
\end{definition}

The next definition captures the essential structure of the commuted constraint equations 
\eqref{AE:MOMENTUMCONSTRAINTCOMMUTED}-\eqref{AE:RAISEDMOMENTUMCONSTRAINTCOMMUTED},
the commuted metric evolution equations \eqref{AE:METRICGAMMACOMMUTED}-\eqref{AE:SECONDFUNDCOMMUTED}, 
and also the CMC condition 
\begin{align}
	\partial_{\vec{I}} \freenewsec_{\ a}^a = 0.	
\end{align}

\begin{definition} [\textbf{The metric equations of variation}] \label{D:METRICEOV}
We define the metric equations of variation to be the following system of equations:
\begin{subequations}
\begin{align}
	\partial_t \dot{\upgamma}_{e \ i}^{\ b} & = - 2 t^{-1} \big[1 + t^{4/3} \newlap\big] \partial_e \dot{\freenewsec}_{\ i}^b 
		+ \frac{2}{3} t^{-1/3} \dot{\dlap}_e \ID_{\ i}^b 
		\label{E:EOVUPGAMMA} \\
	& \ \ + \dot{\mathfrak{g}}_{e \ i}^{\ b},  
		\notag \\
	\partial_t \dot{\freenewsec}_{\ j}^i & = - \frac{1}{2} t^{1/3} \big[1 + t^{4/3} \newlap\big] (\newg^{-1})^{ef} 
		\partial_e \dot{\upgamma}_{f \ j}^{\ i} 
		\label{E:EOVK} \\
	& \ \ + \frac{1}{2} t^{1/3} \big[1 + t^{4/3} \newlap\big] 
		 (\newg^{-1})^{ia} \partial_a \dot{\Gamma}_j \notag \\
	& \ \  + \frac{1}{2} t^{1/3} \big[1 + t^{4/3} \newlap\big]  
		(\newg^{-1})^{ia} \partial_j \dot{\Gamma}_a \notag \\
	& \ \ - t (\newg^{-1})^{ia} \partial_a \dot{\dlap}_j
		+ \frac{1}{3}t^{1/3} A \ID_{\ j}^i \notag \\ 
	& \ \ + \dot{\mathfrak{K}}_{\ j}^i, \notag
\end{align}
\end{subequations}
subject to the constraints
\begin{subequations}
\begin{align}
	\partial_a \dot{\freenewsec}_{\ i}^a & = \frac{2}{3} \newg_{ia}\dot{\newu}^a + \dot{\mathfrak{M}}_i,
		\label{E:EOVM1} \\
	(\newg^{-1})^{ia} \partial_a \dot{\freenewsec}_{\ i}^j & = \frac{2}{3} \dot{\newu}^j 
		+ \dot{\widetilde{\mathfrak{M}}}^j, \label{E:EOVM2}
\end{align}
\end{subequations}
and the CMC condition
\begin{align} \label{E:CMCVARIATION}
	\dot{\freenewsec}_{\ a}^a = 0.
\end{align}
\end{definition}

\begin{remark} \label{E:DANGEROUSTERMCANCELS}
	The precise form of the term $\frac{1}{3}t^{1/3} A \ID_{\ j}^i$ from the right-hand side of	
	equation \eqref{E:EOVK} is not important. Specifically, this term is the next-to-last
	term on the right-hand side of \eqref{AE:SECONDFUNDCOMMUTED}. What matters is that it is proportional to the identity.
	As we discuss at the end of the proof of Lemma~\ref{L:DIVMETRICJ}, this term completely cancels out of our 
	metric energy current divergence identity \eqref{E:DIVMETRICJ}. 
	This cancellation is absolutely essential for the proof of our main stable singularity formation theorem, for otherwise, 
	our fundamental energy integral inequality \eqref{E:FUNDAMENTALENERGYINEQUALITY} 
	would involve a top order quadratic integral that we would have no means of controlling.
\end{remark}

We now introduce our metric energy currents $\dot{\mathbf{J}}_{(Metric)}^{\mu}[\cdot , \cdot].$ 
In Sect.~\ref{S:FUNDAMENATLENERGYINEQUALITIES}, we will apply the divergence theorem to these vectorfields 
in the region in between $\Sigma_t$ and $\Sigma_1$ in order to derive energy identities for the metric variations.
The currents provide a convenient way of bookkeeping during integration by parts.
To the best of our knowledge, it has not previously been noticed that it is possible to derive such energy identities
relative to CMC-transported spatial coordinates; see the discussion near the end of Sect.~\ref{SSS:GLOBALSOLUTIONSINGR}.

\begin{definition} [\textbf{Metric energy current}] \label{D:METRICCURRENT}
To the metric variations $(\dot{\freenewsec}_{\ j}^i, \dot{\upgamma}_{j \ k}^{\ i})_{1 \leq i,j,k \leq 3},$ 
we associate the following spacetime vectorfield, which we refer to as a metric energy current (where $j = 1,2,3$):
\begin{subequations}
	\begin{align}
		\dot{\mathbf{J}}_{(Metric)}^0[(\dot{\freenewsec}, \dot{\upgamma}), (\dot{\freenewsec}, \dot{\upgamma})]  
		& := (\newg^{-1})^{ab} \newg_{ij} \dot{\freenewsec}_{\ a}^i \dot{\freenewsec}_{\ b}^j
			+ \frac{1}{4} t^{4/3} (\newg^{-1})^{ab} 
				(\newg^{-1})^{ef} \newg_{ij} \dot{\upgamma}_{e \ a}^{\ i} \dot{\upgamma}_{f \ b}^{\ j}, 
				\label{E:METRICCURRENT0} \\
		\dot{\mathbf{J}}_{(Metric)}^j[(\dot{\freenewsec}, \dot{\upgamma}), (\dot{\freenewsec}, \dot{\upgamma})]  
		& := t^{1/3} \big[1 + t^{4/3} \newlap\big] (\newg^{-1})^{ab}  
			(\newg^{-1})^{jf} \newg_{ic} \dot{\freenewsec}_{\ a}^i \dot{\upgamma}_{f \ b}^{\ c} 
			+ 2 t (\newg^{-1})^{ia} \dot{\dlap}_a \dot{\freenewsec}_{\ i}^j  
			\label{E:METRICCURRENTJ} \\
		& \ \ - t^{1/3} \big[1 + t^{4/3} \newlap\big] (\newg^{-1})^{ij} \dot{\Gamma}_a \dot{\freenewsec}_{\ i}^a 
			- t^{1/3} \big[1 + t^{4/3} \newlap\big] (\newg^{-1})^{ia} \dot{\Gamma}_a \dot{\freenewsec}_{\ i}^j. \notag
\end{align}	
\end{subequations}

\end{definition}
Note that $\dot{\mathbf{J}}_{(Metric)}^0[\cdot , \cdot]$ can be viewed as a positive definite quadratic form in
$(\dot{\freenewsec}, \dot{\upgamma}).$ This property will result in coercive metric energies for the $\partial_{\vec{I}}-$differentiated 
quantities. The next lemma shows that for solutions to the metric equations of variation,
$\partial_{\mu}\left(\dot{\mathbf{J}}_{(Metric)}^{\mu}[\cdot,\cdot] \right)$ 
can be expressed in terms of quantities that do not depend on the derivatives of the variations $(\dot{\freenewsec}, \dot{\upgamma}).$ This property of 
$\dot{\mathbf{J}}_{(Metric)}^{\mu}[\cdot , \cdot]$ is, of course, an essential ingredient in our derivation of
energy estimates.

\begin{lemma} [\textbf{Differential identity for the metric energy current}] \label{L:DIVMETRICJ}
For a solution $(\dot{\freenewsec}_{\ j}^i, \dot{\upgamma}_{j \ k}^{\ i})_{1 \leq i,j,k \leq 3}$
to \eqref{E:EOVUPGAMMA}-\eqref{E:CMCVARIATION}, we have the following
spacetime coordinate divergence identity:
\begin{align}
	\partial_t \left( \dot{\mathbf{J}}_{(Metric)}^0[(\dot{\freenewsec}, \dot{\upgamma}), (\dot{\freenewsec}, \dot{\upgamma})] \right)
	+ & \partial_j \left( \dot{\mathbf{J}}_{(Metric)}^j[(\dot{\freenewsec}, \dot{\upgamma}), (\dot{\freenewsec}, \dot{\upgamma})] \right)
		\label{E:DIVMETRICJ} \\
	& = \frac{1}{3} t^{1/3} (\newg^{-1})^{ab} (\newg^{-1})^{ef} \newg_{ij}  
		\dot{\upgamma}_{e \ a}^{\ i} \dot{\upgamma}_{f \ b}^{\ j}  
	\notag \\
	& \ \ + \frac{1}{3} t (\newg^{-1})^{ab} \dot{\dlap}_a \dot{\upgamma}_{b \ c}^{\ c} \notag \\
	& \ \ + \frac{4}{3} t \dot{\newu}^a \dot{\dlap}_a 
			\notag \\
	& \ \ - \frac{4}{3} t^{1/3} \dot{\newu}^a \dot{\Gamma}_a \notag \\
	& \ \ + 2 t (\newg^{-1})^{ab} \dot{\mathfrak{M}}_a \dot{\dlap}_b \notag \\
	& \ \ - t^{1/3} \big[1 + t^{4/3} \newlap \big] \dot{\widetilde{\mathfrak{M}}}^a \dot{\Gamma}_a 
			- t^{1/3} \big[1 + t^{4/3} \newlap \big] (\newg^{-1})^{ab} \dot{\mathfrak{M}}_a \dot{\Gamma}_b \notag \\
	& \ \ + \frac{1}{2} t^{4/3} (\newg^{-1})^{ab} \newg_{ij} (\newg^{-1})^{ef} \dot{\mathfrak{g}}_{e \ a}^{\ i} 
		\dot{\upgamma}_{f \ b}^{\ j}
		+ 2 (\newg^{-1})^{ab} \newg_{ij} \dot{\mathfrak{K}}_{\ a}^i \dot{\freenewsec}_{\ b}^j \notag \\
	& \ \ + \triangle_{\dot{\mathbf{J}}_{(Metric);(Border)}[(\dot{\freenewsec}, \dot{\upgamma}), (\dot{\freenewsec}, \dot{\upgamma})]} 
		+ \triangle_{\dot{\mathbf{J}}_{(Metric);(Junk)}[(\dot{\freenewsec}, \dot{\upgamma}), (\dot{\freenewsec}, \dot{\upgamma})]}, \notag
\end{align}
where
\begin{subequations}
\begin{align}
	\triangle_{\dot{\mathbf{J}}_{(Metric);(Border)}[(\dot{\freenewsec}, \dot{\upgamma}), (\dot{\freenewsec}, \dot{\upgamma})]}
	& := 2 t^{-1} \big[1 + t^{4/3} \newlap \big]
		(\newg^{-1})^{ac} \newg_{ij} \freenewsec_{\ c}^b
		\dot{\freenewsec}_{\ a}^i \dot{\freenewsec}_{\ b}^j \label{E:METRICCURRENTERRORBORDERLINE} \\
	& \ \ - 2 t^{-1} \big[1 + t^{4/3} \newlap \big]
			(\newg^{-1})^{ab} \newg_{ic} \freenewsec_{\ j}^c 
			\dot{\freenewsec}_{\ a}^i \dot{\freenewsec}_{\ b}^j \notag \\
	& \ \ + \frac{1}{2} t^{1/3} \big[1 + t^{4/3} \newlap \big] (\newg^{-1})^{ab} (\newg^{-1})^{cf} \newg_{ij}
			\freenewsec_{\ c}^e \dot{\upgamma}_{e \ a}^{\ i} \dot{\upgamma}_{f \ b}^{\ j}   
			\notag \\
	& \ \ + \frac{1}{2} t^{1/3} \big[1 + t^{4/3} \newlap \big] (\newg^{-1})^{cb} (\newg^{-1})^{ef} \newg_{ij} 
			\freenewsec_{\ c}^a \dot{\upgamma}_{e \ a}^{\ i} \dot{\upgamma}_{f \ b}^{\ j} \notag \\
	& \ \ - \frac{1}{2} t^{1/3} \big[1 + t^{4/3} \newlap \big] (\newg^{-1})^{ab} (\newg^{-1})^{ef} \newg_{cj}
				\freenewsec_{\ i}^c \dot{\upgamma}_{e \ a}^{\ i} \dot{\upgamma}_{f \ b}^{\ j}, \notag \\
\triangle_{\dot{\mathbf{J}}_{(Metric);(Junk)}[(\dot{\freenewsec}, \dot{\upgamma}), (\dot{\freenewsec}, \dot{\upgamma})]} 
	& := - \frac{1}{6} t^{5/3} \newlap (\newg^{-1})^{ab} (\newg^{-1})^{ef} \newg_{ij} 
		\dot{\upgamma}_{e \ a}^{\ i} \dot{\upgamma}_{f \ b}^{\ j} \label{E:METRICCURRENTERRORJUNK}  	\\
	& \ \ + t^{1/3} \left[\partial_j\left\lbrace \big[1 + t^{4/3} \newlap \big] (\newg^{-1})^{ab}  
			(\newg^{-1})^{jf} \newg_{ic} \right\rbrace \right] \dot{\freenewsec}_{\ a}^i \dot{\upgamma}_{f \ b}^{\ c}  \notag \\
	& \ \ + 2 t \left[\partial_j (\newg^{-1})^{ia} \right] 
		\dot{\freenewsec}_{\ i}^j \dot{\dlap}_a \notag \\
	& \ \ - t^{1/3} \left[\partial_j \left\lbrace \big[1 + t^{4/3} \newlap \big] (\newg^{-1})^{ij} \right\rbrace
		\right] \dot{\Gamma}_a \dot{\freenewsec}_{\ i}^a \notag \\
	& \ \ - t^{1/3} \left[\partial_j \left\lbrace \big[1 + t^{4/3} \newlap \big] (\newg^{-1})^{ia} \right\rbrace\right] 
			\dot{\Gamma}_a \dot{\freenewsec}_{\ i}^j. \notag
\end{align}
\end{subequations}

\end{lemma}

\begin{remark}
	The first two products from the right-hand side of \eqref{E:METRICCURRENTERRORBORDERLINE} cancel when
	$\dot{\freenewsec}_{\ j}^i := \freenewsec_{\ j}^i$ (this claim follows from the symmetry property $\SecondFund_{ij} = \SecondFund_{ji}$), 
	but they do not generally cancel when
	$\dot{\freenewsec}_{\ j}^i := \partial_{\vec{I}} \freenewsec_{\ j}^i.$
\end{remark}

\begin{proof}
	The derivation of \eqref{E:DIVMETRICJ} involves a series of tedious computations. We first
	discuss the case when the coordinate derivatives hit the coefficients
	of the variations. First, we remark that the error term 
	$\triangle_{\dot{\mathbf{J}}_{(Metric);(Junk)}[(\dot{\freenewsec}, \dot{\upgamma}), (\dot{\freenewsec}, \dot{\upgamma})]}$ 
	contains all of the terms that are generated when the spatial derivatives $\partial_j$ hit the coefficients
	of the variations in \eqref{E:METRICCURRENTJ}. 
	
	In contrast, when $\partial_t$ hits the coefficients of the
	variations in \eqref{E:METRICCURRENT0}, many important terms are generated. 
	Specifically, when $\partial_t$ falls on the $t^{4/3}$ factor in \eqref{E:METRICCURRENT0}, 
	this generates the
	$\frac{1}{3} t^{1/3} (\newg^{-1})^{ab} 
	(\newg^{-1})^{ef} \newg_{ij} \dot{\upgamma}_{e \ a}^{\ i} \dot{\upgamma}_{f \ b}^{\ j}$ 
	term on the right-hand side of \eqref{E:DIVMETRICJ}. When $\partial_t$ falls on any of 
	the factors of $\newg$ or $\newg^{-1}$ in \eqref{E:METRICCURRENT0}, we use equations
	\eqref{E:GEVOLUTION} and \eqref{E:GINVERSEEVOLUTION} to substitute for
	$\partial_t \newg$ and $\partial_t \newg^{-1}.$ We place the resulting products involving
	$\freenewsec$ on the right-hand side of \eqref{E:METRICCURRENTERRORBORDERLINE} as part of
	$\triangle_{\dot{\mathbf{J}}_{(Metric);(Border)}[(\dot{\freenewsec}, \dot{\upgamma}), (\dot{\freenewsec}, \dot{\upgamma})]},$
	while we place the remaining terms on the right-hand side of \eqref{E:METRICCURRENTERRORJUNK} 
	as part of $\triangle_{\dot{\mathbf{J}}_{(Metric);(Junk)}[(\dot{\freenewsec}, \dot{\upgamma}), (\dot{\freenewsec}, \dot{\upgamma})]}.$
	%When $\partial_t$ falls on the factor of $(\sqrt{\gbigdet})^2$ in \eqref{E:METRICCURRENT0},
	%we use equation \eqref{E:LOGVOLFORMEVOLUTION} for substitution and place the resulting term on the
	%on the right-hand side of \eqref{E:METRICCURRENTERRORJUNK} 
	%as part of $\triangle_{\dot{\mathbf{J}}_{(Metric);(Junk)}[\dot{\mathscr{M}}, \dot{\mathscr{M}}]}.$
	
	All remaining terms on the right-hand side of \eqref{E:DIVMETRICJ} are generated when
	the spacetime derivatives $\partial_t$ and $\partial_j$ fall on $(\dot{\freenewsec}, \dot{\upgamma}).$ 
	More precisely, we use the equations of variation 
	\eqref{E:EOVUPGAMMA}-\eqref{E:EOVK}, \eqref{E:EOVM1}-\eqref{E:EOVM2}, and \eqref{E:CMCVARIATION}
	to replace the derivatives of $(\dot{\freenewsec}, \dot{\upgamma})$ with the terms on the right-hand side of 
	\eqref{E:EOVUPGAMMA}-\eqref{E:EOVK}, \eqref{E:EOVM1}-\eqref{E:EOVM2}, and \eqref{E:CMCVARIATION}; 
	we omit the tedious but straightforward
	calculations that correspond to this replacement. However, we do note three important cancellations. First, 
	the term $2 t (\newg^{-1})^{ia} (\partial_j \dot{\dlap}_a) \dot{\freenewsec}_{\ i}^j$ generated by the spatial divergence of 
	$\dot{\mathbf{J}}_{(Metric)}^j$ is canceled by the product arising from the term 
	$- t (\newg^{-1})^{ia} \partial_a \dot{\dlap}_j$ on the right-hand side of \eqref{E:EOVK}; 
	this product appears when $\partial_t$ falls on the $\dot{\freenewsec}$ factors in the term 
	$(\newg^{-1})^{ab} \newg_{ij} \dot{\freenewsec}_{\ a}^i \dot{\freenewsec}_{\ b}^j$ 
	from the right-hand side of \eqref{E:METRICCURRENT0}
	and \eqref{E:EOVK} is used for substitution.
	
	Next, the terms $- t^{1/3} \big[1 + t^{4/3} \newlap\big] (\newg^{-1})^{ij} (\partial_j \dot{\Gamma}_a) \dot{\freenewsec}_{\ 
	i}^a$ and $- t^{1/3} \big[1 + t^{4/3} \newlap\big] (\newg^{-1})^{ia} (\partial_j \dot{\Gamma}_a) \dot{\freenewsec}_{\ i}^j$
	generated by the spatial divergence of $\dot{\mathbf{J}}_{(Metric)}^j$ are canceled by the products arising from the terms
	$\frac{1}{2} t^{1/3} \big[1 + t^{4/3} \newlap\big] (\newg^{-1})^{ia} \partial_a \dot{\Gamma}_j$
	and $\frac{1}{2} t^{1/3} \big[1 + t^{4/3} \newlap\big] (\newg^{-1})^{ia} \partial_j \dot{\Gamma}_a $
	on the right-hand side of \eqref{E:EOVK}; these products appear when $\partial_t$ falls on the $\dot{\freenewsec}$ factors in the term 
	$(\newg^{-1})^{ab} \newg_{ij} \dot{\freenewsec}_{\ a}^i \dot{\freenewsec}_{\ b}^j$ from the right-hand side of \eqref{E:METRICCURRENT0}
	and \eqref{E:EOVK} is used for substitution.
	
	The final cancellation we discuss is the one mentioned in Remark~\ref{E:DANGEROUSTERMCANCELS}. The cancellation is connected
	to the next-to-last term on the right-hand side of \eqref{E:EOVK}. This term enters into the right-hand side of 
	\eqref{E:DIVMETRICJ} when $\partial_t$ falls on the $\dot{\freenewsec}$ factors in the term 
	$(\newg^{-1})^{ab} \newg_{ij} \dot{\freenewsec}_{\ a}^i \dot{\freenewsec}_{\ b}^j$ 
	from the right-hand side of \eqref{E:METRICCURRENT0} and \eqref{E:EOVK} is used for substitution. The
	resulting product is of the form $\frac{2}{3}t^{1/3} A (\newg^{-1})^{ab} \newg_{ij} \ID_{\ a}^i \dot{\freenewsec}_{\ b}^j.$ Thanks to
	equation \eqref{E:CMCVARIATION}, this term vanishes.
	
\end{proof}

\subsection{The fluid equations of variation and the fluid energy currents}

In this section, we extend the discussion of the previous section to apply
to the fluid variables.

The next definition captures the essential structure of the commuted fluid equations
\eqref{AE:PARTIALTPCOMMUTED}-\eqref{AE:PARTIALTUJCOMMUTED}.

\begin{definition} [\textbf{The fluid equations of variation}] \label{D:FLUIDEOV}

We define the fluid equations of variation to be the following system of equations:
\begin{subequations}
\begin{align}
	\partial_t \dot{\adjustednewp} 
		& + 2 t^{1/3} \big[1 + t^{4/3} \newlap\big] \big[1 + t^{4/3} \newg_{ab} \newu^a \newu^b\big]^{1/2} 
		\Big[\adjustednewp + \frac{1}{3} \Big] \partial_c \dot{\newu}^c \label{E:EOVP} \\
		& - 2 t^{5/3} \frac{\big[1 + t^{4/3} \newlap\big] \Big[\adjustednewp + \frac{1}{3} \Big]}{\big[1 + t^{4/3} \newg_{ab} \newu^a 
			\newu^b\big]^{1/2}} \newg_{ef} \newu^e \newu^c \partial_c \dot{\newu}^f	
		 \notag \\
		& = - \frac{2}{3} t^{1/3} \dot{\newlap}
			+  \dot{\mathfrak{P}}, \notag 
\end{align}
\begin{align}
	\partial_t \dot{\newu}^j 
		& - t^{1/3} \big[1 + t^{4/3} \newlap\big] \big[1 + t^{4/3} \newg_{ab} \newu^a \newu^b\big]^{1/2} \newu^j \partial_c \dot{\newu}^c 
			 \label{E:EOVU} \\
		& + t^{5/3} \frac{\big[1 + t^{4/3} \newlap\big]}{\big[1 + t^{4/3} \newg_{ab} \newu^a \newu^b\big]^{1/2}} 
			\newg_{ef} \newu^j \newu^e \newu^c \partial_c \dot{\newu}^f \notag \\
		& + t^{1/3} \frac{\big[1 + t^{4/3} \newlap\big] \newu^c \partial_c \dot{\newu}^j}{\big[1 + t^{4/3} \newg_{ab} \newu^a 
			\newu^b\big]^{1/2}} \notag \\
		& +  t^{-1} \frac{\big[1 + t^{4/3} \newlap\big] \left\lbrace (\newg^{-1})^{jc} + t^{4/3} \newu^j \newu^c \right\rbrace 
			\partial_c \dot{\adjustednewp}}{2\big[1 + t^{4/3} \newg_{ab} \newu^a \newu^b\big]^{1/2} \Big[\adjustednewp + \frac{1}{3} \Big]}  \notag \\
		& = - t^{-1/3}(\newg^{-1})^{ja} \dot{\dlap}_a
			+ \dot{\mathfrak{U}}^j. \notag 
\end{align}
\end{subequations}

\end{definition}

We now introduce our fluid energy currents $\dot{\mathbf{J}}_{(Fluid)}^{\mu}[\cdot , \cdot].$ 
These are the fluid analogs of the metric energy currents from Def.~\ref{D:METRICCURRENT}. 
Roughly speaking, these currents exist because the Euler equations
are hyperbolic and derivable from a Lagrangian. The energy current framework in the context of relativistic
fluid mechanics was first introduced by Christodoulou in \cite{dC2000} and \cite{dC2007}. This framework has
been applied by the second author in various contexts connected to relativistic fluid mechanics; 
see \cites{jS2012,jS2008b,jS2013,jS2008a,jSrS2011}. 

\begin{definition} [\textbf{Fluid energy current}] \label{D:FLUIDCURRENT}
To given fluid variations $(\dot{\adjustednewp},\dot{\newu}^1,\dot{\newu}^2,\dot{\newu}^3),$ we associate the following spacetime vectorfield, which we refer to as a fluid energy current (where $j = 1,2,3$):
	\begin{subequations}
	\begin{align}
		\dot{\mathbf{J}}_{(Fluid)}^0[(\dot{\adjustednewp},\dot{\newu}), (\dot{\adjustednewp},\dot{\newu})] 
			& := \frac{1}{2} \dot{\adjustednewp}^2 
			+ 2 t^{4/3} \frac{\Big[\adjustednewp + \frac{1}{3} \Big]}{\big[1 + t^{4/3} \newg_{ab} \newu^a \newu^b\big]} 
				\newg_{ef} \newu^e \dot{\newu}^f \dot{\adjustednewp} 
				\label{E:FLUIDCURRENT0} \\
			& \ \ + 2 \Big[\adjustednewp + \frac{1}{3} \Big]^2 \left\lbrace t^{4/3} \newg_{ef} \dot{\newu}^e \dot{\newu}^f 
				- \frac{\big[t^{4/3} \newg_{ef} \newu^e \dot{\newu}^f \big]^2}
					{\big[1 + t^{4/3} \newg_{ab} \newu^a \newu^b\big]} \right\rbrace, \notag \\
		\dot{\mathbf{J}}_{(Fluid)}^j[(\dot{\adjustednewp},\dot{\newu}), (\dot{\adjustednewp},\dot{\newu})] 
			& := t^{1/3} \frac{\big[1 + t^{4/3} \newlap\big] \newu^j}{2\big[1 + t^{4/3} \newg_{ab} \newu^a \newu^b\big]^{1/2}} \dot{\adjustednewp}^2 
			+ 2 t^{1/3} \frac{\big[1 + t^{4/3} \newlap\big] \Big[\adjustednewp + \frac{1}{3} \Big]}{\big[1 + t^{4/3} \newg_{ab} \newu^a 
				\newu^b\big]^{1/2}} \dot{\newu}^j \dot{\adjustednewp}  
				\label{E:FLUIDCURRENTJ} \\
			& \ \ + 2 t^{1/3} \frac{\big[1 + t^{4/3} \newlap\big] \Big[\adjustednewp + \frac{1}{3} \Big]^2 \newu^j}
				{\big[1 + t^{4/3} \newg_{ab} \newu^a \newu^b\big]^{1/2}} 
				\left\lbrace t^{4/3} \newg_{ef} \dot{\newu}^e \dot{\newu}^f
					- \frac{\big[t^{4/3} \newg_{ef} \newu^e \dot{\newu}^f\big]^2}
						{\big[1 + t^{4/3} \newg_{ab} \newu^a \newu^b\big]}  \right\rbrace. \notag
	\end{align}
	\end{subequations}
\end{definition}
We note that $\dot{\mathbf{J}}_{(Fluid)}^0[\cdot , \cdot]$ can be viewed as positive definite quadratic form in $(\dot{\adjustednewp},\dot{\newu})$ 
(this is clearly the case whenever $|\newu|_{\newg}$ is sufficiently small, but it is also true when $|\newu|_{\newg}$ is large - see e.g. the discussion in \cite{dC2007}). This property will result in coercive fluid energies (see Lemma~\ref{L:COERCIVITYOFTHEENERGIES}).

The next lemma shows that for solutions to the fluid equations of variation,
the divergence of $\dot{\mathbf{J}}_{(Fluid)}^{\mu}[(\dot{\adjustednewp},\dot{\newu}), (\dot{\adjustednewp},\dot{\newu})]$ can be expressed in terms of
quantities that do not depend on the derivatives of the variations $(\dot{\adjustednewp},\dot{\newu}).$ 
The lemma is the fluid analog of Lemma~\ref{L:DIVMETRICJ}, and equation \eqref{E:DIVFLUIDJ} from the lemma
is analogous to \cite[Equation (1.41)]{dC2007}.

\begin{lemma} [\textbf{Differential identity for the fluid energy current}] \label{L:DIVFLUIDJ}
For a solution $(\dot{\adjustednewp},\dot{\newu}^1,\dot{\newu}^2,\dot{\newu}^3)$ of
\eqref{E:EOVP}-\eqref{E:EOVU}, the spacetime coordinate divergence of 
$\dot{\mathbf{J}}_{(Fluid)}[(\dot{\adjustednewp},\dot{\newu}), (\dot{\adjustednewp},\dot{\newu})]$ 
can be expressed as follows:
\begin{align} \label{E:DIVFLUIDJ}
	\partial_t\left( \dot{\mathbf{J}}_{(Fluid)}^0[(\dot{\adjustednewp},\dot{\newu}), (\dot{\adjustednewp},\dot{\newu})] \right)
	+ & \partial_j \left( \dot{\mathbf{J}}_{(Fluid)}^j[(\dot{\adjustednewp},\dot{\newu}), (\dot{\adjustednewp},\dot{\newu})] \right) 
	\\
	& = - \frac{2}{3} t^{1/3} \dot{\adjustednewp} \dot{\newlap} 
		\notag \\
	& \ \ - 4 t \Big[\adjustednewp + \frac{1}{3} \Big]^2 \dot{\newu}^a \dot{\dlap}_a 
		\notag \\
	& \ \ + \frac{8}{3} t^{1/3} \Big[\adjustednewp + \frac{1}{3} \Big]^2 \newg_{ab} \dot{\newu}^a \dot{\newu}^b  \notag \\
	& \ \ - 4 t^{1/3} \big[1 + t^{4/3} \newlap\big] \Big[\adjustednewp + \frac{1}{3} \Big]^2 \newg_{ia} 
		 \freenewsec_{\ j}^a \dot{\newu}^i \dot{\newu}^j
		 \notag \\
	& \ \ + \dot{\adjustednewp} \dot{\mathfrak{P}} 
		+ 4 t^{4/3} \Big[\adjustednewp + \frac{1}{3} \Big]^2 \newg_{ab} \dot{\newu}^a \dot{\mathfrak{U}}^b 
		\notag \\
	& \ \ + 2 t^{4/3} \frac{\Big[\adjustednewp + \frac{1}{3} \Big]}{\big[1 + t^{4/3} \newg_{ab} \newu^a \newu^b \big]} 
		\dot{\adjustednewp} \newg_{ef} \newu^e \dot{\mathfrak{U}}^f 
		\notag \\
	& \ \ + 2 t^{4/3} \frac{\Big[\adjustednewp + \frac{1}{3} \Big]}
			{\big[1 + t^{4/3} \newg_{ab} \newu^a \newu^b \big]} 
			\newg_{ef} \newu^e \dot{\newu}^f \dot{\mathfrak{P}}
		\notag \\
	& \ \ - 4 t^{8/3} \frac{\Big[\adjustednewp + \frac{1}{3} \Big]^2}
			{\big[1 + t^{4/3} \newg_{ab} \newu^a \newu^b \big]} \newg_{ef} \newu^e 
			\dot{\newu}^f \newg_{ij} \newu^i \dot{\mathfrak{U}}^j \notag \\
	& \ \ + \sum_{l=1}^3 \triangle_{\dot{\mathbf{J}}_{(Fluid);(Junk)_l}[(\dot{\adjustednewp},\dot{\newu}), (\dot{\adjustednewp},\dot{\newu})]}, \notag
\end{align}
where the error terms $\triangle_{\dot{\mathbf{J}}_{(Fluid);(Junk)_l}[(\dot{\adjustednewp},\dot{\newu}), (\dot{\adjustednewp},\dot{\newu})]}$ are
\begin{subequations}
\begin{align}
		\triangle_{\dot{\mathbf{J}}_{(Fluid);(Junk)_1}[(\dot{\adjustednewp},\dot{\newu}), (\dot{\adjustednewp},\dot{\newu})]}
		& := 2 \left\lbrace \partial_t \left(t^{4/3} \frac{\Big[\adjustednewp + \frac{1}{3} \Big]}
			{\big[1 + t^{4/3} \newg_{ab} \newu^a \newu^b\big]} \newg_{ef} \newu^e \right) \right\rbrace  \dot{\newu}^f \dot{\adjustednewp} 
			\label{E:FLUIDCURRENTJUNK1} \\
		& \ \ + 2 t^{4/3} \Big\lbrace \partial_t \Big(\Big[\adjustednewp + \frac{1}{3} \Big]^2 \Big) \Big\rbrace 
			\newg_{ef} \dot{\newu}^e \dot{\newu}^f  \notag \\
		& \ \ - 2 t^{8/3} \Big\lbrace \partial_t \Big(\Big[\adjustednewp + \frac{1}{3} \Big]^2 \Big) \Big\rbrace 
			\left\lbrace\frac{\big[\newg_{ef} \newu^e \dot{\newu}^f\big]^2}{\big[1 + t^{4/3} \newg_{ab} \newu^a \newu^b\big]} 
			 \right\rbrace \notag \\
		& \ \ - 4 t^{4/3} \Big[\adjustednewp + \frac{1}{3} \Big]^2 
			\left\lbrace \frac{\big[\newg_{cd} \newu^c \dot{\newu}^d \big] 
				 \big[\partial_t (t^{4/3} \newg_{ef} \newu^e) \big] \dot{\newu^f}} {\big[1 + t^{4/3} \newg_{ab} \newu^a \newu^b\big]}
				\right\rbrace \notag \\
		& \ \ + 2 t^{8/3} \Big[\adjustednewp + \frac{1}{3} \Big]^2 
			\left\lbrace \frac{\big[\newg_{cd} \newu^c \dot{\newu}^d \big]^2}
			{\big[1 + t^{4/3} \newg_{ab} \newu^a \newu^b \big]^2} \partial_t \big[t^{4/3} \newg_{ef} \newu^e \newu^f \big]
		\right\rbrace, \notag
		\end{align}
		\begin{align}
			\triangle_{\dot{\mathbf{J}}_{(Fluid);(Junk)_2}[(\dot{\adjustednewp},\dot{\newu}), (\dot{\adjustednewp},\dot{\newu})]} 
			& := t^{1/3} \left\lbrace \partial_j \left( \frac{\big[1 + t^{4/3} \newlap\big] \newu^j}
				{2\big[1 + t^{4/3} \newg_{ab} \newu^a \newu^b\big]^{1/2}} \right) \right\rbrace
				\dot{\adjustednewp}^2  \\
			& \ \ + 2 t^{1/3} \left\lbrace \partial_j \left( \frac{\big[1 + t^{4/3} \newlap\big] \Big[\adjustednewp + \frac{1}{3} \Big]}
				{\big[1 + t^{4/3} \newg_{ab} \newu^a \newu^b\big]^{1/2}} \right) \right\rbrace
				\dot{\newu}^j \dot{\adjustednewp} \notag \\
		& \ \ + 2 t^{5/3} \left\lbrace \partial_j \left( \frac{\big[1 + t^{4/3} \newlap\big] \Big[\adjustednewp + \frac{1}{3} \Big]^2 \newu^j}
				{\big[1 + t^{4/3} \newg_{ab} \newu^a \newu^b\big]^{1/2}} \right) \right\rbrace
				\left\lbrace \newg_{ef} \dot{\newu}^e \dot{\newu}^f
					- \frac{t^{4/3} \big[\newg_{ef} \newu^e \dot{\newu}^f\big]^2}
					{\big[1 + t^{4/3} \newg_{ab} \newu^a \newu^b\big]} 
					\right\rbrace \notag \\
		& \ \ + 2 t^{5/3} \frac{\big[1 + t^{4/3} \newlap\big] \Big[\adjustednewp + \frac{1}{3} \Big]^2 \newu^j}
				{\big[1 + t^{4/3} \newg_{ab} \newu^a \newu^b\big]^{1/2}} 
				(\partial_j \newg_{ef}) \dot{\newu^e} \dot{\newu^f} \notag \\
		& \ \ -4 t^3 \frac{\big[1 + t^{4/3} \newlap\big] \Big[\adjustednewp + \frac{1}{3} \Big]^2 \newu^j}
					{\big[1 + t^{4/3} \newg_{ab} \newu^a \newu^b\big]^{1/2}} 
				\left\lbrace \frac{\big[\newg_{cd} \newu^c \dot{\newu}^d\big] 
				 	\big[\partial_j (\newg_{ef} \newu^e) \big] \dot{\newu^f}} {\big[1 + t^{4/3} \newg_{ab} \newu^a \newu^b\big]}
					\right\rbrace \notag \\
		& \ \ + 2 t^{13/3} \frac{\big[1 + t^{4/3} \newlap\big] \Big[\adjustednewp + \frac{1}{3} \Big]^2 \newu^j}
					{\big[1 + t^{4/3} \newg_{ab} \newu^a \newu^b\big]^{1/2}} 
				\left\lbrace \frac{\big[\newg_{cd} \newu^c \dot{\newu}^d \big]^2}
					{\big[1 + t^{4/3} \newg_{ab} \newu^a \newu^b\big]^2} 
					\partial_j \big[\newg_{ef} \newu^e \newu^f \big]
					\right\rbrace, \notag 
\end{align}
\begin{align} \label{E:FLUIDCURRENTJUNK3}
		\triangle_{\dot{\mathbf{J}}_{(Fluid);(Junk)_3}[(\dot{\adjustednewp},\dot{\newu}), (\dot{\adjustednewp},\dot{\newu})]}  
		& := \frac{4}{3} t^{5/3} \newlap \Big[\adjustednewp + \frac{1}{3} \Big]^2 \newg_{ab} \dot{\newu}^a \dot{\newu}^b
			- \frac{4}{3} t^{5/3} \frac{\Big[\adjustednewp + \frac{1}{3} \Big]}
			{\big[1 + t^{4/3} \newg_{ab} \newu^a \newu^b \big]} 
			\newg_{ef} \newu^e \dot{\newu}^f \dot{\newlap} \\
	& \ \ - 2 t \frac{\Big[\adjustednewp + \frac{1}{3} \Big]}{\big[1 + t^{4/3} \newg_{ab} \newu^a \newu^b \big]} \dot{\adjustednewp} 
		\newu^c \dot{\dlap}_c
		+ 4 t^{7/3} \frac{\Big[\adjustednewp + \frac{1}{3} \Big]^2}{\big[1 + t^{4/3} \newg_{ab} \newu^a \newu^b \big]} 
			\newg_{ef} \newu^e \dot{\newu}^f \newu^c \dot{\dlap}_c. \notag 
\end{align}

\end{subequations}

\end{lemma}

\begin{proof}
	The error term $\triangle_{\dot{\mathbf{J}}_{(Fluid);(Junk)_1}[(\dot{\adjustednewp},\dot{\newu}), (\dot{\adjustednewp},\dot{\newu})]}$ 
	contains almost all of the
	terms that are generated when $\partial_t$ hits the coefficients of the variations in \eqref{E:FLUIDCURRENT0}. Three
	exceptional terms of this type are singled out and are not included in 
	$\triangle_{\dot{\mathbf{J}}_{(Fluid);(Junk)_1}[(\dot{\adjustednewp},\dot{\newu}), (\dot{\adjustednewp},\dot{\newu})]}.$ The first exceptional term 
	arises when $\partial_t$ falls on the $t^{4/3}$ factor in the product
	$2 \Big[\adjustednewp + \frac{1}{3} \Big]^2 t^{4/3} \newg_{ef} \dot{\newu}^e \dot{\newu}^f$ from \eqref{E:FLUIDCURRENT0}. This generates the 
	$\frac{8}{3} t^{1/3} \Big[\adjustednewp + \frac{1}{3} \Big]^2 \newg_{ab} \dot{\newu}^a \dot{\newu}^b$ term on the right-hand side of 
	\eqref{E:DIVFLUIDJ}. The second and third exceptional terms arise when
	$\partial_t$ falls on the $\newg_{ef}$ factor in the product 
	$2 \Big[\adjustednewp + \frac{1}{3} \Big]^2 t^{4/3} \newg_{ef} \dot{\newu}^e \dot{\newu}^f$ 
	from \eqref{E:FLUIDCURRENT0}. We use equation \eqref{E:GEVOLUTION} to substitute for $\partial_t \newg_{ef}$
	and place one of the resulting two terms, namely 
	$- 4 t^{1/3} \big[1 + t^{4/3} \newlap\big] \Big[\adjustednewp + \frac{1}{3} \Big]^2 
	\newg_{ia} \freenewsec_{\ j}^a  \dot{\newu}^i 
	\dot{\newu}^j,$ on the right-hand side of \eqref{E:DIVFLUIDJ}. The other resulting term 
	$\frac{4}{3} t^{5/3} \newlap \Big[\adjustednewp + \frac{1}{3} \Big]^2 
	\newg_{ab} \dot{\newu}^a \dot{\newu}^b$ is placed on the right-hand side of \eqref{E:FLUIDCURRENTJUNK3} as
	part of $\triangle_{\dot{\mathbf{J}}_{(Fluid);(Junk)_3}[(\dot{\adjustednewp},\dot{\newu}), (\dot{\adjustednewp},\dot{\newu})]}.$
	
	The error term $\triangle_{\dot{\mathbf{J}}_{(Fluid);(Junk)_2}[(\dot{\adjustednewp},\dot{\newu}), (\dot{\adjustednewp},\dot{\newu})]}$ contains 
	precisely the terms that are generated when $\partial_j$ hits the coefficients of the variations in \eqref{E:FLUIDCURRENTJ}.
	
	The remaining terms on the right-hand sides of \eqref{E:DIVFLUIDJ} and \eqref{E:FLUIDCURRENTJUNK3} are generated when the 
	spacetime derivatives $\partial_t$ and $\partial_j$ fall on $(\dot{\adjustednewp},\dot{\newu}).$
	 More precisely, we use the equations of variation 
	\eqref{E:EOVP}-\eqref{E:EOVU} to replace the derivatives of
	$(\dot{\adjustednewp},\dot{\newu})$ with the terms on the right-hand side of \eqref{E:EOVP}-\eqref{E:EOVU}; we omit the tedious 
	calculations that correspond to this replacement.
	
\end{proof}

\section{Bootstrap Assumptions, Energy Definitions, and Energy Coerciveness} \label{S:BOOTSTRAPANDENERGYDEF}

In this section, we state our bootstrap assumptions for the solution norms. We also define 
the metric and fluid energies and provide a simple lemma that reveals
their coerciveness properties.

\subsection{Bootstrap assumptions}

Recall that our solution norms are defined in \eqref{E:HIGHNORM}-\eqref{E:LOWPOTNORM}.
Our proof of stable singularity formation 
is based on the following bootstrap assumptions,
which we assume on a time interval $t \in (T,1]:$
\begin{subequations}
\begin{align} 
	\highnorm{N}(t) & \leq \epsilon t^{-\upsigma}, && t \in (T,1], 
		\label{E:HIGHBOOT} \\
	\lowkinnorm{N-3}(t) & \leq 1, && t \in (T,1], 
		\label{E:KINBOOT} \\
	\lowpotnorm{N-4}(t) & \leq t^{-\upsigma}, && t \in (T,1]. \label{E:POTBOOT}
\end{align}
\end{subequations}
Above, $\epsilon$ and $\upsigma$ are small positive constants whose smallness will
be adjusted throughout the course of our analysis.
%Our main justifications for the above assumptions are 
%\textbf{i)} they weakly capture certain aspects of the behavior near-FLRW perturbed Kasner solutions;
%\textbf{ii)} they allow us to prove our main theorem.

\subsection{Definitions of the energies}
The main idea of our proof of stable singularity formation is to derive \emph{strict} 
improvements of the bootstrap assumptions \eqref{E:HIGHBOOT}-\eqref{E:POTBOOT}
under near-FLRW assumptions on the data (given at $t=1$). In order to derive these improvements, we will apply integration by parts along the hypersurfaces $\Sigma_t$ to the $\partial_{\vec{I}}-$commuted equations, which
are specific instances of the equations of variation. Equivalently, we will apply the divergence theorem using the energy
currents introduced in the previous section. The energies are the coercive geometric quantities that will
naturally emerge from the integration by parts identities. In this section, we define the energies. In Section
\ref{S:COMPARISON}, we will connect the energies to the norms $\highnorm{M}(t)$ and 
the quantities
$\sum_{1 \leq |\vec{I}| \leq M} \| | \partial_{\vec{I}} \newg |_{\newg} \|_{L^2}^2
	+ \| | \partial_{\vec{I}} \newg^{-1} |_{\newg} \|_{L^2}^2.$

\begin{definition}[\textbf{Metric and fluid energies}] \label{D:METRICANDFLUIDENERGIES}
Let $\Big(\freenewsec_{\ j}^i, \upgamma_{j \ k}^{\ i} \Big)_{1 \leq i,j,k \leq 3}$ be the array of renormalized metric variables, and let 
$\Big(\adjustednewp, \newu^i \Big)_{1 \leq i \leq 3}$ be the array of renormalized fluid variables. 
We define the metric energies $\metricenergy{M}(t) \geq 0$ and the fluid energies $\fluidenergy{M}(t) \geq 0$ by
\begin{subequations}
\begin{align}
		\metricenergy{M}^2(t) 
		&:= \sum_{|\vec{I}| \leq M}\int_{\Sigma_{\tau}} 
		\dot{\mathbf{J}}_{(Metric)}^0[\partial_{\vec{I}} (\freenewsec, \upgamma), \partial_{\vec{I}} (\freenewsec, \upgamma)]  \, dx, 
		\label{E:METRICENERGY} \\
	\fluidenergy{M}^2(t) 
		&:= \sum_{|\vec{I}| \leq M}\int_{\Sigma_{\tau}} 
		\dot{\mathbf{J}}_{(Fluid)}^0[\partial_{\vec{I}} (\adjustednewp, \newu), \partial_{\vec{I}} (\adjustednewp, \newu)]  \, dx,
		\label{E:FLUIDENERGY}
\end{align}
where $\dot{\mathbf{J}}_{(Metric)}^0[\partial_{\vec{I}} (\freenewsec, \upgamma), \partial_{\vec{I}} (\freenewsec, \upgamma)]$ and
$\dot{\mathbf{J}}_{(Fluid)}^0[\partial_{\vec{I}} (\adjustednewp, \newu), \partial_{\vec{I}} (\adjustednewp, \newu)]$ are defined in
\eqref{E:METRICCURRENT0} and \eqref{E:FLUIDCURRENT0}.
\end{subequations}

\end{definition}

\begin{remark}	
	Note that the energies $\metricenergy{M}^2(t)$ do not directly control the quantities
	 $\sum_{1 \leq |\vec{I}| \leq M} \| | \partial_{\vec{I}} \newg |_{\newg} \|_{L^2}^2 
	 + \| | \partial_{\vec{I}} \newg^{-1} |_{\newg} \|_{L^2}^2.$ We will derive separate estimates
	 to control these latter quantities (see Prop.~\ref{P:SOBFORG}).
\end{remark}

\subsection{Coerciveness of the energies}

For the background FLRW solution, we have $\metricenergy{M}(t) \equiv 0$
and $\fluidenergy{M}(t) \equiv 0.$ In the next lemma, we explicitly quantify the coercive nature of the energies in terms of the
$\| \cdot \|_{H_{\newg}^M}$ norms of the solution variables (which are defined in Def.~\ref{D:CMNORMS}).
\begin{lemma} [\textbf{Coerciveness of the metric and fluid energies}] \label{L:COERCIVITYOFTHEENERGIES}
Assume that
\begin{align}
	\left\| \adjustednewp \right\|_{C^0} & \leq \epsilon, 
	&& 
	t^{2/3} \left\| \newu \right\|_{C_{\newg}^0} \leq \epsilon.
\end{align}

There exists a constant $\epsilon_0 > 0$ such that if $\epsilon \leq \epsilon_0,$ 
then the following coerciveness estimates hold:
\begin{subequations}
\begin{align}
	\metricenergy{M} & \approx \left\| \freenewsec \right\|_{H_{\newg}^M}
		+ t^{2/3} \| \upgamma \|_{H_{\newg}^M}, \\
	\fluidenergy{M} & \approx \left\| \adjustednewp \right\|_{H^M} 
		+ t^{2/3} \| \newu \|_{H_{\newg}^M}.
\end{align}
\end{subequations}

\end{lemma}

\begin{remark}
	In Prop.~\ref{P:STRONGPOINTWISE},
	we show that the hypotheses of Lemma~\ref{L:COERCIVITYOFTHEENERGIES} hold for near-FLRW solutions.
\end{remark}

\begin{proof}
	Lemma~\ref{L:COERCIVITYOFTHEENERGIES} is a straightforward consequence of Def.~\ref{D:METRICANDFLUIDENERGIES},
	the expressions \eqref{E:METRICCURRENT0} and \eqref{E:FLUIDCURRENT0} for
	$\dot{\mathbf{J}}_{(Metric)}^0[\partial_{\vec{I}} (\freenewsec, \upgamma), \partial_{\vec{I}} (\freenewsec, \upgamma)] $ and
	$\dot{\mathbf{J}}_{(Fluid)}^0[\partial_{\vec{I}} (\adjustednewp, \newu), \partial_{\vec{I}} (\adjustednewp, \newu)],$ and the 
	$\newg-$Cauchy-Schwarz inequality.
	
\end{proof}

\section{Strong Estimates for the lower-order Derivatives} \label{S:STRONGESTIMATES}
In this section, we use the bootstrap assumptions of Sect.~\ref{S:BOOTSTRAPANDENERGYDEF} to derive strong estimates for the
$C_{\newg}^M$ norms of the lower-order derivatives of the renormalized solution variables. The strong estimates
provide additional information beyond that provided by the bootstrap assumptions and Sobolev embedding.
This additional information plays an essential role in the proof of our main stable singularity formation theorem.
Our derivation of the strong estimates relies on the special structure of the equations verified by the renormalized variables.
In particular, we exhibit and exploit an effective partial dynamical decoupling of various renormalized solution variables.
As an intermediate step, we derive estimates for the \emph{components} of the renormalized solution variables relative to the transported
spatial coordinate frame. Interestingly, in some cases, we are able to prove better estimates for the
$C_{Frame}^M$ norms than we are able to prove for the more invariant $C_{\newg}^M$ norms
[compare the estimates \eqref{E:SECONDFUNDUPGRADEPOINTWISE} and \eqref{E:SECONDFUNDUPGRADEPOINTWISEGNORM} in the cases $1 \leq M \leq N - 3$].
Roughly speaking, we derive the strong estimates by treating the evolution equations verified by the components of the 
lower-order derivatives as ODEs with source terms that have a favorable $t-$weighted structure. The sources depend on higher-order derivatives
and hence our estimates incur some loss in derivatives. We collect together these estimates in Prop.~\ref{P:STRONGPOINTWISE}. 
We stress that the order in which we derive the strong estimates in the proof of the proposition is important.

\begin{proposition}[\textbf{Strong estimates for the lower-order derivatives}]
\label{P:STRONGPOINTWISE}
Assume that on the spacetime slab $(T,1] \times \mathbb{T}^3,$ the renormalized variables 
$\big(\newg_{ij}, (\newg^{-1})^{ij}, \upgamma_{j \ k}^{\ i}, \freenewsec_{\ j}^i, \newlap, \dlap_i, \adjustednewp, \newu^i \big)$ 
verify the constraint equations \eqref{E:RINTERMSOFKPANDU}-\eqref{E:MOMENTUMCONSTRAINTRAISED}, 
the lapse equations \eqref{E:LAPSERESCALEDELLIPTIC}-\eqref{E:LAPSELOWERDERIVATIVES}, 
and the evolution equations 
\eqref{E:GEVOLUTION}-\eqref{E:GINVERSEEVOLUTION},
\eqref{E:LITTLEGAMMAEVOLUTIONDETGFORM}-\eqref{E:RESCALEDKEVOLUTION},
and \eqref{E:RESCALEDPEVOLUTION}-\eqref{E:RESCALEDUEVOLUTION}.
Assume that for some integer $N \geq 8,$ the bootstrap assumptions \eqref{E:HIGHBOOT}-\eqref{E:POTBOOT} hold for $t \in (T,1].$
Assume further that the initial renormalized metric verifies the near-Euclidean condition $\| \newg - \Euc \|_{C_{Frame}^0}(1) \leq \epsilon.$
Then there exist a small constant $\upsigma_N > 0$ and an integer $Z_N > 0$ depending on $N$ but
\textbf{not depending on the other constants} such that if $\epsilon \leq \upsigma \leq \upsigma_N,$ 
then the following estimates also hold on $(T,1]$ for the renormalized metric $\newg,$ its inverse $\newg^{-1},$ and
their derivatives:
\begin{subequations}
\begin{align}
	\left \| \partial_t \big[\newg - \Euc \big] \right \|_{C_{Frame}^M} 
		& \lesssim  \epsilon t^{-1} \| \newg - \Euc \|_{C_{Frame}^M} + \epsilon t^{-1}, &&  (M \leq N - 3),
		\label{E:PARTIALTGCOMMUTEDSTRONGPOINTWISE} 
			\\
	\left \| \partial_t \big[\newg^{-1} - \Euc^{-1} \big] \right \|_{C_{Frame}^M} 
		& \lesssim  \epsilon t^{-1} \left \| \newg^{-1} - \Euc^{-1} \right \|_{C_{Frame}^M} + \epsilon t^{-1}, &&  (M \leq N - 3), 
			\label{E:PARTIALTGINVERSECOMMUTEDSTRONGPOINTWISE} \\
	\| \newg - \Euc \|_{C_{Frame}^M} + \| \newg - \Euc \|_{C_{\newg}^M}
		& \lesssim \epsilon t^{-c \epsilon}, &&  (M \leq N - 3), 
			\label{E:GCOMMUTEDSTRONGPOINTWISE} \\
		\| \newg^{-1} - \Euc^{-1} \|_{C_{Frame}^M} + \| \newg^{-1} - \Euc^{-1} \|_{C_{\newg}^M}
			& \lesssim \epsilon t^{-c \epsilon}, && (M \leq N - 3). 
			\label{E:GINVERSECOMMUTEDSTRONGPOINTWISE} 
\end{align}
\end{subequations}
Above, $\Euc_{ij} = \mbox{diag}(1,1,1)$ denotes the standard Euclidean metric on $\mathbb{T}^3,$
and the norms $\| \cdot \|_{C_{Frame}^M}$ and $\| \cdot \|_{C_{\newg}^M}$ are defined in Def.~\ref{D:CMNORMS}.

The following estimates also hold on $(T,1]$ for the renormalized trace-free second fundamental form $\freenewsec$
and its derivatives:
\begin{subequations}
		\begin{align}
		\left \| \partial_t \freenewsec \right \|_{C_{Frame}^M}
		& \lesssim \epsilon t^{-1/3 - Z_N \upsigma}, && (M \leq N - 3),
			\label{E:PARTIALTSECONDFUNDPOINTWISE} \\
		\left \| \freenewsec  \right \|_{C_{Frame}^M} & \lesssim \epsilon, && (M \leq N - 3),
			\label{E:SECONDFUNDUPGRADEPOINTWISE} \\
		\left \| \freenewsec \right \|_{C_{\newg}^M} 
		& \lesssim \left\lbrace \begin{array}{ll} 
		\epsilon, & (M = 0), \\
	  \epsilon t^{- c \epsilon}, & (1 \leq M \leq N - 3).
		\end{array} \right.	 \label{E:SECONDFUNDUPGRADEPOINTWISEGNORM}
		\end{align}
\end{subequations}

The following estimates also hold on $(T,1]$ for $\upgamma$ and its derivatives:
\begin{subequations}
\begin{align}
		\| \partial_t  \upgamma \|_{C_{Frame}^M}
		& \lesssim \epsilon t^{-1} \| \upgamma \|_{C_{Frame}^M}
			+ \epsilon t^{-1}, && (M \leq N - 4), 
		\label{E:PARTIALTLITTLEGAMMACOMMUTEDPOINTWISE} \\
	\| \upgamma \|_{C_{Frame}^M} + \| \upgamma \|_{C_{\newg}^M}
			& \lesssim \epsilon t^{-c \epsilon}, && (M \leq N - 4).
		\label{E:LITTLEGAMMACOMMUTEDSTRONGPOINTWISE}
\end{align}
\end{subequations}

The following estimates also hold on $(T,1]$ for the renormalized pressure $\adjustednewp$ and its derivatives:
\begin{subequations}
	\begin{align}
		\left\| \partial_t \adjustednewp \right \|_{C^M} 
		& \lesssim \epsilon t^{-1/3 - Z_N \upsigma}, && (M \leq N - 3),	\label{E:PARTIALTPCOMMUTEDPOINTWISE} \\
		\left\| \adjustednewp \right\|_{C^M} 
		& \lesssim \epsilon, && (M \leq N - 3).  \label{E:PCOMMUTEDSTRONGPOINTWISE} 
	\end{align}
\end{subequations}

The following estimates also hold on $(T,1]$ for the renormalized velocity $\newu$ and its derivatives:
\begin{subequations}
	\begin{align}	
		\| \partial_t \newu \|_{C_{Frame}^M}
		& \lesssim \epsilon t^{-1} \| \newu \|_{C_{Frame}^M}
			+ \epsilon t^{-1 - c \epsilon}, && (M \leq N - 4), \label{E:PARTIALTUCOMMUTEDPOINTWISE} \\
		\| \newu \|_{C_{Frame}^M} + \| \newu \|_{C_{\newg}^M}
			& \lesssim \sqrt{\epsilon} t^{-c \sqrt{\epsilon}}, && (M \leq N - 4).
			\label{E:UCOMMUTEDSTRONGPOINTWISE}
	\end{align}
	\end{subequations}
	
	The following estimates also hold on $(T,1]$ for the renormalized lapse $\newlap,$ the renormalized lapse gradient
	$\dlap,$ and their derivatives:
	\begin{subequations}	
	\begin{align}
		\| \newlap \|_{C^M} 
		& \lesssim \epsilon t^{- c \sqrt{\epsilon}}, && (M \leq N - 5),
			\label{E:LAPSECOMMUTEDSTRONGPOINTWISE}  \\
		\| \dlap \|_{C_{Frame}^M}
		+ \| \dlap \|_{C_{\newg}^M} & \lesssim \epsilon t^{2/3 - c \sqrt{\epsilon}}, && (M \leq N - 6),
		\label{E:GNORMLAPSECOMMUTEDSTRONGPOINTWISE}
	\end{align}	
	\end{subequations}
	\begin{align}
		\| \partial_t \newlap \|_{C^0} 
			& \lesssim \epsilon t^{-1 - c \sqrt{\epsilon}}.
			\label{E:PARTIALTLAPSESTRONGPOINTWISE}
	\end{align}
	
\end{proposition}

\begin{proof}
	We first discuss the top level strategy. Recall that the norms $\| \cdot \|_{C_{Frame}^M}$
	are defined in Def.~\ref{D:CMNORMS}.
	To prove the desired estimates, we will have to bound quantities
	that are schematically of the form
	\begin{align} \label{E:SAMPLEEST}
		\left\| F(t;t^{A_1}v_1, t^{A_2}v_2,\cdots, t^{A_l} v_l) \prod_{a=1}^l (\partial_{\vec{I}_a}v_a) \right\|_{C_{Frame}^0},
	\end{align}	
	where for $1 \leq a \leq l,$
	\begin{align}
		v_a \in \left\lbrace 
			\newg, (\newg^{-1}), 
			\freenewsec, 
			\upgamma, 
			\adjustednewp, 
			\newu, 
			\newlap, 
			\dlap 
			\right\rbrace
	\end{align}
	is a renormalized solution tensor, 
	the $A_a$ are positive constants, and $F$ is a smooth scalar-valued function 
	of its arguments that, by virtue of the bootstrap assumptions, will verify 
	\begin{align}
		\| F(t;t^{A_1}v_1, t^{A_2}v_2,\cdots, t^{A_l} v_l) \|_{C^0} \lesssim 1. 
	\end{align}
	For the multi-indices $\vec{I}$ under consideration, the
	bootstrap assumptions \eqref{E:HIGHBOOT}-\eqref{E:POTBOOT} allow us to estimate the quantity \eqref{E:SAMPLEEST} 
	using the following strategy:
	\begin{itemize}
		\item If $|\vec{I}_a|$ is small, then at worst we have $\|\partial_{\vec{I}_a} v_a \|_{C_{Frame}^0} \lesssim t^{- \upsigma}.$
		\item If $|\vec{I}_a|$ is large (but still within the range we are considering in this proof), then at worst we have 
			$\|\partial_{\vec{I}_a} v_a\|_{C_{Frame}^0} 
			\lesssim \epsilon t^{-2/3 - \upsigma}.$ To derive this estimate, we used Sobolev embedding and the norm bootstrap 
			assumption \eqref{E:HIGHBOOT}. More precisely, we will often make use of the Sobolev embedding
			result $H_{Frame}^{M+2}(\mathbb{T}^3) \hookrightarrow C_{Frame}^M(\mathbb{T}^3)$ without explicitly mentioning it.
		\item The integer $N$ has been chosen to be large enough so that in any product, 
			\emph{there is only one term with a ``large'' index}.
	\end{itemize}
	
	Consequently, following this strategy, we deduce the following crude bound:
	\begin{align} \label{E:REPLINFINITYBOUND}
		\left\|F(t^{A_1}v_1, t^{A_2}v_2,\cdots, t^{A_l} v_l) \prod_{a=1}^l (\partial_{\vec{I}_a}v_a) \right\|_{C_{Frame}^0} 
		\lesssim  \prod_{a=1}^l \left\| \partial_{\vec{I}_a}v_a \right\|_{C_{Frame}^0}
		\lesssim \epsilon t^{- 2/3 - Z \upsigma},
	\end{align}
	where 
	\begin{align}
		Z \leq l
	\end{align}
	is a non-negative integer (which we view as a constant that is free to vary from line to line) that counts the number of factors in the 
	product $ \prod_{a=1}^l (\partial_{\vec{I}_a}v_a)$ that contribute a $t^{- \upsigma}$ factor. Clearly $Z$ is
	\textbf{independent of $\epsilon$ and $\upsigma.$} 
	%Note that inequality \eqref{E:VOLFORMLOWERORDERDERIVATIVESIMRPOVED}
	%shows that the term $\sqrt{\gbigdet}$ does not contribute a loss of $t^{- \upsigma}$ to any of the estimates of this section.
	
	For the majority of the quadratic and higher-order terms that we will encounter, the estimate \eqref{E:REPLINFINITYBOUND} will
	suffice for our purposes. However, when we are deriving the strong estimates for $\newg,$ $\newg^{-1},$
	$\upgamma,$ $\newu,$ 
	$\newlap,$ and $\dlap,$ some of the quadratic terms
	will not be bounded with the crude estimate \eqref{E:REPLINFINITYBOUND} but will instead be bounded using the
	already-established strong estimates for other quantities. \emph{That is, we will derive the strong estimates in a particular order, and
	the order is essential.} Some of the linear terms that we will encounter can also be treated with the crude estimate 
	\eqref{E:REPLINFINITYBOUND},
	while other linear terms will require special care. 
		
	We begin the detailed proof of the proposition by noting that the crude	estimate \eqref{E:REPLINFINITYBOUND} 
	implies the following bounds for the junk terms appearing on the right-hand side of equations
	\eqref{AE:GEVOLUTIONCOMMUTED}-\eqref{AE:GINVERSEEVOLUTIONCOMMUTED},
	\eqref{AE:METRICGAMMACOMMUTED}-\eqref{AE:SECONDFUNDCOMMUTED}, 
	and \eqref{AE:PARTIALTPCOMMUTED}-\eqref{AE:PARTIALTUJCOMMUTED}:
	\begin{align} \label{E:STRONGPOINTWISEJUNKNORMS}
		& \left\| \leftexp{(\vec{I});(Junk)}{\mathfrak{G}} \right \|_{C_{Frame}^{N-3}} 
			+ \left\| \leftexp{(\vec{I});(Junk)}{\widetilde{\mathfrak{G}}} \right \|_{C_{Frame}^{N-3}}  
			+ \left\| \leftexp{(\vec{I});(Junk)}{\mathfrak{g}} \right \|_{C_{Frame}^{N-3}} 
			\\
		& \ \ + \left\| \leftexp{(\vec{I});(Junk)}{\mathfrak{K}} \right \|_{C_{Frame}^{N-3}} 
			+ \left\| \leftexp{(\vec{I});(Junk)}{\mathfrak{P}} \right \|_{C^{N-3}} 
			+ \left\| \leftexp{(\vec{I});(Junk)}{\mathfrak{U}} \right \|_{C_{Frame}^{N-3}} 
			\lesssim \epsilon t^{-2/3 - Z \upsigma}.
			\notag
	\end{align}

	\noindent \textbf{Proof of \eqref{E:PARTIALTSECONDFUNDPOINTWISE}, \eqref{E:SECONDFUNDUPGRADEPOINTWISE}, 
	\eqref{E:SECONDFUNDUPGRADEPOINTWISEGNORM} in the case $M=0$ only,
	\eqref{E:PARTIALTPCOMMUTEDPOINTWISE}, and \eqref{E:PCOMMUTEDSTRONGPOINTWISE}}: To derive \eqref{E:PARTIALTSECONDFUNDPOINTWISE}, we let 
	$|\vec{I}| \leq M \leq N - 3$ be a multi-index. We have to estimate the terms on the right-hand side of 
	the evolution equation \eqref{AE:SECONDFUNDCOMMUTED} for $\partial_{\vec{I}} \freenewsec_{\ j}^i.$
	The term $\leftexp{(\vec{I});(Junk)}{\mathfrak{K}}_{\ j}^i$ has already been suitably bounded in \eqref{E:STRONGPOINTWISEJUNKNORMS}.
	The remaining terms, some of which are linear, can be bounded in a similar fashion with the help of  
	the bootstrap assumptions \eqref{E:HIGHBOOT}-\eqref{E:POTBOOT}
	(see equation \eqref{E:ICONTRACTEDGAMMA} for the definition of $\leftexp{(\vec{I})}{\Gamma_a}$):
	\begin{align} \label{E:PARTIALTKSTRONGBOUNDSPROOF}
		& t^{1/3} \left\| \big[1 + t^{4/3} \newlap\big] (\newg^{-1})^{ef} \partial_e 
		\partial_{\vec{I}} \upgamma_{f \ j}^{\ i} \right\|_{C^0}
		+ t^{1/3} \left\| \big[1 + t^{4/3} \newlap\big] 
			(\newg^{-1})^{ia} \partial_a \leftexp{(\vec{I})}{\Gamma_j} \right\|_{C^0} \\
		& \ \ + t^{1/3} \left\| \big[1 + t^{4/3} \newlap\big]  
			(\newg^{-1})^{ia}\partial_j \leftexp{(\vec{I})}{\Gamma_a} \right\|_{C^0} 
			+ t \left\| (\newg^{-1})^{ia} \partial_a \partial_{\vec{I}} \dlap_j \right\|_{C^0} 
			\notag \\
		& \ \ + t^{1/3} \left\| \partial_{\vec{I}} \newlap \ID_{\ j}^i	\right\|_{C^0}	
		\lesssim \epsilon t^{-1/3 - Z \upsigma}. \notag
	\end{align}	
	The one term that requires some special attention is the next-to-last term on the left-hand side of \eqref{E:PARTIALTKSTRONGBOUNDSPROOF}.
	The bootstrap assumptions and Sobolev embedding imply that 
	$\left\| (\newg^{-1})^{ia} \partial_a \partial_{\vec{I}} \dlap_j \right\|_{C^0}
	\lesssim t^{-\upsigma} \left\| \dlap \right\|_{H_{Frame}^N} \lesssim \epsilon t^{-4/3 - 2 \upsigma}.$ Hence, we truly need
	the factor $t$ that multiplies this term in order to reach the desired conclusion. We have thus shown \eqref{E:PARTIALTSECONDFUNDPOINTWISE}.
	
	Inequality \eqref{E:SECONDFUNDUPGRADEPOINTWISE} follows from integrating \eqref{E:PARTIALTSECONDFUNDPOINTWISE} in time
	when $M \leq N - 3:$
	\begin{align} \label{E:KSTRONGPROOF}
		\left\| \freenewsec \right \|_{C_{Frame}^M}(t)
		& \lesssim
		\left\| \freenewsec \right \|_{C_{Frame}^M}(1)
		+ \epsilon \int_{s=t}^1 s^{-1/3 -Z \upsigma} \, ds 
			\lesssim \epsilon, 
	\end{align}
	where we used the small-data estimate
	$ \left\| \freenewsec \right \|_{C_{Frame}^{N-3}}(1) \leq C \epsilon,$
	and we have assumed that $\upsigma$ is small enough for $t^{-1/3 -Z \upsigma}$ to be integrable over the 
	interval $t \in (0,1].$
	
	To derive \eqref{E:SECONDFUNDUPGRADEPOINTWISEGNORM}	in the case $M = 0,$  
	we use the symmetry property $\SecondFund_{ij} = \SecondFund_{ji}$ to deduce
	\begin{align} \label{E:IMPORTANTSECONDFUNDINEQUALITY}
			\left| \freenewsec \right|_{\newg}^2
				:= (\newg^{-1})^{ab} \newg_{ij} \freenewsec_{\ a}^i \freenewsec_{\ b}^j
				& = \left| \freenewsec_{\ b}^a \right| 
					\left| \freenewsec_{\ a}^b \right|
					\lesssim \left\| \freenewsec \right \|_{C_{Frame}^0}^2.
	\end{align}
	The desired estimate \eqref{E:SECONDFUNDUPGRADEPOINTWISEGNORM} then follows from 
	\eqref{E:IMPORTANTSECONDFUNDINEQUALITY} and \eqref{E:SECONDFUNDUPGRADEPOINTWISE}.
	
	Inequalities \eqref{E:PARTIALTPCOMMUTEDPOINTWISE} and \eqref{E:PCOMMUTEDSTRONGPOINTWISE} can be derived in a similar
	fashion with the help of the evolution equation \eqref{AE:PARTIALTPCOMMUTED} for $\partial_{\vec{I}} \newp;$ 
	we omit the details.
	\\
	
	\noindent \textbf{Proof of \eqref{E:PARTIALTGCOMMUTEDSTRONGPOINTWISE}-\eqref{E:GINVERSECOMMUTEDSTRONGPOINTWISE}
	and \eqref{E:SECONDFUNDUPGRADEPOINTWISEGNORM} in the cases $1 \leq M \leq N - 3$}:
	To derive \eqref{E:PARTIALTGCOMMUTEDSTRONGPOINTWISE}, we let $|\vec{I}| \leq M \leq N - 3$
	be a multi-index. We then use equation \eqref{AE:GEVOLUTIONCOMMUTED} to derive an evolution equation 
	of the form $\partial_t \partial_{\vec{I}} \big[\newg_{ij} - \Euc_{ij} \big] = \cdots.$ Next, we
	use the bootstrap assumptions \eqref{E:HIGHBOOT}-\eqref{E:POTBOOT} and the 
	strong estimate \eqref{E:SECONDFUNDUPGRADEPOINTWISE} to 
	bound the right-hand side of the evolution equation in the norm $C_{Frame}^0,$
	which easily leads to the desired estimate \eqref{E:PARTIALTGCOMMUTEDSTRONGPOINTWISE}.
	Inequality \eqref{E:PARTIALTGINVERSECOMMUTEDSTRONGPOINTWISE} can be derived in 
	a similar fashion with the help of the evolution equation \eqref{AE:GINVERSEEVOLUTIONCOMMUTED}.
	We stress that for our upcoming estimates, it is essential that the coefficient of 
	$\left \| \newg - \Euc  \right \|_{C_{Frame}^M}$
	on the right-hand side of \eqref{E:PARTIALTGCOMMUTEDSTRONGPOINTWISE} is precisely
	$\epsilon t^{-1},$ and not a worse power of $t.$ The same remark applies to inequalities 
	\eqref{E:PARTIALTGINVERSECOMMUTEDSTRONGPOINTWISE},
	\eqref{E:PARTIALTLITTLEGAMMACOMMUTEDPOINTWISE},
	and \eqref{E:PARTIALTUCOMMUTEDPOINTWISE}.
	
	To derive the estimate for $\| \newg - \Euc \|_{C_{Frame}^M}$ in \eqref{E:GCOMMUTEDSTRONGPOINTWISE}, 
	we integrate \eqref{E:PARTIALTGCOMMUTEDSTRONGPOINTWISE} in time and use the small-data estimate
	$\left\| \newg - \Euc \right \|_{C_{Frame}^{N-3}}(1) \leq C \epsilon$ to deduce the following
	inequality, which is valid for $M \leq N - 3:$
	\begin{align} \label{E:GIMRPROVEDGRONWALLREADY}
		\left\| \newg - \Euc  \right \|_{C_{Frame}^M}(t)
		& \leq \left\| \newg - \Euc \right \|_{C_{Frame}^M}(1)
			+ C \epsilon \int_{s = t}^1 s^{-1} \, ds 
			+ c \epsilon \int_{s = t}^1 s^{-1} \left\| \newg - \Euc \right \|_{C_{Frame}^M}(s) \, ds \\
		& \leq C \epsilon (1 + |\ln t|) + c \epsilon \int_{s = t}^1 s^{-1} \left\| \newg - \Euc \right \|_{C_{Frame}^M}(s) \, ds. \notag
	\end{align}
	Applying Gronwall's inequality to \eqref{E:GIMRPROVEDGRONWALLREADY}, we deduce that
	\begin{align}
		\left\| \newg - \Euc  \right \|_{C_{Frame}^M}(t) & \leq C \epsilon (1 + |\ln t|) t^{- c \epsilon}.
	\end{align}
	We have thus bounded $\left\| \newg - \Euc  \right \|_{C_{Frame}^M}$ by the right-hand side of \eqref{E:GCOMMUTEDSTRONGPOINTWISE} 
	as desired. Similarly, $\left\| \newg^{-1} - \Euc^{-1} \right \|_{C_{Frame}^M}$ can be bounded by the right-hand side of
	\eqref{E:GINVERSECOMMUTEDSTRONGPOINTWISE} with the help of inequality \eqref{E:PARTIALTGINVERSECOMMUTEDSTRONGPOINTWISE}.
	Finally, from the definition \eqref{E:CGMFRAMENORMDEF} of the norm $\| \cdot \|_{C_{\newg}^M},$ we deduce that 
	\begin{align}
		\left\| \newg - \Euc \right \|_{C_{\newg}^M} & \lesssim \left\| \newg^{-1} \right \|_{C_{Frame}^0} \left\| \newg - \Euc \right 
		\|_{C_{Frame}^M}, \\ 
		\left\| \newg^{-1} - \Euc^{-1} \right \|_{C_{\newg}^M} & \lesssim \left \| \newg \right \|_{C_{Frame}^0} 
		\left\| \newg^{-1} - \Euc^{-1} \right \|_{C_{Frame}^M}.
	\end{align}	
	Therefore, $\left\| \newg - \Euc \right \|_{C_{\newg}^M}$ and 
	$\left\| \newg^{-1} - \Euc^{-1} \right \|_{C_{\newg}^M}$
	can each be bounded by $C \epsilon (1 + |\ln t|) t^{- c \epsilon}$ by using the already-established bounds for 
	$\left\| \newg - \Euc \right \|_{C_{Frame}^M}$ and $\left\| \newg^{-1} - \Euc^{-1} \right \|_{C_{Frame}^M}.$
	
	Similarly, the estimate \eqref{E:SECONDFUNDUPGRADEPOINTWISEGNORM}	in the cases $1 \leq M \leq N - 3$
	follows from the estimate \eqref{E:SECONDFUNDUPGRADEPOINTWISE}
	and the already-established bounds \eqref{E:GCOMMUTEDSTRONGPOINTWISE} and \eqref{E:GINVERSECOMMUTEDSTRONGPOINTWISE} for 
	$\left \|\newg - \Euc \right \|_{C_{Frame}^0}$ and $\left \|\newg^{-1} - \Euc^{-1} \right\|_{C_{Frame}^0}.$
	\\
	
	\noindent \textbf{Proof of \eqref{E:PARTIALTLITTLEGAMMACOMMUTEDPOINTWISE}-\eqref{E:LITTLEGAMMACOMMUTEDSTRONGPOINTWISE}}:	To derive 
	\eqref{E:PARTIALTLITTLEGAMMACOMMUTEDPOINTWISE}, we let $|\vec{I}| \leq M \leq N-4$ be a multi-index. We 
	will use the bootstrap assumptions \eqref{E:HIGHBOOT}-\eqref{E:POTBOOT} and the strong estimate
	\eqref{E:SECONDFUNDUPGRADEPOINTWISE} to estimate the terms appearing on the right-hand side of 
	the evolution equation \eqref{AE:METRICGAMMACOMMUTED} for $\partial_{\vec{I}} \upgamma_{e \ i}^{\ b}.$ 
	The first two terms on the right-hand side of \eqref{AE:METRICGAMMACOMMUTED} are bounded by
	\begin{align}
		t^{-1} \left|\big[1 + t^{4/3} \newlap\big] \partial_e \partial_{\vec{I}} \freenewsec_{\ i}^b  \right| & \lesssim \epsilon t^{-1}, \\
		t^{-1/3}\left| \partial_{\vec{I}} \dlap_e \ID_{\ i}^b \right| & \lesssim \epsilon t^{-1/3 - \upsigma}. 
			\label{E:DLAPBOUNDFORGAMMAEVOLUTION}
	\end{align}
	Note in particular that in deriving \eqref{E:DLAPBOUNDFORGAMMAEVOLUTION}, we used the
	assumption $|\vec{I}| \leq N - 4$ in order to bound 
	$|\partial_{\vec{I}} \dlap_e| \lesssim \| \dlap \|_{H_{Frame}^{N-2}} \lesssim \epsilon t^{-\upsigma}.$
	Using similar reasoning, we bound the $t^{-1} \leftexp{(\vec{I});(Border)}{\mathfrak{g}_{e \ i}^{\ b}}$ term on the right-hand side of
	\eqref{AE:METRICGAMMACOMMUTED} by
	\begin{align}
		t^{-1} \left| \leftexp{(\vec{I});(Border)}{\mathfrak{g}_{e \ i}^{\ b}} \right| 
			& \lesssim \epsilon t^{-1} \| \upgamma \|_{C_{Frame}^M} + \epsilon t^{-1/3 - \upsigma}.
	\end{align}
	The term $\leftexp{(\vec{I});(Junk)}{\mathfrak{g}_{e \ i}^{\ b}}$ on the right-hand side of \eqref{AE:METRICGAMMACOMMUTED}
	has been already been suitably bounded in \eqref{E:STRONGPOINTWISEJUNKNORMS}.
	Combining these estimates, we have thus shown \eqref{E:PARTIALTLITTLEGAMMACOMMUTEDPOINTWISE}.
	
	Inequality \eqref{E:LITTLEGAMMACOMMUTEDSTRONGPOINTWISE} for $\| \upgamma \|_{C_{Frame}^M}$ 
	(when $M \leq N - 4$)  
	then follows from integrating \eqref{E:PARTIALTLITTLEGAMMACOMMUTEDPOINTWISE} in time,
	using the small-data estimate $\| \upgamma \|_{C_{Frame}^{N-4}}(1) \leq C \epsilon,$ and applying Gronwall's inequality 
	(as in our proof of \eqref{E:GCOMMUTEDSTRONGPOINTWISE}). 

	To obtain the desired bound for $\| \upgamma \|_{C_{\newg}^M},$ we first use
	the definition \eqref{E:CGMFRAMENORMDEF} of the norm $\| \cdot \|_{C_{\newg}^M}$ to deduce 
	\begin{align} \label{E:GAMMAGNORMINTERMSOFFRAMENORM}
		\| \upgamma \|_{C_{\newg}^M} 
		\lesssim \| \newg^{-1} \|_{C_{Frame}^0} \| \newg \|_{C_{Frame}^0}^{1/2} 
		\| \upgamma \|_{C_{Frame}^M}.
	\end{align}
	The desired bound for $\| \upgamma \|_{C_{\newg}^M}$ in \eqref{E:LITTLEGAMMACOMMUTEDSTRONGPOINTWISE}
	then follows from  \eqref{E:GAMMAGNORMINTERMSOFFRAMENORM},
	the already-established bound for $\| \upgamma \|_{C_{Frame}^M},$
	and the strong estimates \eqref{E:GCOMMUTEDSTRONGPOINTWISE} and \eqref{E:GINVERSECOMMUTEDSTRONGPOINTWISE}
	for $\| \newg - \Euc \|_{C_{Frame}^0}$ and $\| \newg^{-1} - \Euc^{-1} \|_{C_{Frame}^0}.$
	\\
	
	%\noindent \textbf{Proof of \eqref{E:GCOMMUTEDSTRONGPOINTWISE}-\eqref{E:GINVERSECOMMUTEDSTRONGPOINTWISE}}: 
	%These estimates follow from Lemma~\ref{L:INTERMSOFGAMMA} and the estimates 
	%\eqref{E:GSTRONGPOINTWISE}, \eqref{E:GINVERSESTRONGPOINTWISE}, and
	%\eqref{E:LITTLEGAMMACOMMUTEDSTRONGPOINTWISE}. 
	%\\

	\noindent \textbf{Proof of \eqref{E:PARTIALTUCOMMUTEDPOINTWISE}-\eqref{E:UCOMMUTEDSTRONGPOINTWISE}}: 
	Let $|\vec{I}| \leq M \leq N - 4$ be a multi-index.
	To deduce the estimate \eqref{E:PARTIALTUCOMMUTEDPOINTWISE}, we have to estimate the terms in the evolution equation \eqref{AE:PARTIALTUJCOMMUTED}
	for $\partial_{\vec{I}} \newu^j.$
	The $\leftexp{(\vec{I});(Junk)}{\mathfrak{U}}^j$ term on the right-hand side of equation \eqref{AE:PARTIALTUJCOMMUTED} was suitably 
	bounded in \eqref{E:STRONGPOINTWISEJUNKNORMS}.	Using the bootstrap assumptions \eqref{E:HIGHBOOT}-\eqref{E:POTBOOT},
	we bound the term $- t^{-1/3}(\newg^{-1})^{ja} \partial_{\vec{I}} \dlap_a$ on the right-hand side by
	\begin{align}
		t^{-1/3} \left|(\newg^{-1})^{ja} \partial_{\vec{I}} \dlap_a \right| & \lesssim \epsilon t^{-1/3 - 2\upsigma}.
	\end{align}
	
	To estimate the term $t^{-1} \leftexp{(\vec{I});(Border)}{\mathfrak{U}}^j$ on the right-hand side of \eqref{AE:PARTIALTUJCOMMUTED},
	we use the bootstrap assumptions \eqref{E:HIGHBOOT}-\eqref{E:POTBOOT} and the strong estimates
	\eqref{E:GCOMMUTEDSTRONGPOINTWISE}, \eqref{E:GINVERSECOMMUTEDSTRONGPOINTWISE},
	\eqref{E:SECONDFUNDUPGRADEPOINTWISE}, \eqref{E:LITTLEGAMMACOMMUTEDSTRONGPOINTWISE},
	and \eqref{E:PCOMMUTEDSTRONGPOINTWISE} to deduce
	\begin{align}
		t^{-1} \left|\leftexp{(\vec{I});(Border)}{\mathfrak{U}}^j \right| 
			& \lesssim \epsilon t^{-1} \| \newu \|_{C_{Frame}^M} + \epsilon^2 t^{-1-c \epsilon},
	\end{align} 
	 which is clearly bounded by the right-hand side of \eqref{E:PARTIALTUCOMMUTEDPOINTWISE}.
	 
	We similarly estimate the last four terms on the \emph{left}-hand side of \eqref{AE:PARTIALTUJCOMMUTED}, thus
	arriving at the following inequalities:
	\begin{align}
		t^{1/3} \left|\big[1 + t^{4/3} \newlap\big] \big[1 + t^{4/3} \newg_{ab} \newu^a \newu^b\big]^{1/2} \newu^j \partial_c 
			\partial_{\vec{I}} \newu^c \right|
		& \lesssim \epsilon^2 t^{-1/3 - Z \upsigma}, \label{E:UJCOMMUTEDFIRSTTERMLEFTHAND} \\
		t^{5/3} \left| \frac{\big[1 + t^{4/3} \newlap\big]}{\big[1 + t^{4/3} \newg_{ab} \newu^a \newu^b\big]^{1/2}} 
			\newg_{cf}\newu^c  \newu^j \newu^e \partial_e \partial_{\vec{I}} \newu^f \right| 
		& \lesssim \epsilon^2 t^{1 - Z \upsigma}, \label{E:UJCOMMUTEDSECONDTERMLEFTHAND} \\	
			t^{1/3} \left| \frac{\big[1 + t^{4/3} \newlap\big] \newu^c \partial_c \partial_{\vec{I}} \newu^j}
			{\big[1 + t^{4/3} \newg_{ab} \newu^a \newu^b\big]^{1/2}} 
			\right|
		& \lesssim \epsilon^2 t^{-1/3 - Z \upsigma}, \label{E:UJCOMMUTEDTHIRDTERMLEFTHAND} \\	
			t^{-1} \left| \frac{\big[1 + t^{4/3} \newlap\big]
				\left\lbrace (\newg^{-1})^{jc} + t^{4/3} \newu^j \newu^c \right\rbrace \partial_c \partial_{\vec{I}} 	
				\adjustednewp}{2\big[1 + t^{4/3} \newg_{ab} \newu^a \newu^b\big]^{1/2} 
			\big[\adjustednewp + \frac{1}{3} \big]} \right|
		& \lesssim \epsilon t^{-1 - c \epsilon}. \label{E:UJCOMMUTEDFOURTHTERMLEFTHAND}
	\end{align}
	Note in particular that \eqref{E:UJCOMMUTEDFOURTHTERMLEFTHAND} involves a linear term, which yields only a single power of 
	$\epsilon$ on the right-hand side of \eqref{E:UJCOMMUTEDFOURTHTERMLEFTHAND}. Combining the above estimates, we arrive at inequality 
	\eqref{E:PARTIALTUCOMMUTEDPOINTWISE}.
	
	To prove \eqref{E:UCOMMUTEDSTRONGPOINTWISE}, we integrate inequality \eqref{E:PARTIALTUCOMMUTEDPOINTWISE}
	in time, use the trivial estimate $s^{-1 - c \epsilon} \leq s^{-1 - c \sqrt{\epsilon}}$
	for $s \in (0,1],$ and use the small-data estimate
	$\| \newu \|_{C_{Frame}^{N-4}}(1) \leq C \epsilon,$
	thereby arriving at the following inequality, which is valid for 
	$M \leq N-4:$ 
	\begin{align} \label{E:UMGRONWALLREADY}
		\| \newu \|_{C_{Frame}^M}(t)
		& \leq 	\| \newu \|_{C_{Frame}^M}(1) 
		+ C \epsilon \int_{s = t}^1  s^{-1 - c \sqrt{\epsilon}} \, ds
		+ c \epsilon \int_{s = t}^1 s^{-1} \| \newu \|_{C_{Frame}^M}(s) \, ds
		 \\
		& \leq C \sqrt{\epsilon} t^{-1 - c \sqrt{\epsilon}}
			+ c \epsilon \int_{s = t}^1  \| \newu \|_{C_{Frame}^M}(s) \, ds.
			\notag
	\end{align}
	Applying Gronwall's inequality to \eqref{E:UMGRONWALLREADY}, we deduce 
	\begin{align}  \label{E:UCOMMUTEDSTRONGPOINTWISEPROOF}
		\| \newu \|_{C_{Frame}^M} & \leq  C \sqrt{\epsilon} t^{-c \sqrt{\epsilon}}. 
	\end{align}
	We have thus bounded the term $\| \newu \|_{C_{Frame}^M}$ on the left-hand side of \eqref{E:UCOMMUTEDSTRONGPOINTWISE}
	by the right-hand side of \eqref{E:UCOMMUTEDSTRONGPOINTWISE}. 
	
	To obtain the desired bound for the term $\| \newu \|_{C_{\newg}^M},$ we first
	use the definition \eqref{E:CGMFRAMENORMDEF} of the norm $\| \cdot \|_{C_{\newg}^M}$ to deduce the inequality
	\begin{align} \label{E:GNORMUINTERMSOFFRAME}
		\| \newu \|_{C_{\newg}^M} 
		\leq 
		\| \newg \|_{C_{Frame}^0}^{1/2}
		\| \newu \|_{C_{Frame}^M}.
	\end{align}
	We then insert the bound \eqref{E:UCOMMUTEDSTRONGPOINTWISE} for $\| \newu \|_{C_{Frame}^M}$ 
	and the strong estimate \eqref{E:GCOMMUTEDSTRONGPOINTWISE} for $\| \newg - \Euc \|_{C_{Frame}^0}$ 
	into \eqref{E:GNORMUINTERMSOFFRAME}, which yields the desired estimate for 
	$\| \newu \|_{C_{\newg}^M}$ in \eqref{E:UCOMMUTEDSTRONGPOINTWISE}. 
	\\
	
	\noindent \textbf{Proof of \eqref{E:LAPSECOMMUTEDSTRONGPOINTWISE}-\eqref{E:GNORMLAPSECOMMUTEDSTRONGPOINTWISE}
		and \eqref{E:PARTIALTLAPSESTRONGPOINTWISE}}:
	To prove \eqref{E:LAPSECOMMUTEDSTRONGPOINTWISE}, we will apply the maximum principle
	to the elliptic PDE \eqref{AE:ALTERNATELAPSEICOMMUTEDLOWER} verified by $\partial_{\vec{I}} \newlap.$ 
	To this end, we first bound the inhomogeneous terms $\leftexp{(\vec{I});(Border)}{\widetilde{\mathfrak{N}}}$
	and $\leftexp{(\vec{I});(Junk)}{\widetilde{\mathfrak{N}}}$ on the right-hand side of \eqref{AE:ALTERNATELAPSEICOMMUTEDLOWER} 
	\emph{using only the strong estimates 
	\eqref{E:GCOMMUTEDSTRONGPOINTWISE}, 
	\eqref{E:GINVERSECOMMUTEDSTRONGPOINTWISE},
	\eqref{E:SECONDFUNDUPGRADEPOINTWISE},
	\eqref{E:LITTLEGAMMACOMMUTEDSTRONGPOINTWISE},
	\eqref{E:PCOMMUTEDSTRONGPOINTWISE}, 
	\eqref{E:UCOMMUTEDSTRONGPOINTWISE}, 
	together with the bootstrap assumptions $t^{2/3} \| \newlap \|_{H^{N-1}} + \| \dlap \|_{H_{Frame}^{N-2}} \lesssim \epsilon t^{- \upsigma}.$} 
	This results in the following bounds for $M \leq N - 5:$
	\begin{align} \label{E:WIDETILDENESTIMATEBORDER}
		\sum_{|\vec{I}| \leq M} \left\| \leftexp{(\vec{I});(Border)}{\widetilde{\mathfrak{N}}} \right\|_{C^0}
		& \lesssim \epsilon t^{- c \sqrt{\epsilon}},	\\
		\sum_{|\vec{I}| \leq M} \left\| \leftexp{(\vec{I});(Junk)}{\widetilde{\mathfrak{N}}} \right\|_{C^0}
		& \lesssim \epsilon t^{- \upsigma - c \sqrt{\epsilon}}. 
		\label{E:WIDETILDENESTIMATEJUNK}
	\end{align} 
	
	From similar reasoning, it follows that the term $\widetilde{f}$ from \eqref{E:TILDEELLIPTICOPERATORJUNKTERM} can be bounded 
	by
	\begin{align} \label{E:WIDETILDEFBOUND}
		\| \widetilde{f} \|_{C^0} & \lesssim \epsilon.
	\end{align}
	
	Then by examining equation \eqref{AE:ALTERNATELAPSEICOMMUTEDLOWER}
	and applying the maximum principle for the operator $t^{4/3} (\newg^{-1})^{ab} \partial_a \partial_b,$ 
	we see that at a maximum point for $\partial_{\vec{I}} \newlap,$ we must have
	$(1 + \widetilde{f})\partial_{\vec{I}} \newlap \leq - \leftexp{(\vec{I});(Border)}{\widetilde{\mathfrak{N}}} 
		- t^{2/3} \leftexp{(\vec{I});(Junk)}{\widetilde{\mathfrak{N}}}.$ Similarly, at a 
	minimum, we have $(1 + \widetilde{f}) \partial_{\vec{I}} \newlap \geq - \leftexp{(\vec{I});(Border)}{\widetilde{\mathfrak{N}}}
	- t^{2/3} \leftexp{(\vec{I});(Junk)}{\widetilde{\mathfrak{N}}}.$
	Thus, 
	\begin{align} \label{E:LAPSEMAXIMUMPRINCIPLEBOUND}
		\| \newlap \|_{C^M} \lesssim
			\sum_{|\vec{I}| \leq M} \left\| \leftexp{(\vec{I});(Border)}{\widetilde{\mathfrak{N}}}
				+ t^{2/3} \leftexp{(\vec{I});(Junk)}{\widetilde{\mathfrak{N}}} \right\|_{C^0}.
	\end{align}
	Combining \eqref{E:WIDETILDENESTIMATEBORDER}-\eqref{E:WIDETILDENESTIMATEJUNK} and \eqref{E:LAPSEMAXIMUMPRINCIPLEBOUND}, 
	we deduce that the following estimate holds for $M \leq N - 5:$
	\begin{align} \label{E:NEWLAPBOUNDEDPROOF}
		\| \newlap \|_{C^M} & \leq \epsilon t^{- c \sqrt{\epsilon}}.
	\end{align}
	Also taking into account \eqref{E:GINVERSECOMMUTEDSTRONGPOINTWISE} and the fact that 
	$\dlap = t^{2/3} \partial \newlap,$ we have proved \eqref{E:LAPSECOMMUTEDSTRONGPOINTWISE} and
	\eqref{E:GNORMLAPSECOMMUTEDSTRONGPOINTWISE} for the quantity $\| \dlap \|_{C_{Frame}^M}.$
	To deduce the desired bound for \eqref{E:GNORMLAPSECOMMUTEDSTRONGPOINTWISE} $\| \dlap \|_{C_{\newg}^M},$
	we first use the definition \eqref{E:CGMFRAMENORMDEF} of the norm $\| \cdot \|_{C_{\newg}^M}$ to deduce the inequality
	\begin{align} \label{E:GNORMDLAPINTERMSOFFRAMEDLAP}
		\| \dlap \|_{C_{\newg}^M} \leq  \| \newg^{-1} \|_{C_{Frame}^0}^{1/2} \| \dlap \|_{C_{Frame}^M}.
	\end{align}
	We then insert the bound \eqref{E:GNORMLAPSECOMMUTEDSTRONGPOINTWISE} for $\| \dlap \|_{C_{Frame}^M}$
	and the strong estimate \eqref{E:GINVERSECOMMUTEDSTRONGPOINTWISE} for $\| \newg^{-1} - \Euc^{-1} \|_{C_{Frame}^0}$ 
	into \eqref{E:GNORMDLAPINTERMSOFFRAMEDLAP}, which yields the desired estimate for $\| \dlap \|_{C_{\newg}^M}.$ 
	We remark that the restriction $M \leq N - 5$ comes from the 
	terms in $\leftexp{(\vec{I});(Border)}{\widetilde{\mathfrak{N}}}$
	that depend on $|\vec{I}| + 1$ derivatives of $\upgamma.$ These terms can be estimated (in part) using the strong estimate
	\eqref{E:LITTLEGAMMACOMMUTEDSTRONGPOINTWISE} whenever $|\vec{I}| \leq N - 5.$
	 
	 The proof of \eqref{E:PARTIALTLAPSESTRONGPOINTWISE} is similar to the proof of \eqref{E:LAPSECOMMUTEDSTRONGPOINTWISE},
	 so we only give partial details. Upon commuting \eqref{AE:ALTERNATELAPSEICOMMUTEDLOWER} with $\partial_t,$ we see that
	 the quantity $\partial_t \newlap$ verifies an elliptic PDE of the form 
	 \eqref{AE:ALTERNATELAPSEICOMMUTEDLOWER}, where
	 the differential operator $\partial_{\vec{I}}$ is replaced with $\partial_t.$ The previously proven 
	 estimates imply that at the points of maximum and minimum for $\partial_t \newlap$
	 (where $\partial_t \dlap_i = \frac{2}{3}t^{-1/3} \partial_i \newlap$),
	 the inhomogeneous terms in the elliptic PDE can be bounded in $C^0$ by 
	 $\lesssim \epsilon t^{-1 - c \sqrt{\epsilon}}.$ 
	 The desired estimate \eqref{E:PARTIALTLAPSESTRONGPOINTWISE} thus follows in the same way that \eqref{E:NEWLAPBOUNDEDPROOF} follows from 
	 \eqref{E:LAPSEMAXIMUMPRINCIPLEBOUND}.
	 
\end{proof}

\section{Preliminary Sobolev Estimates for the Lapse and the Key Coercive Quadratic Integral} \label{S:LAPSEKEYLINEAR}

In this section, we use the strong estimates of Prop.~\ref{P:STRONGPOINTWISE}
to derive some preliminary Sobolev estimates for the lapse variables $\newlap$ and $\dlap_i.$ The first lemma
and the corresponding corollary provide standard elliptic $L^2-$type bounds for these variables in terms of the 
inhomogeneous terms in the PDEs that they satisfy. 

In contrast, the estimates of Prop.~\ref{P:KEYSIGNED}
are subtle and rely on the special structure of the Einstein equations in CMC-transported spatial coordinates. 
%The analog of this proposition in the second model problem is the identity \eqref{E:MODELSCALARFIELDKEYIDENTITY}.
In Prop.~\ref{P:KEYSIGNED}, we analyze the spatial integral of the term $- \frac{2}{3} t^{1/3} \dot{\adjustednewp} \dot{\newlap}$
on the right-hand side of the expression \eqref{E:DIVFLUIDJ} for the divergence of the fluid energy current. This integral
will appear in Sect.~\ref{S:FUNDAMENATLENERGYINEQUALITIES} during our derivation of our main energy integral inequalities for the fluid.
In order to bound this quadratic integral, we use the lapse + Hamiltonian constraint equation \eqref{E:LAPSERESCALEDELLIPTIC} to substitute
for $\dot{\adjustednewp}.$ After integrating by parts, \textbf{we will discover that the resulting integral identity 
is coercive (with a good sign ``towards the past'') in the lapse variables.} This key estimate, which will lead to suitable 
control for $\newlap$ and $\dlap_i,$ is one of the main reasons that we are able to prove our main stable singularity formation theorem.
We think of this estimate as ``extra control of the lapse that comes for free from the fluid estimates''
because of the special structure of the equations.

We now derive the first lemma and its corollary.

\begin{lemma} [\textbf{Negative definiteness of the operator $\mathcal{L}$}] \label{L:NEGATIVEDEFINITE}
	Assume that the hypotheses and conclusions of Prop.~\ref{P:STRONGPOINTWISE} hold on the spacetime slab $(T,1] \times \mathbb{T}^3.$
	Then there exist a small constant $\upsigma_N > 0$ 
	and a large constant $C > 0$ 
	such that if $\epsilon \leq \upsigma \leq \upsigma_N,$ 
	then the operator 
	$\mathcal{L} := t^{4/3} (\newg^{-1})^{ab} \partial_a \partial_b - (1 + f)$  
	defined in \eqref{E:LDEF} is negative definite in the following sense for $t \in (T,1]:$
	\begin{align} \label{E:LNEGDEF}
		\int_{\Sigma_t} 
			\big(\partial_{\vec{I}} \newlap \big) \mathcal{L} \partial_{\vec{I}} \newlap  \, dx 
		& \leq - (1 - C \epsilon) \int_{\Sigma_t}  \Big|\partial_{\vec{I}} \dlap \big|_{\newg}^2 \, dx 
			- (1 - C \epsilon) \int_{\Sigma_t} \Big|\partial_{\vec{I}} \newlap \Big|^2 \, dx.
	\end{align}
	
\end{lemma}

\begin{proof}
	We first integrate by parts in the integral on the left-hand side of \eqref{E:LNEGDEF} to obtain
	\begin{align}  \label{E:LAPSENEGDEFID}
		\int_{\Sigma_t} 
			\big(\partial_{\vec{I}} \newlap \big) \mathcal{L} \partial_{\vec{I}} \newlap  \, dx 
		& = -  t^{4/3} \int_{\Sigma_t}  \Big|\partial_{\vec{I}} \partial \newlap \big|_{\newg}^2 \, dx 
			-  \int_{\Sigma_t} \Big|\partial_{\vec{I}} \newlap \Big|^2 \, dx \\
		& \ \ - \int_{\Sigma_t} f \Big|\partial_{\vec{I}} \newlap \Big|^2 \, dx 
			- t^{4/3} \int_{\Sigma_t} \left\lbrace \partial_a (\newg^{-1})^{ab} \right\rbrace
			\Big(\partial_{\vec{I}} \newlap \Big) \Big(\partial_b \partial_{\vec{I}} \newlap \Big) \, dx. 
			\notag
	\end{align}
	The term $f$ from \eqref{E:ELLIPTICOPERATORJUNKTERM} can be bounded as follows by using
	\eqref{E:GCOMMUTEDSTRONGPOINTWISE},
	\eqref{E:SECONDFUNDUPGRADEPOINTWISE}, 
	\eqref{E:PCOMMUTEDSTRONGPOINTWISE}, and \eqref{E:UCOMMUTEDSTRONGPOINTWISE}:
	\begin{align} \label{E:LAPSEOPERATORJUNKTERMBOUND}
		|f| & \leq C \epsilon.
	\end{align}
	From the $\newg-$Cauchy-Schwarz inequality and the estimates  
	\eqref{E:GINVERSECOMMUTEDSTRONGPOINTWISE} and \eqref{E:LAPSEOPERATORJUNKTERMBOUND}, it follows that the magnitude of the last three integrals
	on the right-hand side of \eqref{E:LAPSENEGDEFID} can be bounded by
	\begin{align} \label{E:LAPSEIBPJUNKESTIMATE}
			\lesssim \epsilon \left\| \partial_{\vec{I}} \newlap \right\|_{L^2}^2
			+ \epsilon t^{4/3} \left\| \Big| \partial_{\vec{I}} \partial \newlap \Big|_{\newg} \right\|_{L^2}^2.
	\end{align}
	The desired estimate \eqref{E:LNEGDEF} follows from \eqref{E:LAPSENEGDEFID}, \eqref{E:LAPSEIBPJUNKESTIMATE}, and
	the fact that $\partial_i \newlap = t^{-2/3} \dlap_i.$
		
\end{proof}

\begin{corollary}[\textbf{Preliminary bound for the lapse}] \label{C:PRELIMINARYLAPSEBOUND}
	Assume that the hypotheses and conclusions of Prop.~\ref{P:STRONGPOINTWISE} hold on the spacetime slab $(T,1] \times \mathbb{T}^3.$
	In particular, assume that $\newlap$ verifies the lapse equations \eqref{E:LAPSERESCALEDELLIPTIC}-\eqref{E:LAPSELOWERDERIVATIVES}, 
	and let $t^{-4/3} \leftexp{(\vec{I});(Border)}{\mathfrak{N}} + \leftexp{(\vec{I});(Junk)}{\mathfrak{N}}$
	and $\leftexp{(\vec{I});(Border)}{\widetilde{\mathfrak{N}}} + t^{2/3} \leftexp{(\vec{I});(Junk)}{\widetilde{\mathfrak{N}}}$
	be the inhomogeneous terms from the right-hand sides of the $\partial_{\vec{I}}-$commuted lapse equations 
	\eqref{AE:LAPSEICOMMUTED} and \eqref{AE:ALTERNATELAPSEICOMMUTEDLOWER}. 
	Then there exists a small constant $\upsigma_N > 0$  
	such that if $\epsilon \leq \upsigma \leq \upsigma_N,$ 
	then the following estimates hold for $t \in (T,1]:$
	\begin{subequations}
	\begin{align} 
		\left\| \partial_{\vec{I}} \newlap \right\|_{L^2}
		+ \left\| \Big| \partial_{\vec{I}} \dlap \Big|_{\newg} \right\|_{L^2}
		+ t^{2/3} \left\| \Big| \partial_{\vec{I}} \partial \dlap \Big|_{\newg} \right\|_{L^2}
		& \lesssim t^{-4/3} \left\| \partial_{\vec{I}} \adjustednewp \right\|_{L^2} & &
			\label{E:LAPSEICOMMUTEDINITIALELLIPTICESTIMATE} \\
		& \ \ + t^{-4/3} \left\| \leftexp{(\vec{I});(Border)}{\mathfrak{N}} \right\|_{L^2}
			+ \left\| \leftexp{(\vec{I});(Junk)}{\mathfrak{N}} \right\|_{L^2}, & & (|\vec{I}| \leq N),
		\notag \\
	\left\| \partial_{\vec{I}} \newlap \right\|_{L^2}
		+ \left\| \Big| \partial_{\vec{I}} \dlap \Big|_{\newg} \right\|_{L^2}
		& \lesssim  \left\| \leftexp{(\vec{I});(Border)}{\widetilde{\mathfrak{N}}} \right\|_{L^2}
			+ 	t^{2/3} \left\| \leftexp{(\vec{I});(Junk)}{\widetilde{\mathfrak{N}}} \right\|_{L^2}, & & (|\vec{I}| \leq N - 1). 
		\label{E:ALTERNATELAPSEICOMMUTEDINITIALELLIPTICESTIMATE}
	\end{align} 
	\end{subequations}
		
\end{corollary}

\begin{proof}
	Using equation \eqref{AE:LAPSEICOMMUTED} and Cauchy-Schwarz, 
	we estimate the left-hand side of \eqref{E:LNEGDEF} by
	\begin{align} \label{E:LAPSEICOMMUTEDINHOML2ESTIMATE}
		\left|\int_{\Sigma_t}
			\big(\partial_{\vec{I}} \newlap \big) \mathcal{L} \partial_{\vec{I}} \newlap \, dx \right|
		& \leq \frac{1}{2} \left\| \partial_{\vec{I}} \newlap \right\|_{L^2}^2 
			+ C t^{-8/3} \left\| \partial_{\vec{I}} \adjustednewp \right\|_{L^2}^2
			\\
		& \ \ + C t^{-8/3} \left\| \leftexp{(\vec{I});(Border)}{\mathfrak{N}} \right\|_{L^2}^2
		+ C \left\| \leftexp{(\vec{I});(Junk)}{\mathfrak{N}} \right\|_{L^2}^2. \notag
	\end{align}
	The desired estimate for the first two terms
	on the left-hand side of \eqref{E:LAPSEICOMMUTEDINITIALELLIPTICESTIMATE}
	now follows from \eqref{E:LNEGDEF}
	and \eqref{E:LAPSEICOMMUTEDINHOML2ESTIMATE}. The estimate
	\eqref{E:ALTERNATELAPSEICOMMUTEDINITIALELLIPTICESTIMATE} follows similarly with
	the help of equation \eqref{AE:ALTERNATELAPSEICOMMUTEDLOWER}. 
	
	To derive the estimate \eqref{E:LAPSEICOMMUTEDINITIALELLIPTICESTIMATE}
	for $\left\| \Big| \partial_{\vec{I}} \partial \dlap \Big|_{\newg} \right\|_{L^2},$ 
	we first use integration by parts to deduce
	\begin{align} \label{E:LAPSETOPDERIVATIVEIBPIDENTITY}
		  \int_{\Sigma_t} &	(\newg^{-1})^{ab} (\newg^{-1})^{cd}
			(\partial_a \partial_c \partial_{\vec{I}} \newlap)
			(\partial_b \partial_d \partial_{\vec{I}} \newlap)  \, dx 
			\\
		& =  \int_{\Sigma_t} \Big| (\newg^{-1})^{ab} 
				\partial_a \partial_b \partial_{\vec{I}} \newlap \Big|^2
				\, dx 
				\notag \\
			& \ \ +  \int_{\Sigma_t} \left(\partial_b \left\lbrace (\newg^{-1})^{ab} (\newg^{-1})^{cd}
					\right\rbrace \right) (\partial_a \partial_{\vec{I}} \newlap) 
					(\partial_c \partial_d \partial_{\vec{I}} \newlap) \, dx 
					\notag \\
			& \ \ -  \int_{\Sigma_t} \left( \partial_c \left\lbrace (\newg^{-1})^{ab} (\newg^{-1})^{cd} \right\rbrace \right) 
				(\partial_a \partial_{\vec{I}} \newlap) (\partial_b \partial_d \partial_{\vec{I}} \newlap) \, dx. 
				\notag
	\end{align}
	From \eqref{E:LAPSETOPDERIVATIVEIBPIDENTITY}, the estimate 
	\eqref{E:GINVERSECOMMUTEDSTRONGPOINTWISE}, and $\newg-$Cauchy-Schwarz, it follows that
	\begin{align} \label{E:LAPSETOPDERIVATIVEIBPFIRSTESTIMATE}
		\int_{\Sigma_t} (\newg^{-1})^{ab} (\newg^{-1})^{cd}
			(\partial_a \partial_c \partial_{\vec{I}} \newlap)
			(\partial_b \partial_d \partial_{\vec{I}} \newlap)  \, dx
			& \leq C \int_{\Sigma_t} \Big| (\newg^{-1})^{ab} 
				\partial_a \partial_b \partial_{\vec{I}} \newlap \Big|^2
				\, dx  \\
			& \ \ + C \epsilon t^{- c \epsilon} \int_{\Sigma_t} (\newg^{-1})^{ab}
						(\partial_a \partial_{\vec{I}}\newlap) 
						(\partial_b \partial_{\vec{I}}\newlap) 
						\, dx. \notag
	\end{align}
	The second integral on the right-hand side of \eqref{E:LAPSETOPDERIVATIVEIBPFIRSTESTIMATE} has already 
	been suitably bounded. To deduce the desired estimate for $ \partial_{\vec{I}} \partial \dlap,$ 
	it remains to estimate the first integral on the right-hand side of 
	\eqref{E:LAPSETOPDERIVATIVEIBPFIRSTESTIMATE}. To this end, we first use equation \eqref{AE:LAPSEICOMMUTED}
	to deduce
	\begin{align} \label{E:LAPSETOPDERIVATIVEISOLATED}
		t^{4/3} (\newg^{-1})^{ab} \partial_a \partial_b \partial_{\vec{I}} \newlap
		& = (1 + f) \partial_{\vec{I}} \newlap 
			+ 2 t^{-4/3} \partial_{\vec{I}} \adjustednewp
			+ t^{-4/3} \leftexp{(\vec{I});(Border)}{\mathfrak{N}} 
			+ \leftexp{(\vec{I});(Junk)}{\mathfrak{N}}.
	\end{align}
	We now square both sides of \eqref{E:LAPSETOPDERIVATIVEISOLATED},
	integrate over $\mathbb{T}^3,$
	and use \eqref{E:LAPSEOPERATORJUNKTERMBOUND} to deduce
	\begin{align} \label{E:LAPSETOPDERIVATIVECAUCHYSCHWARZ}
		t^{8/3} & \int_{\Sigma_t} \Big| (\newg^{-1})^{ab} \partial_a \partial_b \partial_{\vec{I}} \newlap \Big|^2
			\, dx \\
		& \leq C \int_{\Sigma_t} \Big|\partial_{\vec{I}} \newlap \Big|^2 \, dx 
			+ C t^{-8/3} \int_{\Sigma_t} \Big|\partial_{\vec{I}} \adjustednewp \Big|^2 \, dx
			\notag \\
		& \ \ + C t^{-8/3} \int_{\Sigma_t} \Big|\leftexp{(\vec{I});(Border)}{\mathfrak{N}}\Big|^2 \, dx
			+ C \int_{\Sigma_t} \Big|\leftexp{(\vec{I});(Junk)}{\mathfrak{N}}\Big|^2 \, dx.
			\notag
	\end{align}
	We now combine \eqref{E:LAPSETOPDERIVATIVEIBPFIRSTESTIMATE}, \eqref{E:LAPSETOPDERIVATIVECAUCHYSCHWARZ} and the previously proven 
	estimates for the first two terms on the left-hand side of \eqref{E:LAPSEICOMMUTEDINITIALELLIPTICESTIMATE} to deduce
	the desired estimate for
	 $t^{2/3} \left\| \Big| \partial_{\vec{I}} \partial \dlap \Big|_{\newg} \right\|_{L^2}
	 = t^{4/3} \left\| \Big| \partial_{\vec{I}} \partial \partial \newlap \Big|_{\newg} \right\|_{L^2}:$
	\begin{align}
			t^{8/3} \int_{\Sigma_t} (\newg^{-1})^{ab} (\newg^{-1})^{cd}
			(\partial_a \partial_c \partial_{\vec{I}} \newlap)
			(\partial_b \partial_d \partial_{\vec{I}} \newlap)  \, dx
		& \leq C t^{-8/3} \int_{\Sigma_t} \Big|\partial_{\vec{I}} \adjustednewp \Big|^2 \, dx
				\notag \\
		& \ \ + C t^{-8/3} \int_{\Sigma_t} \Big|\leftexp{(\vec{I});(Border)}{\mathfrak{N}}\Big|^2 \, dx
			+ C \int_{\Sigma_t} \Big|\leftexp{(\vec{I});(Junk)}{\mathfrak{N}}\Big|^2 \, dx.
			\notag
	\end{align}
	
\end{proof}

We now provide the key proposition that will enable us to control the lapse variables.

\begin{proposition} [\textbf{The key coercive quadratic integral}] \label{P:KEYSIGNED}
	Assume that the hypotheses and conclusions of Prop.~\ref{P:STRONGPOINTWISE} hold on the spacetime slab $(T,1] \times \mathbb{T}^3.$
	In particular, assume that the lapse equation \eqref{E:LAPSERESCALEDELLIPTIC}
	(which was derived under the assumption that the Hamiltonian constraint \eqref{E:RINTERMSOFKPANDU} holds)
	is verified on the same slab. Then there exist 
	a small constant $\upsigma_N > 0$ 
	and a large constant $C > 0$ 
	such that if $\epsilon \leq \upsigma \leq \upsigma_N$ 
	and $\upalpha > 0$ is any positive constant,
	then the following estimate holds for $t \in (T,1]:$
	\begin{align} \label{E:KEYSIGNEDICOMMUTEDINTEGRALESTIMATE}
		- 2 t^{1/3} & \int_{\Sigma_t} 
			\big(\partial_{\vec{I}} \newlap \big) \big(\partial_{\vec{I}} \adjustednewp \big) \, dx
			\\
	& \geq (1 - C \epsilon) t^{5/3} \int_{\Sigma_t} \Big|\partial_{\vec{I}} \dlap \big|_{\newg}^2 \, dx  
		+ (1 - C \epsilon - \upalpha) t^{5/3} \int_{\Sigma_t} \big|\partial_{\vec{I}} \newlap \big|^2 \, dx 
		\notag \\
	& \ \ - \upalpha^{-1} t^{-1} \int_{\Sigma_t} \Big|\leftexp{(\vec{I});(Border)}{\mathfrak{N}} \Big|^2  \, dx
			- \upalpha^{-1} t^{5/3} \int_{\Sigma_t}  \Big|\leftexp{(\vec{I});(Junk)}{\mathfrak{N}} \Big|^2 \, dx.
		\notag
		\end{align}
\end{proposition}

\begin{proof}
	We multiply the commuted equation \eqref{AE:LAPSEICOMMUTED} by $t^{5/3} \partial_{\vec{I}} \newlap$
	and use the Cauchy-Schwarz inequality for integrals 
	to deduce that
	\begin{align} \label{E:KEYLAPSEICOMMUTEDERROR}
		- 2 t^{1/3} \int_{\Sigma_t} 
			\big(\partial_{\vec{I}} \newlap \big) \big(\partial_{\vec{I}} \adjustednewp \big) \, dx 
		& \geq - t^{5/3} \int_{\Sigma_t} \big(\partial_{\vec{I}} \newlap \big) \mathcal{L} \partial_{\vec{I}} \newlap \, dx	\\
		& \ \ - \upalpha t^{5/3} \int_{\Sigma_t} |\partial_{\vec{I}} \newlap|^2 \, dx \notag \\
		& \ \ - \upalpha^{-1} t^{-1} \int_{\Sigma_t} \Big|\leftexp{(\vec{I});(Border)}{\mathfrak{N}} \Big|^2  \, dx
			- \upalpha^{-1} t^{5/3} \int_{\Sigma_t} \Big|\leftexp{(\vec{I});(Junk)}{\mathfrak{N}} \Big|^2 \, dx.
			\notag
	\end{align}
	The desired estimate \eqref{E:KEYSIGNEDICOMMUTEDINTEGRALESTIMATE} now follows easily from \eqref{E:KEYLAPSEICOMMUTEDERROR}
	and inequality \eqref{E:LNEGDEF}.

\end{proof}

\begin{remark}
	The integral on the left-hand side of \eqref{E:KEYSIGNEDICOMMUTEDINTEGRALESTIMATE} 
	arises in our main energy identity for the fluid (see the proof of Lemma~\ref{L:FLUIDLAPSEENERGYID}).	
	If we had tried to crudely bound this integral in absolute value by bounding 
	 $\partial_{\vec{I}} \adjustednewp$
	 in terms of the fluid energy and by using elliptic estimates to bound $\newlap$
	 in terms of the metric and fluid energies, then our top order energy estimates would not close.
	 Hence, we truly need the special structure revealed by Prop.~\ref{P:KEYSIGNED}.
	 One reason that the estimates would not close without Prop.~\ref{P:KEYSIGNED} 
	 is that the first term on the right-hand side of
	 \eqref{E:LAPSEICOMMUTEDINITIALELLIPTICESTIMATE} comes with a large (implicit) constant that would lead
	 to a damaging top-order energy integral. 
	 Another reason is that we need the positive
	 terms on the right-hand side of \eqref{E:KEYSIGNEDICOMMUTEDINTEGRALESTIMATE} in order to counter
	 some other dangerous top-order quadratic integrals (see the proof of Prop.~\ref{P:FUNDAMENTALENERGYINEQUALITY}).
\end{remark}

\section{The Fundamental Energy Integral Inequalities} \label{S:FUNDAMENATLENERGYINEQUALITIES}

In this section, we derive our fundamental energy integral inequalities for the near-FLRW solutions. The main ingredients are
the divergence identities for the metric and fluid currents provided in \eqref{E:DIVMETRICJ} and \eqref{E:DIVFLUIDJ},
the strong estimates of Prop.~\ref{P:STRONGPOINTWISE},
and the key estimate proved in Prop.~\ref{P:KEYSIGNED}, which will provide $L^2$ control of the lapse variables.

\subsection{The fundamental energy integral inequalities}

We begin by defining a family of energies for the metric + fluid solutions. Our fundamental energy 
integral inequality involves a member of this family.

\begin{definition}[\textbf{Total  metric + fluid energies}] \label{D:TOTALENERGY}
	Let $M \geq 0$ be an integer, and let $\metricenergy{M}(t)$ and $\fluidenergy{M}(t)$
	be the metric and fluid energies defined in \eqref{E:METRICENERGY} and \eqref{E:FLUIDENERGY}.
	For each real number $\smallparameter > 0,$ 
	we define $\totalenergy{\smallparameter}{M}(t) \geq 0$ by
	\begin{align} \label{E:TOTALENERGY}
		\totalenergy{\smallparameter}{M}^2 & := \smallparameter \metricenergy{M}^2 + \fluidenergy{M}^2.
	\end{align}

\end{definition}

Our main goal in this section is show that there exists a real number
$\smallparameter_* > 0$ such that $\totalenergy{\smallparameter_*}{M}^2(t)$
verifies a useful a priori integral inequality. The main result is contained in the 
next proposition.

\begin{proposition} [\textbf{The fundamental integral inequality for the total energies}] \label{P:FUNDAMENTALENERGYINEQUALITY}
Assume that the hypotheses and conclusions of Prop.~\ref{P:STRONGPOINTWISE} hold on the spacetime slab $(T,1] \times \mathbb{T}^3.$ 
Let $\totalenergy{\smallparameter}{M}(t)$ be the total solution energy defined in Def.~\ref{D:TOTALENERGY}. Then there exist 
a small positive constant $\smallparameter_* > 0$
and a small constant $\upsigma_N > 0$ such that if
$\epsilon \leq \upsigma \leq \upsigma_N$ and $0 \leq M \leq N,$
then $\totalenergy{\smallparameter_*}{M}(t)$ verifies the following integral inequality for $t \in (T,1]:$
\begin{align} \label{E:FUNDAMENTALENERGYINEQUALITY}
	\totalenergy{\smallparameter_*}{M}^2(t) 
	& + \int_{s=t}^{s=1} s^{1/3} 
		\left\| \upgamma \right\|_{H_{\newg}^M}^2 \, d s
		+ \int_{s=t}^{s=1} s^{1/3} \left\|  \newu \right\|_{H_{\newg}^M}^2  \, d s  \\
	& \ \ + \int_{s=t}^{s=1} s^{5/3} \left\| \newlap \right\|_{H^M}^2  \, d s
		+ \int_{s=t}^{s=1} s^{5/3} \left\| \dlap \right\|_{H_{\newg}^M}^2 \, d s \notag \\
	& \lesssim  \totalenergy{\smallparameter_*}{M}^2(1)
	\notag \\
	& \ \ + \int_{s=t}^{s=1} s^{-1/3} \left\| \freenewsec \right\|_{H_{\newg}^M}^2 
		\, d s 
		+ \int_{s=t}^{s=1} s^{-1/3} \left\| \adjustednewp \right\|_{H^M}^2 
		\, d s \notag \\
	& \ \ + \sum_{|\vec{I}| \leq M} \int_{s = t}^{s = 1} s^{5/3}   
		\left\| \leftexp{(\vec{I});(Junk)}{\mathfrak{N}} \right\|_{L^2}^2 \, ds \notag \\
	& \ \ + \sum_{|\vec{I}| \leq M} \int_{s = t}^{s = 1} s
			\left\| \left| \leftexp{(\vec{I});(Junk)}{\mathfrak{K}} \right|_{\newg} \right\|_{L^2}^2 \, d s	
			+ \sum_{|\vec{I}| \leq M} \int_{s = t}^{s = 1} 
			s^3 \left\| \left| \leftexp{(\vec{I});(Junk)}{\mathfrak{g}} \right|_{\newg} \right\|_{L^2}^2 \, ds 
			\notag \\
	& \ \ + \sum_{|\vec{I}| \leq M} \int_{s = t}^{s = 1} 
			s \left\|  \leftexp{(\vec{I});(Junk)}{\mathfrak{P}} \right\|_{L^2}^2\, ds 
			\notag \\
	& \ \ + \sum_{|\vec{I}| \leq M} \int_{s = t}^{s = 1} 
			 	s^3 \left\|  \left| \leftexp{(\vec{I});(Junk)}{\mathfrak{M}} \right|_{\newg} \right\|_{L^2}^2  \, d s	
		+ \sum_{|\vec{I}| \leq M} \int_{s = t}^{s = 1}
			s^3 \left\|  \left| \leftexp{(\vec{I});(Junk)}{\widetilde{\mathfrak{M}}} \right|_{\newg} \right\|_{L^2}^2 \, ds 
			\notag \\
	& \ \ + \sum_{|\vec{I}| \leq M} \int_{s = t}^{s = 1}
			s^3 \left\| \left| \leftexp{(\vec{I});(Junk)}{\mathfrak{U}} \right|_{\newg} \right\|_{L^2}^2 \, ds 
			\notag \\
	& \ \ + \epsilon \int_{s=t}^{s=1} 
		s^{-1} \left\| \freenewsec \right\|_{H_{\newg}^M}^2 \, d s  
		\notag \\
	& \ \ + \sum_{|\vec{I}| \leq M} \int_{s = t}^{s = 1} 
			s^{-1} \left\| \leftexp{(\vec{I});(Border)}{\mathfrak{N}} \right\|_{L^2}^2 \, ds 
			\notag \\
	& \ \ + \sum_{|\vec{I}| \leq M} \int_{s = t}^{s = 1} 
			s^{1/3} \left\| \left| \leftexp{(\vec{I});(Border)}{\mathfrak{g}} \right|_{\newg} \right\|_{L^2}^2 \, ds 
			\notag \\
	& \ \ + \sum_{|\vec{I}| \leq M} \int_{s = t}^{s = 1}
				s^{1/3} \left\|  \left| \leftexp{(\vec{I});(Border)}{\mathfrak{M}} \right|_{\newg} \right\|_{L^2}^2 \, ds 
		+ \sum_{|\vec{I}| \leq M} \int_{s = t}^{s = 1}
			s^{1/3} \left\|  \left| \leftexp{(\vec{I});(Border)}{\widetilde{\mathfrak{M}}} \right|_{\newg} \right\|_{L^2}^2 \, ds 
			\notag \\
	& \ \ +  \sum_{|\vec{I}| \leq M} 
			\int_{s = t}^{s = 1} 
				s^{1/3} \left\| \left| \leftexp{(\vec{I});(Border)}{\mathfrak{U}} \right|_{\newg} \right\|_{L^2}^2 \, ds. 
				\notag 
\end{align}

\end{proposition}

\begin{remark}
	Note that inequality \eqref{E:FUNDAMENTALENERGYINEQUALITY} involves the norms 
	$|\cdot|_{\newg}.$ These are the norms that we can access via integration by parts and hence these norms play an essential
	role in our analysis.
\end{remark}

The proof of Prop.~\ref{P:FUNDAMENTALENERGYINEQUALITY} is located in Sect.~\ref{SS:ENERGYINEQUALITYPROOF}. In the next two sections, we separately derive preliminary inequalities for the metric and fluid energies; we will combine these separate estimates in a suitable fashion to deduce \eqref{E:FUNDAMENTALENERGYINEQUALITY}.

\subsubsection{Preliminary metric energy inequalities}

\begin{lemma} [\textbf{Preliminary energy inequalities for the metric}] \label{L:METRICENERGYINTEGRALINEQUALITY}
Assume the hypotheses of Prop.~\ref{P:FUNDAMENTALENERGYINEQUALITY}, and let
$\metricenergy{M}(t)$ be the metric energy defined in \eqref{E:METRICENERGY}.
Then there exist a small constant $\upsigma_N > 0$
and a large constant $C > 0$ 
such that if $\epsilon \leq \upsigma \leq \upsigma_N$
and $0 \leq M \leq N,$ then $\metricenergy{M}(t)$
verifies the following integral inequality for $t \in (T,1]:$
\begin{align} \label{E:METRICENERGYMINTEGRALINEQUALITY}
	\metricenergy{M}^2(t) & + \left(\frac{1}{6} - C \sqrt{\epsilon} \right)
			\sum_{|\vec{I}| \leq M} \int_{s=t}^{s=1} s^{1/3} \int_{\Sigma_s} 
			\left|\partial_{\vec{I}} \upgamma \right|_{\newg}^2 \, dx \, ds 
			\\
	& \leq \metricenergy{M}^2(1)
			+ C \sum_{|\vec{I}| \leq M} \int_{s=t}^{s=1} s^{-1/3} \int_{\Sigma_s}	
				\left| \partial_{\vec{I}} \freenewsec \right|_{\newg}^2 \, dx \, ds 
			\notag \\ 
	& \ \ + C \epsilon \sum_{|\vec{I}| \leq M} \int_{s=t}^{s=1} s^{-1} \int_{\Sigma_s}	
				\left| \partial_{\vec{I}} \freenewsec \right|_{\newg}^2 \, dx \, ds 
			\notag \\ 
	& \ \ + C \sum_{|\vec{I}| \leq M} \int_{s=t}^{s=1} s^{1/3} \int_{\Sigma_s} 
		\left| \partial_{\vec{I}} \newu \right|_{\newg}^2 \, dx \, ds 
		\notag \\
	& \ \ + C \sum_{|\vec{I}| \leq M} \int_{s=t}^{s=1} s^{5/3} \int_{\Sigma_s} 
		\left|\partial_{\vec{I}} \dlap \right|_{\newg}^2 \, dx \, ds 
		\notag \\
	& \ \ + C \sum_{|\vec{I}| \leq M} \int_{s=t}^{s=1} s^{1/3} \int_{\Sigma_s}  
		 \left| \leftexp{(\vec{I});(Border)}{\mathfrak{M}} \right|_{\newg}^2  \, dx \, ds
		\notag \\
	& \ \ + C \sum_{|\vec{I}| \leq M} \int_{s=t}^{s=1} s^3 \int_{\Sigma_s}  
		\left| \leftexp{(\vec{I});(Junk)}{\mathfrak{M}} \right|_{\newg}^2 \, dx \, ds
		\notag \\
	& \ \ + C \sum_{|\vec{I}| \leq M} \int_{s=t}^{s=1} s^{1/3} \int_{\Sigma_s}  
		\left| \leftexp{(\vec{I});(Border)}{\widetilde{\mathfrak{M}}} \right|_{\newg}^2 \, dx \, ds
		\notag \\
	& \ \ + C \sum_{|\vec{I}| \leq M} \int_{s=t}^{s=1} s^3 \int_{\Sigma_s}  
		\left| \leftexp{(\vec{I});(Junk)}{\widetilde{\mathfrak{M}}} \right|_{\newg}^2 \, dx \, ds 
		\notag \\
	& \ \ + C \sum_{|\vec{I}| \leq M} 
		\int_{s=t}^{s=1} s^{1/3} \int_{\Sigma_s} 
		\left| \leftexp{(\vec{I});(Border)}{\mathfrak{g}} \right|_{\newg}^2 \, dx \, ds
		\notag \\
	& \ \ + C \sum_{|\vec{I}| \leq M} 
		\int_{s=t}^{s=1} s^3 \int_{\Sigma_s} 
		\left| \leftexp{(\vec{I});(Junk)}{\mathfrak{g}} \right|_{\newg}^2 \, dx \, ds
		\notag \\
	& \ \ + C \sum_{|\vec{I}| \leq M} 
		\int_{s=t}^{s=1} s \int_{\Sigma_s} 
		\left| \leftexp{(\vec{I});(Junk)}{\mathfrak{K}} \right|_{\newg}^2 \, dx \, ds.
		\notag 
\end{align}

\end{lemma}

\begin{proof}
	By the divergence theorem, we have that
	\begin{align}  \label{E:DIVTHMMETRIC}
		\metricenergy{M}^2(t) - \metricenergy{M}^2(1) & = 
			- \sum_{|\vec{I}| \leq M} \int_{s=t}^1 
			\int_{\Sigma_s} \partial_{\mu} \left( \dot{\mathbf{J}}_{(Metric)}^{\mu}
				[\partial_{\vec{I}} (\freenewsec, \upgamma), \partial_{\vec{I}} (\freenewsec, \upgamma)] \right) \, dx \, ds.
	\end{align}
	Our goal therefore is to estimate the right-hand side of \eqref{E:DIVTHMMETRIC} using the expression \eqref{E:DIVMETRICJ}. 
	The expression \eqref{E:DIVMETRICJ} is valid because the differentiated metric variables $\partial_{\vec{I}} (\freenewsec, \upgamma)$ 
	appearing in \eqref{E:DIVTHMMETRIC} verify the $\partial_{\vec{I}}-$commuted
	metric equations of Sect.~\ref{SSS:COMMUTEDMETRICRENORM}, 
	which are specific instances of the metric equations of variation.
	
	We recall that the variations $\dot{\freenewsec}$ etc. are defined in Def.~\ref{D:VARIATIONS}. 
	We first discuss the integrals corresponding to the term $\frac{1}{3} t^{1/3} (\newg^{-1})^{ab} 
	(\newg^{-1})^{ef} \newg_{ij} \dot{\upgamma}_{e \ a}^{\ i} \dot{\upgamma}_{f \ b}^{\ j}$ on the right-hand side of \eqref{E:DIVMETRICJ}. 
	The corresponding integrals appearing on the right-hand side of \eqref{E:DIVTHMMETRIC} are therefore
	\begin{align} \label{E:METRICGOODSPATIALINTEGRAL}
		- \frac{1}{3} \int_{s=t}^{s=1} s^{1/3} \int_{\Sigma_s}  (\newg^{-1})^{ab} (\newg^{-1})^{ef} \newg_{ij}  
		\big(\partial_{\vec{I}} \upgamma_{e \ a}^{\ i} \big) \big(\partial_{\vec{I}}\upgamma_{f \ b}^{\ j} \big) \,dx \, ds.
	\end{align}
	We then move the integrals \eqref{E:METRICGOODSPATIALINTEGRAL} over to the left-hand side of \eqref{E:METRICENERGYMINTEGRALINEQUALITY}. 
	We note that at this stage in the proof, these integrals are multiplied by a $+ \frac{1}{3}$ factor
	rather than the final $+ \frac{1}{6}$ factor that appears on the left-hand side of \eqref{E:METRICENERGYMINTEGRALINEQUALITY}.
	
	We next discuss the integrals corresponding to the term 
	$\frac{1}{3} t (\newg^{-1})^{ab} \dot{\dlap}_a \dot{\upgamma}_{b \ c}^{\ c}$ on the right-hand side of \eqref{E:DIVMETRICJ}. 
	We first note the following simple pointwise inequality, which is valid for any constant $\upbeta > 0:$
	\begin{align}
		\frac{1}{3} t \left| (\newg^{-1})^{ab} \dot{\dlap}_a \dot{\upgamma}_{b \ c}^{\ c} \right| 
			\leq \frac{\upbeta}{3}t^{1/3} |\dot{\upgamma}|_{\newg}^2 + t^{5/3} \frac{\upbeta^{-1}}{3} 
			|\dot{\dlap}|_{\newg}^2.
	\end{align}
	The corresponding integrals are therefore bounded in magnitude by
	\begin{align} \label{E:FLUIDCURRENTMAININTEGRALS}
		  \frac{\upbeta}{3} \int_{s=t}^{s=1} s^{1/3} \int_{\Sigma_s} \left| \partial_{\vec{I}} \upgamma \right|_{\newg}^2 \, dx \, ds
		+ \frac{\upbeta^{-1}}{3}  \int_{s=t}^{s=1} s^{5/3} \int_{\Sigma_s} \left| \partial_{\vec{I}} \dlap \right|_{\newg}^2 \, dx \, ds.
	\end{align}
	If $\upbeta$ is small enough, then the first integral in 
	\eqref{E:FLUIDCURRENTMAININTEGRALS} can be absorbed into the left-hand side of \eqref{E:METRICENERGYMINTEGRALINEQUALITY}, 
	while the second integral
	is clearly bounded by the right-hand side of \eqref{E:METRICENERGYMINTEGRALINEQUALITY}.
	The absorbing reduces the aforementioned $+ \frac{1}{3}$ factor (by a small amount if $\upbeta$ is small).
	The spacetime integrals corresponding to the terms 
	\begin{align*}
		& - \frac{4}{3} t^{1/3} \dot{\newu}^a \dot{\Gamma}_a, 
		\qquad - t^{1/3} \big[1 + t^{4/3} \newlap \big] \dot{\widetilde{\mathfrak{M}}}^a \dot{\Gamma}_a,
		\\
		& - t^{1/3} \big[1 + t^{4/3} \newlap \big] (\newg^{-1})^{ab} \dot{\mathfrak{M}}_a \dot{\Gamma}_b,
		\qquad
		\frac{1}{2} t^{4/3} (\newg^{-1})^{ab} \newg_{ij} (\newg^{-1})^{ef} \dot{\mathfrak{g}}_{e \ a}^{\ i} 
		\dot{\upgamma}_{f \ b}^{\ j}
	\end{align*}
	on the right-hand side of \eqref{E:DIVMETRICJ}
	can be treated similarly.  We recall that the inhomogeneous terms $\dot{\mathfrak{M}}$ etc.
	are defined in Def.~\ref{D:INHOMSHORTHAND}.
	The $\dot{\mathfrak{M}}$ terms result in the presence of the inhomogeneous 
	term integrals such as 
	\begin{align*}
		C \sum_{|\vec{I}| \leq M} \int_{s=t}^{s=1} s^{1/3} \int_{\Sigma_s} \left| 
		\leftexp{(\vec{I});(Border)}{\mathfrak{M}} \right|_{\newg}^2 \, dx \, ds
	\end{align*}	
	on the right-hand side of 
	\eqref{E:METRICENERGYMINTEGRALINEQUALITY}. After all of the absorbing, the positive integrals 
	on the left-hand side of \eqref{E:METRICENERGYMINTEGRALINEQUALITY}
	appear with their ``final'' constant factor $+ \frac{1}{6}.$
	
	The same reasoning allows us to bound the spacetime integrals corresponding to the terms
	\begin{align*}
		& \frac{4}{3} t \dot{\newu}^a \dot{\dlap}_a, 
		\qquad
		2 t (\newg^{-1})^{ab} \dot{\mathfrak{M}}_a \dot{\dlap}_b, 
		\qquad
		2 (\newg^{-1})^{ab} \newg_{ij} \dot{\mathfrak{K}}_{\ a}^i \dot{\freenewsec}_{\ b}^j 
	\end{align*}
	appearing on the right-hand side of \eqref{E:DIVMETRICJ}, except that we don't
	absorb any of the integrals.
	
	It remains for us to discuss how to bound the integrals corresponding to the terms 
	$\triangle_{\dot{\mathbf{J}}_{(Metric);(Border)}[(\dot{\freenewsec}, \dot{\upgamma}), (\dot{\freenewsec}, \dot{\upgamma})]}$ 
	and $\triangle_{\dot{\mathbf{J}}_{(Metric);(Junk)}[(\dot{\freenewsec}, \dot{\upgamma}), (\dot{\freenewsec}, \dot{\upgamma})]}$
	on the right-hand side of \eqref{E:DIVMETRICJ}. We first discuss the first two terms on the right-hand side of
	the expression \eqref{E:METRICCURRENTERRORBORDERLINE} for
	$\triangle_{\dot{\mathbf{J}}_{(Metric);(Border)}[(\dot{\freenewsec}, \dot{\upgamma}), (\dot{\freenewsec}, \dot{\upgamma})]}.$ These two 
	terms can be handled similarly, so we will only carefully analyze the first term. We use the $\newg-$Cauchy-Schwarz inequality and
	estimates \eqref{E:SECONDFUNDUPGRADEPOINTWISEGNORM} (in the case $M = 0$)
	and \eqref{E:LAPSECOMMUTEDSTRONGPOINTWISE} to deduce that
	\begin{align} \label{E:BORDERLINECUBICTERM}
		2 t^{-1} \Big|\big[1 + t^{4/3} \newlap \big]
		(\newg^{-1})^{ac} \newg_{ij} \freenewsec_{\ c}^b 
		\dot{\freenewsec}_{\ a}^i \dot{\freenewsec}_{\ b}^j \Big| \lesssim 
		\epsilon t^{-1} |\dot{\freenewsec}|_{\newg}^2.
	\end{align}
	Therefore, the corresponding spacetime integrals are bounded by
	the integrals 
	\begin{align}
		C \epsilon \sum_{|\vec{I}| \leq M} \int_{s=t}^{s=1} s^{-1} \int_{\Sigma_s}	
		\left| \partial_{\vec{I}} \freenewsec \right|_{\newg}^2 \, dx \, ds
	\end{align}
	on the right-hand side 
	of \eqref{E:METRICENERGYMINTEGRALINEQUALITY}. 
	Similarly, the integrals corresponding to the last three terms on the right-hand side of
	the expression \eqref{E:METRICCURRENTERRORBORDERLINE} can each be bounded by
	\begin{align} \label{E:METRICCURRENTBORDERLINESOAKABLE}
			C \epsilon \sum_{|\vec{I}| \leq M} \int_{s=t}^{s=1} s^{1/3} \int_{\Sigma_s} \left|\partial_{\vec{I}} \upgamma \right|_{\newg}^2 \, dx  \, ds.
	\end{align}
	When $\epsilon$ is sufficiently small, these integrals can be absorbed into the
	positive integral on the left-hand side of \eqref{E:METRICENERGYMINTEGRALINEQUALITY}.
	
	Finally, we discuss how to bound the integrals corresponding to the term
	$\triangle_{\dot{\mathbf{J}}_{(Metric);(Junk)}[(\dot{\freenewsec}, \dot{\upgamma}), (\dot{\freenewsec}, \dot{\upgamma})]}$
	on the right-hand side of \eqref{E:DIVMETRICJ}. These terms all make negligible contributions
	to the dynamics. To proceed, we use the $\newg-$Cauchy-Schwarz inequality and the
	estimates \eqref{E:GCOMMUTEDSTRONGPOINTWISE},
	\eqref{E:GINVERSECOMMUTEDSTRONGPOINTWISE}, and \eqref{E:LAPSECOMMUTEDSTRONGPOINTWISE} to deduce that
	\begin{align} \label{E:METRICJUNKFIRSTESTIMATE}
		\left|\triangle_{\dot{\mathbf{J}}_{(Metric);(Junk)}[(\dot{\freenewsec}, \dot{\upgamma}), (\dot{\freenewsec}, \dot{\upgamma})]} \right|
		& \lesssim \sqrt{\epsilon} t^{-1/3} |\dot{\freenewsec}|_{\newg}^2
			+ \sqrt{\epsilon} t^{1 - c \sqrt{\epsilon}} |\dot{\upgamma}|_{\newg}^2
			+ \sqrt{\epsilon} t^{7/3 - c \sqrt{\epsilon}} |\dot{\dlap}|_{\newg}^2.
	\end{align}
	We now integrate inequality \eqref{E:METRICJUNKFIRSTESTIMATE} over $[t,1) \times \mathbb{T}^3.$ 
	The integral corresponding to the term $|\dot{\upgamma}|_{\newg}^2$ can be absorbed into the
	positive integral on the left-hand side of \eqref{E:METRICENERGYMINTEGRALINEQUALITY} when 
	$\epsilon$ is sufficiently small. The integrals corresponding to the terms
	$|\dot{\freenewsec}|_{\newg}^2$ and $|\dot{\dlap}|_{\newg}^2$ are clearly bounded by the right-hand
	side of \eqref{E:METRICENERGYMINTEGRALINEQUALITY} as desired. 
	
%	and using the bootstrap assumption
%	\eqref{E:HIGHBOOT}, we conclude that
	
%	\begin{align} \label{E:TRIVIALLMETRICINTEGRALS}
%		\sum_{|\vec{I}| \leq M} 
%			& \left|\triangle_{\dot{\mathbf{J}}_{(Metric);(Junk)}[\partial_{\vec{I}} (\freenewsec, \upgamma), 
%				\partial_{\vec{I}}(\freenewsec, \upgamma)]} \right| \, dx \, ds \\
%			& \lesssim \epsilon^3 \int_{s=t}^{s=1} s^{-1/3 - c \epsilon - 2 \upsigma} \, ds 
%			\lesssim \epsilon^3. \notag
%	\end{align}	
%	Clearly, the right-hand side of \eqref{E:TRIVIALLMETRICINTEGRALS} is bounded by the last term on the right-hand side of 
%	\eqref{E:FLUIDENERGYMINTEGRALINEQUALITY} as desired.
	 
\end{proof}

\subsubsection{Preliminary fluid and lapse energy inequalities}

We now derive preliminary energy inequalities for the fluid. Thanks to Prop.~\ref{P:KEYSIGNED}, our estimates will also lead to 
control over the lapse variables. 
%The analogous energy inequality in the second model problem is \eqref{E:SFENERGYESTIMATE}.

\begin{lemma}[\textbf{Preliminary energy inequalities for the fluid and lapse}] \label{L:FLUIDLAPSEENERGYID}
Assume the hypotheses of Prop.~\ref{P:FUNDAMENTALENERGYINEQUALITY}, and let
$\fluidenergy{M}(t)$ be the fluid energy defined in \eqref{E:FLUIDENERGY}.
Then there exist a small constant $\upsigma_N > 0$ 
and a large constant $C > 0$ 
such that if $\epsilon \leq \upsigma \leq \upsigma_N$
and $0 \leq M \leq N,$ then $\fluidenergy{M}(t)$ verifies the
following integral inequality for $t \in (T,1]:$
\begin{align} \label{E:FLUIDENERGYMINTEGRALINEQUALITY}
	\fluidenergy{M}^2(t) & + \left(\frac{1}{9} - C \sqrt{\epsilon} \right) \sum_{|\vec{I}| \leq M} \int_{s = t}^{s = 1} s^{5/3} \int_{\Sigma_s}  
	  	\left| \partial_{\vec{I}} \dlap \right|_{\newg}^2 \, dx \, ds \\
	 & + \left(\frac{1}{6} - C \sqrt{\epsilon} \right) \sum_{|\vec{I}| \leq M} \int_{s = t}^{s = 1} s^{5/3} \int_{\Sigma_s} 
			\left| \partial_{\vec{I}} \newlap \right|^2 \, dx \, ds \notag \\
	& + \left(\frac{1}{27} - C \sqrt{\epsilon} \right) \sum_{|\vec{I}| \leq M} \int_{s = t}^{s = 1} s^{1/3} \int_{\Sigma_s} 
	  		\left| \partial_{\vec{I}} \newu \right|_{\newg}^2 \, dx \, ds \notag \\
	& \leq \fluidenergy{M}^2(1)
		\notag \\
	& \ \ + C \sum_{|\vec{I}| \leq M} \int_{s=t}^{s=1} s^{-1/3} 
		\left| \partial_{\vec{I}} \adjustednewp \right|^2 
		\, dx \, ds \notag \\ 
	 & \ \ + C \sum_{|\vec{I}| \leq M} 
			\int_{s = t}^{s = 1} s^{-1} \int_{\Sigma_s} 
			\left| \leftexp{(\vec{I});(Border)}{\mathfrak{N}} \right|^2 \, dx \, ds  \notag \\
	& \ \ + C \sum_{|\vec{I}| \leq M} 
			\int_{s = t}^{s = 1} s^{5/3} \int_{\Sigma_s} \left| \leftexp{(\vec{I});(Junk)}{\mathfrak{N}} \right|^2 \, dx \, ds \notag \\
	& \ \ + C \sum_{|\vec{I}| \leq M} \int_{s = t}^{s = 1} s \int_{\Sigma_s} \left| \leftexp{(\vec{I});(Junk)}{\mathfrak{P}} \right|^2 \, dx \, ds 
		\notag \\
	& \ \ + C \sum_{|\vec{I}| \leq M} \int_{s = t}^{s = 1} s^{-1} \int_{\Sigma_s} \left| \leftexp{(\vec{I});(Border)}{\mathfrak{U}} \right|_{\newg}^2 \, dx \, ds 
		\notag \\
	& \ \ + C \sum_{|\vec{I}| \leq M} \int_{s = t}^{s = 1} s^3 \int_{\Sigma_s} \left| \leftexp{(\vec{I});(Junk)}{\mathfrak{U}} \right|_{\newg}^2 \, dx \, ds. 
		\notag 
\end{align}
\end{lemma}

\begin{proof}
	By the divergence theorem, we have that
	\begin{align}  \label{E:DIVTHMFLUID}
		\fluidenergy{M}^2(t) - \fluidenergy{M}^2(1) & = 
			- \sum_{|\vec{I}| \leq M} \int_{s=t}^1 
			\int_{\Sigma_s} \partial_{\mu} \left( \dot{\mathbf{J}}_{(Fluid)}^{\mu}
				[\partial_{\vec{I}} (\adjustednewp, \newu), \partial_{\vec{I}} (\adjustednewp, \newu)] \right) 
			\, dx \, ds.
	\end{align}
	Our goal therefore is to estimate the right-hand side of \eqref{E:DIVTHMFLUID} using the expression \eqref{E:DIVFLUIDJ}.
	The expression \eqref{E:DIVFLUIDJ} is valid because the differentiated fluid variables 
	$\partial_{\vec{I}} (\adjustednewp, \newu)$ 
	appearing in \eqref{E:DIVTHMFLUID} verify the $\partial_{\vec{I}}-$commuted
	fluid equations of Sect.~\ref{SSS:COMMUTEDFLUIDRENORM}, 
	which are specific instances of the fluid equations of variation.
	
	We first estimate the integrals corresponding to the first three terms
	on the right-hand side of \eqref{E:DIVFLUIDJ}. We recall that the variations $\dot{\freenewsec}$ etc. are defined 
	in Def.~\ref{D:VARIATIONS}. The
	corresponding integrals on the right-hand side of \eqref{E:DIVTHMFLUID} are
	\begin{align} \label{E:TOPORDERFLUIDINTEGRALS}
		& \frac{2}{3} \int_{s=t}^{s=1} s^{1/3} \int_{\Sigma_s} \big(\partial_{\vec{I}} \adjustednewp \big) \big(\partial_{\vec{I}} \newlap 
			\big) \, dx \, ds \\
		& \ \ + 4 \int_{s=t}^{s=1} s \int_{\Sigma_s} \Big[\adjustednewp + \frac{1}{3} \Big]^2 \big(\partial_{\vec{I}} \newu^a \big) 
				\big(\partial_{\vec{I}} \dlap_a \big)  \, dx \, ds
		 		\notag \\
		& \ \ - \frac{8}{3} \int_{s=t}^{s=1} s^{1/3} \int_{\Sigma_s} \Big[\adjustednewp + \frac{1}{3} \Big]^2 \newg_{ab} 
				\big(\partial_{\vec{I}} \newu^a \big) \big(\partial_{\vec{I}} \newu^b \big) \, dx \, ds. \notag
	\end{align}
	Applying the crucially important Prop.~\ref{P:KEYSIGNED} to the first spatial integral, and using
	the estimate \eqref{E:PCOMMUTEDSTRONGPOINTWISE},
	the algebraic estimate $\frac{4}{9} \left|\big(\partial_{\vec{I}} \newu^a \big) \big(\partial_{\vec{I}} \dlap_a \big)\right| 
	\leq \frac{2}{9} \left( \big|\partial_{\vec{I}} \newu \big|_{\newg}^2 + \big|\partial_{\vec{I}} \dlap \big|_{\newg}^2 \right),$
	and Cauchy-Schwarz, we can bound the sum of the integrals in \eqref{E:TOPORDERFLUIDINTEGRALS} by
\begin{align}	 \label{E:TOPORDERFLUIDINTEGRALSESTIMATED}
 \leq & - \frac{1}{3}\left(1 - \epsilon - \frac{2}{3}\right) 
 	\int_{s=t}^{s=1} s^{5/3} \int_{\Sigma_s} \Big|\partial_{\vec{I}} \dlap \big|_{\newg}^2 	\, dx \, ds  
 	 \\
	& - \frac{1}{3}\left(1 - C \epsilon - \upalpha\right) \int_{s=t}^{s=1} s^{5/3} \int_{\Sigma_s} \big|\partial_{\vec{I}} \newlap \big|^2 \, dx \, ds 
		\notag \\
	& - \left(\frac{8}{27} - \frac{2}{9} - C \epsilon \right) \int_{s=t}^{s=1} s^{1/3} \int_{\Sigma_s} 
		\big|\partial_{\vec{I}} \newu \big|_{\newg}^2 \, dx \, ds \notag \\
	& + \frac{1}{3} \upalpha^{-1} \int_{s=t}^{s=1} s^{-1} \int_{\Sigma_s} \Big|\leftexp{(\vec{I});(Border)}{\mathfrak{N}} \Big|^2  \, dx
		\, ds
		+ \frac{1}{3} \upalpha^{-1} \int_{s=t}^{s=1} s^{5/3} \int_{\Sigma_s}  \Big|\leftexp{(\vec{I});(Junk)}{\mathfrak{N}} \Big|^2 \, dx
		\, ds. 
		\notag
\end{align}
Setting $\upalpha = 1/2$ in \eqref{E:TOPORDERFLUIDINTEGRALSESTIMATED}, we have produced the first three terms on the
left-hand side of \eqref{E:FLUIDENERGYMINTEGRALINEQUALITY} [where at this stage in the proof, the integral 
$\int_{s=t}^{s=1} s^{1/3} \int_{\Sigma_s} \big|\partial_{\vec{I}} \newu \big|_{\newg}^2 \, dx \, ds$
is multiplied by the factor $\frac{2}{27}$  
rather than the factor $\frac{1}{27}$ that appears on the left-hand side of \eqref{E:FLUIDENERGYMINTEGRALINEQUALITY}], 
as well as the integrals involving $\Big|\leftexp{(\vec{I});(Border)}{\mathfrak{N}} \Big|^2$
and $\Big|\leftexp{(\vec{I});(Junk)}{\mathfrak{N}} \Big|^2$ on the right-hand side. 

To bound the integrals corresponding to the 
$- 4 t^{1/3} \big[1 + t^{4/3} \newlap\big] \Big[\adjustednewp + \frac{1}{3} \Big]^2 \newg_{ia} \freenewsec_{\ j}^a \dot{\newu}^i \dot{\newu}^j$ 
term on the right-hand side of
\eqref{E:DIVFLUIDJ}, we use the $\newg-$Cauchy-Schwarz inequality and  
the estimates \eqref{E:SECONDFUNDUPGRADEPOINTWISEGNORM} (in the case $M = 0$),
\eqref{E:PCOMMUTEDSTRONGPOINTWISE}, and \eqref{E:LAPSECOMMUTEDSTRONGPOINTWISE} to deduce
\begin{align}
		\int_{s=t}^{s=1} s^{1/3} \int_{\Sigma_s} \left|\big[1 + t^{4/3} \newlap\big] \Big[\adjustednewp + \frac{1}{3} \Big]^2 
			\newg_{ia} \freenewsec_{\ j}^a  \big(\partial_{\vec{I}} \newu^i \big) 
		 	\big( \partial_{\vec{I}} \newu^j \big) \right| \, dx \, ds
		 & \lesssim \epsilon \int_{s=t}^{s=1} s^{1/3} \int_{\Sigma_s} \left| \partial_{\vec{I}} \newu \right|_{\newg}^2 \, dx \, ds.
\end{align}
Thus, when $\epsilon$ is small, we can absorb these integrals into the left-hand side of \eqref{E:FLUIDENERGYMINTEGRALINEQUALITY}.

To bound the integrals corresponding to the term $\dot{\adjustednewp} \dot{\mathfrak{P}}$ on the right-hand side of \eqref{E:DIVFLUIDJ}
(where $\dot{\mathfrak{P}} = t^{1/3} \leftexp{(\vec{I});(Junk)}{\mathfrak{P}}$), we simply use Cauchy-Schwarz to deduce
\begin{align} \label{E:FLUIDCURRENTDOTPPINHOMINTEGRALESTIMATE}
	\int_{s=t}^{s=1} s^{1/3} \left|\int_{\Sigma_s} \big(\partial_{\vec{I}} \adjustednewp \big) 
			\leftexp{(\vec{I});(Junk)}{\mathfrak{P}}\right| \, dx \, ds
		& \leq \int_{s=t}^{s=1} s^{-1/3} \int_{\Sigma_s} \left| \partial_{\vec{I}} \adjustednewp \right|^2 \, dx \, ds
			\notag \\
		& \ \ + \int_{s=t}^{s=1} s \int_{\Sigma_s} \left| \leftexp{(\vec{I});(Junk)}{\mathfrak{P}} \right|^2 \, dx \, ds. 
		\notag
\end{align}
Clearly, the right-hand side of \eqref{E:FLUIDCURRENTDOTPPINHOMINTEGRALESTIMATE} is bounded by the right-hand side of 
\eqref{E:FLUIDENERGYMINTEGRALINEQUALITY} as desired.

To bound the integrals corresponding to the term
$4 t^{4/3} \Big[\adjustednewp + \frac{1}{3} \Big]^2 \newg_{ab} \dot{\newu}^a \dot{\mathfrak{U}}^b$
on the right-hand side of \eqref{E:DIVFLUIDJ} 
(where $\dot{\mathfrak{U}} = t^{-1} \leftexp{(\vec{I});(Border)}{\mathfrak{U}}^j + t^{1/3} \leftexp{(\vec{I});(Junk)}{\mathfrak{U}}^j$), 
we first use the estimate \eqref{E:PCOMMUTEDSTRONGPOINTWISE} to deduce that there exists a constant $C > 0$ such that 
for any constant $\upbeta > 0,$ we have
\begin{align}
		t^{4/3} \left| \Big[\adjustednewp + \frac{1}{3} \Big]^2 \newg_{ab} \dot{\newu}^a \dot{\mathfrak{U}}^b \right|
		& \leq C \upbeta t^{1/3} |\dot{\newu}|_{\newg}^2 + C \upbeta^{-1} t^{7/3} |\dot{\mathfrak{U}}|_{\newg}^2.
\end{align}
Thus, the corresponding integrals can be bounded by 
\begin{align} \label{E:FLUIDCURRENTUINHOMOGENEOUSCROSSTERM}
	4 \int_{s=t}^{s=1} s^{4/3} & \int_{\Sigma_s}  
		\left| \Big[\adjustednewp + \frac{1}{3} \Big]^2 \newg_{ab} \big(\partial_{\vec{I}} \newu^a 
		\big) \left(s^{-1} \leftexp{(\vec{I});(Border)}{\mathfrak{U}}^b 
			+ s^{1/3} \leftexp{(\vec{I});(Junk)}{\mathfrak{U}}^b \right) \right|
		\, dx \, ds \\
	& \leq C \upbeta \int_{s=t}^{s=1} s^{1/3} \int_{\Sigma_s} \left| \partial_{\vec{I}} \newu \right|_{\newg}^2 \, dx \, ds
		\notag \\
	& \ \ + C \upbeta^{-1} \int_{s=t}^{s=1} s^{1/3} \int_{\Sigma_s} \left| \leftexp{(\vec{I});(Border)}{\mathfrak{U}} \right|_{\newg}^2 
			\, dx \, ds
		+ C \upbeta^{-1} \int_{s=t}^{s=1} s^3 \int_{\Sigma_s} \left| \leftexp{(\vec{I});(Junk)}{\mathfrak{U}} \right|_{\newg}^2 \, dx
			\, ds. \notag
\end{align}		
If $\upbeta$ is chosen so that $C \upbeta < \frac{1}{27},$ then we can absorb the first integral on the right-hand side
of \eqref{E:FLUIDCURRENTUINHOMOGENEOUSCROSSTERM} into the left-hand side of \eqref{E:FLUIDENERGYMINTEGRALINEQUALITY}, 
which reduces the constant in front of the third integral to its listed value of $\frac{1}{27}.$ Furthermore, the second
and third integrals on the right-hand side of \eqref{E:FLUIDCURRENTUINHOMOGENEOUSCROSSTERM} are clearly bounded by
the right-hand side of \eqref{E:FLUIDENERGYMINTEGRALINEQUALITY} as desired.

The integrals corresponding to the next three terms on the right-hand side of \eqref{E:DIVFLUIDJ}
are easy to estimate because of the large powers of $t$ available. To estimate the integral corresponding to the term 
$2 t^{4/3} \frac{\Big[\adjustednewp + \frac{1}{3} \Big]}{\big[1 + t^{4/3} \newg_{ab} \newu^a \newu^b \big]} \dot{\adjustednewp} \newg_{ef} \newu^e \dot{\mathfrak{U}}^f,$ from the right-hand side of \eqref{E:DIVFLUIDJ}, we use the estimates 
\eqref{E:PCOMMUTEDSTRONGPOINTWISE} and \eqref{E:UCOMMUTEDSTRONGPOINTWISE} to deduce that
\begin{align} \label{E:JDOTFLUIDEASYTERMALGEBRAIC}
	t^{4/3} \left| \frac{\Big[\adjustednewp + \frac{1}{3} \Big]}{\big[1 + t^{4/3} \newg_{ab} \newu^a \newu^b \big]} 
		\dot{\adjustednewp} \newg_{ef} \newu^e \dot{\mathfrak{U}}^f \right| 
	& \lesssim t^{1/3 - c \sqrt{\epsilon}} |\dot{\adjustednewp}|^2 + \epsilon t^{7/3} 
	|\dot{\mathfrak{U}}|_{\newg}^2. 
\end{align}
From \eqref{E:JDOTFLUIDEASYTERMALGEBRAIC}, it follows that the corresponding integrals can be bounded by
\begin{align} \label{E:JDOTFLUIDEASYINTEGRALS}
	& \lesssim \int_{s=t}^{s=1} s^{1/3 - c \sqrt{\epsilon}} \int_{\Sigma_s} 
		\left| \partial_{\vec{I}} \adjustednewp \right|^2 \, dx \, ds \\
	& \ \ + \epsilon \int_{s=t}^{s=1} s^{1/3} \int_{\Sigma_s} 
			\left| \leftexp{(\vec{I})}{\leftexp{(Border)}{\mathfrak{U}}} \right|_{\newg}^2 \, dx \, ds 
		+ \epsilon \int_{s=t}^{s=1} s^3 \int_{\Sigma_s} 
			\left| \leftexp{(\vec{I})}{\leftexp{(Junk)}{\mathfrak{U}}} \right|_{\newg}^2 \, dx \, ds. \notag
\end{align}
Clearly, the integrals on the right-hand side of
\eqref{E:JDOTFLUIDEASYINTEGRALS} are bounded by the right-hand side of \eqref{E:FLUIDENERGYMINTEGRALINEQUALITY}
as desired. The integrals corresponding to the terms
\begin{align*}
	&  2 t^{4/3} \frac{\Big[\adjustednewp + \frac{1}{3} \Big]}{\big[1 + t^{4/3} \newg_{ab} \newu^a \newu^b \big]} 
		\newg_{ef} \newu^e \dot{\newu}^f \dot{\mathfrak{P}}, \\
	& - 4 t^{8/3} \frac{\Big[\adjustednewp + \frac{1}{3} \Big]^2}
		{\big[1 + t^{4/3} \newg_{ab} \newu^a \newu^b \big]} \newg_{ef} \newu^e 
		\dot{\newu}^f \newg_{ij} \newu^i \dot{\mathfrak{U}}^j
\end{align*}	
from the right-hand side of \eqref{E:DIVFLUIDJ} can similarly be bounded with the help of the estimates \eqref{E:PCOMMUTEDSTRONGPOINTWISE} and \eqref{E:UCOMMUTEDSTRONGPOINTWISE}; we omit the straightforward details.

	It remains to bound the integrals corresponding to the terms 
	$\triangle_{\dot{\mathbf{J}}_{(Fluid);(Junk)_l}[(\dot{\adjustednewp},\dot{\newu}), (\dot{\adjustednewp},\dot{\newu})]}$ from the right-hand 
	side of \eqref{E:DIVFLUIDJ}. These terms make only very minor contributions to the dynamics. 
	To proceed, we insert the strong estimates of Prop.~\ref{P:STRONGPOINTWISE} into 
	the expressions \eqref{E:FLUIDCURRENTJUNK1}-\eqref{E:FLUIDCURRENTJUNK3}
	and use the $\newg-$Cauchy-Schwarz inequality plus simple estimates of the form $ab \lesssim \upzeta^{-1}a^2 + \upzeta b^2$ to deduce that
	\begin{align} \label{E:JFLUIDJUNKFIRSTESTIMATE}
		\sum_{l=1}^3 \left| \triangle_{\dot{\mathbf{J}}_{(Fluid);(Junk)_l}[(\dot{\adjustednewp},\dot{\newu}), (\dot{\adjustednewp},\dot{\newu})]} \right|
		& \lesssim \sqrt{\epsilon}
		\left\lbrace t^{1 - c \sqrt{\epsilon}} |\dot{U}|_{\newg}^2 
			+ t^{-1/3} |\dot{\adjustednewp}|^2 
			+ t^{7/3 - c \sqrt{\epsilon}} |\dot{\newlap}|^2 
			+ t^{7/3 - c \sqrt{\epsilon}} |\dot{\dlap}|_{\newg}^2 \right\rbrace. 		
	\end{align}
	We now integrate \eqref{E:JFLUIDJUNKFIRSTESTIMATE} over $[t,1) \times \mathbb{T}^3.$
	For sufficiently small $\epsilon,$ the integrals corresponding to the terms
	$|\dot{U}|_{\newg}^2,$ $|\dot{\newlap}|^2,$ and $|\dot{\dlap}|_{\newg}^2$ can be absorbed
	into the positive integrals on the left-hand side of \eqref{E:FLUIDENERGYMINTEGRALINEQUALITY}.
	The integral corresponding to $|\dot{\adjustednewp}|^2$ is clearly bounded by the right-hand
	side of \eqref{E:FLUIDENERGYMINTEGRALINEQUALITY} as desired.
	
	%and using the Sobolev norm bootstrap assumption
	%\eqref{E:HIGHBOOT}, we conclude that there exists an integer $\widetilde{Z} > 0$ such that
	
	%\begin{align} \label{E:TRIVIALLFLUIDINTEGRALS}
	%	\sum_{|\vec{I}| \leq M} 
	%		\sum_{l=1}^3 & \int_{s = t}^{s = 1} \int_{\Sigma_s} 
	%		|\triangle_{\dot{\mathbf{J}}_{(Fluid);(Junk)_l}[\partial_{\vec{I}} (\adjustednewp, \newu), 
	%			\partial_{\vec{I}} (\adjustednewp, \newu)]}| \, dx \, ds \\
	%		& \lesssim \epsilon^{5/2} \int_{s=t}^{s=1} s^{-1/3 - c \sqrt{\epsilon} - \widetilde{Z} \upsigma} \, ds 
	%		\lesssim \epsilon^{5/2}. \notag
	%\end{align}	
	%Clearly, the right-hand side of \eqref{E:TRIVIALLFLUIDINTEGRALS} is bounded by the last term on the right-hand side of 
	%\eqref{E:FLUIDENERGYMINTEGRALINEQUALITY} as desired.
	
	\end{proof}

\subsection{Proof of Prop.~\ref{P:FUNDAMENTALENERGYINEQUALITY}} \label{SS:ENERGYINEQUALITYPROOF}
Recall that the total energies are
\begin{align}
	\totalenergy{\smallparameter}{M}^2 & :=  \smallparameter \metricenergy{M}^2 + \fluidenergy{M}.
\end{align}
We simply add $\smallparameter$ times inequality \eqref{E:METRICENERGYMINTEGRALINEQUALITY} to
the inequality \eqref{E:FLUIDENERGYMINTEGRALINEQUALITY}. If $\smallparameter = \smallparameter_* > 0$ is chosen to be
sufficiently small, then $\smallparameter_*$ times the third and fourth integrals on the right-hand side of \eqref{E:METRICENERGYMINTEGRALINEQUALITY} can be absorbed into the positive integrals on the left-hand side of \eqref{E:FLUIDENERGYMINTEGRALINEQUALITY}. The desired inequality \eqref{E:FUNDAMENTALENERGYINEQUALITY} thus follows.

\hfill $\qed$

\section{Pointwise Bounds for the Inhomogeneous Terms} \label{S:POINTWISEINHOMOGENEOUS}

In this section, we derive pointwise bounds for the $|\cdot|_{\newg}$ norms of the inhomogeneous terms 
$\leftexp{(\vec{I});(Junk)}{\mathfrak{N}},$ $\leftexp{(\vec{I});(Border)}{\mathfrak{N}},$ 
$\leftexp{(\vec{I});(Junk)}{\mathfrak{g}},$ $\leftexp{(\vec{I});(Border)}{\mathfrak{g}},$
etc., 
appearing on the right-hand sides of the $\partial_{\vec{I}}-$commuted equations.
The right-hand sides of the bounds feature the $\partial_{\vec{I}}-$differentiated renormalized solution variables 
multiplied by various $t-$weights, 
some of which are of crucial importance. Because these same inhomogeneous terms also appear in the fundamental energy integral inequality 
\eqref{E:FUNDAMENTALENERGYINEQUALITY}, these bounds are an extremely important ingredient in deriving the a priori energy and norm estimates of 
Sect.~\ref{S:FUNDAMENTALAPRIORI}. 
The main tools used in the proof are the strong estimates afforded by Prop.~\ref{P:STRONGPOINTWISE}. 
%To streamline the notation, we define

%\begin{align} \label{E:POINTWISEARRAY} 
%	\Big|\partial_{\vec{I}} \left(\newg, \newg^{-1}, \freenewsec, \upgamma, \adjustednewp, \newu, \newlap, \dlap \right) \Big|_{\newg} 
%	&:=  \left|\partial_{\vec{I}} \newg \right|_{\newg} 
%		+ \left|\partial_{\vec{I}} \newg^{-1} \right|_{\newg}
%		+ \left|\partial_{\vec{I}} \freenewsec \right|_{\newg} 
%		+ \left| \partial_{\vec{I}} \upgamma \right|_{\newg}
%		+ \left| \partial_{\vec{I}} \newlap \right|
%		+ \left| \partial_{\vec{I}} \dlap \right|_{\newg} 
%		+ \left| \partial_{\vec{I}} \adjustednewp \right|
%		+ \left| \partial_{\vec{I}} \newu \right|_{\newg}. 
%\end{align}
%When $\vec{I} = 0,$ we will omit the tensors $\newg$ and $\newg^{-1}$ from the above array.

%We now state and prove the aforementioned proposition.

\begin{proposition}[\textbf{Pointwise bounds for the inhomogeneous terms}] \label{P:POINTWISEESTIMATES}
Assume that the hypotheses and conclusions of Prop.~\ref{P:STRONGPOINTWISE} hold on the spacetime slab $(T,1] \times \mathbb{T}^3.$
Then there exist a small constant $\upsigma_N > 0$ and a large constant $c > 0$ and such that if $\epsilon \leq \upsigma \leq \upsigma_N$ and $0 \leq M \leq N,$
then the following pointwise bounds are verified by the junk inhomogeneous terms appearing
in the commuted equations \eqref{AE:MOMENTUMCONSTRAINTCOMMUTED}-\eqref{AE:RAISEDMOMENTUMCONSTRAINTCOMMUTED},
\eqref{AE:METRICGAMMACOMMUTED}-\eqref{AE:SECONDFUNDCOMMUTED}, 
and \eqref{AE:PARTIALTPCOMMUTED}-\eqref{AE:PARTIALTUJCOMMUTED} on $(T,1] \times \mathbb{T}^3:$
\begin{align} \label{E:POINTWISEJUNKNORMS}
		\sum_{|\vec{I}| \leq M} 
		\left| \leftexp{(\vec{I});(Junk)}{\mathfrak{M}} \right|_{\newg} 
				+ \left| \leftexp{(\vec{I});(Junk)}{\widetilde{\mathfrak{M}}} \right|_{\newg} &
				\\
		+ \sum_{|\vec{I}| \leq M} \left| \leftexp{(\vec{I});(Junk)}{\mathfrak{K}} \right|_{\newg}
				+ \left| \leftexp{(\vec{I});(Junk)}{\mathfrak{g}} \right|_{\newg} &
				\notag \\
		+ \sum_{|\vec{I}| \leq M}  \left| \leftexp{(\vec{I});(Junk)}{\mathfrak{P}} \right| 
				+ \left| \leftexp{(\vec{I});(Junk)}{\mathfrak{U}} \right|_{\newg} &
		\lesssim  \underbrace{\sqrt{\epsilon} t^{- c \sqrt{\epsilon}} \sum_{1 \leq |\vec{I}| \leq M} \left| \partial_{\vec{I}} \newg \right|_{\newg}}_{
				\mbox{absent if $M = 0$}} 
			+ \underbrace{\sqrt{\epsilon} t^{- c \sqrt{\epsilon}} \sum_{1 \leq |\vec{I}| \leq M} \left| \partial_{\vec{I}} \newg^{-1} \right|_{\newg}}_{
				\mbox{absent if $M = 0$}}
			\notag \\
		& \ \ \sqrt{\epsilon} t^{- c \sqrt{\epsilon}} \sum_{|\vec{I}| \leq M} \left| \partial_{\vec{I}} \Big(
			\freenewsec, \upgamma, \adjustednewp, \newu, \newlap, \dlap \Big) \right|_{\newg}.
			\notag 
		\end{align}
	
	In addition, the following estimates for the junk inhomogeneous terms appearing in 
	equations \eqref{AE:GEVOLUTIONCOMMUTED}-\eqref{AE:GINVERSEEVOLUTIONCOMMUTED} 
	hold on $(T,1] \times \mathbb{T}^3$ for $1 \leq M \leq N:$
	\begin{subequations}
	\begin{align} \label{E:GJUNKNORM}
		\sum_{1 \leq |\vec{I}| \leq M} \left| \leftexp{(\vec{I});(Junk)}{\mathfrak{G}} \right|_{\newg}
		& \lesssim \sqrt{\epsilon} t^{- c \sqrt{\epsilon}} \sum_{1 \leq |\vec{I}| \leq M} \left|\partial_{\vec{I}} \newg \right|_{\newg}
			\\
		& \ \ +  \sqrt{\epsilon} t^{- c \sqrt{\epsilon}} \sum_{|\vec{I}| \leq M} \left|\partial_{\vec{I}} \freenewsec \right|_{\newg}
			\notag \\
		& \ \ + \sum_{|\vec{I}| = M} \left|\partial_{\vec{I}} \newlap \right|
			\notag 
			\\
		& \ \ + \sqrt{\epsilon} t^{- c \sqrt{\epsilon}} \sum_{|\vec{J}| \leq M - 1} \left|\partial_{\vec{J}} \newlap \right|,
			\notag \\
		\sum_{1 \leq |\vec{I}| \leq M} \left| \leftexp{(\vec{I});(Junk)}{\widetilde{\mathfrak{G}}} \right|_{\newg}
		& \lesssim \sqrt{\epsilon} t^{- c \sqrt{\epsilon}} \sum_{1 \leq |\vec{I}| \leq M} \left|\partial_{\vec{I}} \newg^{-1} \right|_{\newg}
			\label{E:GINVERSEJUNKNORM} \\
		& \ \ +  \sqrt{\epsilon} t^{- c \sqrt{\epsilon}} \sum_{|\vec{I}| \leq M} \left|\partial_{\vec{I}} \freenewsec \right|_{\newg}
			\notag \\
		& \ \ + \sum_{|\vec{I}| = M} \left|\partial_{\vec{I}} \newlap \right|
			\notag 
			\\
		& \ \ + \sqrt{\epsilon} t^{- c \sqrt{\epsilon}} \sum_{|\vec{J}| \leq M - 1} \left|\partial_{\vec{J}} \newlap \right|.
			\notag 		
	\end{align}
	\end{subequations}

	The following estimates for the borderline inhomogeneous terms appearing in 
	the commuted equations \eqref{AE:GEVOLUTIONCOMMUTED}-\eqref{AE:GINVERSEEVOLUTIONCOMMUTED} 
	hold on $(T,1] \times \mathbb{T}^3$ for $1 \leq M \leq N:$
	\begin{subequations}
	\begin{align} \label{E:GBORDERNORM}
		\sum_{|\vec{I}| \leq M} \left| \leftexp{(\vec{I});(Border)}{\mathfrak{G}} \right|_{\newg}
		& \ \ \lesssim \underbrace{\sqrt{\epsilon} \sum_{|\vec{I}| = M} \Big|\partial_{\vec{I}} \newg \Big|_{\newg}}_{\mbox{borderline, principal order}}
				+ \underbrace{\sqrt{\epsilon} t^{- c \sqrt{\epsilon}} \sum_{1 \leq |\vec{J}| \leq M - 1} \Big|\partial_{\vec{J}} \newg \Big|_{\newg}}_{\mbox{
				borderline, below principal order}}
				\\
		& + \underbrace{\sum_{|\vec{I}| = M} \left|\partial_{\vec{I}} \freenewsec \right|_{\newg}}_{\mbox{borderline, principal order}}
			+ \underbrace{\sqrt{\epsilon} t^{- c \sqrt{\epsilon}} \sum_{|\vec{J}| \leq M - 1} 
				\left|\partial_{\vec{J}} \freenewsec \right|_{\newg}}_{\mbox{borderline, below principal order}},
				\notag \\	
		\sum_{|\vec{I}| \leq M} \left| \leftexp{(\vec{I});(Border)}{\widetilde{\mathfrak{G}}} \right|_{\newg}
			& \ \ \lesssim \underbrace{\sqrt{\epsilon} \sum_{|\vec{I}| = M} \Big|\partial_{\vec{I}} \newg^{-1} \Big|_{\newg}}_{\mbox{borderline, principal order}}
				+ \underbrace{\sqrt{\epsilon} t^{- c \sqrt{\epsilon}} \sum_{1 \leq |\vec{J}| \leq M - 1} \Big|\partial_{\vec{J}} \newg^{-1} 
					\Big|_{\newg}}_{\mbox{borderline, below principal order}}
				\label{E:GINVERSEBORDERNORM} \\
		& + \underbrace{\sum_{|\vec{I}| = M} \left|\partial_{\vec{I}} \freenewsec \right|_{\newg}}_{\mbox{borderline, principal order}}
			+ \underbrace{\sqrt{\epsilon} t^{- c \sqrt{\epsilon}} \sum_{|\vec{J}| \leq M - 1} 
				\left|\partial_{\vec{J}} \freenewsec \right|_{\newg}}_{\mbox{borderline, below principal order}}.
				\notag
	\end{align}
	\end{subequations}

	In addition, the following estimates for the junk inhomogeneous terms appearing in 
	the commuted equation \eqref{AE:LAPSEICOMMUTED} hold on $(T,1] \times \mathbb{T}^3$
	for $0 \leq M \leq N:$
	\begin{align} \label{E:NIJUNK}
		\sum_{|\vec{I}| \leq M} \left| \leftexp{(\vec{I});(Junk)}{\mathfrak{N}} \right|
			& \lesssim  \underbrace{\sqrt{\epsilon} t^{- c \sqrt{\epsilon}} \sum_{1 \leq |\vec{I}| \leq M} \Big|\partial_{\vec{I}} \newg \Big|_{\newg}}_{\mbox{
				absent if $M=0$}}
				+ \underbrace{\sqrt{\epsilon} t^{- c \sqrt{\epsilon}} \sum_{1 \leq |\vec{I}| \leq M} \Big|\partial_{\vec{I}} \newg^{-1} \Big|_{\newg}}_{\mbox{
					absent if $M=0$}}
			\\
		& \ \ \sqrt{\epsilon} t^{- c \sqrt{\epsilon}} \sum_{|\vec{I}| \leq M} 
					\left|\partial_{\vec{I}} \freenewsec \right|_{\newg}  
				+ \sqrt{\epsilon} \sum_{|\vec{I}| \leq M} \Big| \partial_{\vec{I}} \upgamma \Big|_{\newg} 
					\notag \\
			& \ \ + \sqrt{\epsilon} \sum_{|\vec{J}| \leq M - 1} \Big|\partial_{\vec{J}} \newlap \Big|
				+ \sqrt{\epsilon} \sum_{|\vec{I}| \leq M} \left|\partial_{\vec{I}} \dlap \right|_{\newg}
				\notag \\
			& \ \ + \sqrt{\epsilon} t^{- c \sqrt{\epsilon}} \sum_{|\vec{I}| \leq M} \Big| \partial_{\vec{I}} \adjustednewp \Big|
				+ \sqrt{\epsilon} t^{- c \sqrt{\epsilon}} 
				\sum_{|\vec{I}| \leq M} \Big| \partial_{\vec{I}} \newu \Big|_{\newg}
			\notag \\
			& \ \ +  \sqrt{\epsilon} t^{- c \sqrt{\epsilon}} \sum_{|\vec{J}| \leq M-1} \Big| \partial_{\vec{J}} \upgamma \Big|_{\newg}.
				\notag 
	\end{align}
	
	The following estimates for the borderline inhomogeneous terms appearing in 
	the commuted equation \eqref{AE:LAPSEICOMMUTED} hold on $(T,1] \times \mathbb{T}^3$
	for $0 \leq M \leq N:$
	\begin{align} \label{E:NIBORDER}
		\sum_{|\vec{I}| \leq M} 
			\Big| \leftexp{(\vec{I});(Border)}{\mathfrak{N}} \Big|
		& \lesssim \underbrace{\sqrt{\epsilon} \sum_{|\vec{I}| = M} 
			\left|\partial_{\vec{I}} \freenewsec \right|_{\newg}}_{\mbox{borderline, principal order}}
			+ \underbrace{\sqrt{\epsilon} t^{- c \sqrt{\epsilon}} \sum_{|\vec{J}| \leq M - 1} 
				\left|\partial_{\vec{J}} \freenewsec \right|_{\newg}}_{\mbox{borderline, below principal order}}.
	\end{align}
	
	The following estimates for the junk inhomogeneous terms appearing in 
	the commuted equation \eqref{AE:ALTERNATELAPSEICOMMUTEDLOWER} hold on $(T,1] \times \mathbb{T}^3$
	for $0 \leq M \leq N - 1:$ 
	\begin{align} \label{E:WIDETILDENJUNKIPOINTWISE}
	\sum_{|\vec{I}| \leq M} \Big| \leftexp{(\vec{I});(Junk)}{\widetilde{\mathfrak{N}}} \Big| 
	& \lesssim \sqrt{\epsilon} t^{2/3 -c \sqrt{\epsilon}} \sum_{|\vec{H}| = M + 1} \Big| \partial_{\vec{H}} \upgamma \Big|_{\newg} 
		\\
	& \ \ + \underbrace{\sqrt{\epsilon} t^{- c \sqrt{\epsilon}} \sum_{1 \leq |\vec{I}| \leq M} \left| \partial_{\vec{I}} \newg \right|_{\newg}}_{
		\mbox{absent if $M = 0$}} 
		+ \underbrace{\sqrt{\epsilon} t^{- c \sqrt{\epsilon}} \sum_{1 \leq |\vec{I}| \leq M} \left| \partial_{\vec{I}} \newg^{-1} \right|_{\newg}}_{
			\mbox{absent if $M = 0$}}
		\notag \\
	& \ \ + \sqrt{\epsilon} t^{- c \sqrt{\epsilon}} \sum_{|\vec{I}| \leq M} \left| \partial_{\vec{I}} \Big(
			\freenewsec, \upgamma, \adjustednewp, \newu, \newlap, \dlap \Big) \right|_{\newg}.
		\notag
\end{align}

	The following estimates for the borderline inhomogeneous terms appearing in 
	the commuted equation \eqref{AE:ALTERNATELAPSEICOMMUTEDLOWER} hold on $(T,1] \times \mathbb{T}^3$
	for $0 \leq M \leq N - 1:$
	\begin{align}
	\sum_{|\vec{I}| \leq M} \Big| \leftexp{(\vec{I});(Border)}{\widetilde{\mathfrak{N}}} \Big|
		& \lesssim  \sum_{|\vec{H}| = M + 1} \Big| \partial_{\vec{H}} \upgamma \Big|_{\newg}
			\label{E:WIDETILDENIPOINTWISEBORDER} \\
		&\ \ + \underbrace{\sqrt{\epsilon} t^{- c \sqrt{\epsilon}} \sum_{1 \leq |\vec{I}| \leq M} \Big|\partial_{\vec{I}} \newg \Big|_{\newg}}_{
			\mbox{absent if $M=0$}} 
			+ 
			\underbrace{\sqrt{\epsilon} t^{- c \sqrt{\epsilon}} \sum_{1 \leq |\vec{I}| \leq M} \Big|\partial_{\vec{I}} \newg^{-1} \Big|_{\newg}}_{
				\mbox{absent if $M=0$}} 
				\notag \\
		& \ \ + \sqrt{\epsilon} t^{- c \sqrt{\epsilon}} \sum_{|\vec{I}| \leq M} \Big| \partial_{\vec{I}} \upgamma \Big|_{\newg}
			 \notag \\
		& \ \ + \sqrt{\epsilon} t^{- c \sqrt{\epsilon}} \sum_{|\vec{I}| \leq M} 
			\Big| \partial_{\vec{I}} \adjustednewp \Big|
			+ \sqrt{\epsilon} t^{- c \sqrt{\epsilon}} \sum_{|\vec{I}| \leq M} \Big| \partial_{\vec{I}} \newu \Big|_{\newg}.
			\notag 
	\end{align}

	Finally, the following estimates for the borderline inhomogeneous terms
	appearing in the commuted equations 
	\eqref{AE:MOMENTUMCONSTRAINTCOMMUTED}-\eqref{AE:RAISEDMOMENTUMCONSTRAINTCOMMUTED},
	\eqref{AE:METRICGAMMACOMMUTED}-\eqref{AE:SECONDFUNDCOMMUTED}, and 
	\eqref{AE:PARTIALTPCOMMUTED}-\eqref{AE:PARTIALTUJCOMMUTED} hold on $(T,1] \times \mathbb{T}^3$
	for $0 \leq M \leq N:$
	\begin{subequations}
	\begin{align}
		\sum_{|\vec{I}| \leq M} \Big| \leftexp{(\vec{I});(Border)}{\mathfrak{g}} \Big|_{\newg}
		& \lesssim  \sqrt{\epsilon} t^{- c \sqrt{\epsilon}} \sum_{|\vec{I}| \leq M} 
			 \left|\partial_{\vec{I}} \freenewsec \right|_{\newg}
			+ \underbrace{\sqrt{\epsilon} \sum_{|\vec{I}| = M} \Big| \partial_{\vec{I}} \upgamma 	
				\Big|_{\newg}}_{\mbox{borderline, principal order}} 
			\label{E:GAMMAIPOINTWISEGNORMBORDER}  \\
		& + \underbrace{\sqrt{\epsilon} t^{2/3} \sum_{|\vec{I}| = M} \Big| \partial_{\vec{I}} \dlap \Big|_{\newg}}_{\mbox{borderline, principal order}} 
				\notag \\
		& \ \  + \underbrace{\sqrt{\epsilon} t^{- c \sqrt{\epsilon}} \sum_{|\vec{J}| \leq M - 1} 
			\Big| \partial_{\vec{J}} \upgamma \Big|_{\newg}}_{\mbox{borderline, below principal order}}
			\notag \\
		& \ \ + \sqrt{\epsilon} t^{2/3 - c \sqrt{\epsilon}} \sum_{|\vec{J}| \leq M - 1} \left|\partial_{\vec{J}} \dlap \right|_{\newg},
			\notag 
	\end{align}
	\begin{align}   \label{E:UIPOINTWISEGNORM}
		\sum_{|\vec{I}| \leq M} \Big| \leftexp{(\vec{I});(Border)}{\mathfrak{U}} \Big|_{\newg}
		& \lesssim  \underbrace{\sqrt{\epsilon} t^{- c \sqrt{\epsilon}} \sum_{1 \leq |\vec{I}| \leq M} \Big|\partial_{\vec{I}} \newg \Big|_{\newg}}_{
			\mbox{absent if $M=0$}}
			+ \underbrace{\sqrt{\epsilon} t^{- c \sqrt{\epsilon}} \sum_{1 \leq |\vec{I}| \leq M} \Big|\partial_{\vec{I}} \newg^{-1} \Big|_{\newg}}_{
			\mbox{absent if $M=0$}}
			\\
		& \ \ + \sqrt{\epsilon} t^{- c \sqrt{\epsilon}} \sum_{|\vec{I}| \leq M} 
				\left|\partial_{\vec{I}} \freenewsec \right|_{\newg}
			\notag \\
		& \ \ + \sqrt{\epsilon} t^{- c \sqrt{\epsilon}} \sum_{|\vec{I}| \leq M} 
			\Big| \partial_{\vec{I}} \adjustednewp \Big|
			+ \underbrace{\sqrt{\epsilon} \sum_{|\vec{I}| = M} \Big| \partial_{\vec{I}} \newu \Big|_{\newg}}_{\mbox{borderline, 
				principal order}} \notag \\
		& \ \ + \underbrace{\sqrt{\epsilon} t^{- c \sqrt{\epsilon}} 
			\sum_{|\vec{J}| \leq M - 1} \Big| \partial_{\vec{J}} \newu \Big|_{\newg}}_{\mbox{borderline, below principal order}},
			\notag 
	\end{align}	
	\begin{align}
		\sum_{|\vec{I}| \leq M} \left| \leftexp{(\vec{I});(Border)}{\mathfrak{M}} \right|_{\newg} 
		& + \left| \leftexp{(\vec{I});(Border)}{\widetilde{\mathfrak{M}}} \right|_{\newg}
		 	\label{E:MIPOINTWISEGNORM} \\
		& \lesssim  \underbrace{\sqrt{\epsilon} t^{- c \sqrt{\epsilon}} \sum_{1 \leq |\vec{I}| \leq M} \Big|\partial_{\vec{I}} \newg \Big|_{\newg}}_{
			\mbox{absent if $M=0$}}
			+ \underbrace{\sqrt{\epsilon} t^{- c \sqrt{\epsilon}} \sum_{1 \leq |\vec{I}| \leq M} \Big|\partial_{\vec{I}} \newg^{-1} \Big|_{\newg}}_{
			\mbox{absent if $M=0$}}
			\notag \\
		& \ \ + \sqrt{\epsilon} t^{- c \sqrt{\epsilon}} \sum_{|\vec{I}| \leq M} 
			\left|\partial_{\vec{I}} \freenewsec \right|_{\newg}
			+ \underbrace{\sqrt{\epsilon} \sum_{|\vec{I}| = M} \Big| \partial_{\vec{I}} \upgamma 
			\Big|_{\newg}}_{\mbox{borderline, principal order}} \notag \\
		& \ \ + \sqrt{\epsilon} t^{- c \sqrt{\epsilon}} \sum_{|\vec{I}| \leq M} \Big| \partial_{\vec{I}} \adjustednewp \Big|
			+ \underbrace{\sqrt{\epsilon} \sum_{|\vec{I}| = M} \Big| \partial_{\vec{I}} \newu \Big|_{\newg}}_{\mbox{borderline, principal order}}
			\notag \\
		& \ \ + \underbrace{\sqrt{\epsilon} t^{- c \sqrt{\epsilon}} \sum_{|\vec{J}| \leq M - 1} 
				\Big| \partial_{\vec{J}} \upgamma \Big|_{\newg}}_{\mbox{borderline, below principal order}} 
			+ \underbrace{\sqrt{\epsilon} t^{- c \sqrt{\epsilon}} 
				\sum_{|\vec{J}| \leq M - 1} \Big| \partial_{\vec{J}} \newu \Big|_{\newg}}_{\mbox{borderline, below principal 
				order}}. \notag 
	\end{align}
	\end{subequations}

\end{proposition}

\begin{remark}
	In the above estimates, sums of the form $\sum_{|\vec{J}| \leq M - 1}$ 
	are understood to be absent when $M = 0.$ Furthermore, sums of the form $\sum_{1 \leq |\vec{J}| \leq M - 1}$ 
	are understood to be absent when $M=1.$ 
\end{remark}

\begin{remark}
	Some of the above powers of $\epsilon$ stated above are non-optimal.
\end{remark}

\begin{proof}
	
We first discuss the top level strategy. Much like in \eqref{E:SAMPLEEST}, we will be deriving pointwise estimates
for the $|\cdot|_{\newg}$ norm of tensorial products. In particular, we will estimating quantities of the form
\begin{align} \label{E:SAMPLEESTPOINTWISE}
	\left| F(t;t^{A_1}v_1, t^{A_2}v_2,\cdots, t^{A_l} v_l) \prod_{a=1}^l (\partial_{\vec{I}_a}v_a) \right|_{\newg}, 
\end{align}	
	where for $1 \leq a \leq l,$
	\begin{align}
		v_a \in \left\lbrace 
			\newg, (\newg^{-1}), 
			\freenewsec, 
			\upgamma, 
			\adjustednewp, 
			\newu, 
			\newlap, 
			\dlap 
			\right\rbrace
	\end{align}
	is a renormalized solution tensor,
	the $A_a$ are non-negative constants for $1 \leq a \leq l$
	(the $t^{A_a}$ appear explicitly in the equations), and $F$ is a smooth scalar-valued function 
	of its arguments that, by virtue of the bootstrap assumptions, will verify 
	\begin{align} \label{E:FHARMLESS}
		\| F(t;t^{A_1}v_1, t^{A_2}v_2,\cdots, t^{A_l} v_l) \|_{C^0} \lesssim 1. 
	\end{align}	
	Aside from a few exceptional cases involving the quantities $t^{4/3} \newlap$ and $t^{2/3} \dlap,$ 
	(the exceptional cases are easy to estimate and are handled later in the proof), 
	the estimates in Prop.~\ref{P:POINTWISEESTIMATES} are derived using the following strategy,
	which is a slight modification of the strategy used in our proof of Prop.~\ref{P:STRONGPOINTWISE}:
	\begin{itemize}
		\item We bound terms using the norms $|\cdot|_{\newg}$ and $\| \cdot \|_{C_{\newg}^M}$ instead of the frame component 
		norms used in the proof of Prop.~\ref{P:STRONGPOINTWISE}.
		\item In the non-exceptional cases, there is at most one ``large index'' factor $\partial_{\vec{I}_{a_{Large}}} v_{a_{Large}}$ in \eqref{E:SAMPLEESTPOINTWISE} 
			with so many derivatives on it that the strong estimates of Prop.~\ref{P:STRONGPOINTWISE} don't apply.
			If there are no large index factors, then choose one of $v's$ and designate it as ``$v_{a_{Large}}.$''
		\item For $a \neq a_{Large},$ we bound $\partial_{\vec{I}_a }v_a$ in $\| \cdot \|_{C_{\newg}^0}$ using 
			the strong estimates for the norms $\| \cdot \|_{C_{\newg}^M}$ proved in Prop.~\ref{P:STRONGPOINTWISE}.
			The proposition implies that these quantities can be bounded by at worst
			$\| \partial_{\vec{I}_{a \neq a_{Large}}} v_{a \neq a_{Large}} \|_{C_{\newg}^0} \lesssim \sqrt{\epsilon} t^{- c \sqrt{\epsilon}},$ 
			or in some important cases (depending on $\vec{I}_{a \neq a_{Large}}$ and the field variable $v_{a \neq a_{Large}}$) by 
			$\| \partial_{\vec{I}_{a \neq a_{Large}}} v_{a \neq a_{Large}} \|_{C_{\newg}^0} \lesssim \epsilon.$ 
		\item Since $|\newg|_{\newg} = |\newg^{-1}|_{\newg} = \sqrt{3},$ the quantities $|\newg|_{\newg}$ and 
			$|\newg^{-1}|_{\newg}$ never appear explicitly in our estimates.
		\item In the non-exceptional cases, this strategy allows us to 
			apply the $\newg-$Cauchy-Schwarz inequality 
			and to bound all quadratic and higher-order products
			of the form \eqref{E:SAMPLEESTPOINTWISE} by 
			\begin{align} \label{E:REPPOINTWISEBOUND}
				\left| F(t;t^{A_1}v_1,t^{A_2}v_2,\cdots, t^{A_l} v_l) \prod_{a=1}^l (\partial_{\vec{I}_a}v_a) \right|_{\newg} 
					& \lesssim \sqrt{\epsilon} t^{-B} \left|\partial_{\vec{I}_{a_{Large}}} v_{a_{Large}} \right|_{\newg},
			\end{align}
			where $B$ is either $0$ or at worst $c \sqrt{\epsilon}.$
			The right-hand side of \eqref{E:REPPOINTWISEBOUND}
			then appears as one of the terms on the right-hand side of the inequalities of Prop.~\ref{P:POINTWISEESTIMATES}.			
			\item In a few cases, the product to be estimated is multiplied by an overall factor $t^A$ for some $A > 0;$
			this factor of $t^A$ is incorporated into the bounds on right-hand side of the inequalities of Prop.~\ref{P:POINTWISEESTIMATES}.
			\item In many cases, the bound \eqref{E:REPPOINTWISEBOUND} is non-optimal with respect to the powers of $\epsilon$ involved.
			However, for simplicity, we will often be content with such a non-optimal bound when it 
			is strong enough to allow us to prove our main stable singularity formation theorem.
		\end{itemize}

The majority of the products that we have to bound are ``non-borderline:''
\begin{itemize}
	\item For these terms, the loss of $t^{-B}$ is harmless in the following sense: the estimate \eqref{E:REPPOINTWISEBOUND}
		leads to a ``non-borderline'' integral in inequality \eqref{E:GCOMMUTEDSOB}
		or the energy integral inequality \eqref{E:ENERGYINTEGRALINEQUALITIES}.
		For these terms, it would be possible to derive the main a priori norm estimate 
		\eqref{E:MAINHIGHNORMESTIMATE} even with a significantly worse (i.e., larger) value of $B$
		in \eqref{E:REPPOINTWISEBOUND}.
\end{itemize}

On the other hand, there are also ``borderline'' terms, which we treat as follows: 
\begin{itemize}
	\item For the principal order (i.e., order $M$) borderline terms, \emph{we must derive estimates with $B = 0$
	in \eqref{E:REPPOINTWISEBOUND}.} 
		Otherwise, our proof of the main a priori norm estimate \eqref{E:MAINHIGHNORMESTIMATE} would break down. 
	\item For the lower-order (i.e., order $\leq M-1$) borderline terms, we derive estimates with 
		$B = c \sqrt{\epsilon},$ where $c > 0$ is a constant. Such an estimate is good enough to allow our proof of
		the main a priori norm estimate \eqref{E:MAINHIGHNORMESTIMATE} to go through. 
\end{itemize}

\noindent \textbf{Proof of \eqref{E:POINTWISEJUNKNORMS}}: 
First, by inspecting the equation \eqref{AE:SECONDFUNDCOMMUTEDERRORTERMS} for $\leftexp{(\vec{I});(Junk)}{\mathfrak{K}}$
(and similarly for the other terms on the left-hand side of \eqref{E:POINTWISEJUNKNORMS}), we see that 
the lapse gradient variable $\dlap$ always appears in products that feature the weight $t^{2/3}.$
The vast majority of the terms appearing in $\leftexp{(\vec{I});(Junk)}{\mathfrak{K}},$ etc. 
(all of which are quadratic) are non-exceptional and can be bounded 
by using \eqref{E:REPPOINTWISEBOUND} with $B = c \sqrt{\epsilon}.$ 
There are a few exceptional products 
involving $t^{4/3} \newlap$ and $t^{2/3} \dlap.$ 
In these exceptional cases, it is possible to have a product such that two of the factors 
(one of which is $t^{4/3} \partial_{\vec{I}_a} \psi$ or $t^{2/3} \partial_{\vec{I}_a} \dlap$) 
have too many derivatives to bound them by the strong estimates of Prop.~\ref{P:STRONGPOINTWISE}.
Two examples are the products 
$- \frac{1}{2} \left[\partial_{\vec{I}}, \big[1 + t^{4/3} \newlap \big] (\newg^{-1})^{ef}  \right] \partial_e \upgamma_{f \ j}^{\ i}$
and $- t^{2/3} \left[\partial_{\vec{I}}, (\newg^{-1})^{ia} \right] \partial_a \dlap_j$
on the right-hand side of \eqref{AE:SECONDFUNDCOMMUTEDERRORTERMS}. To estimate such products, we first note that we have the inequalities
$t^{4/3} \| \newlap \|_{C^{N-3}} + t^{2/3} \| \partial \dlap \|_{C_{\newg}^{N-5}} \lesssim \epsilon t^{2/3 - Z \upsigma} \highnorm{N} \lesssim 
\epsilon t^{2/3 - Z \upsigma},$ where $Z>0$ is an integer. 
These inequalities follow from the bootstrap assumptions \eqref{E:HIGHBOOT}-\eqref{E:POTBOOT} and Sobolev embedding.
The remaining factors in the product (i.e., the non-lapse factors) can be bounded in $|\cdot|_{\newg}$ norm by at worst 
$t^{-c\sqrt{\epsilon}} \left|\partial_{\vec{I}_{a_{Large}}} v_{a_{Large}} \right|_{\newg}$
(much like in \eqref{E:REPPOINTWISEBOUND}). It follows that in total, 
the exceptional products are bounded by 
$\epsilon t^{2/3 - Z \upsigma - c \sqrt{\epsilon}}\left|\partial_{\vec{I}_{a_{Large}}} v_{a_{Large}} \right|_{\newg},$
where $v_{a_{Large}}$ is a non-lapse variable. This is a far better bound than needed, and
inequality \eqref{E:POINTWISEJUNKNORMS} thus follows.

\ \\
\noindent \textbf{Proof of \eqref{E:GJUNKNORM}-\eqref{E:GINVERSEJUNKNORM}}:
The proof is essentially the same as the proof of \eqref{E:POINTWISEJUNKNORMS}, and there are no exceptional products. 
However, some of the terms that occur when all derivatives fall on $\newlap$ are linear and
we bound them by the term $\sum_{|\vec{I}| = M} \left|\partial_{\vec{I}} \newlap \right|$
on the right-hand sides of \eqref{E:GJUNKNORM}-\eqref{E:GINVERSEJUNKNORM}. All of the other terms that
occur when all derivatives fall on $\newlap$ are quadratic and involve the factor $\freenewsec;$ for these products, 
we use the strong estimate \eqref{E:SECONDFUNDUPGRADEPOINTWISEGNORM} 
to also bound them by $\sum_{|\vec{I}| = M} \left|\partial_{\vec{I}} \newlap \right|$ (without a loss of $t^{-c \sqrt{\epsilon}}$).
The remaining terms to be estimated are quadratic. In these cases, we simply allow the loss 
$B = c \sqrt{\epsilon}$ in \eqref{E:REPPOINTWISEBOUND}, and the desired bounds \eqref{E:GJUNKNORM}-\eqref{E:GINVERSEJUNKNORM} follow.  
\ \\

\noindent \textbf{Proof of \eqref{E:GBORDERNORM}-\eqref{E:GINVERSEBORDERNORM}}:
We only prove \eqref{E:GBORDERNORM} since the proof of \eqref{E:GINVERSEBORDERNORM} is nearly identical.
We need to pointwise bound the $|\cdot|_{\newg}$ norm of the right-hand side of \eqref{E:METRICEVOLUTIONBORDERCOMMUTED}. 
That is, we have to bound tensors of the form 
\begin{align} \label{E:GBORDERTENSOR}
	T_{ij} & = \Big(\partial_{\vec{I}_1} \newg_{ia}  \Big) 
		\Big(\partial_{\vec{I}_2} \freenewsec_{\ j}^a \Big),
\end{align}
where $|\vec{I}_1| + |\vec{I}_2| \leq M.$ If $|\vec{I}_1| = M$ and $|\vec{I}_2| = 0,$
then the strong estimate \eqref{E:SECONDFUNDUPGRADEPOINTWISEGNORM} implies the bound
\begin{align}
	|T|_{\newg} & \lesssim 
		\underbrace{\epsilon \sum_{|\vec{I}| = M} \Big| \partial_{\vec{I}} \newg \Big|_{\newg}}_{\mbox{borderline, principal order}},
\end{align}
which is $\lesssim$ the right-hand side of \eqref{E:GBORDERNORM}. Similarly,
if $|\vec{I}_1| \leq M - 1$ and $|\vec{I}_2| \leq N - 3,$
then the strong estimate \eqref{E:SECONDFUNDUPGRADEPOINTWISEGNORM} implies the bound
\begin{align}
	|T|_{\newg} & \lesssim 
		\underbrace{\epsilon t^{- c \sqrt{\epsilon}} \sum_{|\vec{J}| \leq M - 1} \Big| \partial_{\vec{J}} \newg \Big|_{\newg}}_{\mbox{borderline, below principal 	
			order}},
\end{align}
which is $\lesssim$ the right-hand side of \eqref{E:GBORDERNORM}. If $|\vec{I}_1| = 0$ and $|\vec{I}_2| = M,$
then 
\begin{align}
	|T|_{\newg} & \lesssim 
		\underbrace{\sum_{|\vec{I}| = M} \Big| \partial_{\vec{I}} \freenewsec \Big|_{\newg}}_{\mbox{borderline, principal order}},
\end{align}
which is $\lesssim$ the right-hand side of \eqref{E:GBORDERNORM}. Finally, if
$|\vec{I}_1| \leq N - 4$ and $|\vec{I}_2| \leq M - 1,$
then the strong estimate \eqref{E:GCOMMUTEDSTRONGPOINTWISE} implies the bound
\begin{align}
	|T|_{\newg} & \lesssim 
		\underbrace{\epsilon t^{- c \sqrt{\epsilon}} \sum_{|\vec{J}| \leq M - 1} \Big| \partial_{\vec{I}} \freenewsec \Big|_{\newg}}_{\mbox{borderline, below 
		principal order}},
\end{align}
which is $\lesssim$ the right-hand side of \eqref{E:GBORDERNORM}.

\ \\

\noindent \textbf{Proof of \eqref{E:NIJUNK}}:
To prove \eqref{E:NIJUNK}, we have to estimate the term $\leftexp{(\vec{I});(Junk)}{\mathfrak{N}}$
from equation \eqref{AE:LAPSEICOMMUTEDINHOMOGENEOUSTERMJUNK}. 

\emph{Step 1: Bound for the $\partial_{\vec{I}} \leftexp{(Junk)}{\mathfrak{N}}$ term from \eqref{AE:LAPSEICOMMUTEDINHOMOGENEOUSTERMJUNK}.}
The term $\leftexp{(Junk)}{\mathfrak{N}}$ is defined in \eqref{AE:LAPSENJUNK}. 
To bound $\partial_{\vec{I}}$ of the first product on the right-hand
side of \eqref{AE:LAPSENJUNK}, we have to pointwise bound e.g. possibly exceptional products of the form
\begin{align} \label{E:EXCEPTIONALTERMI}
	T & = t^{2/3} \Big(\partial_{\vec{I}_1} (\newg^{-1})^{ef} \Big) 
		\Big(\partial_{\vec{I}_2} \upgamma_{e \ f}^{\ b} \Big) 
		\Big(\partial_{\vec{I}_3} \dlap_b \Big),
\end{align}
where $|\vec{I}_1| + |\vec{I}_2| + |\vec{I}_3| \leq M.$ 
Exceptional products [in which more than one factor in \eqref{E:EXCEPTIONALTERMI}
has too many derivatives to apply Prop.~\ref{P:STRONGPOINTWISE}] 
can occur if $|\vec{I}_3| \leq N-4.$ 
To handle these cases, we use the estimate
$t^{2/3} \| \dlap \|_{C_{\newg}^{N-4}} \lesssim  \epsilon t^{2/3 - Z \upsigma},$ 
which can be proved in the same way as the estimate 
$t^{2/3} \| \partial \dlap \|_{C_{\newg}^{N-5}} \lesssim \epsilon t^{2/3 - Z \upsigma}$
used in the proof of \eqref{E:POINTWISEJUNKNORMS}.
Also using the strong estimate \eqref{E:GINVERSECOMMUTEDSTRONGPOINTWISE}, 
we deduce that when $|\vec{I}_3| \leq N - 4,$ we have
\begin{align} \label{E:EXCEPTIONALTERMIBOUNDED}
	|T| & \lesssim \sqrt{\epsilon} t^{2/3 - Z \upsigma - c \sqrt{\epsilon}} 
		\sum_{|\vec{I}| \leq M} \Big| \partial_{\vec{I}} \upgamma \Big|_{\newg}.
\end{align}
Clearly \eqref{E:EXCEPTIONALTERMIBOUNDED} implies that $|T|$ is $\lesssim$ the right-hand side of \eqref{E:NIJUNK} as desired.
In the remaining (non-exceptional) cases of \eqref{E:EXCEPTIONALTERMI}, we can simply bound
$|T|$ by using the estimate \eqref{E:REPPOINTWISEBOUND} and then taking into account the factor $t^{2/3}$ in \eqref{E:EXCEPTIONALTERMI}; 
the desired bound follows.

To bound $\partial_{\vec{I}}$ of the second product on the right-hand
side of \eqref{AE:LAPSENJUNK}, we have to pointwise bound (non-exceptional) terms of the form
\begin{align} \label{E:NONEXCEPTIONALTERM}
	T & = \Big(\partial_{\vec{I}_1} \adjustednewp \Big) 
		\Big(\partial_{\vec{I}_2} \newg_{ab} \Big) 
		\Big(\partial_{\vec{I}_3} \newu^a \Big)
		\Big(\partial_{\vec{I}_4} \newu^b \Big),
\end{align}
where $|\vec{I}_1| + |\vec{I}_2| + |\vec{I}_3| + |\vec{I}_4| \leq M.$
In all cases, we simply allow the loss $B = c \sqrt{\epsilon}$ in \eqref{E:REPPOINTWISEBOUND}, 
and the desired bound for the term \eqref{E:NONEXCEPTIONALTERM} follows. 
A similar argument applies to the last product on the right-hand
side of \eqref{AE:LAPSENJUNK}.

\emph{Step 2: Bound for the commutator terms in \eqref{AE:LAPSEICOMMUTEDINHOMOGENEOUSTERMJUNK}.} 
The commutator term $t^{2/3} \left[ \partial_{\vec{I}} , (\newg^{-1})^{ab} \right] \partial_a \dlap_b,$
which involves some exceptional terms, can be bounded by using an argument similar
to the one we used to bound \eqref{E:EXCEPTIONALTERMI}.

To bound the $\left[ \partial_{\vec{I}} , \freenewsec_{\ b}^a \freenewsec_{\ a}^b \right] \newlap$
term from the right-hand side of \eqref{AE:LAPSEICOMMUTEDINHOMOGENEOUSTERMJUNK}, we have to
pointwise bound products of the form
\begin{align} \label{E:KKLAPTERM}
	T & = \Big(\partial_{\vec{I}_1}\freenewsec_{\ b}^a \Big) 
		\Big(\partial_{\vec{I}_2} \freenewsec_{\ a}^b \Big) 
		\Big(\partial_{\vec{I}_3} \newlap \Big),
\end{align}
where $1 \leq |\vec{I}_1| + |\vec{I}_2| + |\vec{I}_3| \leq M$ and $|\vec{I}_3| \leq M - 1.$
If $|\vec{I}_3| \leq N - 5,$ then we use the strong estimates 
\eqref{E:SECONDFUNDUPGRADEPOINTWISEGNORM} and \eqref{E:LAPSECOMMUTEDSTRONGPOINTWISE}
to derive the bound
\begin{align} \label{E:KKLAPBOUNDED}
	|T| & \lesssim \sqrt{\epsilon} t^{- c \sqrt{\epsilon}} 
		\sum_{|\vec{I}| \leq M} \left|\partial_{\vec{I}} \freenewsec \right|_{\newg}.
\end{align}
The bound \eqref{E:KKLAPBOUNDED} implies that $|T|$ is $\lesssim$ the right-hand side of \eqref{E:NIJUNK} as desired.
In all other cases, we use the \emph{frame component norm} estimates
from \eqref{E:SECONDFUNDUPGRADEPOINTWISE} to derive the desired bound
\begin{align} \label{E:KKLAPBOUNDEDII}
	|T| & \lesssim \sqrt{\epsilon}  \sum_{|\vec{I}| \leq M-1} \Big| \partial_{\vec{I}} \newlap \Big|.
\end{align}

The commutator term $2 \left[ \partial_{\vec{I}} , \adjustednewp  \right] \newlap$
on the right-hand side of \eqref{AE:LAPSEICOMMUTEDINHOMOGENEOUSTERMJUNK} can similarly
be bounded with the help of the strong estimate \eqref{E:PCOMMUTEDSTRONGPOINTWISE}.

The last two terms on the right-hand side of \eqref{AE:LAPSEICOMMUTEDINHOMOGENEOUSTERMJUNK} are
non-exceptional and easy to bound. We simply allow the loss $B = c \sqrt{\epsilon}$ in \eqref{E:REPPOINTWISEBOUND} and then use the fact
that these terms are multiplied by the factor $t^{4/3}.$

\ \\

\noindent \textbf{Proof of \eqref{E:NIBORDER}}:
To prove \eqref{E:NIBORDER}, we have to pointwise bound the term $\leftexp{(\vec{I});(Border)}{\mathfrak{N}}$ from equation 
\eqref{AE:LAPSEICOMMUTEDINHOMOGENEOUSTERMBORDERLINE}. That is, we have to bound products of the form
\begin{align} \label{E:TNREPTERM}
	T & = \Big(\partial_{\vec{I}_1} \freenewsec_{\ b}^a  \Big)
		\Big(\partial_{\vec{I}_2} \freenewsec_{\ a}^b \Big),
\end{align}
where $|\vec{I}_1| + |\vec{I}_2| \leq M.$ Principal order borderline terms occur when either $|\vec{I}_1| = M$ or $|\vec{I}_2| = M.$
In either case, one of the factors occurs with no derivatives, and the estimate \eqref{E:SECONDFUNDUPGRADEPOINTWISEGNORM} allows us
to bound the $|\cdot|_{\newg}$ norm of this factor by $\epsilon.$ Thus, we have
\begin{align}
	|T| & \lesssim \underbrace{\epsilon \sum_{|\vec{I}| = M} 
			\left|\partial_{\vec{I}} \freenewsec \right|_{\newg}}_{\mbox{borderline, principal order}}
\end{align}
in this case. In the remaining cases of \eqref{E:TNREPTERM}, we have $1 \leq |\vec{I}_1|, |\vec{I}_2| \leq M - 1.$
In these cases, we allow the loss $B = c \sqrt{\epsilon}$ in \eqref{E:REPPOINTWISEBOUND}, and the desired bound follows.
\ \\

\noindent \textbf{Proof of \eqref{E:WIDETILDENJUNKIPOINTWISE}}:
We have to pointwise bound the term $\leftexp{(\vec{I});(Junk)}{\widetilde{\mathfrak{N}}}$ 
from equation \eqref{AE:LAPSETILDEICOMMUTEDINHOMOGENEOUSTERMJUNK}. 
The proof of \eqref{E:WIDETILDENJUNKIPOINTWISE} is essentially the same as the proof of \eqref{E:POINTWISEJUNKNORMS}.
The only new feature is that there are products on the right-hand side of \eqref{AE:LAPSETILDEICOMMUTEDINHOMOGENEOUSTERMJUNK}
that depend on $|\vec{I}| + 1$ derivatives of $\upgamma,$ and
these products are multiplied by the factor $t^{2/3}.$

\ \\

\noindent \textbf{Proof of \eqref{E:WIDETILDENIPOINTWISEBORDER}}:
We have to bound the term $\leftexp{(\vec{I});(Border)}{\widetilde{\mathfrak{N}}}$ 
from equation \eqref{AE:LAPSETILDEICOMMUTEDINHOMOGENEOUSTERMBORDERLINE}.
That is, we have to bound the products that arise when $\partial_{\vec{I}}$ 
is applied to the right-hand side of \eqref{AE:LAPSENTILDEBORDERLINE}
and $|\vec{I}| \leq M \leq N - 1.$ No exceptional products occur. 
If all derivatives fall on $\partial \upgamma,$ 
then the resulting terms are each $\lesssim \sum_{|\vec{H}| = M + 1} \Big| \partial_{\vec{H}} \upgamma \Big|_{\newg},$ 
which is $\lesssim$ the right-hand side of \eqref{E:WIDETILDENIPOINTWISEBORDER} as desired. Otherwise, we allow
the loss $B = c \sqrt{\epsilon}$ in \eqref{E:REPPOINTWISEBOUND}, and the desired bound follows.

\ \\

\noindent \textbf{Proof of \eqref{E:GAMMAIPOINTWISEGNORMBORDER}:}
We have to bound the term $\partial_{\vec{I}} \leftexp{(Border)}{\mathfrak{g}}_{e \ i}^{\ b}$ from the right-hand side of equation \eqref{AE:METRICGAMMACOMMUTEDERRORTERMS}. That is, we have to bound the products that arise when $\partial_{\vec{I}}$ is applied to the right-hand side of \eqref{AE:GAMMAEVOLUTIONERROR}. The first two products on the right-hand side of \eqref{AE:GAMMAEVOLUTIONERROR} can be estimated in the same way. By the product rule, to bound these terms, we must estimate the $|\cdot|_{\newg}$ pointwise norm of tensors $T_{e \ i}^{\ b}$ of the form
\begin{align} \label{E:GREPTERM}
	T_{e \ i}^{\ b} = \Big(\partial_{\vec{I_1}} \freenewsec_{\ a}^b \Big)
		\Big(\partial_{\vec{I_2}} \upgamma_{e \ i}^{\ a} \Big),
\end{align}
where $|\vec{I}_1| + |\vec{I}_2| \leq M.$ 
If $|\vec{I}_2| = M,$ principal order borderline terms are generated. However, in this case, $|\vec{I_1}| = 0,$ and the 
strong estimate \eqref{E:SECONDFUNDUPGRADEPOINTWISEGNORM} implies that
\begin{align} \label{E:TPRINCIPALGESTIMATE}
	|T|_{\newg} & \underbrace{\lesssim \epsilon \sum_{|\vec{I}| = M} 
		\Big| \partial_{\vec{I}} \upgamma \Big|_{\newg}}_{\mbox{borderline, principal order}}.
\end{align}
In the remaining cases, we have $|\vec{I}_2| \leq M - 1.$ In these cases, we allow
the loss $B = c \sqrt{\epsilon}$ in \eqref{E:REPPOINTWISEBOUND}, and the desired bound follows. 

To bound $\partial_{\vec{I}}$ of the last product on the right-hand side of \eqref{AE:GAMMAEVOLUTIONERROR}, we have to 
estimate the $|\cdot|_{\newg}$ norm of
\begin{align} \label{E:GREPTERMLAST}
	T_{e \ a}^{\ b} = 
		t^{2/3} \Big(\partial_{\vec{I_1}} \dlap_e \Big) 
		\Big(\partial_{\vec{I_2}} \freenewsec_{\ a}^b \Big),
\end{align}
where $|\vec{I}_1| + |\vec{I}_2| \leq M.$ Principal order borderline terms occur when $|\vec{I_1}| = M$
and $|\vec{I_2}| = 0.$ Again using \eqref{E:SECONDFUNDUPGRADEPOINTWISEGNORM}, we see that when $|\vec{I}_1| = M,$
\begin{align} \label{E:TPRINCIPALGESTIMATEII}
	|T|_{\newg} 
		& \underbrace{\lesssim \epsilon t^{2/3} \sum_{|\vec{I}| = M} 
		\Big| \partial_{\vec{I}} \dlap \Big|_{\newg}}_{\mbox{borderline, principal order}}.
\end{align}
In the remaining cases of \eqref{E:GREPTERMLAST}, we have $|\vec{I}_1| \leq M - 1.$ In these cases, we allow
the loss $B = c \sqrt{\epsilon}$ in \eqref{E:REPPOINTWISEBOUND} and account for the factor $t^{2/3}$ in \eqref{E:GREPTERMLAST}, 
and the desired bound follows.

\ \\

\noindent \textbf{Proof of \eqref{E:UIPOINTWISEGNORM}}
To prove \eqref{E:UIPOINTWISEGNORM}, we have to bound the $|\cdot|_{\newg}$ norm of the quantity
$\leftexp{(\vec{I});(Border)}{\mathfrak{U}}^j$ defined in \eqref{AE:UJICOMMUTEDERRORBORDERLINE}, where $|\vec{I}| \leq M.$

\emph{Step 1: Bound for the $\partial_{\vec{I}} \leftexp{(Border)}{\mathfrak{U}}^j$ term from \eqref{AE:UJICOMMUTEDERRORBORDERLINE}.} 
To derive the desired bound, we must apply $\partial_{\vec{I}}$ to the product of terms on the right-hand side of \eqref{AE:UJBORDERLINE} and estimate the $|\cdot|_{\newg}$ norm of the resulting terms. More precisely,
we have to estimate the $|\cdot|_{\newg}$ norm of
\begin{align} \label{E:TUREPFIRST}
	T^j & = \Big(\partial_{\vec{I}_1} \freenewsec_{\ a}^j \Big)
		\Big(\partial_{\vec{I}_2} \newu^a\Big),
\end{align}
where $|\vec{I}_1| + |\vec{I}_2| \leq M.$ 
Principal order borderline terms occur when $|\vec{I_1}| = 0$ and $|\vec{I_2}| = M.$ From the strong estimate
\eqref{E:SECONDFUNDUPGRADEPOINTWISEGNORM}, it follows that when $|\vec{I}_2| = M,$ we have
\begin{align} \label{E:TUPRINCIPALUESTIMATE}
	|T|_{\newg} 
		& \underbrace{\lesssim \epsilon \sum_{|\vec{I}| = M} 
		\Big| \partial_{\vec{I}} \newu \Big|_{\newg}}_{\mbox{borderline, principal order}}.
\end{align}
In the remaining cases of \eqref{E:TUREPFIRST}, we have $|\vec{I}_2| \leq M - 1.$ In these cases, we allow
the loss $B = c \sqrt{\epsilon}$ in \eqref{E:REPPOINTWISEBOUND}, and the desired bound follows.

\emph{Step 2: Bound for the commutator term from \eqref{AE:UJICOMMUTEDERRORBORDERLINE}.} In this step, we have to bound the 
$|\cdot|_{\newg}$ norm of
\begin{align} \label{E:TUREPTHIRD}
	T^j & = \Big(\partial_{\vec{I_1}} (\newg^{-1})^{jc}\Big) 
		\Big(\partial_{\vec{I_2}} \partial_c \adjustednewp \Big)
		\times \partial_{\vec{I_3}} \left(\frac{1}{2\big[1 + t^{4/3} \newg_{ab} \newu^a \newu^b \big]^{1/2} 
			\Big[\adjustednewp + \frac{1}{3} \Big]} \right), 
\end{align}
where $1 \leq |\vec{I}_1| + |\vec{I}_2| + |\vec{I}_3| \leq M,$ and $|\vec{I_2}| \leq M - 1.$ In all of these cases, we simply allow the loss 
$B = c \sqrt{\epsilon}$ in \eqref{E:REPPOINTWISEBOUND}, and the desired bound follows. 
%We point out that lower-order borderline terms are generated whenever $|\vec{I}_2| \geq N - 2.$ In these cases, we use inequality 
%\eqref{E:GNORMPARTIALIGANDGINVERSEINTERMSOFGAMMA} to estimate $\partial_{\vec{I_2}} (\newg^{-1})^{jc}$ in terms of 
%$\partial_{|\vec{I_2}|-1} \upgamma$ and the lower-order derivatives of $\upgamma,$ which leads to the bound

%\begin{align} \label{E:TUPRINCIPALUESTIMATEIIDONE}
%	|T|_{\newg} & \lesssim \underbrace{\sqrt{\epsilon} t^{- c \sqrt{\epsilon}} \sum_{|\vec{I}| \leq M - 1} 
%		\Big| \partial_{\vec{I}} \upgamma \Big|_{\newg}}_{\mbox{borderline, below principal order}}.
%\end{align}

\ \\

\noindent \textbf{Proof of \eqref{E:MIPOINTWISEGNORM}:} 
We only discuss the estimates for $\left| \leftexp{(\vec{I});(Border)}{\mathfrak{M}} \right|_{\newg};$ the estimates for 
$\left| \leftexp{(\vec{I});(Border)}{\widetilde{\mathfrak{M}}} \right|_{\newg}$ are nearly identical.
We have to pointwise bound the $|\cdot|_{\newg}$ norm of the term $\leftexp{(\vec{I});(Border)}{\mathfrak{M}}_i$ from \eqref{AE:MOMENTUMCONSTRAINTCOMMUTEDERRORBORDERLINE}.

\emph{Step 1: Bound for $\partial_{\vec{I}} \leftexp{(Border)}{\mathfrak{M}}_i.$}
We must bound the $|\cdot|_{\newg}$ norm of the products that arise when $\partial_{\vec{I}}$ is applied to the right-hand side of \eqref{AE:MOMENTUMERRORBORDERLINE}. The first product can be bounded in the same way that we bound the second product. For the second product, we have to pointwise bound the $|\cdot|_{\newg}$ norm of
\begin{align} \label{E:TMREPSECOND}
	T_i & = \Big(\partial_{\vec{I_1}} \left(\upgamma_{a \ i}^{\ b} + \upgamma_{i \ a}^{\ b} 
			- (\newg^{-1})^{bl} \newg_{ma} \upgamma_{l \ i}^{\ m} \right) \Big) 
		\Big(\partial_{\vec{I_2}} \freenewsec_{\ b}^a \Big),
\end{align}
where $|\vec{I}_1| + |\vec{I}_2| \leq M.$ Principal order borderline terms are generated when 
$|\vec{I}_1| = M,$ $|\vec{I}_2| =  0,$ and all derivatives fall on $\upgamma.$
In these cases, the strong estimate \eqref{E:SECONDFUNDUPGRADEPOINTWISEGNORM} implies that
\begin{align} \label{E:TMPRINCIPALUESTIMATEII}
	|T|_{\newg} & \lesssim 
			\underbrace{\epsilon \sum_{|\vec{I}| = M} \Big| \partial_{\vec{I}} \upgamma \Big|_{\newg}}_{\mbox{borderline, principal order}}.
\end{align}
Clearly, the right-hand side of \eqref{E:TMPRINCIPALUESTIMATEII} is $\lesssim$ the right-hand side of \eqref{E:MIPOINTWISEGNORM} as desired. 
In all of the remaining cases of \eqref{E:TMREPSECOND}, we allow the loss $B = c \sqrt{\epsilon}$ in \eqref{E:REPPOINTWISEBOUND}, 
and the desired bound follows.

For the third product of terms right-hand side of \eqref{AE:MOMENTUMERRORBORDERLINE}, we have to
pointwise bound the $|\cdot|_{\newg}$ norm of
\begin{align} \label{E:TMREPTHIRD}
	T_i & = \Big(\partial_{\vec{I_1}} \adjustednewp \Big)
		\Big(\partial_{\vec{I_2}} \newg_{ia} \Big)
		\Big(\partial_{\vec{I_3}} \newu^a \Big), 
\end{align}
where $|\vec{I}_1| + |\vec{I}_2| + |\vec{I}_3| \leq M.$ 
Principal order borderline terms are generated when 
$|\vec{I}_1| = |\vec{I}_2| = 0$ and $|\vec{I}_3| = M.$ In these cases, 
the estimate \eqref{E:PCOMMUTEDSTRONGPOINTWISE} implies that 
\begin{align} \label{E:TMPRINCIPALUESTIMATEIII}
	|T|_{\newg} & \lesssim 
		\underbrace{\epsilon \sum_{|\vec{I}| = M} \Big| \partial_{\vec{I}}  \newu \Big|_{\newg}}_{\mbox{borderline, principal order}},
	\end{align}
which is clearly $\lesssim$ the right-hand side of \eqref{E:UIPOINTWISEGNORM} as desired.
In the remaining cases of \eqref{E:TMREPTHIRD}, we have $|\vec{I}_3| \leq M - 1.$ In these cases, we allow
the loss $B = c \sqrt{\epsilon}$ in \eqref{E:REPPOINTWISEBOUND}, and the desired bound follows.

\emph{Step 2: Bound for $\Big[\partial_{\vec{I}} , \newg_{ia} \Big] \newu^a.$}
We must bound the $|\cdot|_{\newg}$ norm of
\begin{align}
	T_i & = \Big(\partial_{\vec{I}_1} \newg_{ia} \Big) \Big(\partial_{\vec{I}_2} \newu^a \Big),
\end{align}
where $1 \leq |\vec{I}_1| + |\vec{I}_2| \leq M$ and $|\vec{I}_2| \leq M - 1.$  In all of these cases, we 
simply allow the loss $B = c \sqrt{\epsilon}$ in \eqref{E:REPPOINTWISEBOUND}, and the desired bound follows.

%The estimates  
%\eqref{E:GNORMPARTIALIGANDGINVERSEINTERMSOFGAMMA}, \eqref{E:LITTLEGAMMACOMMUTEDSTRONGPOINTWISE}, 
%and \eqref{E:UCOMMUTEDSTRONGPOINTWISE} imply that

%\begin{align} \label{E:TMCOMMUTATORESTIMATE}
%	|T|_{\newg} & \lesssim 
%		\begin{cases}	\overbrace{\epsilon \sum_{|\vec{I}| \leq M - 1} \Big| \partial_{\vec{I}} \newu \Big|_{\newg}}^{\mbox{borderline, 
%			below principal order}}, 
%			\qquad \mbox{when \ } |\vec{I}_1| \leq N - 3, \\
%			\sqrt{\epsilon} t^{- c \sqrt{\epsilon}} \sum_{|\vec{I}| \leq M} 
%			\Big| \partial_{\vec{I}_2} \newu \Big|_{\newg}, 
%			\qquad \mbox{when \ } |\vec{I}_2| \leq N - 4.
%		\end{cases}
%\end{align}
%Clearly, the right-hand side of \eqref{E:TMCOMMUTATORESTIMATE} is $\lesssim$ the right-hand side of \eqref{E:UIPOINTWISEGNORM} as desired.

\end{proof}

\section{Sobolev Bounds for the Lapse and for the Inhomogeneous Terms} \label{S:LAPSEANDINHOMEINTEGRALESTIMATES}

In this section, we use elliptic estimates to bound $\| \newlap\|_{H^M}$ and $ \| \dlap \|_{H_{\newg}^M}$ 
in terms of the total energies $\totalenergy{\smallparameter_*}{M}$
and $ \sum_{1 \leq |\vec{I}| \leq M }\| |\partial_{\vec{I}} \newg |_{\newg} \|_{L^2} + \| |\partial_{\vec{I}} \newg^{-1} |_{\newg} \|_{L^2}.$
We also derive similar bounds for the $L^2$ norms of the inhomogeneous terms 
$\leftexp{(\vec{I});(Junk)}{\mathfrak{G}},$ 
$\leftexp{(\vec{I});(Junk)}{\widetilde{\mathfrak{G}}},$ 
$\leftexp{(\vec{I});(Junk)}{\mathfrak{K}},$ 
etc., that appear in the PDEs verified by the differentiated quantities
$\partial_{\vec{I}} \newg,$ 
$\partial_{\vec{I}} \newg^{-1},$
$\partial_{\vec{I}} \freenewsec,$ 
$\partial_{\vec{I}} \upgamma,$ 
$\partial_{\vec{I}} \newlap,$
$\partial_{\vec{I}} \dlap,$
$\partial_{\vec{I}} \newp,$ 
and $\partial_{\vec{I}} \newu.$ 
In particular, these $L^2$ bounds will allow us to control
the \emph{spatial} integrals involving the inhomogeneous terms $\leftexp{(\vec{I});(Junk)}{\mathfrak{K}},$ etc.,
appearing in the fundamental energy integral inequality \eqref{E:FUNDAMENTALENERGYINEQUALITY}. 
All of the estimates will follow easily thanks to the pointwise bounds of Prop.~\ref{P:POINTWISEESTIMATES}.

\subsection{Sobolev bounds for the lapse variables and the corresponding inhomogeneous terms} \label{SS:SOBOLEVBOUNDSFORLAPSE}

In the next proposition, we provide the aforementioned Sobolev bounds for the lapse variables and the corresponding
inhomogeneous terms appearing in the $\partial_{\vec{I}}-$commuted lapse equations.

\begin{proposition} [\textbf{Sobolev bounds for the lapse in terms of the total energies and the metric norms}]
\label{P:LAPSEINTERMSOFENERGY}
Assume that the hypotheses and conclusions of Prop.~\ref{P:STRONGPOINTWISE} hold on the spacetime slab $(T,1] \times \mathbb{T}^3.$
Then there exist a small constant $\upsigma_N > 0$ and a large constant $c > 0$  
such that if $\epsilon \leq \upsigma \leq \upsigma_N,$
then the following $L^2$ estimates for the inhomogeneous terms hold on $(T,1]:$
	\begin{subequations}
	\begin{align}
	\sum_{|\vec{I}| \leq M}	\left\| \leftexp{(\vec{I});(Border)}{\mathfrak{N}} \right\|_{L^2}
		& \lesssim \sqrt{\epsilon} \totalenergy{\smallparameter_*}{M} + \sqrt{\epsilon} t^{- c \sqrt{\epsilon}} \totalenergy{\smallparameter_*}{M-1}, && (M \leq N),
		 \label{E:ICOMMUTEDINHOMOGENEOUSLAPSEBORDER} \\
	\sum_{|\vec{I}| \leq M} \left\| \leftexp{(\vec{I});(Junk)}{\mathfrak{N}} \right\|_{L^2} 
		& \lesssim \sqrt{\epsilon} t^{-2/3 - c \sqrt{\epsilon}} \totalenergy{\smallparameter_*}{M} &&
			\label{E:ICOMMUTEDINHOMOGENEOUSLAPSEJUNK}  
			\\
		& \ \ + \sqrt{\epsilon} \| \newlap \|_{H_{\newg}^{M-1}}
			+ \sqrt{\epsilon} \| \dlap \|_{H_{\newg}^M} &&  
			\notag \\
		& \ \ + \underbrace{\sqrt{\epsilon} t^{- c \sqrt{\epsilon}} \sum_{1 \leq |\vec{I}| \leq M} 
			\left\| \Big| \partial_{\vec{I}} \newg \Big|_{\newg} \right\|_{L^2}}_{\mbox{absent if $M=0$}} 
			+ \underbrace{\sqrt{\epsilon} t^{- c \sqrt{\epsilon}} \sum_{1 \leq |\vec{I}| \leq M} 
				\left\| \Big|\partial_{\vec{I}} \newg^{-1} \Big|_{\newg} \right\|_{L^2}}_{\mbox{absent if $M=0$}}, && (M \leq N),  
			\notag \\
	\sum_{|\vec{I}| \leq M}	\left\| \leftexp{(\vec{I});(Border)}{\widetilde{\mathfrak{N}}} \right\|_{L^2} 
	 	& \lesssim t^{-2/3} \totalenergy{\smallparameter_*}{M+1}
	 		+ \sqrt{\epsilon} t^{-2/3 - c \sqrt{\epsilon}} \totalenergy{\smallparameter_*}{M}, && (M \leq N - 1),
	 		\label{E:ICOMMUTEDALTERNATEINHOMOGENEOUSLAPSEBORDER2} 
	 		\\
	\sum_{|\vec{I}| \leq N-1}	 \left\| \leftexp{(\vec{I});(Junk)}{\widetilde{\mathfrak{N}}} \right\|_{L^2}
		& \lesssim \sqrt{\epsilon} t^{-2/3 - c \sqrt{\epsilon}} \totalenergy{\smallparameter_*}{M}
			+ \sqrt{\epsilon} t^{- c \sqrt{\epsilon}} \| \dlap \|_{H_{\newg}^M}, && (M \leq N - 1). 
		\label{E:ICOMMUTEDALTERNATEINHOMOGENEOUSLAPSEJUNK2}
	\end{align}
	\end{subequations}
	
	Furthermore, the following Sobolev estimates for $\newlap,$ $\dlap,$ and $\partial \dlap$ also hold on $(T,1]:$
	\begin{subequations}
	\begin{align}
		\| \newlap \|_{H^M} + \| \dlap \|_{H_{\newg}^M} + t^{2/3} \| \partial \dlap \|_{H_{\newg}^M}
		& \lesssim t^{-4/3} \totalenergy{\smallparameter_*}{M} 
			+ \underbrace{\sqrt{\epsilon} t^{-4/3 - c \sqrt{\epsilon}} \totalenergy{\smallparameter_*}{M-1}}_{\mbox{absent if $M=0$}}
			&&  
			\label{E:TOPLAPSESOBOLEVBOUND} \\ 
		& \ \ + \underbrace{\sqrt{\epsilon} t^{- c \sqrt{\epsilon}} \sum_{1 \leq |\vec{I}| \leq M} 
				\left\| \Big|\partial_{\vec{I}} \newg \Big|_{\newg} \right\|_{L^2}}_{\mbox{absent if $M=0$}} 
			+ \underbrace{\sqrt{\epsilon} t^{- c \sqrt{\epsilon}} \sum_{1 \leq |\vec{I}| \leq M} 
				\left\| \Big|\partial_{\vec{I}} \newg^{-1} \Big|_{\newg} \right\|_{L^2}}_{\mbox{absent if $M=0$}}, && (M \leq N), 
				\notag \\
		\| \newlap \|_{H^M} + \| \dlap \|_{H_{\newg}^M}
		& \lesssim  t^{-2/3} \totalenergy{\smallparameter_*}{M+1} 
				+ \sqrt{\epsilon} t^{-2/3 - c \sqrt{\epsilon}} \totalenergy{\smallparameter_*}{M} && 
			\label{E:LOWERLAPSEGOODSOBOLEVBOUND}  \\
		& \ \ + \underbrace{\sqrt{\epsilon} t^{- c \sqrt{\epsilon}} \sum_{1 \leq |\vec{I}| \leq M} 
			\left\| \Big|\partial_{\vec{I}} \newg \Big|_{\newg} \right\|_{L^2}}_{\mbox{absent if $M=0$}} 
			+ \underbrace{\sqrt{\epsilon} t^{- c \sqrt{\epsilon}} \sum_{1 \leq |\vec{I}| \leq M} 
				\left\| \Big|\partial_{\vec{I}} \newg^{-1} \Big|_{\newg} \right\|_{L^2}}_{\mbox{absent if $M=0$}}, && (M \leq N - 1).
			\notag
	\end{align}
	\end{subequations}
	
\end{proposition}

\begin{remark}
	Some of the above powers of $\epsilon$ stated above are non-optimal. 
\end{remark}

\begin{proof}
	We first note that the estimates \eqref{E:PCOMMUTEDSTRONGPOINTWISE} and \eqref{E:UCOMMUTEDSTRONGPOINTWISE}
	imply that the hypotheses of Lemma~\ref{L:COERCIVITYOFTHEENERGIES} are verified when $\epsilon$ is sufficiently small.
	We now square both sides of the inequalities \eqref{E:NIJUNK}-\eqref{E:WIDETILDENIPOINTWISEBORDER}
	and then integrate over $\mathbb{T}^3.$	With the help of Lemma~\ref{L:COERCIVITYOFTHEENERGIES}, we 
	bound the integrals corresponding to $\leftexp{(\vec{I});(Border)}{\mathfrak{N}}$
	(i.e., the integrals corresponding to the terms on the right-hand side of \eqref{E:NIBORDER}) 
	by
	\begin{align} \label{E:NIBORDERINTEGRALS}
		\sum_{|\vec{I}| \leq M} \left\| \leftexp{(\vec{I});(Border)}{\mathfrak{N}} \right\|_{L^2}^2 
		& \lesssim \sum_{|\vec{I}| \leq M} 
			\epsilon \int_{\Sigma_t} \left|\partial_{\vec{I}} \freenewsec \right|_{\newg}^2 \, dx
		+ \sum_{|\vec{J}| \leq M - 1} 
			\epsilon t^{-c \epsilon} 
			\int_{\Sigma_t} \left|\partial_{\vec{J}} \freenewsec \right|_{\newg}^2 \, dx  \\
		& \lesssim \epsilon \totalenergy{\smallparameter_*}{M}^2 + \epsilon t^{-c \epsilon} \totalenergy{\smallparameter_*}{M-1}^2, \notag
	\end{align}
	which implies the desired estimate \eqref{E:ICOMMUTEDINHOMOGENEOUSLAPSEBORDER}. 
	Inequalities \eqref{E:ICOMMUTEDINHOMOGENEOUSLAPSEJUNK}-\eqref{E:LOWERLAPSEGOODSOBOLEVBOUND}
	can be proved in a similar fashion; we omit the straightforward details.

	Inequalities \eqref{E:TOPLAPSESOBOLEVBOUND}-\eqref{E:LOWERLAPSEGOODSOBOLEVBOUND}
	then follow from Lemma~\ref{L:COERCIVITYOFTHEENERGIES},
	the estimates \eqref{E:LAPSEICOMMUTEDINITIALELLIPTICESTIMATE}-\eqref{E:ALTERNATELAPSEICOMMUTEDINITIALELLIPTICESTIMATE}, and
	the estimates \eqref{E:ICOMMUTEDINHOMOGENEOUSLAPSEBORDER}-\eqref{E:ICOMMUTEDALTERNATEINHOMOGENEOUSLAPSEJUNK2}.
	We remark that in deriving these bounds,
	we soaked the $\sqrt{\epsilon} \| \newlap \|_{H^{M-1}}$ and $\sqrt{\epsilon} \| \dlap \|_{H_{\newg}^M}$ terms on the right-hand sides of 	
	\eqref{E:ICOMMUTEDINHOMOGENEOUSLAPSEJUNK} and \eqref{E:ICOMMUTEDALTERNATEINHOMOGENEOUSLAPSEJUNK2}
	into the left-hand sides of \eqref{E:TOPLAPSESOBOLEVBOUND}-\eqref{E:LOWERLAPSEGOODSOBOLEVBOUND} 
	(this is possible when $\epsilon$ is sufficiently small).
	
\end{proof}

\subsection{Sobolev bounds for the metric and fluid inhomogeneous terms}

In the next proposition, we extend the analysis of
Prop.~\ref{P:LAPSEINTERMSOFENERGY} to the non-lapse variables. 
Specifically, we bound the $L^2$ norms of the inhomogeneous terms 
$\leftexp{(\vec{I});(Junk)}{\mathfrak{G}},$ 
$\leftexp{(\vec{I});(Junk)}{\widetilde{\mathfrak{G}}},$ etc.,
that appear in the PDEs verified by the differentiated quantities
$\partial_{\vec{I}} \newg,$ 
$\partial_{\vec{I}} \newg^{-1},$
$\partial_{\vec{I}} \freenewsec,$ 
$\partial_{\vec{I}} \upgamma,$ 
$\partial_{\vec{I}} \newp,$ 
and $\partial_{\vec{I}} \newu.$ 
As in Prop.~\ref{P:LAPSEINTERMSOFENERGY}, the right-hand sides of the bounds feature 
the total energies $\totalenergy{\smallparameter_*}{M}$
and $ \sum_{1 \leq |\vec{I}| \leq M }\| |\partial_{\vec{I}} \newg |_{\newg} \|_{L^2} + \| |\partial_{\vec{I}} \newg^{-1} |_{\newg} \|_{L^2}.$

\begin{proposition} \label{P:INHOMOGENEOUSTERMSOBOLEVESTIMATES}
	Assume that the hypotheses and conclusions of Prop.~\ref{P:STRONGPOINTWISE} hold on the spacetime slab $(T,1] \times \mathbb{T}^3.$ 
	Let $\leftexp{(\vec{I});(Junk)}{\mathfrak{M}},$ $\leftexp{(\vec{I});(Border)}{\mathfrak{M}},$ etc.
	be the inhomogeneous terms appearing in the $\partial_{\vec{I}}-$commuted equations
	\eqref{AE:MOMENTUMCONSTRAINTCOMMUTED}-\eqref{AE:RAISEDMOMENTUMCONSTRAINTCOMMUTED},
	\eqref{AE:GEVOLUTIONCOMMUTED}-\eqref{AE:GINVERSEEVOLUTIONCOMMUTED}, 
	\eqref{AE:METRICGAMMACOMMUTED}-\eqref{AE:SECONDFUNDCOMMUTED},
	and \eqref{AE:PARTIALTPCOMMUTED}-\eqref{AE:PARTIALTUJCOMMUTED}.	
	Then there exist a small constant $\upsigma_N > 0$ and a large constant $c > 0$
	such that if $\epsilon \leq \upsigma \leq \upsigma_N,$
	then the following $L^2$ estimates hold on the interval $(T,1]$ for $1 \leq M \leq N:$
	\begin{subequations}
	\begin{align}
	\sum_{1 \leq |\vec{I}| \leq M} \left\| \Big| \leftexp{(\vec{I});(Junk)}{\mathfrak{G}} \Big|_{\newg} \right \|_{L^2}^2
	& \lesssim t^{- 4/3} \totalenergy{\smallparameter_*}{M}^2
		+ \epsilon t^{- 4/3 - c \sqrt{\epsilon}} \totalenergy{\smallparameter_*}{M-1}^2
		+ \epsilon t^{- c \sqrt{\epsilon}} \sum_{1 \leq |\vec{I}| \leq M} 
			\left\| \Big|\partial_{\vec{I}} \newg \Big|_{\newg} \right\|_{L^2}^2,
			\label{E:JUNKGSOBOLEV}  \\
	\sum_{1 \leq |\vec{I}| \leq M} \left\| \Big| \leftexp{(\vec{I});(Junk)}{\widetilde{\mathfrak{G}}} \Big|_{\newg} \right\|_{L^2}^2 
		&\lesssim t^{- 4/3} \totalenergy{\smallparameter_*}{M}^2
			+ \epsilon t^{- 4/3 - c \sqrt{\epsilon}} \totalenergy{\smallparameter_*}{M-1}^2
			+ \epsilon t^{- c \sqrt{\epsilon}} \sum_{1 \leq |\vec{I}| \leq M} 
				\left\| \Big|\partial_{\vec{I}} \newg^{-1} \Big|_{\newg} \right\|_{L^2}^2,
			\label{E:JUNKGINVERSESOBOLEV}
	\end{align}
	\end{subequations}
	\begin{subequations}
	\begin{align}
	\sum_{1 \leq |\vec{I}| \leq M} \left\| \Big| \leftexp{(\vec{I});(Border)}{\mathfrak{G}} \Big|_{\newg} \right \|_{L^2}^2
	& \lesssim	\totalenergy{\smallparameter_*}{M}^2
		+ \epsilon t^{- c \sqrt{\epsilon}} \totalenergy{\smallparameter_*}{M-1}^2
		\label{E:BORDERGSOBOLEV} \\
	& \ \ + \epsilon \sum_{|\vec{I}| = M} \left\| \Big|\partial_{\vec{I}} \newg \Big|_{\newg} \right\|_{L^2}^2
		+ \underbrace{\epsilon t^{- c \sqrt{\epsilon}} \sum_{1 \leq |\vec{J}| \leq M - 1} \left\| \Big|\partial_{\vec{I}} \newg \Big|_{\newg} 
			\right\|_{L^2}^2}_{\mbox{absent if $M=1$}},
		\notag \\
	\sum_{1 \leq |\vec{I}| \leq M} \left\| \Big| \leftexp{(\vec{I});(Border)}{\widetilde{\mathfrak{G}}} \Big|_{\newg} \right\|_{L^2}^2 
		& \lesssim \totalenergy{\smallparameter_*}{M}^2
		+ \epsilon t^{- c \sqrt{\epsilon}} \totalenergy{\smallparameter_*}{M-1}^2
		\label{E:BORDERGINVERSESOBOLEV}  \\
	& \ \ + \epsilon \sum_{|\vec{I}| = M} \left\| \Big|\partial_{\vec{I}} \newg^{-1} \Big|_{\newg} \right\|_{L^2}^2
		+ \underbrace{\epsilon t^{- c \sqrt{\epsilon}}  \sum_{1 \leq |\vec{J}| \leq M-1} \left\| \Big|\partial_{\vec{J}} \newg^{-1} \Big|_{\newg} 
			\right\|_{L^2}^2}_{\mbox{absent if $M=1$}}.
		\notag
		\end{align}
	\end{subequations}
		
	In addition, the following $L^2$ estimates hold on the interval $(T,1]$ for $0 \leq M \leq N:$	
	\begin{align}
	& \sum_{|\vec{I}| \leq M} \left \| \Big|\leftexp{(\vec{I});(Junk)}{\mathfrak{M}} \Big|_{\newg} \right \|_{L^2}^2
		+ \left \| \Big|\leftexp{(\vec{I});(Junk)}{\widetilde{\mathfrak{M}}} \Big|_{\newg} \right \|_{L^2}^2 
		+ \left \| \Big| \leftexp{(\vec{I});(Junk)}{\mathfrak{K}} \Big|_{\newg} \right \|_{L^2}^2	&&
			\label{E:JUNKINHOMSOBOLEV}  \\
	& \ \ + \sum_{|\vec{I}| \leq M} \left \| \Big| \leftexp{(\vec{I});(Junk)}{\mathfrak{g}} \Big|_{\newg} \right \|_{L^2}^2
		+ \left\| \leftexp{(\vec{I});(Junk)}{\mathfrak{P}} \right\|_{L^2}^2
		+ \left\| \Big|\leftexp{(\vec{I});(Junk)}{\mathfrak{U}} \Big|_{\newg} \right\|_{L^2}^2 &&
		\notag \\
	& \lesssim \epsilon t^{-4/3 - c \sqrt{\epsilon}} \totalenergy{\smallparameter_*}{M}^2
			+ \underbrace{\epsilon t^{- c \sqrt{\epsilon}} \sum_{1 \leq |\vec{I}| \leq M} 
				\left\| \Big|\partial_{\vec{I}} \newg \Big|_{\newg} \right\|_{L^2}^2}_{\mbox{absent if $M=0$}} 
			+ \underbrace{\epsilon t^{- c \sqrt{\epsilon}} \sum_{1 \leq |\vec{I}| \leq M} 
				\left\| \Big|\partial_{\vec{I}} \newg^{-1} \Big|_{\newg} \right\|_{L^2}^2}_{\mbox{absent if $M=0$}},
		\notag
	\end{align}
	\begin{align} \label{E:BORDERLINESOBOLEVESTIMATES}
		& \sum_{|\vec{I}| \leq M} \left \| \Big|\leftexp{(\vec{I});(Border)}{\mathfrak{M}} \Big|_{\newg} \right \|_{L^2}^2
			+ \left \| \Big|\leftexp{(\vec{I});(Border)}{\widetilde{\mathfrak{M}}} \Big|_{\newg} \right \|_{L^2}^2
			+ \left\| \Big| \leftexp{(\vec{I});(Border)}{\mathfrak{g}} \Big|_{\newg} \right\|_{L^2}^2
			+ \left\| \Big| \leftexp{(\vec{I});(Border)}{\mathfrak{U}} \Big|_{\newg} \right\|_{L^2}^2
			\\
		& \ \ \lesssim \epsilon t^{-4/3} \totalenergy{\smallparameter_*}{M}^2 
			+ \underbrace{\epsilon t^{-4/3 - c \sqrt{\epsilon}} \totalenergy{\smallparameter_*}{M-1}^2}_{\mbox{absent if $M=0$}}
			\notag \\
		& \ \ \ \ + \underbrace{\epsilon t^{- c \sqrt{\epsilon}} \sum_{1 \leq |\vec{I}| \leq M} 
			\left\| \Big|\partial_{\vec{I}} \newg \Big|_{\newg} \right\|_{L^2}^2}_{\mbox{absent if $M=0$}} 
			+ \underbrace{\epsilon t^{- c \sqrt{\epsilon}} \sum_{1 \leq |\vec{I}| \leq M} 
			\left\| \Big|\partial_{\vec{I}} \newg^{-1} \Big|_{\newg} \right\|_{L^2}^2}_{\mbox{absent if $M=0$}}. 
				\notag
	\end{align}

	%\begin{subequations}
	%\begin{align}
	%\sum_{|\vec{I}| \leq M} \left \| \Big|\leftexp{(\vec{I});(Border)}{\mathfrak{M}} \Big|_{\newg} \right \|_{L^2}^2
	%	& \lesssim  \epsilon t^{-4/3} \totalenergy{\smallparameter_*}{M}^2 
	%		+ \epsilon t^{-4/3 - c \sqrt{\epsilon}} \totalenergy{\smallparameter_*}{M-1}^2,
	%		&& (M \leq N),
	%		\label{E:MISOBOLEV} \\
	%	\sum_{|\vec{I}| \leq M} \left \| \Big|\leftexp{(\vec{I});(Border)}{\widetilde{\mathfrak{M}}} \Big|_{\newg} \right \|_{L^2}^2
	%		& \lesssim  \epsilon t^{-4/3} \totalenergy{\smallparameter_*}{M}^2 
	%		+ \epsilon t^{-4/3 - c \sqrt{\epsilon}} \totalenergy{\smallparameter_*}{M-1}^2, 
	%			&& (M \leq N),
	%		\label{E:TILDEMISOBOLEV}	
	%\end{align}	

	%\begin{align}
	%	\sum_{|\vec{I}| \leq M} 
	%		\left\| \Big| \leftexp{(\vec{I});(Border)}{\mathfrak{g}} \Big|_{\newg} \right\|_{L^2}^2
	%	& \lesssim \epsilon t^{-4/3} \totalenergy{\smallparameter_*}{M}^2 + \epsilon t^{-4/3 - c \sqrt{\epsilon}} \totalenergy{\smallparameter_*}{M-1}^2,
	%		&& (M \leq N),
	%		\label{E:GISOBOLEV} 
	%\end{align}	
	
	%\begin{align}
	%	\sum_{|\vec{I}| \leq M} \left\| \Big|\leftexp{(\vec{I});(Border)}{\mathfrak{U}} \Big|_{\newg} \right\|_{L^2}^2
	%	& \lesssim \epsilon t^{-4/3} \totalenergy{\smallparameter_*}{M}^2 +
	%		\epsilon t^{-4/3 - c \sqrt{\epsilon}} \totalenergy{\smallparameter_*}{M-1}^2, && (M \leq N).
	%		\label{E:UISOBOLEV} 
	%\end{align}
	
	%\end{subequations}
	
\end{proposition}

\begin{proof}
The proof is very similar to the proof of Prop.~\ref{P:LAPSEINTERMSOFENERGY}: we square
the inequalities in Prop.~\ref{P:POINTWISEESTIMATES} and integrate over $\mathbb{T}^3.$
For example, to derive the estimates for $\left\| \Big| \leftexp{(\vec{I});(Border)}{\mathfrak{g}} \Big|_{\newg} \right\|_{L^2}^2$
in \eqref{E:BORDERLINESOBOLEVESTIMATES}, we square both sides of inequality \eqref{E:GAMMAIPOINTWISEGNORMBORDER} and integrate over 
$\mathbb{T}^3.$ The integrals involving $\partial_{\vec{I}} \dlap$ 
are bounded with the help of Prop.~\ref{P:LAPSEINTERMSOFENERGY}, while the remaining integrals are controlled by the 
terms $\| |\partial_{\vec{I}} \newg|_{\newg} \|_{L^2}^2$ and 
$\| |\partial_{\vec{I}} \newg^{-1}|_{\newg} \|_{L^2}^2$ on the right-hand side of \eqref{E:BORDERLINESOBOLEVESTIMATES}, 
the definition of the energies $\totalenergy{\smallparameter_*}{M},$ and Lemma~\ref{L:COERCIVITYOFTHEENERGIES}. The remaining estimates in the proposition can be proved in a similar fashion.

\end{proof}

\section{A Comparison of the Sobolev Norms and the Energies} \label{S:COMPARISON}

Our bootstrap assumptions involve the solution norm $\highnorm{N},$ while  
the fundamental energy integral inequality of Sect.~\ref{S:FUNDAMENATLENERGYINEQUALITIES}
will allow us to derive a priori estimates for the quantities $\totalenergy{\smallparameter_*}{M}$
(see the proof of Corollary \ref{C:ENERGYINTEGRALINEQUALITIES}).
These a priori estimates will also involve the quantities
$\sum_{1 \leq |\vec{I}| \leq M} \| | \partial_{\vec{I}} \newg |_{\newg} \|_{L^2}^2$ 
and $\sum_{1 \leq |\vec{I}| \leq M} \| | \partial_{\vec{I}} \newg^{-1} |_{\newg} \|_{L^2}^2.$
In this short section, we establish comparison estimates between these quantities, 
which will ultimately allow us to derive an a priori estimate for $\highnorm{N}.$

\begin{proposition} [\textbf{Comparison of norms and energies}] \label{P:COMPARISON}
	Assume that the hypotheses and conclusions of Prop.~\ref{P:STRONGPOINTWISE} hold on the spacetime slab $(T,1] \times \mathbb{T}^3.$
	Let $\highnorm{M}(t)$ be the solution norm defined in \eqref{E:HIGHNORM} and let
	$\totalenergy{\smallparameter_*}{M}$ be the total solution energy defined by \eqref{E:TOTALENERGY}, where the choice of 
	the constant $\smallparameter_* > 0$ was made in Prop.~\ref{P:FUNDAMENTALENERGYINEQUALITY}.
	Then there exist a small constant $\upsigma_N > 0$ 
	and large constants $C_M, c_M > 0$
	such that if $\epsilon \leq \upsigma \leq \upsigma_N,$
	then the following comparison estimates hold for $t \in (T,1]:$
	\begin{subequations}
	\begin{align}
		\totalenergy{\smallparameter_*}{M}(t) & \leq C_M t^{- c_M \sqrt{\epsilon}} \highnorm{M}(t), && (0 \leq M \leq N),
			\label{E:ENERGYBOUNDEDBYNORM} \\
		\highnorm{M}(t) & \leq 
			\underbrace{C_M t^{- c_M \sqrt{\epsilon}} \sum_{1 \leq |\vec{I}| \leq M} \left\| \left| \partial_{\vec{I}} \newg \right|_{\newg} \right\|_{L^2}(t)
				}_{\mbox{absent if $M=0$}} 
			+ \underbrace{C_M t^{- c_M \sqrt{\epsilon}} \sum_{1 \leq |\vec{I}| \leq M}+ \left\| \left| \partial_{\vec{I}} \newg^{-1} \right|_{\newg}\right\|_{
				L^2}(t)}_{\mbox{absent if $M=0$}} &&
				\label{E:NORMBOUNDEDBYENERGY} \\
		& \ \ + C_M t^{- c_M \sqrt{\epsilon}} \totalenergy{\smallparameter_*}{M}(t), && (0 \leq M \leq N).
			\notag 			
	\end{align}
	\end{subequations}
\end{proposition}

\begin{proof}
	Inequality \eqref{E:ENERGYBOUNDEDBYNORM} follows from 
	Lemma~\ref{L:COERCIVITYOFTHEENERGIES}, the strong estimates
	\eqref{E:GCOMMUTEDSTRONGPOINTWISE}, \eqref{E:GINVERSECOMMUTEDSTRONGPOINTWISE},
	\eqref{E:PCOMMUTEDSTRONGPOINTWISE}, and \eqref{E:UCOMMUTEDSTRONGPOINTWISE},
	and the $\newg-$Cauchy-Schwarz inequality. Inequality \eqref{E:NORMBOUNDEDBYENERGY} 
	follows similarly with the help of the estimates \eqref{E:TOPLAPSESOBOLEVBOUND}-\eqref{E:LOWERLAPSEGOODSOBOLEVBOUND}.
	%Finally, the terms $\lowkinnorm{M-1}(t) + \lowpotnorm{M-1}(t)$ were bounded by the right-hand side of \eqref{E:NORMBOUNDEDBYENERGY} in
	%Prop.~\ref{P:VOLUMEFORMINTERMSOFLAPSE} and Prop.~\ref{P:STRONGPOINTWISE}.
\end{proof}

\section{The Fundamental A Priori Estimates for the Solution Norms} \label{S:FUNDAMENTALAPRIORI}
In this section, we use all of the results obtained thus far in order to
derive a priori estimates for the solution norms
$\highnorm{N},$ $\lowkinnorm{N-3},$ and $\lowpotnorm{N-4}$ ($N \geq 8$) from Def.~\ref{D:NORMS}.
This is the main step in our proof of stable singularity formation.
Our estimates will require the assumption that the data are near-FLRW.
As preliminary steps, in the first two propositions of this section, we derive
hierarchies of time-integral inequalities for the quantities 
$\left\| \left| \partial_{\vec{I}} \newg \right|_{\newg} \right\|_{L^2}$
and $\left\| \left| \partial_{\vec{I}} \newg^{-1} \right|_{\newg} \right\|_{L^2}$
and for the total energies $\totalenergy{\smallparameter_*}{M}$ 
from Def.~\ref{D:TOTALENERGY}.
The important point is that these hierarchies are amenable to analysis via Gronwall's inequality.
We then use these hierarchies to derive the main a priori norm estimates in Corollary \ref{C:ENERGYINTEGRALINEQUALITIES}.

In the next proposition, we derive the first hierarchy.

\begin{proposition} [\textbf{A hierarchy of integral inequalities for $\left\| \left| \partial_{\vec{I}} \newg \right|_{\newg} \right\|_{L^2}(t)$ and $\left\| \left| \partial_{\vec{I}} \newg^{-1} \right|_{\newg} \right\|_{L^2}(t)$}] \label{P:SOBFORG}
Assume that the hypotheses and conclusions of Prop.~\ref{P:STRONGPOINTWISE} hold on the spacetime slab $(T,1] \times \mathbb{T}^3.$
Then there exist a small constant $\upsigma_* > 0$
and large constants $C_N, c_N > 0$ 
such that if $\epsilon \leq \upsigma \leq \upsigma_*,$
then the following hierarchy of integral inequalities is verified for $t \in (T,1]$ whenever 
$1 \leq M \leq N:$
	\begin{align} \label{E:GCOMMUTEDSOB}
		\sum_{1 \leq |\vec{I}| \leq M} \left\| \left| \partial_{\vec{I}} \newg \right|_{\newg} \right\|_{L^2}^2(t)
		& \leq \sum_{1 \leq |\vec{I}| \leq M} \left\| \left| \partial_{\vec{I}} \newg \right|_{\newg} \right\|_{L^2}^2(1)
			\\
		& \ \ + C_N \sqrt{\epsilon} \sum_{1 \leq |\vec{I}| \leq M} \int_{s=t}^1 s^{-1} 
				\left\| \left| \partial_{\vec{I}} \newg \right|_{\newg} \right\|_{L^2}^2(s) \, ds
				\notag \\
		& \ \ + \underbrace{C_N \sqrt{\epsilon}  \sum_{1 \leq |\vec{J}| \leq M - 1} \int_{s=t}^1 s^{-1 - c_N \sqrt{\epsilon}} 
				\left\| \left| \partial_{\vec{J}} \newg \right|_{\newg} \right\|_{L^2}^2(s) \, ds}_{\mbox{absent if $M=1$}}	
			\notag \\
		%& \ \ + C_N \sum_{1 \leq |\vec{I}| \leq M} \int_{s=t}^1 s^{- c \sqrt{\epsilon}} 
		%		\left\| \left| \partial_{\vec{I}} \newg \right|_{\newg} \right\|_{L^2}^2(s) \, ds
		%	\notag \\
		& \ \ + \frac{C_N}{\sqrt{\epsilon}} \int_{s=t}^1 s^{-1} \totalenergy{\smallparameter_*}{M}^2(s) \, ds
			+ C_N \sqrt{\epsilon} \int_{s=t}^1 s^{-1 - c_N \sqrt{\epsilon}} \totalenergy{\smallparameter_*}{M-1}^2(s) \, ds.
			\notag
	\end{align}
	Furthermore, the same estimates hold if we replace $\left| \partial_{\vec{I}} \newg \right|_{\newg}$ and $\left| 
	\partial_{\vec{J}} \newg \right|_{\newg}$ respectively with
	$\left| \partial_{\vec{I}} \newg^{-1} \right|_{\newg}$ and $\left| \partial_{\vec{J}} \newg^{-1} \right|_{\newg}$
	in \eqref{E:GCOMMUTEDSOB}.		
	%,
	%		\notag \\
	%	\sum_{1 \leq |\vec{I}| \leq M} \left\| \left| \partial_{\vec{I}} \newg^{-1} \right|_{\newg} \right\|_{L^2}^2(t)
	%	& \lesssim \sum_{1 \leq |\vec{I}| \leq M} \left\| \left| \partial_{\vec{I}} \newg^{-1} \right|_{\newg} \right\|_{L^2}^2(1)
	%		+ \totalenergy{\smallparameter_*}{M}^2(t) 
	%		+ t^{- c \sqrt{\epsilon}} \totalenergy{\smallparameter_*}{M-1}^2(t) \\
	%	& \ \ + \epsilon \sum_{1 \leq |\vec{I}| \leq M} \int_{s=t}^1 s^{-1} \left\| \left| \partial_{\vec{I}} \newg^{-1} \right|_{\newg} \right\|_{L^2}^2(s) \, ds
	%		+ \underbrace{\epsilon \sum_{1 \leq |\vec{J}| \leq M - 1} \int_{s=t}^1 s^{-1 - c \sqrt{\epsilon}} 
	%			\left\| \left| \partial_{\vec{J}} \newg^{-1} \right|_{\newg} \right\|_{L^2}^2(t) \, ds}_{\mbox{absent if $M=1$}}
		
\end{proposition}

\begin{proof}
	We first use the evolution equation \eqref{E:GINVERSEEVOLUTION} for $\partial_t (\newg^{-1})^{ij}$
	and the $\newg-$Cauchy-Schwarz inequality to deduce the estimate
	\begin{align} \label{E:PARTIALTPARTIALIGNORMINITIALPOINTWISE}
		\left|\partial_t \left(|\partial_{\vec{I}} \newg|_{\newg}^2 \right) \right|
		& \leq 4 t^{-1} |\freenewsec|_{\newg} |\partial_{\vec{I}} \newg|_{\newg}^2
			+ 4 t^{1/3} |\newlap| |\freenewsec|_{\newg} |\partial_{\vec{I}} \newg|_{\newg}^2
			+ \frac{4}{3} t^{1/3} |\newlap| |\partial_{\vec{I}} \newg|_{\newg}^2
			+ 2 |\partial_t \partial_{\vec{I}} \newg|_{\newg} |\partial_{\vec{I}} \newg|_{\newg}.
	\end{align}
	We next integrate the quantity $\partial_t \left(|\partial_{\vec{I}} \newg|_{\newg}^2 \right)$ 
	over the spacetime slab $[t,1] \times \mathbb{T}^3$ 
	and sum over $1 \leq |\vec{I}| \leq M,$
	which yields the difference 
	$\sum_{1 \leq |\vec{I}| \leq M} \left\| \left| \partial_{\vec{I}} \newg \right|_{\newg} \right\|_{L^2}^2(t) - 
	\sum_{1 \leq |\vec{I}| \leq M} \left\| \left| \partial_{\vec{I}} \newg \right|_{\newg} \right\|_{L^2}^2(1).$
	
	To complete the proof of \eqref{E:GCOMMUTEDSOB}, we have to bound the integral of the right-hand
	side of \eqref{E:PARTIALTPARTIALIGNORMINITIALPOINTWISE} over the spacetime slab $[t,1] \times \mathbb{T}^3$
	by the integrals on the right-hand side of \eqref{E:GCOMMUTEDSOB}. 
	To this end, we first pointwise bound the coefficients of $|\partial_{\vec{I}} \newg|_{\newg}^2$ 
	in the first three products on the right-hand side of \eqref{E:PARTIALTPARTIALIGNORMINITIALPOINTWISE}
	by using the strong estimates \eqref{E:SECONDFUNDUPGRADEPOINTWISEGNORM} and \eqref{E:LAPSECOMMUTEDSTRONGPOINTWISE}.
	We then integrate the first three products on the right-hand side of \eqref{E:PARTIALTPARTIALIGNORMINITIALPOINTWISE} over $[t,1] \times \mathbb{T}^3.$ 
	We have thus bounded these spacetime integrals by the right-hand side of \eqref{E:GCOMMUTEDSOB}
	as desired. 
	
	It remains to bound the spacetime integral of the product $2 |\partial_t \partial_{\vec{I}} \newg|_{\newg} |\partial_{\vec{I}} \newg|_{\newg}$
	from the right-hand side of \eqref{E:PARTIALTPARTIALIGNORMINITIALPOINTWISE}
	by the right-hand side of \eqref{E:GCOMMUTEDSOB}. To derive such a bound, we first 
	use the following simple inequality:
	\begin{align} \label{E:ANNOYINGCAUCHYSCHWARZ}
		2 |\partial_t \partial_{\vec{I}} \newg|_{\newg} |\partial_{\vec{I}} \newg|_{\newg}
		& \leq \frac{1}{\sqrt{\epsilon}} t |\partial_t \partial_{\vec{I}} \newg|_{\newg}^2 
			+ \sqrt{\epsilon} t^{-1} |\partial_{\vec{I}} \newg|_{\newg}^2.
	\end{align}
	The spacetime integral of the second term on the right-hand side of \eqref{E:ANNOYINGCAUCHYSCHWARZ} is 
	manifestly bounded by the right-hand side of \eqref{E:GCOMMUTEDSOB}. To bound the spacetime integral of the
	first term on the right-hand side of \eqref{E:ANNOYINGCAUCHYSCHWARZ}, we use the evolution equation
	\eqref{AE:GEVOLUTIONCOMMUTED} to replace $\partial_t \partial_{\vec{I}} \newg$ with inhomogeneous terms 
	and then bound the integrals of the inhomogeneous terms by using the estimates \eqref{E:JUNKGSOBOLEV} and \eqref{E:BORDERGSOBOLEV}.
	
 	With the help of the evolution equation \eqref{AE:GINVERSEEVOLUTIONCOMMUTED} for $\partial_{\vec{I}} \newg^{-1},$
	we can use a similar argument to derive analogous estimates with $\newg^{-1}$ in place of $\newg.$
\end{proof}

We now derive the hierarchy of integral inequalities for the total energies.

\begin{proposition} [\textbf{A hierarchy of integral inequalities for $\totalenergy{\smallparameter_*}{M}(t)$}] \label{P:ENERGYINTEGRALINEQUALITIES}
Assume that the hypotheses and conclusions of Prop.~\ref{P:STRONGPOINTWISE} hold on the spacetime slab $(T,1] \times \mathbb{T}^3.$
Let $\totalenergy{\smallparameter_*}{M}$ be the total solution energy defined by \eqref{E:TOTALENERGY}, 
where the choice of the constant $\smallparameter_* > 0$ was made in Prop.~\ref{P:FUNDAMENTALENERGYINEQUALITY}.
Then there exist a small constant $\upsigma_* > 0$ 
and large constants $C_N, c_N > 0$ depending on $N$ 
such that if $\epsilon \leq \upsigma \leq \upsigma_*,$ 
then the following hierarchy of integral inequalities is verified for $t \in (T,1]$ whenever $0 \leq M \leq N:$
	\begin{align} \label{E:ENERGYINTEGRALINEQUALITIES}
		\totalenergy{\smallparameter_*}{M}^2(t) 
		& \leq C_N \highnorm{M}^2(1)  
			+ \underbrace{C_N \epsilon \sum_{1 \leq |\vec{I}| \leq M} \int_{s=t}^1 s^{1/3 - c_N \sqrt{\epsilon}} 
				\left\| \left| \partial_{\vec{I}} \newg \right|_{\newg} \right\|_{L^2}^2(s) \, ds}_{\mbox{absent if $M=0$}}
				\\
		& \ \ + \underbrace{C_N \epsilon \sum_{1 \leq |\vec{I}| \leq M} \int_{s=t}^1 s^{1/3 - c_N \sqrt{\epsilon}} 
				\left\| \left| \partial_{\vec{I}} \newg^{-1} \right|_{\newg} \right\|_{L^2}^2(s) \, ds}_{\mbox{absent if $M=0$}}
			\notag \\
		& \ \ + C_N \int_{s=t}^{s=1} s^{-1/3} \totalenergy{\smallparameter_*}{M}^2(s) \, ds 
			\notag \\
		& \ \ + C_N \epsilon \int_{s=t}^{s=1} s^{-1} \totalenergy{\smallparameter_*}{M}^2(s) \, ds  
	 		+ \underbrace{C_N \epsilon \int_{s=t}^{s=1} s^{-1 - c_N \sqrt{\epsilon}} \totalenergy{\smallparameter_*}{M-1}^2(s) \, 
	 		ds}_{\mbox{absent if $M = 0$}}.
	 		\notag 
	 	\end{align}
\end{proposition}

	%& \ \ + \underbrace{C_N \epsilon \sum_{1 \leq |\vec{I}| \leq M} \int_{s=t}^{s=1}  
	 %			s^{1/3 - c_N \sqrt{\epsilon}} \left\| \partial_{\vec{I}} \newg \right\|_{L_{\newg}^2}^2(s) \, ds}_{\mbox{absent if $M=0$}}
	 	%	+ \underbrace{C_N \epsilon \sum_{1 \leq |\vec{I}| \leq M} \int_{s=t}^{s=1}  
	 	%		s^{1/3 - c_N \sqrt{\epsilon}} \left\| \partial_{\vec{I}} \newg^{-1} \right\|_{L_{\newg}^2}^2(s) \, ds}_{\mbox{absent if $M=0$}}
	 	%		\notag \\

\begin{proof}
	To derive inequality \eqref{E:ENERGYINTEGRALINEQUALITIES}, we have to bound the integrals on the right-hand side of
	inequality \eqref{E:FUNDAMENTALENERGYINEQUALITY} by the right-hand side of \eqref{E:ENERGYINTEGRALINEQUALITIES}.
	To this end, we first use the estimates of Prop.~\ref{P:LAPSEINTERMSOFENERGY}
	and Prop.~\ref{P:INHOMOGENEOUSTERMSOBOLEVESTIMATES} to bound the inhomogeneous term
	integrals such as $\int_{s = t}^{s = 1} s^{5/3}  \left\| \leftexp{(\vec{I});(Junk)}{\mathfrak{N}} \right\|_{L^2}^2 \, ds,$ 
	$\int_{s = t}^{s = 1} s^{-1} \left\| \leftexp{(\vec{I});(Border)}{\mathfrak{N}} \right\|_{L^2}^2 \, ds,$ etc.
	on the right-hand side of inequality \eqref{E:FUNDAMENTALENERGYINEQUALITY} and make use of the bound
	$\totalenergy{\smallparameter_*}{M}^2(1) \leq C_M \highnorm{M}^2(1)$ implied by Prop.~\ref{P:COMPARISON}. 
	Some of these bounds result in the presence of some spacetime integrals with 
	the integrands $s^{- c \sqrt{\epsilon}} \left\| \left| \partial_{\vec{I}} \newg \right|_{\newg} \right\|_{L^2}^2$ or
	$s^{- c \sqrt{\epsilon}} \left\| \left| \partial_{\vec{I}} \newg^{-1} \right|_{\newg} \right\|_{L^2}^2,$
	where the loss of $s^{- c \sqrt{\epsilon}}$ comes from the right-hand sides of
	the estimates of Prop.~\ref{P:LAPSEINTERMSOFENERGY}
	and Prop.~\ref{P:INHOMOGENEOUSTERMSOBOLEVESTIMATES}.
	By inspection of the right-hand side of inequality \eqref{E:FUNDAMENTALENERGYINEQUALITY}, 
	we see that every time such an integrand appears, it is 
	multiplied by an extra factor of $s^{1/3}$ or better
	(this is clearly true for all integrands except for
	$s^{-1} \left\| \leftexp{(\vec{I});(Border)}{\mathfrak{N}} \right\|_{L^2}^2,$
	which by \eqref{E:ICOMMUTEDINHOMOGENEOUSLAPSEBORDER} can be bounded independently of
	$\left\| \left| \partial_{\vec{I}} \newg \right|_{\newg} \right\|_{L^2}^2$ and
	$\left\| \left| \partial_{\vec{I}} \newg^{-1} \right|_{\newg} \right\|_{L^2}^2$).
	Hence, such spacetime integrals can be bounded by the first two integrals 
	on the right-hand side of \eqref{E:ENERGYINTEGRALINEQUALITIES}.
	
	The remaining spacetime integrals on the right-hand side of inequality \eqref{E:FUNDAMENTALENERGYINEQUALITY} can be bounded 
	by the integrals $C_N \int_{s=t}^{s=1} s^{-1/3} \totalenergy{\smallparameter_*}{M}^2(s) \, ds$
	and $C_N \epsilon \int_{s=t}^{s=1} s^{-1} \totalenergy{\smallparameter_*}{M}^2(s) \, ds$
	on the right-hand side of \eqref{E:ENERGYINTEGRALINEQUALITIES} 
	with the help of Lemma~\ref{L:COERCIVITYOFTHEENERGIES}.
\end{proof}

We now use Prop.~\ref{P:SOBFORG} and Prop.~\ref{P:ENERGYINTEGRALINEQUALITIES} 
to derive our main a priori estimates for the solution norms. We will use
the following Gronwall lemma to estimate solutions to the hierarchies of integral inequalities
that appear in the propositions.

\begin{lemma} [\textbf{A Gronwall estimate}] \label{L:GRONWALLSYSTEM}
Let $y(t) \geq 0, z(t) \geq 0$ be continuous functions.
Suppose that there exist constants $C,c>0$ such that the following inequalities hold for $t \in (T,1]:$	
\begin{subequations}	
\begin{align} 
	y(t) & \leq C \epsilon^3 t^{- c \sqrt{\epsilon}}
		+ c \sqrt{\epsilon} \int_{s=t}^1 s^{-1} y(s) \, ds
		+ \frac{c}{\sqrt{\epsilon}} \int_{s=t}^1 s^{-1} z(s) \, ds, 
			\label{E:YINTEGRALINEQUALITY} \\
	z(t) & \leq C \epsilon^4 t^{- c \sqrt{\epsilon}}
			+ C \epsilon \int_{s=t}^1 s^{1/3 - c \sqrt{\epsilon}} y(s) \, ds
		\label{E:ZINTEGRALINEQUALITY} 
				\\
		& \ \ + C \int_{s=t}^1 s^{-1/3} z(s) \, ds
			+ c \epsilon \int_{s=t}^1 s^{-1} z(s) \, ds. \notag
\end{align}
\end{subequations}	
Then there exist constants $C', c' > 0$ such that the following inequalities hold for
$t \in (T,1]:$	
\begin{subequations}
\begin{align}
	y(t) & \leq C' \epsilon^3 t^{- c' \sqrt{\epsilon}}, 
		\label{E:YGRONWALL} \\
	z(t) & \leq C' \epsilon^4 t^{- c' \sqrt{\epsilon}}.
		\label{E:ZGRONWALL}
\end{align}
\end{subequations}

\end{lemma}

\begin{proof}
	We define
	\begin{align} \label{E:QISYPLUSZ}
		Q(t) := \epsilon y(t) + z(t).
	\end{align}
	With the help of \eqref{E:YINTEGRALINEQUALITY}-\eqref{E:ZINTEGRALINEQUALITY}, we see that
	there exist constants $\widetilde{C}, \widetilde{c} > 0$  such that $Q$ verifies the following integral inequality
	for $t \in (T,1]:$
	\begin{align} \label{E:QGRONWALLREADY}
	Q(t) & \leq \widetilde{C} \epsilon^4 t^{- \widetilde{c} \sqrt{\epsilon}}
		+ \widetilde{C} \int_{s=t}^1 s^{-1/3} Q(s) \, ds
		+ \widetilde{c} \sqrt{\epsilon} \int_{s=t}^1 s^{-1} Q(s) \, ds.
\end{align}		
Applying Gronwall's inequality to \eqref{E:QGRONWALLREADY}, we deduce that
there exist constants $C', c' > 0$ such that
\begin{align} \label{E:QGRONWALLED}
	Q(t) & \leq C' \epsilon^4 t^{- c' \sqrt{\epsilon}}, && t \in (T,1].
\end{align}
The desired inequalities \eqref{E:YGRONWALL}-\eqref{E:ZGRONWALL} now follow from \eqref{E:QGRONWALLED}.

\end{proof}

We derive our main a priori norm estimates in the next corollary. These estimates are the backbone of our
main stable singularity formation theorem.

\begin{corollary} [\textbf{The main a priori estimate for the norms when $N \geq 8$}] \label{C:ENERGYINTEGRALINEQUALITIES}
	Let $\highnorm{N},$ $\lowkinnorm{N-3},$ and $\lowpotnorm{N-4}$
	be the solution norms from Def.~\ref{D:NORMS}.
	Assume that the hypotheses and conclusions of Prop.~\ref{P:SOBFORG} 
	and Prop.~\ref{P:ENERGYINTEGRALINEQUALITIES} hold on the spacetime slab 
	$(T,1] \times \mathbb{T}^3.$ In particular, assume that for some integer $N \geq 8,$ the solution norm
	bootstrap assumptions \eqref{E:HIGHBOOT}-\eqref{E:POTBOOT} hold for $t \in (T,1].$
	Assume in addition that the initial renormalized metric verifies the
	stronger near-Euclidean condition $\| \newg - \Euc \|_{C_{Frame}^0}(1) \leq \epsilon^2$
	in place of the previous assumption $\| \newg - \Euc \|_{C_{Frame}^0}(1) \leq \epsilon.$
	Then there exist a small constant $\epsilon_N > 0$ and a large constant $c_N > 0$
	such that if $\epsilon \leq \epsilon_N$ and $\highnorm{N}(1) \leq \epsilon^2,$ 
	then the following norm inequalities are verified for $t \in (T,1]:$
	\begin{subequations}
	\begin{align} 
		\highnorm{N}(t) & \leq \frac{1}{2} \epsilon t^{- c_N \sqrt{\epsilon}}, 
			\label{E:MAINHIGHNORMESTIMATE} \\
		\lowkinnorm{N-3}(t) & \leq \frac{1}{2} \epsilon, 
			\label{E:MAINLOWKINNORMESTIMATE} \\
		\lowpotnorm{N-4}(t) & \leq \frac{1}{2} \epsilon t^{- c_N \sqrt{\epsilon}}.
			\label{E:MAINLOWPOTNORMESTIMATE}
	\end{align}
	\end{subequations}
	\end{corollary}

\begin{remark}
	The powers of $\epsilon$ stated in the estimates \eqref{E:MAINHIGHNORMESTIMATE}-\eqref{E:MAINLOWPOTNORMESTIMATE}
	are non-optimal and could be improved with additional effort.
\end{remark}

\begin{proof}
	We will prove by induction in $M$ that
	there exist constants $C_M', c_M' > 0$
	such that if $\epsilon$ is sufficiently small, 
	then the following inequalities hold for $t \in (T,1]:$ 
	\begin{align} 
			\sum_{1 \leq |\vec{I}| \leq M} \left\lbrace
				\left\| \left| \partial_{\vec{I}} \newg \right|_{\newg} \right\|_{L^2}^2(t)
			+ \left\| \left| \partial_{\vec{I}} \newg^{-1} \right|_{\newg} \right\|_{L^2}^2(t)
			\right\rbrace
		& \leq C_M' \epsilon^3 t^{- c_M' \sqrt{\epsilon}}, && (1 \leq M \leq N),
	 		\label{E:GCOMMUTEDSOBGRONWALLED} \\
		\totalenergy{\smallparameter_*}{M}^2(t) & \leq C_M' \epsilon^4 t^{- c_M' \sqrt{\epsilon}}, && (0 \leq M \leq N).
			\label{E:MAINTOTALENERGYESTIMATEM} 
	\end{align}
	Then from Prop.~\ref{P:COMPARISON},  
	\eqref{E:GCOMMUTEDSOBGRONWALLED},
	and \eqref{E:MAINTOTALENERGYESTIMATEM},
	we conclude that there exist constants 
	$\widetilde{C}_M, \widetilde{c}_M > 0$ such that
	\begin{align} \label{E:MAINHIGHNORMESTIMATEM}
		\highnorm{M}^2(t) & \leq \widetilde{C}_M \epsilon^3 t^{- \widetilde{c}_M \sqrt{\epsilon}},
	\end{align}
	and the desired estimate \eqref{E:MAINHIGHNORMESTIMATE} thus follows whenever 
	$\epsilon$ is sufficiently small. The estimates \eqref{E:MAINLOWKINNORMESTIMATE} and \eqref{E:MAINLOWPOTNORMESTIMATE} 
	then follow from revisiting the estimates proved in Prop.~\ref{P:STRONGPOINTWISE},
	where the assumption $\| \newg - \Euc \|_{C_{Frame}^0}(1) \leq \epsilon^2$ and the
	estimate \eqref{E:MAINHIGHNORMESTIMATEM} (with $M=N$) are 
	now taken as hypotheses in these propositions in place of the original assumption
	$\| \newg - \Euc \|_{C_{Frame}^0}(1) \leq \epsilon$ and the
	original bootstrap assumption $\highnorm{N}(t) \leq \epsilon t^{- \upsigma}.$

	It remains for us to prove \eqref{E:GCOMMUTEDSOBGRONWALLED}-\eqref{E:MAINTOTALENERGYESTIMATEM}. 
	To this end, we first inductively assume that the case $M - 1$ has been shown;
	we remark that our argument will also apply to the base case $M=0.$ We now define
	\begin{align}
		y(t) & := \left\lbrace \begin{array}{ll} 
		0, & (M = 0), \\
	  \sum_{1 \leq |\vec{I}| \leq M} 
			\left\| \left| \partial_{\vec{I}} \newg \right|_{\newg} \right\|_{L^2}^2(t)
			+ \sum_{1 \leq |\vec{I}| \leq M} \left\| \left| \partial_{\vec{I}} \newg^{-1} \right|_{\newg} \right\|_{L^2}^2(t), 
			& (1 \leq M),
		\end{array} \right.
		\\
		z(t) & := \totalenergy{\smallparameter_*}{M}^2(t).
	\end{align}
	Inserting the induction assumption estimates into inequalities \eqref{E:GCOMMUTEDSOB} and \eqref{E:ENERGYINTEGRALINEQUALITIES}
	(including the analogous but unwritten version of \eqref{E:GCOMMUTEDSOB} that holds for $\partial_{\vec{I}} \newg^{-1}$)
	and carrying out straightforward computations,
	we deduce that there exist constants $C, c > 0$
	such that hypothesis inequalities \eqref{E:YINTEGRALINEQUALITY}-\eqref{E:ZINTEGRALINEQUALITY} hold.
	The inequalities \eqref{E:YGRONWALL}-\eqref{E:ZGRONWALL} 
	immediately imply the desired estimates \eqref{E:GCOMMUTEDSOBGRONWALLED}-\eqref{E:MAINTOTALENERGYESTIMATEM}.
	This closes the induction.

\end{proof}

\section{Local Well-Posedness, the Existence of a CMC Hypersurface, and Continuation Criteria} \label{S:LWPANDCMC}

Before we can prove our main stable singularity formation theorem, we first need to address 
several aspects of the local-in-time theory. Hence, in this section, we briefly discuss local-in-time theory for the 
Einstein-stiff fluid system in two different gauges. 
We have two main goals. The first is to show that near-FLRW data launch a local solution that contains a spacelike hypersurface with constant mean curvature equal to $-1/3.$ 
We need the existence of such a hypersurface because
most of our previous analysis was carried out relative to CMC-transported spatial coordinates. In service of this goal, we 
sketch a proof of local well-posedness in a harmonic map gauge. This result (see Prop.~\ref{P:HARMONIC}) 
provides us with estimates that are sufficient to apply methods similar to those used by Bartnik \cite{rB1984}, which will guarantee the existence of the desired CMC hypersurface (see Prop.~\ref{P:CMCEXISTS}). To simplify our proofs, we have chosen to adopt near-FLRW hypotheses in both Prop.~\ref{P:HARMONIC} and Prop.~\ref{P:CMCEXISTS}; these hypotheses could be relaxed with additional effort. Our second goal is to sketch a proof of local well-posedness and continuation criteria in CMC-transported spatial coordinates (see Theorem~\ref{T:LOCAL}).

\subsection{Near-FLRW local well-posedness in harmonic map gauge}

In this section, we sketch a proof of local well-posedness for the Einstein-stiff fluid system in harmonic map gauge.

\begin{proposition}[\textbf{Near-FLRW local well-posedness in harmonic map gauge}] \label{P:HARMONIC}
Let $N \geq 4$ be an integer. Let $\big(\mathring{g}, \mathring{k}, \mathring{p}, \mathring{u} \big)$
be initial data (at time $1$) on the manifold $\Sigma_1 \simeq \mathbb{T}^3$ for the Einstein-stiff fluid system 
\eqref{E:EINSTEININTRO}-\eqref{E:EOS}, 
$\speed = 1$ (which by definition verify the constraints \eqref{E:GAUSSINTRO}-\eqref{E:CODAZZIINTRO}). Assume that 
\begin{align}
	\inf_{x \in \mathbb{T}^3} \mathring{p}(x) > 0,
\end{align}
and that relative to the standard coordinates on $\mathbb{T}^3$
(see Def.~\ref{D:STANDARDATLAS}), the components of the data verify the following estimates 
for $i,j = 1,2,3:$
\begin{align}
		\| \mathring{g}_{ij} - \Euc_{ij} \|_{H^{N+1}}
		+ \left\| \mathring{k}_{\ j}^i + \frac{1}{3} \ID_{\ j}^i \right\|_{H^N}
		+ \left\| \mathring{p} - \frac{1}{3} \right\|_{H^N}
		+ \| \mathring{u}^i \|_{H^N} & \leq \epsilon < \infty.
	\end{align}
Above, $\Euc_{ij} = \mbox{diag}(1,1,1)$ denotes the standard Euclidean metric on $\mathbb{T}^3$ and
$\ID_{\ j}^i = \mbox{diag}(1,1,1)$ denotes the identity. Let $\big(\mathbf{M}, (\gfour, p, \ufour) \big)$ be the maximal globally hyperbolic development of the data (see the discussion in Sect.~\ref{S:INTRO}). In particular, the variables $(\gfour, p, \ufour)$ verify the Einstein-stiff fluid equations \eqref{E:EINSTEININTRO}-\eqref{E:EOS}, $\speed = 1.$ Then if $\epsilon$ is sufficiently small, there exists a collection of coordinates $(t,x^1,x^2,x^3)$ covering a subset $\mathbf{V} := \cup_{t \in [1/2, 3/2]} \Sigma_t \simeq [1/2, 3/2] \times \mathbb{T}^3$ of $\mathbf{M}.$ The coordinate $t$ is a time function defined on $\mathbf{V},$ the $(x^1,x^2,x^3)$ are (spatially locally defined)
standard local coordinates on $\mathbb{T}^3$ (see Def.~\ref{D:STANDARDATLAS}), and each $\Sigma_s := \lbrace q \in \mathbf{V} \ | \ t(q) = s \rbrace \simeq \mathbb{T}^3$ is a Cauchy hypersurface. Relative to these coordinates, the FLRW solution can be expressed as
	\begin{align} \label{E:FLRWSTANDARDCOORDS}
		\widetilde{\gfour} & = - dt^2 + t^{2/3} \sum_{i=1}^3 (dx^i)^2, 
			& \widetilde{p} & = \frac{1}{3} t^{-2}, 
			& \widetilde{\ufour}^{\mu} & = \delta_0^{\mu}.
	\end{align}
	Above, $\delta_{\nu}^{\mu}$ ($\mu, \nu = 0,1,2,3$) is the standard Kronecker delta. 
	Furthermore, relative to these coordinates, the perturbed solution's components verify the 
	following estimates for $\mu, \nu = 0,1,2,3:$
	\begin{subequations}
	\begin{align}
		\sup_{t \in [1/2, 3/2]} \left\| \partial_t^M (\gfour_{\mu \nu} - \widetilde{\gfour}_{\mu \nu} \big) \right\|_{H^{N+1 - M}} & \leq C 
			\epsilon, && (0 \leq M \leq N + 1), \\
		\sup_{t \in [1/2, 3/2]} \left\| \partial_t^M \big(p - \widetilde{p} \big) \right\|_{H^{N-M}} & \leq C \epsilon, 
			&& (0 \leq M \leq N), \\
		\sup_{t \in [1/2, 3/2]} \left\| \partial_t^M \big(\ufour^{\mu} - \widetilde{\ufour}^{\mu} \big) \right\|_{H^{N-M}} 
			& \leq C \epsilon, && (0 \leq M \leq N).
	\end{align}
	\end{subequations}
	In addition, the harmonic map condition $(\gfour^{-1})^{\alpha \beta} \big(\Chfour_{\alpha \ \beta}^{\ \mu} - 
	\widetilde{\Chfour}_{\alpha \ \beta}^{\ \mu} \big) = 0$ is verified on $\mathbf{V}.$ 
	Here, $\Chfour_{\alpha \ \beta}^{\ \mu}$ is a Christoffel symbol of $\gfour$
	and $\widetilde{\Chfour}_{\alpha \ \beta}^{\ \mu}$ is a Christoffel symbol of $\widetilde{\gfour}.$ 
	Finally, the map from the initial data to the 
	solution is continuous. By ``continuous,'' we mean continuous relative to the norms on the data and the norms on the 
	solution that are stated in the hypotheses and above conclusions of this proposition.
	
\end{proposition}

\begin{proof}
	We use a hyperbolic reduction of Einstein's equations 
	described in \cite{hFaR2000}. This reduction is an extension of the fundamental ideas developed by
	Choquet-Bruhat \cite{CB1952}. More precisely, as discussed in \cite[Section 5]{hFaR2000}, there exists a 
	\emph{harmonic map gauge} for the Einstein equations in which
	\begin{align} \label{E:HARMONICGAUGE}
		(\gfour^{-1})^{\alpha \beta} \big(\Chfour_{\alpha \ \beta}^{\ \mu} - \widetilde{\Chfour}_{\alpha \ \beta}^{\ \mu} \big)
		& \equiv 0.
		\end{align}
 	Since $\Chfour_{\alpha \ \beta}^{\ \mu} - \widetilde{\Chfour}_{\alpha \ \beta}^{\ \mu}$
	is a tensor, the condition \eqref{E:HARMONICGAUGE} has a coordinate invariant meaning. It also allows us to work
	globally in space, even though $\mathbb{T}^3$ cannot be covered by a single coordinate chart.
	Relative to the standard local coordinates on $\mathbb{T}^3,$ the non-zero FLRW Christoffel symbols are
	\begin{align}
		\widetilde{\Chfour}_{j \ k}^{\ 0} 
		& = \frac{1}{3} t^{-1/3} \delta_{jk}, 
		& & 
		\widetilde{\Chfour}_{0 \ k}^{\ j}	= \frac{1}{3} t^{-1} \delta_{\ k}^j, 
		\label{E:NONZEROFLRWCHRISTOFFEL}
	\end{align}
	where $\delta_{jk} = \mbox{diag}(1,1,1)$ and
	$\delta_{\ k}^j = \mbox{diag}(1,1,1)$ are standard Kronecker deltas.
	In these coordinates, \eqref{E:HARMONICGAUGE} is equivalent to (for $i = 1,2,3$)
	\begin{align}
		\Chfour^0 
		& = \frac{1}{3} t^{-1/3} (\gfour^{-1})^{ab} \delta_{ab}, 
		&&
		\Chfour^i = \frac{2}{3} t^{-1/3} (\gfour^{-1})^{0i}.
	\end{align}
	Above, $\Chfour^{\mu}:=(\gfour^{-1})^{\alpha \beta} \Chfour_{\alpha \ \beta}^{\ \mu}$ is a
	contracted Christoffel symbol of $\gfour.$ To be compatible with both the Einstein initial data and 
	\eqref{E:HARMONICGAUGE}, $\gfour_{\mu \nu}$ and $\partial_t \gfour_{\mu \nu}$ are equipped with 
	the following initial data (given at $t = 1$) relative to the standard coordinates on 
	$\mathbb{T}^3:$
	\begin{align}
		\gfour_{00} & = - 1, & \gfour_{0i} & = 0, & \gfour_{ij} & = \mathring{g}_{ij}, \\
		\partial_t \gfour_{00} & = 2 \left(\frac{1}{3} \mathring{g}^{ab} \delta_{ab} + \mathring{\SecondFund}_{\ a}^a \right), 
			& \partial_t \gfour_{0i} & = \mathring{g}^{ab}\left(\partial_a \mathring{g}_{bi} 
				- \frac{1}{2} \partial_i \mathring{g}_{ab} \right), & \partial_t \gfour_{ij} & = - 2 \mathring{\SecondFund}_{ij}.
	\end{align}
	The above data enforce the harmonic mapping condition \eqref{E:HARMONICGAUGE} at $t = 1.$
	The data for the pressure and the four velocity are
	\begin{align}
		p & = \mathring{p}, 
		&& \ufour^0 = \sqrt{1 + \mathring{g}_{ab}\mathring{u}^a\mathring{u}^b}, 
		& & \ufour^i = \mathring{u}^i.
	\end{align}
	
	In this gauge, the Einstein-stiff fluid equations \eqref{E:EINSTEININTRO}-\eqref{E:EOS}, $\speed = 1$ 
	are equivalent to a reduced system comprising constraint equations,
	quasilinear wave equations for the spacetime metric 
	components, and first order hyperbolic equations for the fluid variables. Local well-posedness in the Sobolev spaces stated in 
	the proposition, the estimates stated in the proposition, 
	the continuous dependence of the solution on the data, and 
	the preservation of the constraints and the harmonic mapping condition \eqref{E:HARMONICGAUGE} are all standard results; see e.g. the 
	discussion in 
	\cite{hR2008}, which addresses all aspects of the theorem except continuous dependence on the data and how to estimate the fluid. 
	The article \cite{jS2012} provides details
	concerning the derivation of energy/Sobolev estimates for the fluid using energy currents and the divergence theorem; similar
	analysis was carried out in Sect.~\ref{S:FUNDAMENATLENERGYINEQUALITIES}. For the main ideas behind continuous dependence 
	on the data, see e.g. \cite{jS2008a}.
	
\end{proof}

\subsection{The existence of a CMC hypersurface}

We would now like to show that the spacetimes launched by Prop.~\ref{P:HARMONIC} contain a spacelike hypersurface graph of 
constant mean curvature precisely equal to $-\frac{1}{3}.$ In the next lemma, we recall the quasilinear elliptic PDE verified by such graphs.

\begin{lemma}[\textbf{Prescribed mean curvature PDE}]
	Let $\gfour$ be the spacetime metric solution from the conclusions of Prop.~\ref{P:HARMONIC}, and
	let $(t,x^1,x^2,x^3)$ be the corresponding coordinates (the $x^i$ are spatially locally defined). Let 
	$\Sigma:= \lbrace (t, x) \in \mathbb{R} \times \mathbb{T}^3 \ | \ t = \varphi(x) \rbrace$ be a hypersurface
	graph. Let $\mathscr{M}_{\varphi}(x)$ be $3 \times$ the mean curvature of $\Sigma$ (relative to $\gfour$) 
	evaluated at the spacetime point $(\varphi(x),x).$
	Let $\iota_{\varphi}: \mathbb{T}^3 \rightarrow \mathbb{R} \times \mathbb{T}^3$ be the function 
	defined by
	\begin{align} \label{E:IOTAVARPHIDEF}
		\iota_{\varphi}(x) := (\varphi(x),x).
	\end{align}
	Let $T_+ [1/2, 3/2] \times \mathbb{T}^3$ denote the bundle of future-directed unit timelike vectors over 
	$[1/2, 3/2] \times \mathbb{T}^3,$ 
	and let $F(t,x;\mathbf{X})$ be a function on $T_+ [1/2, 3/2] \times \mathbb{T}^3.$ 
	Let $\Chfour_{\alpha \ \beta}^{\ \mu}$ be the Christoffel symbols of $\gfour$ relative 
	to the above coordinates,
	let $\hfour^{\mu \nu}$ be the reciprocal first fundamental form of $\Sigma$ relative to $\gfour_{\mu \nu},$  
	and let $\mathbf{H}^{\mu \nu}$ be a rescaled version of $\hfour^{\mu \nu},$
	which are defined as follows:
	\begin{subequations}
	\begin{align} 
		\hfour^{\mu \nu}[\gfour \circ \iota_{\varphi}, \partial \varphi] 
			& := (\gfour^{-1})^{\mu \nu} \circ \iota_{\varphi} 
				+ \Nml^{\mu}[\gfour \circ \iota_{\varphi}, \partial \varphi] 
				\Nml^{\nu}[\gfour \circ \iota_{\varphi}, \partial \varphi],
			\label{E:FIRSTFUNDPHI} \\
		\mathbf{H}^{\mu \nu}[\gfour \circ \iota_{\varphi}, \partial \varphi] 
		& := |\longNml|_{\gfour}^{-1}[\gfour \circ \iota_{\varphi}, \partial \varphi] 
			\hfour^{\mu \nu}[\gfour \circ \iota_{\varphi}, \partial \varphi].
			\label{E:FIRSTFUNDPHIRESCALED}
	\end{align}
	\end{subequations}
	Above, $\Nml^{\mu}$ is the future-directed normal to $\Sigma,$ which is defined by
	\begin{subequations}
	\begin{align}
		\Nml^{\mu}[\gfour \circ \iota_{\varphi}, \partial \varphi] 
		& := |\longNml|_{\gfour}^{-1}[\gfour \circ \iota_{\varphi}, \partial \varphi]  (\gfour^{-1})^{\mu \nu} \circ \iota_{\varphi}
			\longNml_{\nu}[\partial \varphi], \\
		\longNml_{\mu}[\partial \varphi] & := \Dfour_{\mu} (- t + \varphi), 
			\label{E:LONGNORMALCOVECTOR} \\
		|\longNml|_{\gfour}[\gfour \circ \iota_{\varphi}, \partial \varphi] 
			& := \Big[- (\gfour^{-1})^{00} \circ \iota_{\varphi}
			+ 2 (\gfour^{-1})^{0a} \circ \iota_{\varphi} \partial_a \varphi  
			- (\gfour^{-1})^{ab} \circ \iota_{\varphi} (\partial_a \varphi)(\partial_b \varphi) \Big]^{1/2}.
			\label{E:NORMAL}
	\end{align}
	\end{subequations}
	Note that \eqref{E:NORMAL} slightly contradicts \eqref{E:TSPACETIMEGFOURNORM}
	in the sense that in \eqref{E:NORMAL}, we require $|\longNml|_{\gfour}$ to be non-negative.
	
	Then $\mathscr{M}_{\varphi}$ can be expressed as follows:
	\begin{align} \label{E:MEANCURVATUREEXPRESSION}
		\mathscr{M}_{\varphi} 
		& =  - \mathbf{H}^{ab}[\gfour \circ \iota_{\varphi}, \partial \varphi] \partial_a \partial_b \varphi
			- \mathbf{H}^{\alpha \beta}[\gfour \circ \iota_{\varphi}, \partial \varphi]  
				\Chfour_{\alpha \ \beta}^{\ 0}[\gfour \circ \iota_{\varphi}, ((\partial_t, \partial) \gfour) \circ \iota_{\varphi}] \\
		& \ \ + \mathbf{H}^{\alpha \beta}[\gfour \circ \iota_{\varphi}, \partial \varphi]  
				\Chfour_{\alpha \ \beta}^{\ a}[\gfour \circ \iota_{\varphi}, ((\partial_t, \partial) \gfour)\circ \iota_{\varphi}] 
				\partial_a \varphi.
				\notag
	\end{align}	

	Therefore, $\Sigma$ is a hypersurface of prescribed mean curvature 
	$\frac{1}{3} F(\varphi(x),x;\Nml[\gfour \circ \iota_{\varphi}, \partial \varphi])$ 
	if and only if $\varphi$ verifies the quasilinear elliptic PDE
		\begin{align} \label{E:PMCPDE}
			\mathbf{H}^{ab}[\gfour \circ \iota_{\varphi}, \partial \varphi] \partial_a \partial_b \varphi 
			& = - F(\iota_{\varphi};\Nml[\gfour \circ \iota_{\varphi}, \partial \varphi]) 
				- \mathbf{H}^{\alpha \beta}[\gfour \circ \iota_{\varphi}, \partial \varphi] 
					\Chfour_{\alpha \ \beta}^{\ 0}[\gfour \circ \iota_{\varphi}, ((\partial_t, \partial) \gfour)\circ \iota_{\varphi}] 
					\\
			& \ \ + \mathbf{H}^{\alpha \beta}[\gfour \circ \iota_{\varphi}, \partial \varphi] 
					\Chfour_{\alpha \ \beta}^{\ a}[\gfour \circ \iota_{\varphi}, ((\partial_t, \partial) \gfour)\circ \iota_{\varphi}] 
					\partial_a \varphi.
				\notag
	\end{align}
	
\end{lemma}

\begin{proof}
	The hypersurface $\Sigma:= \lbrace (t, x) \in \mathbb{R} \times \mathbb{T}^3 \ | \ t = \varphi(x) \rbrace$
	is the $0$ level set of the function
	\begin{align} \label{E:LEVELSETFUNCN}
		\Phi(t,x) & := - t + \varphi(x).
	\end{align}	
	The spacetime one-form that is normal to the level sets of $\Phi$ is
	\begin{align}
		\longNml_{\mu} & = \Dfour_{\mu} \Phi.
	\end{align}
	The corresponding future-directed $\gfour-$unit normal vector is
	\begin{align}
		\Nml^{\mu} & = \frac{1}{|\longNml|_{\gfour}} (\gfour^{-1})^{\mu \alpha} \Dfour_{\alpha} \Phi,
	\end{align}
	where 
	\begin{align}
		 |\longNml|_{\gfour}:= |(\gfour^{-1})^{\alpha \beta} \Dfour_{\alpha} \Phi \Dfour_{\beta} \Phi|^{1/2}.
	\end{align}
	
	The reciprocal first fundamental form of $\Sigma$ is
	\begin{align} \label{E:FIRSTFUNDOFSIGMA}
		\hfour^{\mu \nu} := (\gfour^{-1})^{\mu \nu} + \Nml^{\mu} \Nml^{\nu},
	\end{align}
	while the second fundamental form is
	\begin{align} \label{E:SECONDFUNDOFSIGMA}
		\kfour_{\ \nu}^{\mu} := - \gfour_{\nu \kappa} \gfour_{\beta \lambda} \hfour^{\mu \alpha} \hfour^{\kappa \lambda} \Dfour_{\alpha} 
		\Nml^{\beta}.
	\end{align}
	The trace of the second fundamental form, which is defined to be $3$ times the mean curvature of $\Sigma,$
	is therefore given by
	\begin{align} \label{E:DIVNORMAL}
		- (\Dfour_{\alpha} \Nml^{\alpha})|_{t = \varphi(x)}
	 		= - \left( \frac{1}{|\longNml|_{\gfour}} \hfour^{\alpha \beta} \Dfour_{\alpha} \Dfour_{\beta} \Phi \right)
	 		\circ \iota_{\varphi}.
	\end{align}	
	
	Setting \eqref{E:DIVNORMAL} equal to $F,$ 
	expanding covariant derivatives in terms of the 
	Christoffel symbols of $\gfour,$ and carrying out straightforward computations,
	we thus deduce equation \eqref{E:PMCPDE} from \eqref{E:DIVNORMAL}.
	
\end{proof}

We now provide a proof of the existence of a hypersurface of constant mean curvature precisely equal to $-1/3$ 
in the spacetimes of interest. Our proof makes use of the estimates of Prop.~\ref{P:HARMONIC}
as well as the following fundamental Leray-Schauder fixed point result from degree theory.

\begin{proposition} \cite[Theorem 4.4.3]{nL1978} \label{P:LLOYDDEGREE}
 Let $\mathfrak{K}$ be a closed, bounded subset of the Banach space $\mathfrak{B}.$ Assume that
 $0 \in \mbox{interior}(\mathfrak{K}).$ Let $\mathscr{T}: \mathfrak{K} \rightarrow \mathfrak{B}$
 be a compact map. Assume that for all real numbers $\uplambda \in (0,1),$ 
 there is no element $k \in \partial \mathfrak{K}$ such that $\uplambda \mathscr{T}(k) = k.$
 Then there exists an element $k_* \in \mathfrak{K}$ such that $\mathscr{T}(k_*) = k_*.$
\end{proposition}

\hfill $\qed$

Our proof of existence is simpler than Bartnik's similar proof \cite[Theorem 4.1]{rB1984}
because ours takes advantage of the smallness of $\epsilon.$ One key difference 
is that Bartnik's proof only addressed the case in which $\gfour_{0i} \equiv 0,$ ($i=1,2,3$), while in our harmonic
map gauge, we generally have $\gfour_{0i} \neq 0.$

\begin{proposition} [\textbf{Existence of a CMC slice}] \cite[c.f. Lemma 6.2]{cG1983}\cite[c.f. Theorem 4.1]{rB1984}  \label{P:CMCEXISTS}
Assume the hypotheses and conclusions of Prop.~\ref{P:HARMONIC} (in particular assume that $\epsilon$ is sufficiently small).
Assume in addition that $N \geq 5.$ Let 
$\big(\gfour_{\mu \nu}, p, \ufour^{\mu} \big)$ be the solution on $[1/2,3/2] \times \mathbb{T}^3$ guaranteed by the proposition. Let $(t,x^1,x^2,x^3)$ be the corresponding coordinates (the $x^i$ are spatially locally defined). Let $\Sigma_t := \lbrace t \rbrace \times \mathbb{T}^3$ denote a constant-time hypersurface and let $\mathscr{M}_0(t,x)$ denote $3 \times$ its corresponding mean curvature (with respect to $\gfour$) 
evaluated at the spacetime point $(t,x).$
Then there exists a constant $C_* > 0$ such that if $\epsilon$ is sufficiently small, then there exists a pair of times
\begin{align}
	t_{\pm} & = 1 \pm C_* \epsilon, 
	\label{E:TIMES}
\end{align}
such that
\begin{align} \label{E:CMCTRAPPING}
	\sup_{x \in \mathbb{T}^3} - \mathscr{M}_0(t_+,x) < 1 < \inf_{x \in \mathbb{T}^3} - \mathscr{M}_0(t_-,x).
\end{align}
Furthermore, there exists a function $\varphi$ on $\mathbb{T}^3$ such that 
\begin{subequations}
\begin{align} 
	\varphi & \in H^{N+2}(\mathbb{T}^3), \\
	t_- & < \inf_{x \in \mathbb{T}^3} \varphi(x) \leq \sup_{x \in \mathbb{T}^3} \varphi(x) < t_+, \label{E:VARPHITRAPPED}
\end{align}
\end{subequations}
and such that $\Sigma := \lbrace (\varphi(x), x) \ | \ x \in \mathbb{T}^3 \rbrace 
\subset (t_-, t_+) \times \mathbb{T}^3$ is a spacelike hypersurface with mean curvature 
(relative to $\gfour$) constantly equal to $-\frac{1}{3}$ (i.e., the PDE \eqref{E:PMCPDE} is verified with $F(\iota_{\varphi};\Nml) := - 1$). 

Furthermore, there exists a constant $C >0$ such that if $\epsilon $ is sufficiently small, then the following estimates are verified by $\varphi$ 
and the components of $\hfour^{\mu \nu}$ relative to the above coordinates: 
\begin{align} \label{E:PHIFULLYUPGRADEDESTIMATES}
	\left \| \varphi - 1 \right \|_{H^{N+2}} & \leq C \epsilon,
\end{align}
\begin{subequations}
\begin{align}
	 \|\hfour^{00}[\gfour \circ \iota_{\varphi}, \partial \varphi] 
		+ 1 \|_{H^{N+1}} & \leq C \epsilon, & 
		\label{E:FIRSTFUND00} \\
	\|\hfour^{0i}[\gfour \circ \iota_{\varphi}, \partial \varphi] 	
			\|_{H^{N+1}} & \leq C \epsilon, & & (i = 1,2,3), 
		\label{E:FIRSTFUND0I} \\
	\| \hfour^{ij}[\gfour \circ \iota_{\varphi}, \partial \varphi] - (\Euc^{-1})^{ij} \|_{H^{N+1}} & \leq C \epsilon, & & (i,j = 1,2,3).
		\label{E:FIRSTFUNDIJ}
\end{align}
\end{subequations}
In the above estimates, $\iota_{\varphi}$ is defined in \eqref{E:IOTAVARPHIDEF}, 
$\hfour^{\mu \nu}[\gfour \circ \iota_{\varphi}, \partial \varphi]$ is defined in \eqref{E:FIRSTFUNDPHI}, and 
$(\Euc^{-1})^{ij} = \mbox{diag}(1,1,1)$ is the standard inverse Euclidean metric on $\mathbb{T}^3.$

Finally, relative to the above coordinates, the solution's components have the following regularity properties
for $\mu, \nu = 0,1,2,3:$
\begin{subequations}
	\begin{align}
		 \left\| \left\lbrace \partial_t^M (\gfour_{\mu \nu} - \widetilde{\gfour}_{\mu \nu} \big) \right\rbrace
			 \circ \iota_{\varphi} 
			\right\|_{H^{N+1 - M}} & \leq C 
			\epsilon, && (0 \leq M \leq N + 1), \label{E:GCOMPOSEDWITHPHIUPGRADED} \\
		\left\| \left\lbrace \partial_t^M \big(p - \widetilde{p} \big) \right\rbrace
			\circ \iota_{\varphi}\right\|_{H^{N-M}} & \leq C \epsilon, 
			&& (0 \leq M \leq N), \label{E:PCOMPOSEDWITHPHIUPGRADED} \\
		\left\| \left\lbrace \partial_t^M \big(\ufour^{\mu} - \widetilde{\ufour}^{\mu} \big) \right\rbrace
			\circ \iota_{\varphi} \right\|_{H^{N-M}} 
			& \leq C \epsilon, && (0 \leq M \leq N), \label{E:UCOMPOSEDWITHPHIUPGRADED}
	\end{align}
	\end{subequations}
	where the components $ \widetilde{\gfour}_{\mu \nu},$ $\widetilde{p},$ $\widetilde{\ufour}^{\mu}$ 
	of the FLRW solution are given by \eqref{E:FLRWSTANDARDCOORDS}.

\end{proposition}

\begin{proof}
	Note that the function
	$F$ from \eqref{E:PMCPDE} is constantly $-1$ in the present context.
	Throughout this proof, we will repeatedly use the assumption that $\epsilon$ is sufficiently small without
	mentioning it every time. The main idea is to derive a priori estimates for a modified version of the elliptic PDE 
	\eqref{E:PMCPDE}. The modification depends on a small positive real number $\updelta.$
	The a priori estimates will allow us to apply Prop.~\ref{P:LLOYDDEGREE} to deduce an existence result for the modified 
	PDE. We then take a limit as $\updelta$ goes to $0$ in order to deduce existence for the actual PDE of interest. The final 
	step is to improve the Sobolev regularity of both $(\gfour, p, \ufour) \circ \iota_{\varphi}$ and $\varphi(x).$
	The former estimates require us to revisit the estimates of Prop.~\ref{P:HARMONIC}; instead of using the divergence 
	theorem/energy estimates to derive $L^2$ estimates along the hypersurfaces $\Sigma_t,$ as was done in the proposition,
	we instead derive analogous $L^2$ estimates along $\Sigma.$ 

\textbf{Step 1 - Linearize, modify, and change variables via $\varphi = 1 + \chi$}: 
Let $\mathscr{M}_0(t,x)$ be $3$ times the mean curvature (relative to $\gfour$) of the constant-time hypersurface $\Sigma_t$ evaluated at the spacetime point $(t,x).$ We now note that a slightly modified version of the PDE \eqref{E:PMCPDE} can be written in the following form:
\begin{align} \label{E:PMCPDERESCALED}
			\mathbf{H}^{ab}[\gfour \circ \iota_{\varphi}, \partial \varphi] \partial_a \partial_b \varphi 
			& = \overbrace{- F(\iota_{\varphi}; \Nml[\gfour \circ \iota_{\varphi}, \partial \varphi])}^1
			 		+ \mathscr{M}_0 \circ \iota_{\varphi} \\
			& \ \ + Y^a[\gfour \circ \iota_{\varphi}, ((\partial_t, \partial) \gfour)\circ \iota_{\varphi}, \partial \varphi] \partial_a \varphi
			 		+ \updelta (\varphi - 1). \notag
	\end{align}
Above, $\updelta \in (0,1]$ is a small positive constant and the term $\updelta (\varphi - 1)$ was artificially added to the equation for reasons to be explained. We will let $\updelta \downarrow 0$ later in the proof. Furthermore, $Y^a[\cdots]$ is a smooth function of its arguments as long as $\gfour$ is invertible and $|\longNml|_{\gfour} > 0.$ The $\mathscr{M}_0 \circ \iota_{\varphi}$ term is derived by isolating and subtracting off this term from the
$\mathbf{H}^{\alpha \beta}[\gfour \circ \iota_{\varphi},\partial \varphi]  
\Chfour_{\alpha \ \beta}^{\ 0} [\gfour \circ \iota_{\varphi}, ((\partial_t, \partial) \gfour)\circ \iota_{\varphi}]$
term on the right-hand side of \eqref{E:PMCPDE}. More precisely, we view
$-\mathbf{H}^{\alpha \beta}[\gfour \circ \iota_{\varphi}, \partial \varphi]  
\Chfour_{\alpha \ \beta}^{\ 0} [\gfour \circ \iota_{\varphi}, ((\partial_t, \partial) \gfour)\circ \iota_{\varphi}]$
$:= f[\gfour \circ \iota_{\varphi}, ((\partial_t, \partial) \gfour)\circ \iota_{\varphi}, \partial \varphi],$
where $f$ is a smooth function of its arguments. Then 
$\mathscr{M}_0(t,x) = f[\gfour(t,x), (\partial \gfour)(t,x), 0],$ 
$\mathscr{M}_0 \circ \iota_{\varphi} = f[\gfour \circ \iota_{\varphi}, (\partial \gfour) \circ \iota_{\varphi}, 0],$
while the difference
$f[\gfour \circ \iota_{\varphi}, ((\partial_t, \partial) \gfour)\circ \iota_{\varphi}, \partial \varphi] 
- f[\gfour \circ \iota_{\varphi}, ((\partial_t, \partial) \gfour)\circ \iota_{\varphi}, 0]$ 
is (by Taylor expansion) incorporated into the $Y^a[\gfour \circ \iota_{\varphi}, ((\partial_t, \partial) \gfour)\circ \iota_{\varphi}, \partial \varphi] \partial_a \varphi$ term on the right-hand side of \eqref{E:PMCPDERESCALED}. We also make the change of variables $\chi = \varphi - 1,$ where we think of $\chi$ as small. We then consider a linearized version of \eqref{E:PMCPDERESCALED}. More precisely, given a known function 
$\varrho \in H^{N-1},$ we consider the linear (in $\chi$) PDE
\begin{align} \label{E:PMCPDERESCALEDLINEARIZED}
			\mathbf{H}^{ab}[\gfour \circ \iota_{1 + \varrho}, \partial \varrho] 
				\partial_a \partial_b \chi 
			& = 1 + \mathscr{M}_0 \circ \iota_{1 + \varrho}
			 		+ Y^a[\gfour \circ \iota_{1 + \varrho}, ((\partial_t, \partial) \gfour)\circ \iota_{1 + \varrho}, \partial \varrho] 
			 		\partial_a \chi
					+ \updelta \chi.
	\end{align}
The reason for adding the term $\updelta (\varphi - 1) = \updelta \chi$ is the following: by the maximum principle for classical solutions, this $0th$ order term 
ensures that the linearized PDE has unique (in the class $C^2$) solutions.

\textbf{Step 2 - Identifying a fixed point space}: We choose to first prove existence in a space of relatively high regularity. 
The estimates of Prop.~\ref{P:HARMONIC} imply that there exists a constant $C_* > 0$ such that 
\eqref{E:TIMES} and \eqref{E:CMCTRAPPING} hold. We will work with closed, bounded subsets of $H^{N-1}$ of the form
\begin{align} \label{E:SMALLERSETDEF}
	\mathfrak{K}_{\epsilon} := \left\lbrace \varrho  \ | \  \| \varrho \|_{H^{N-1}} \leq 1, 
	\left \| |\longNml|_{\gfour}^{-1}[\gfour \circ \iota_{1 + \varrho}, \partial \varrho] \right \|_{L^{\infty}} \leq 2, \| \varrho \|_{L^{\infty}} \leq 
	C_* \epsilon \right\rbrace.
\end{align}	
Note that the function $\varrho \equiv 0$ belongs to the interior of $\mathfrak{K}_{\epsilon}.$
This fact is an important hypothesis in Prop.~\ref{P:LLOYDDEGREE}. The reason that we assume $N \geq 5$ is that
Sobolev embedding implies the existence of a constant $C > 0$ such that 
\begin{align} \label{E:PARTIALHC0ESTIMATE}
	\left \| \partial (\mathbf{H}^{ij}[\gfour \circ \iota_{1 + \varrho}, \partial \varrho] ) \right \|_{C^0} \leq C 
\end{align}
whenever $\epsilon$ is sufficiently small and $\varrho \in \mathfrak{K}_{\epsilon};$ the availability of this estimate simplifies our analysis.

\textbf{Step 3 - The linearized map}: For each $\updelta > 0,$ let $\mathscr{T}_{\updelta}$ be the map that sends
elements $\varrho \in \mathfrak{K}_{\epsilon}$ to the corresponding solution $\chi$ to the linear PDE \eqref{E:PMCPDERESCALEDLINEARIZED},
i.e., $\mathscr{T}_{\updelta} \varrho = \chi.$ 
Note that for $\varrho \in \mathfrak{K}_{\epsilon},$ the modified linear PDE, which is of the form 
$a^{ij}(x)\partial_i \partial_j \chi + b^i(x) \partial_i \chi + c \chi = d(x),$ is uniformly elliptic with $a^{ij}(x)$ positive definite, $c < 0,$ 
and $a^{ij}, b^i, c, d \in H^{N-2}.$ Standard elliptic theory 
and the Rellich-Kondrachov theorem ensure that $\mathscr{T}_{\updelta}: \mathfrak{K}_{\epsilon} \rightarrow H^N \hookrightarrow H^{N-1}$ is a well-defined \emph{compact} (in $H^{N-1}$) map. We are using in particular the maximum principle and the fact that $c < 0$
to conclude that the map $\chi \to a^{ij}(x)\partial_i \partial_j \chi + b^i(x) \partial_i \chi + c \chi$
has a trivial kernel. We are also using the fact that the components $\gfour_{\mu \nu}$
are elements of $C^{N-1}([1/2, 3/2] \times \mathbb{T}^3)$ and the Christoffel symbols $\Chfour_{\alpha \ \beta}^{\ \mu}$
are elements of $C^{N-2}([1/2, 3/2] \times \mathbb{T}^3).$ This ensures, for example, that 
$\gfour_{\mu \nu} \circ \iota_{1 + \varrho} \in H^{N-1},$ 
$\Chfour_{\alpha \ \beta}^{\ \mu}[\gfour \circ \iota_{1 + \varrho}, ((\partial_t, \partial) \gfour)\circ \iota_{1 + \varrho}] 
\in H^{N-2},$ and $\mathbf{H}^{ij}[\gfour \circ \iota_{1 + \varrho}, \partial \varrho] \in H^{N-2}.$

\textbf{Step 4 - A priori estimates and fixed points}: In order to apply Prop.~\ref{P:LLOYDDEGREE}, we will show that the following fact holds: for each real number $\uplambda \in [0,1],$ any $\uplambda-$quasi-fixed point $\varrho$ of $\mathscr{T}_{\updelta}$ that belongs to $\mathfrak{K}_{\epsilon}$ must lie in the interior of $\mathfrak{K}_{\epsilon}.$ By definition, 
a $\uplambda-$quasi-fixed point is a solution to the equation
\begin{align} \label{E:QUASIFIXED}
	\uplambda \mathscr{T}_{\updelta} \varrho = \varrho,
\end{align}
which can equivalently be expressed as
\begin{align} \label{E:PMCPDEQUASIFIXED}
			\mathbf{H}^{ab}[\gfour \circ \iota_{1 + \varrho}, \partial \varrho] 
				\partial_a \partial_b \varrho 
			& = \uplambda \left\lbrace 1 + \mathscr{M}_0 \circ \iota_{1 + \varrho} \right\rbrace
			 		+ Y^a[\gfour \circ \iota_{1 + \varrho}, ((\partial_t, \partial) \gfour)\circ \iota_{1 + \varrho}, \partial \varrho] 
			 		\partial_a \varrho
					+ \updelta \varrho.
\end{align}
After we have shown the aforementioned fact, Prop.~\ref{P:LLOYDDEGREE} will immediately imply the desired existence result for the modified nonlinear PDE. To prove the fact, we will derive the following two a priori estimates: \textbf{i)} if $\uplambda \in [0,1],$ 
$\epsilon$ is sufficiently small, and $\varrho \in \mathfrak{K}_{\epsilon}$ verifies \eqref{E:QUASIFIXED}, then $\| \varrho \|_{H^N} \leq C \epsilon;$ \textbf{ii)} if $\uplambda \in [0,1],$ $\epsilon$ is sufficiently small,
and $\varrho \in \mathfrak{K}_{\epsilon}$ solves \eqref{E:QUASIFIXED}, then $\| \varrho \|_{C^0} < C_* \epsilon,$ where $C_*$ is the constant 
appearing in definition \eqref{E:SMALLERSETDEF}. \textbf{i)} and \textbf{ii)} together imply that when $\uplambda \in [0,1]$
and $\epsilon$ is sufficiently small, there are no 
$\uplambda-$quasi-fixed points $\varrho$ of $\mathscr{T}_{\updelta}$ with
$\varrho \in \partial \mathfrak{K}_{\epsilon}.$ Our proofs of \textbf{i)} and \textbf{ii)} are based on standard $L^2-$type Sobolev estimates and the maximum principle. We first prove \textbf{i)} via Sobolev estimates. Specifically, we multiply \eqref{E:PMCPDEQUASIFIXED} by $\varrho,$ integrate by parts, use the bounds implied by \eqref{E:SMALLERSETDEF}, use the uniform positive definiteness of $\mathbf{H}^{ij}$ in the regime we are considering, use the estimate \eqref{E:PARTIALHC0ESTIMATE}, use the standard Sobolev calculus, and apply Cauchy-Schwarz (for the inverse metric $\mathbf{H}^{ij}$) to deduce that there exists a constant $C > 0$ (independent of $\updelta \in (0,1]$) such that 
\begin{align} \label{E:CMCFIRSTENERGY}
	\int_{\mathbb{T}^3} \mathbf{H}^{ab}[\gfour \circ \iota_{1 + \varrho}, \partial \varrho] 
		(\partial_a \varrho)(\partial_b \varrho)
		\, dx
	\leq C \uplambda^2 \int_{\mathbb{T}^3} \left|1 + \mathscr{M}_0 \circ \iota_{1 + \varrho} \right|^2 \, dx
		+ C \int_{\mathbb{T}^3} \varrho^2 \, dx.
\end{align}
The estimates of Prop.~\ref{P:HARMONIC} imply that for $\varrho \in \mathfrak{K}_{\epsilon},$ we have
\begin{align}
	\| 1 + \mathscr{M}_0 \circ \iota_{1 + \varrho} \|_{C^0} \leq C \epsilon.
\end{align}
Furthermore, the definition of $\mathfrak{K}_{\epsilon}$ implies that $\| \varrho \|_{C^0} \leq C \epsilon.$ Since
$\mathbf{H}^{ij}$ is uniformly comparable to $(\Euc^{-1})^{ij}$ in the regime we are considering, we 
therefore conclude from \eqref{E:CMCFIRSTENERGY} that
\begin{align}
	\int_{\mathbb{T}^3} (\Euc^{-1})^{ab} (\partial_a \varrho)(\partial_b \varrho) \, dx & \leq C \epsilon^2.
\end{align} 
 
Similarly, we commute equation \eqref{E:PMCPDEQUASIFIXED} with $\partial_{\vec{I}},$ $|\vec{I}| \leq N-2$ and repeat the argument to
inductively derive
\begin{align}
	\int_{\mathbb{T}^3} (\Euc^{-1})^{ab} (\partial_a \partial_{\vec{I}} \varrho)(\partial_b \partial_{\vec{I}} \varrho) \, dx 
	& \leq C \epsilon^2, & |\vec{I}| & \leq N-2.
\end{align} 
To estimate the $Nth$ order derivatives of $\varrho,$ we first note that for $|\vec{I}| = N-2,$ we have 
\begin{align}
	\partial_a \big(\mathbf{H}^{ab}[\gfour \circ \iota_{1 + \varrho}, \partial \varrho] 
	\partial_b \partial_{\vec{I}} \varrho \big) = f_{\vec{I}},
\end{align}
where $\| f_{\vec{I}} \|_{L^2} \leq C \epsilon.$ It follows that $\partial_{\vec{I}} \varrho$ is a weak solution to a 
uniformly elliptic divergence form PDE with $\mathbf{H}^{ab}[\gfour \circ \iota_{1 + \varrho}, \partial \varrho] \in H^{N-2} \hookrightarrow C^{N-4}.$ From standard elliptic theory (c.f. \cite[Theorem 1 of Section 6.3.1]{lE1998}), we conclude that $\partial_{\vec{I}} \varrho \in H^2$ and
\begin{align}
	\| \partial_{\vec{I}} \varrho \|_{H^2} \leq C (\| f_{\vec{I}} \|_{L^2} + \| \partial_{\vec{I}} \varrho \|_{H^1}) \leq C \epsilon
\end{align}
as desired. In summary, we have shown that
\begin{align} \label{E:PHIHNBOUND}
	\| \varphi - 1 \|_{H^N} & \leq C \epsilon,
\end{align}
where $C$ does not depend on $\updelta \in (0,1].$

To prove $\textbf{ii)},$ we show that when $\uplambda \in [0,1],$ any $\uplambda-$quasi fixed point $\varrho \in \mathfrak{K}_{\epsilon}$ must avoid the two obstacles $\pm C_* \epsilon.$ We argue by contradiction using the maximum principle. If $\varrho(x_*) = C_* \epsilon,$ then $x_*$ must be a maximum. Thus, $\partial_i \varrho(x_*) = 0,$ $(i=1,2,3).$ Therefore, it follows from \eqref{E:CMCTRAPPING} and \eqref{E:PMCPDEQUASIFIXED} that
\begin{align}
	\mathbf{H}^{ab}[\gfour(1 + C_* \epsilon, x_*), 0]
		\partial_a \partial_b \varrho(x_*) 
		& = \uplambda \left\lbrace 1 + \mathscr{M}_0 (1 + C_* \epsilon, x_*) \right\rbrace
		+ C_* \epsilon \updelta \\
		& > 0. \notag
\end{align}
Since $\mathbf{H}^{ij}[\gfour(1 + C_* \epsilon, \cdot), 0]$ is a uniformly positive definite $3 \times 3$ matrix, this 
contradicts the maximum principle for the operator 
$\mathbf{H}^{ab}[\gfour(1 + C_* \epsilon, x_*), 0] \partial_a \partial_b$ and shows that $\varrho(x_*) = C_* \epsilon$ is impossible. Similarly, 
it is impossible for $\varrho$ to touch $-C_* \epsilon.$ It follows that 
\begin{align}
	\| \varrho \|_{C^0} < C_* \epsilon.
\end{align}
We can now apply Prop.~\ref{P:LLOYDDEGREE} to conclude that the PDE \eqref{E:PMCPDERESCALED}
has a solution $\varphi = 1 + \varrho,$ where $\varrho \in \mathfrak{K}_{\epsilon}.$

\textbf{Step 5 - Limit as $\updelta \downarrow 0$}: We have now produced a family of solutions
$\varphi_{\updelta}$ to the modified nonlinear PDE \eqref{E:PMCPDERESCALED} that verify the bound
\eqref{E:PHIHNBOUND}. By weak compactness and the Rellich-Kondrachov theorem, 
there exists a sequence of numbers $\updelta_n \overset{n \to \infty}{\downarrow} 0$
such that the sequence $\varphi_{\updelta_n}$ converges
(weakly in $H^N$ and strongly in $H^{N-1}$) to the desired solution $\varphi \in H^N \cap \mathfrak{K}_{\epsilon}$ of the nonlinear PDE \eqref{E:PMCPDERESCALED} with $\updelta = 0.$ The solution $\varphi$ also verifies the bound \eqref{E:PHIHNBOUND}.

\textbf{Step 6 - Upgraded regularity of the solution to the Einstein-stiff fluid equations}:
	The above arguments have produced a $C^{N-2}$ hypersurface $\Sigma = \lbrace (\varphi(x),x) \ | \ x \in \mathbb{T}^3 \rbrace$ 
	of mean curvature $-1/3$ that is sandwiched in between the hypersurfaces $\Sigma_{t_-}$ and $\Sigma_{t_+},$ where $t$ is the time 
	coordinate from Prop.~\ref{P:HARMONIC} and $t_- = 1 - C_* \epsilon,$ $t_+ = 1 + C_* \epsilon.$
	Let $\mathscr{Z} := (\gfour_{\mu \nu}, \partial \gfour_{\mu \nu}, p, \ufour^{\mu})_{0 \leq \mu, \nu \leq 3}$
	denote the array of solution components appearing in the conclusions of Prop.~\ref{P:HARMONIC}, and let
	$\widetilde{\mathscr{Z}}$ denote the array of FLRW components [see equation \eqref{E:FLRWSTANDARDCOORDS}]. 
	The main difficulty that remains to be resolved is that even though Prop.~\ref{P:HARMONIC} guarantees that
	$\partial_t^M(\mathscr{Z} - \widetilde{\mathscr{Z}}) (t,\cdot) \in H^{N-M}$ for $0 \leq M \leq N,$ it does not
	automatically follow that, for example,  
	$\big[\partial_{\vec{I}}(\mathscr{Z} - \widetilde{\mathscr{Z}}) \big] \circ \iota_{\varphi} \in L^2$
	when $|\vec{I}| = N;$ this difficulty would be present even if $\varphi$ were known to be $C^{\infty}.$ The standard trace theorems
	allow for the possibility that the spacetime function 
	$\partial_t^M(\mathscr{Z} - \widetilde{\mathscr{Z}})$ loses some Sobolev 
	differentiability when restricted to the hypersurface $\Sigma.$ To avoid any loss, we will revisit the   
	estimates proved in Prop.~\ref{P:HARMONIC} along the hypersurfaces $\Sigma_t$ and deduce analogous
	estimates along $\Sigma.$ Specifically, we claim that one can derive the following Sobolev estimate for the solution
	along $\Sigma$ whenever $M + |\vec{I}| \leq N:$
	\begin{align} \label{E:DIVTHMFORSIGMA}
		\int_{\mathbb{T}^3} \left|\partial_t^M \partial_{\vec{I}}(\mathscr{Z} - \widetilde{\mathscr{Z}})\right|^2 \circ \iota_{\varphi} \, d 
		\mu
		& \leq C \int_{\mathbb{T}^3} \left|\partial_t^M \partial_{\vec{I}}(\mathscr{Z} - \widetilde{\mathscr{Z}})\right|^2(t_-,x) \, dx 
			+ C \int_{s = t_-}^{t_+} \left|\partial_t^M \partial_{\vec{I}}(\mathscr{Z} - \widetilde{\mathscr{Z}})\right|^2(s,x)  \, dx \, ds \\
		& \leq C \epsilon. \notag
	\end{align}
	Above, $d \mu = \upsilon \, dx,$
	where $\upsilon:= \sqrt{\mbox{det}(E^{\Sigma})}$ 
	is the volume form factor corresponding to $E_{ij}^{\Sigma} := \iota_{\varphi}^* \Eucfour^{\Sigma}_{ij},$ 
	and $\iota_{\varphi}^*$ denotes pullback by $\iota_{\varphi}.$
	Here, $\Eucfour^{\Sigma}$ is the first fundamental form of $\Sigma$ relative to $\Eucfour,$ 
	where $\Eucfour_{\mu \nu} := \mbox{diag}(1,1,1,1)$ is the standard Euclidean metric on $[1/2, 3/2] \times \mathbb{T}^3.$ 
	In components, we have
	\begin{align}
		\Eucfour^{\Sigma}_{\mu \nu} = \Eucfour_{\mu \nu} - \frac{\longNml_{\mu} \longNml_{\nu}}{|\longNml|_{\Eucfour}^2},
	\end{align}
	where $\longNml_{\mu}$ is defined in \eqref{E:LONGNORMALCOVECTOR} and 
	$|\longNml|_{\Eucfour}^2 := (\Eucfour^{-1})^{\alpha \beta} \longNml_{\alpha} \longNml_{\beta}.$
	The estimate \eqref{E:PHIHNBOUND} implies that
	\begin{align}
		|\upsilon - 1| \leq C \epsilon. 
	\end{align}
	The estimate \eqref{E:DIVTHMFORSIGMA} can be derived by 
	commuting the Einstein-stiff fluid equations in harmonic map gauge
	with the operators $\partial_t^M \partial_{\vec{I}},$ by
	applying the divergence theorem (with the metric $\Eucfour$)
	to the region in between $\Sigma_{t_-}$ and $\Sigma,$ and by overestimating the 
	corresponding spacetime integral by the spacetime integral in \eqref{E:DIVTHMFORSIGMA}. 
	The ``$\epsilon$'' on the right-hand side of \eqref{E:DIVTHMFORSIGMA} is guaranteed by the 
	estimates of Prop.~\ref{P:HARMONIC}
	and the estimate \eqref{E:PHIHNBOUND}. The vectorfield used in the divergence theorem is 
	the same vectorfield that is used to derive the energy estimates in the proof of Prop.~\ref{P:HARMONIC}. 

\textbf{Step 7 - The $H^{N+2}$ regularity of $\varphi$}:
Now that we have the estimates \eqref{E:GCOMPOSEDWITHPHIUPGRADED}-\eqref{E:UCOMPOSEDWITHPHIUPGRADED},
we can return to the PDE \eqref{E:PMCPDEQUASIFIXED} (with $\updelta = 0$ now) and derive the additional Sobolev
regularity/estimates for up to order $N+2$ derivatives of $\varphi,$ as in the proof of \eqref{E:PHIHNBOUND}. The desired estimate \eqref{E:PHIFULLYUPGRADEDESTIMATES} thus follows.

\textbf{Step 8 - Sobolev estimates for $\hfour^{\mu \nu}[\gfour \circ \iota_{\varphi}, \partial \varphi]$}:
The estimates \eqref{E:FIRSTFUND00}-\eqref{E:FIRSTFUNDIJ} follow from definition \eqref{E:FIRSTFUNDPHI},
the estimates \eqref{E:PHIFULLYUPGRADEDESTIMATES} and \eqref{E:GCOMPOSEDWITHPHIUPGRADED}-\eqref{E:UCOMPOSEDWITHPHIUPGRADED},
and the standard Sobolev calculus.
\end{proof}

In the next corollary, we show that the fields induced on the CMC hypersurface $\Sigma$ are near-FLRW. 
This implies that they can be used as the ``data'' in our main stable singularity formation theorem.

\begin{corollary}[\textbf{Near-FLRW fields on the CMC hypersurface}] \label{C:NEARFLRWFIELDSONCMC}
Assume the hypotheses and conclusions of Prop.~\ref{P:CMCEXISTS}, and let $\Sigma$
be the CMC hypersurface provided by the proposition. Let 
$\hfour_{\mu \nu} := \gfour_{\mu \alpha} \gfour_{\nu \beta} \hfour^{\alpha \beta}$ denote the first fundamental 
form of $\Sigma$ [see \eqref{E:FIRSTFUNDPHI}],
let $\mathbf{v}_{\mu} = \hfour_{\mu \alpha} \ufour^{\alpha}$ denote the one-form that is dual
to the $\gfour-$orthogonal projection of $\ufour$ onto $\Sigma,$ and let 
$\kfour_{\mu \nu} := \gfour_{\mu \alpha} \kfour_{\ \nu}^{\alpha}$
denote the second fundamental form of $\Sigma$ [see \eqref{E:SECONDFUNDOFSIGMA}].
Let $\iota_{\varphi}$ denote the embedding 
$\iota_{\varphi}: \mathbb{T}^3 \hookrightarrow \Sigma,$ $\iota_{\varphi}(x) = \big(\varphi(x), x \big),$ 
and let $\iota_{\varphi}^*$ denote pullback\footnote{For example, $\iota_{\varphi}^* \hfour_{ij} = (\partial_i \iota_{\varphi}^{\mu})(\partial_j \iota_{\varphi}^{\nu}) 
\gfour_{\mu \alpha} \circ \iota_{\varphi} \gfour_{\nu \beta} \circ \iota_{\varphi}
\hfour^{\alpha \beta}[\gfour \circ \iota_{\varphi}, \partial \varphi].$} 
by $\iota_{\varphi}.$ Then there exists a constant $C > 0$ such that if $\epsilon$ is sufficiently small,
then the following estimates for components 
hold relative to the coordinates of Prop.~\ref{P:CMCEXISTS} (for $i,j = 1,2,3$):
\begin{align} \label{E:INDUCEDDATASMALL}
	\| \iota_{\varphi}^* \hfour_{ij} - \Euc_{ij} \|_{H^{N+1}}
	+ \left\| \iota_{\varphi}^* \kfour_{ij} + \frac{1}{3} \Euc_{i j} \right\|_{H^N}
	+ \left\| \iota_{\varphi}^* p - \frac{1}{3} \right\|_{H^N}
	+ \| \iota_{\varphi}^* \mathbf{v}_i \|_{H^N} & \leq C \epsilon,
\end{align}	
where $\Euc_{ij} = \mbox{diag}(1,1,1)$ denotes the standard Euclidean metric on $\mathbb{T}^3$
(see Def.~\ref{D:STANDARDATLAS}).

Furthermore, the fields $\iota_{\varphi}^* \hfour_{ij},$
$\iota_{\varphi}^* \kfour_{ij},$
$\iota_{\varphi}^* p,$
and $\iota_{\varphi}^* \mathbf{v}_i$
verify the Einstein constraints
\eqref{E:GAUSSINTRO}-\eqref{E:CODAZZIINTRO}.

\end{corollary}

\begin{proof}
	The estimates in \eqref{E:INDUCEDDATASMALL} follow from the estimates of Prop.~\ref{P:CMCEXISTS}, the relation 
	\eqref{E:SECONDFUNDOFSIGMA}, and the standard Sobolev calculus. The fact that the fields verify the Einstein constraints
	is a consequence of the diffeomorphism invariance of the Einstein-stiff fluid equations.
	
\end{proof}

%$\big(\newg_{ij}, \newsec_{\ j}^i, \sqrt{\gbigdet}, \newlap, \dlap_i, \newp, \newu^i \big)$
%to the Rescaled Euler-Einstein-CMC-transported system \eqref{E:RINTERMSOFKPANDU}, \eqref{E:MOMENTUMCONSTRAINT}, \eqref{E:LAPSERESCALEDELLIPTIC}, \eqref{E:LOGVOLFORMEVOLUTION}, \eqref{E:LITTLEGAMMAEVOLUTIONDETGFORM}, \eqref{E:RESCALEDKEVOLUTION}, 
%\eqref{E:GEVOLUTION}, \eqref{E:RESCALEDPEVOLUTION}, \eqref{E:RESCALEDUEVOLUTION}.
%Let $\mathring{\newg}_{ij}, \mathring{\upgamma}_{j \ k}^{\ i}, \mathring{\newsec}_{\ j}^i, \mathring{\newp}, \mathring{\newu}^i$ denote the components of the corresponding rescaled variables of Def.~\ref{D:RESCALEDVAR}. 

\subsection{Local well-posedness and continuation criteria relative to CMC-transported spatial coordinates}

By Corollary \ref{C:NEARFLRWFIELDSONCMC}, we can now assume that the perturbed spacetime contains a 
spacelike Cauchy hypersurface $\Sigma$ equipped with near-FLRW fields verifying the Einstein constraints
and with mean curvature constantly equal to $-\frac{1}{3}.$ We now discuss 
local well-posedness and continuation criteria for the Einstein-stiff fluid system relative to CMC-transported 
spatial coordinates.

\begin{theorem}[\textbf{Local well-posedness relative to CMC-transported spatial coordinates. Continuation criteria}]
\label{T:LOCAL}
Let $N \geq 4$ be an integer. Let $\big(\mathring{g}, \mathring{k}, \mathring{p}, \mathring{u} \big)$
be initial data on the manifold $\Sigma_1 \simeq \lbrace 1 \rbrace \times \mathbb{T}^3$ for the 
Einstein-stiff fluid system \eqref{E:EINSTEININTRO}-\eqref{E:EOS}, $\speed = 1$ (which by definition verify the constraints
\eqref{E:GAUSSINTRO}-\eqref{E:CODAZZIINTRO}). Assume that $\mathring{k}_{\ a}^a = - 1.$ Assume 
that relative to standard coordinates on $\Sigma_1 = \mathbb{T}^3$ (see Def.~\ref{D:STANDARDATLAS}),
the eigenvalues of the $3 \times 3$ matrix $\mathring{g}_{ij}$ are uniformly bounded from above by a positive constant 
and strictly from below by $0,$ that
\begin{align}
	\inf_{x \in \mathbb{T}^3} \mathring{p}(x) > 0,
\end{align}
and that the components of the data verify the following estimates for $i,j = 1,2,3:$
\begin{align}
		\| \mathring{g}_{ij} - \Euc_{ij} \|_{H^{N+1}}
		+ \left\| \mathring{k}_{\ j}^i + \frac{1}{3} \ID_{\ j}^i \right\|_{H^N}
		+ \left\| \mathring{p} - \frac{1}{3} \right\|_{H^N}
		+ \| \mathring{u}^i \|_{H^N} & < \infty.
	\end{align}	
Then these data launch a unique classical solution $\big(g_{ij}, \SecondFund_{\ j}^i, n, p, u^i \big)$ to the Einstein-stiff fluid CMC-transported spatial coordinates equations \eqref{E:HAMILTONIAN}-\eqref{E:MOMENTUM},
\eqref{E:PARTIALTGCMC}-\eqref{E:PARTIALTKCMC}, \eqref{E:EULERPCMC}-\eqref{E:EULERUCMC}, \eqref{E:LAPSE}. The solution exists on a non-trivial spacetime slab $(T,1] \times \mathbb{T}^3$ upon which the CMC condition 
$k_{\ a}^a = - t^{-1}$ holds and upon which its components have the following properties:
\begin{align}
	p(t,x) & > 0, & n & > 0, 
\end{align}
\begin{align}
	g_{ij} & \in C^{N-1}((T,1] \times \mathbb{T}^3); & \SecondFund_{\ j}^i & \in C^{N-2}((T,1] \times \mathbb{T}^3); \\
	n & \in C^N((T,1] \times \mathbb{T}^3); & p, u^i & \in C^{N-2}((T,1] \times \mathbb{T}^3).
\end{align}
The quantities $(\gfour_{\mu \nu}, p, \ufour^{\mu})$ verify the Einstein-stiff fluid system 
\eqref{E:EINSTEININTRO}-\eqref{E:EOS}, $\speed = 1.$ Here, $\gfour := - n^2 dt^2 + g_{ab} dx^a dx^b,$ and
$\ufour$ is the future-directed vectorfield such that $\ufour^i = u^i$ and $\gfour(\ufour, \ufour) = - 1.$
Moreover, on $(T,1] \times \mathbb{T}^3,$ the eigenvalues of the $3 \times 3$ matrix $g_{ij}$ are 
uniformly bounded from above by a positive constant
and strictly from below by $0.$ Furthermore, $\gfour_{\mu \nu}$ is a Lorentzian matrix on $(T,1] \times \mathbb{T}^3$ and
for $t \in (T,1],$ the sets $\lbrace t \rbrace \times \mathbb{T}^3$ are Cauchy hypersurfaces in the Lorentzian manifold-with-boundary
$\big((T,1] \times \mathbb{T}^3, \gfour \big).$  

The solution's components have the following Sobolev regularity:
\begin{align}
	g_{ij} - \widetilde{g}_{ij} & \in C^0((T,1],H^{N+1}); & k_{\ j}^i - \widetilde{k}_{\ j}^i & \in C^0((T,1],H^N); 
	\\
	n - \widetilde{n} & \in C^0((T,1],H^{N+2}); & p - \widetilde{p}, u^i - \widetilde{u}^i & \in C^0((T,1],H^N).
	\notag
\end{align}
Above,
\begin{align}
	\widetilde{g}_{ij} & = t^{2/3} \Euc_{ij}, & \widetilde{k}_{\ j}^i & = - \frac{1}{3} t^{-1} \ID_{\ j}^i, 
	& \widetilde{n} & = 1, & \widetilde{p} & = \frac{1}{3} t^{-2}, & \widetilde{u}^i & = 0
\end{align}
are the components of the FLRW solution. 
	
	In addition, there exists an open neighborhood $\mathcal{O}$ of
	$\big(\mathring{g}_{ij}, \mathring{k}_{\ j}^i, \mathring{p}, \mathring{u}^i \big)$
	such that all data belonging to $\mathcal{O}$ launch solutions that also exist on the slab 
	$(T, 1] \times \mathbb{T}^3$ and that have the 
	same regularity properties as $(g_{ij}, \SecondFund_{\ j}^i, n, p, u^i).$ 
	Furthermore, on $\mathcal{O},$ the map from the initial data to the 
	solution is continuous; by ``continuous,'' we mean continuous relative to the norms for the data and the 
	norms for the solution that are stated in the hypotheses and above conclusions of this theorem.

Finally, if $T_{min}$ denotes the $\inf$ over all times $T$ such that the solution exists classically and has the above 
properties, then either $T_{min} = 0,$ or one of the following \textbf{breakdown scenarios} must occur:
\begin{enumerate}
	\item There exists a sequence $\left \lbrace (t_m ,x_m) \right \rbrace_{m=1}^{\infty} \subset (T_{min},1] \times \mathbb{T}^3$ such that
		the minimum eigenvalue of the $3 \times 3$ matrix $g_{ij}(t_m,x_m)$ converges to $0$ as $m \to \infty.$
	\item There exists a sequence $\left \lbrace (t_m ,x_m) \right \rbrace_{m=1}^{\infty} \subset (T_{min},1] \times \mathbb{T}^3$ such that
		$n(t_m,x_m)$ converges to $0$ as $m \to \infty.$   
	\item There exists a sequence $\left \lbrace (t_m ,x_m) \right \rbrace_{m=1}^{\infty} \subset (T_{min},1] \times \mathbb{T}^3$ such that
		$p(t_m,x_m)$ converges to $0$ as $m \to \infty.$
	\item $\lim{t \downarrow T_{min}} \sup_{t \leq s \leq 1} \| g(s) \|_{C_{Frame}^2} + \| k(s) \|_{C_{Frame}^1} 
		+ \| n(s) \|_{C^2} + \| p(s) \|_{C^1} + \| u(s) \|_{C_{Frame}^1} = \infty.$

	%where $\highnorm{N},$ $\lowkinnorm{N},$ and $\lowpotnorm{N}$ are the norms of the rescaled solution variables from
	%Def.~\ref{D:NORMS}.
\end{enumerate}

\end{theorem}

\begin{remark}
	Conditions $(1)-(4)$ are known as \emph{continuation criteria}.
	Conditions $(1)$ and $(2)$ correspond to a breakdown in the Lorentzian nature of $\gfour.$ Condition
	$(3)$ is connected to the fact that the Euler equations can degenerate when the pressure vanishes.
\end{remark}

\begin{proof}[Sketch of a proof] 
Theorem~\ref{T:LOCAL} can be proved using the ideas of \cite[Theorem 6.2]{aS2010}, which is based on the proof of 
\cite[Theorem 10.2.2]{dCsK1993}. The main difficulty is that $R_{ij}$ cannot generally be viewed as an elliptic operator
acting on the components of the Riemannian $3-$metric $g.$ Hence, equations \eqref{E:PARTIALTGCMC}-\eqref{E:PARTIALTKCMC} do not immediately imply
that the components $g_{ij}$ verify wave equations corresponding to the wave operator of the spacetime metric $\gfour.$
The main idea behind circumventing this difficulty 
is to replace the evolution equation \eqref{E:PARTIALTKCMC} for $\SecondFund_{ij}$ with a wave equation
corresponding to the wave operator of the spacetime metric $\gfour.$
The wave equation is obtained by commuting \eqref{E:PARTIALTKCMC} with $\partial_t,$ 
and \eqref{E:PARTIALTKCMC} is treated as an additional constraint. 
The other equations are used to substitute for the terms that appear 
on the right-hand side of the wave equation for $\SecondFund_{ij}.$ If this procedure is properly implemented,
then the resulting ``modified'' system is mixed elliptic-hyperbolic 
(the elliptic part comes from the lapse equation).
Local well-posedness for such systems, including the solution properties stated 
in the conclusions of the theorem, can be
derived using the standard methods described in \cite{lAvM2003}. It is important to prove
that if a solution to the modified system initially verifies the constraints, then the 
constraints remain verified throughout the evolution. To this end, one shows that
for a solution to the modified system, the constraint quantities 
verify a homogeneous system of evolution equations for which energy methods imply the uniqueness 
of the $0$ solution; the conclusion is that the constraint quantities remain $0$ if they start out $0.$
The proof that the sets $\lbrace t \rbrace \times \mathbb{T}^3$ are Cauchy hypersurfaces
can be found in \cite[Proposition 1]{hR2008}. Finally, the continuation criteria (1) - (4) are quite standard; see e.g.
the proof of \cite[Theorem 6.4.11]{lH1997} for the main ideas.
\end{proof}

%\begin{corollary}{\textbf{Smallness of geometric data implies the initial smallness of the rescaled solution norm}}
%	\label{C:SMALLDATAIMPLIESSMALLINITIALNORM}
%	
%	\begin{align}
%		\| \mathring{g} - \Euc \|_{H^{N+1}}
%		+ \| \mathring{k} + \frac{1}{3} \ID \|_{H^N}
%		+ \| \mathring{p} - \frac{1}{3} \|_{H^N}
%		+ \| \mathring{u} \|_{H^N} & \leq \epsilon.
%	\end{align}	
%	
%	\begin{align}
%		\highnorm{N}(1) + \lowkinnorm{N}(1) + \lowpotnorm{N}(1) & \leq C \epsilon.
%	\end{align}
	
%\end{corollary}

\section{Statement and Proof of the Stable Singularity Formation Theorem} \label{S:STABLESINGULARITY}

In this section, we prove our main theorem demonstrating the global nonlinear past stability of the FLRW Big Bang spacetime. 
By Prop.~\ref{P:CMCEXISTS} and Corollary \ref{C:NEARFLRWFIELDSONCMC}, we may assume that the perturbed spacetime
contains a spacelike hypersurface $\Sigma_1$ with constant mean curvature equal to $-\frac{1}{3}$ and equipped with near-FLRW field ``data'' that verify the Einstein constraints. 
Specifically, the data are the fields $\iota_{\varphi}^* \hfour_{ij}, \iota_{\varphi}^* \kfour_{ij}, 
\iota_{\varphi}^* p, \iota_{\varphi}^* \mathbf{v}_i$ from the conclusions of Corollary \ref{C:NEARFLRWFIELDSONCMC}. 
In this section, we denote these fields by $\mathring{g}_{ij},$ $\mathring{k}_{ij},$ $\mathring{p},$ $\mathring{u}_i.$
Furthermore, we renormalize the time coordinate so that $t = 1$ along $\Sigma_1.$ There are two kinds of statements
presented in the conclusions of the theorem: \textbf{i)} existence on the entire spacetime slab $(0,1] \times \mathbb{T}^3,$ and
\textbf{ii)} sharp asymptotics/convergence estimates as $t \downarrow 0.$ As in our proof of the strong estimates of 
Prop.~\ref{P:STRONGPOINTWISE}, 
our proofs of \textbf{ii)} incur a loss in derivatives. The reason is that in deriving these estimates, we freeze the spatial point $x$ and treat the Einstein-stiff fluid equations as ODEs with small sources. The sources depend on higher-order spatial derivatives, which leads to the loss.
The main ingredients in the proof of the theorem are the a priori norm estimates of Corollary \ref{C:ENERGYINTEGRALINEQUALITIES}.

\begin{theorem}[\textbf{Main Theorem: Stable Big Bang Formation}] \label{T:BIGBANG}
Let $\big(\mathring{g}, \mathring{\SecondFund}, \mathring{p}, \mathring{u} \big)$ 
be initial data on the manifold $\Sigma_1 = \mathbb{T}^3$
for the Einstein-stiff fluid system \eqref{E:EINSTEININTRO}-\eqref{E:EOS}, $\speed = 1$ (which by definition verify the constraints
\eqref{E:GAUSSINTRO}-\eqref{E:CODAZZIINTRO}). Assume that the data verify the CMC condition 
\begin{align}
	\mathring{k}_{\ a}^a = -1.
\end{align}
Assume further that the components of the data verify the following near-FLRW condition 
relative to the standard coordinates on $\mathbb{T}^3$ (see Def.~\ref{D:STANDARDATLAS})
for some integer $N \geq 8$ and $i,j = 1,2,3:$
\begin{align} \label{E:GEOMETRICDATAARESMALL}
	\| \mathring{g}_{ij} - \Euc_{ij} \|_{H^{N+1}}
	+ \left\| \mathring{k}_{\ j}^i + \frac{1}{3} \ID_{\ j}^i \right\|_{H^N}
	+ \left\| \mathring{p} - \frac{1}{3} \right\|_{H^N}
	+ \| \mathring{u}^i \|_{H^N} & \leq \epsilon^2.
\end{align}	
Above, $\Euc_{ij} = \mbox{diag}(1,1,1)$ denotes the standard Euclidean metric on $\mathbb{T}^3$
and $\ID_{\ j}^i = \mbox{diag}(1,1,1)$ denotes the identity. Let $\big(\mathbf{M}, (\gfour, p, \ufour) \big)$ be the maximal globally hyperbolic development of the data (see the discussion in Sect.~\ref{S:INTRO}). Then there exist a small constant $\epsilon_* > 0$ and large constants $C, c > 0,$
where the constants depend on $N,$ 
such that if $\epsilon \leq \epsilon_*$ and \eqref{E:GEOMETRICDATAARESMALL} holds, then the following conclusions hold. 
\begin{itemize}
\item The field variables $(\gfour, p, \ufour)$ are classical solutions to the Einstein-stiff fluid equations \eqref{E:EINSTEININTRO}-\eqref{E:EOS},
	$\speed = 1.$
\item There exists a collection of CMC-transported spatial coordinates $(t,x^1,x^2,x^3)$ 
	covering $\mathbf{V} := \cup_{t \in (0, 1]} \Sigma_t \simeq (0, 1] \times \mathbb{T}^3,$ 
	where $\mathbf{V}$ is the past of $\Sigma_1$ in $\mathbf{M}.$ The coordinate $t$ is a time function 
	defined on $\mathbf{V},$ the $(x^1,x^2,x^3)$ are (spatially locally defined) transported spatial coordinates 
	(see Def.~\ref{D:STANDARDATLAS}), and each 
	$\Sigma_s := \lbrace q \in \mathbf{V} \ | \ t(q) = s \rbrace \simeq \mathbb{T}^3$ is a Cauchy hypersurface. The CMC condition 
	$\SecondFund_{\ a}^a = - t^{-1}$ holds along each $\Sigma_t.$ Relative to these coordinates, the FLRW solution can be expressed as
\begin{align} \label{E:MAINPROOFFLRWSTANDARDCOORDS}
	\widetilde{\gfour} & = - dt^2 + t^{2/3} \sum_{i=1}^3 (dx^i)^2, 
		& \widetilde{p} & = \frac{1}{3} t^{-2}, 
		& \widetilde{\ufour}^{\mu} & = \delta_0^{\mu},
\end{align}
where $\delta_{\nu}^{\mu}$ ($\mu, \nu = 0,1,2,3$) is the standard Kronecker delta.
\item Let $(g_{ij}, \SecondFund_{\ j}^i, n, p, u^i)_{1 \leq i,j \leq 3}$ denote the components of the perturbed solution relative to the 
CMC-transported spatial coordinates, where
%$(\newsec_{\ j}^i = t \SecondFund_{\ j}^i, \newg_{ij} = t^{-2/3} g_{ij}, \upgamma_{j \ k}^{\ i},\sqrt{\gbigdet}, \newlap, \dlap_i, \newp, \newu^i)$ 
\begin{subequations}
\begin{align}
	\gfour & = - n^2 dt^2 + g_{ab} dx^a dx^b, \\
	\ufour & = (1 + g_{ab} u^a u^b)^{1/2} \Nml + u^a \partial_a
\end{align}
\end{subequations}
are respectively the spacetime metric 
and the fluid four-velocity, $\Nml = n^{-1} \partial_t,$ 
$\SecondFund_{\ j}^i = - \frac{1}{2} n^{-1} g^{ia} \partial_t g_{aj}$ 
are the components of the mixed second fundamental form of $\Sigma_t,$
$n$ is the lapse, and
$p$ is the fluid pressure.
The following norm estimates (see Def.~\ref{D:NORMS})
are verified by the renormalized solution variables' frame components
(see Def.~\ref{D:RESCALEDVAR}) 
for $t \in (0,1]:$ 
\begin{subequations}
	\begin{align} 
		\highnorm{N}(t) & \leq \epsilon t^{- c \sqrt{\epsilon}}, 
			\label{E:MAINTHEOREMNORMESTIMATE} \\
		\lowkinnorm{N-3}(t) & \leq \epsilon, \\
 		\lowpotnorm{N-4}(t) & \leq \epsilon t^{- c \sqrt{\epsilon}}.
		\label{E:MAINTHEOREMPOTNORMESTIMATE}
	\end{align}
	\end{subequations}

\end{itemize}

In addition, the solution has the following properties.
\ \\

\noindent{\textbf{Causal disconnectedness:}} Let $\pmb{\zeta}(s)$ be a past-directed causal curve in $\big((0,1] \times \mathbb{T}^3, \gfour \big)$ with domain $s \in [s_1, s_{Max})$ such that $\pmb{\zeta}(s_1) \in \Sigma_t.$ Let $\pmb{\zeta}^{\mu}$ denote the coordinates of this curve in the universal covering space of the spacetime (i.e., $(0,1] \times \mathbb{R}^3).$ The length 
$\ell[\pmb{\zeta}] := \int_{s_1}^{s_{Max}} \sqrt{(\Euc_{ab} \circ \pmb{\zeta}) \dot{\zeta}^a \dot{\zeta}^b} \, ds$ of the spatial part of the curve as measured by the Euclidean metric is bounded from above by
\begin{align} \label{E:LENGTHEST}
	\ell[\pmb{\zeta}]
	\leq 
		\left(\frac{3}{2} + C \epsilon \right) t^{2/3 - c \epsilon}.
\end{align}
The constant $C$ \textbf{can be chosen to be independent of} the curve $\pmb{\zeta}.$ Thus, if 
$q,r \in \Sigma_t$ are separated by a Euclidean distance greater than 
$2\left(\frac{3}{2} + C \epsilon \right) t^{2/3 - c \epsilon},$
then the past of $q$ does not intersect the past of $r.$ 
\ \\

\noindent{\textbf{Geodesic incompleteness:}}
Every past-directed causal geodesic $\pmb{\zeta}$ that emanates from $\Sigma_1$ crashes into the singular hypersurface $\Sigma_0$ in finite affine parameter time
	\begin{align} \label{E:AFFINTEBLOWUPTIME}
	\mathscr{A}(0) \leq \left(\frac{3}{2} + C \epsilon \right) \left|\mathscr{A}^{'}(1)\right|,
\end{align}	
where $\mathscr{A}(t)$ is the affine parameter along $\pmb{\zeta}$ viewed as a function of $t$ along $\pmb{\zeta}$ 
(normalized by $\mathscr{A}(1) = 0$).
\ \\

\noindent{\textbf{Convergence of time-rescaled variables:}}	
	There exist functions $\upsilon_{Bang}, \newp_{Bang} \in C^{N-3}(\mathbb{T}^3)$ 
	and a type $\binom{1}{1}$ tensorfield $(\newsec_{Bang})_{\ j}^i \in C^{N-3}(\mathbb{T}^3)$ 
	such that the lapse, time-rescaled volume form factor, 
	time-rescaled mixed second fundamental form, time-rescaled pressure,
	three-velocity, and four-velocity $\Sigma_t-$normal component
	$\gfour(\ufour, \Nml) = - (1 + g_{ab} u^a u^b)^{1/2}$
	verify the following convergence estimates for $t \in (0,1]:$
\begin{subequations}
\begin{align}
	\left \| n - 1  \right \|_{C^M} 
	& \leq \left\lbrace \begin{array}{ll} 
		C \epsilon t^{4/3 - c \sqrt{\epsilon}}, & (M \leq N-5), \\
	  C \epsilon t^{2/3 - c \sqrt{\epsilon}}, & (M \leq N-3),
		\end{array} \right.
		\label{E:LAPSELIMIT} \\
	\left \| t^{-1} \sqrt{\gdet} - \upsilon_{Bang} \right \|_{C_{Frame}^M} 
	& \leq \left\lbrace \begin{array}{ll} 
		C \epsilon t^{4/3 - c \sqrt{\epsilon}}, & (M \leq N-5), \\
	  C \epsilon t^{2/3 - c \sqrt{\epsilon}}, & (M \leq N-3),
		\end{array} \right.
		\label{E:VOLFORMLIMIT} \\
	\left \| t \SecondFund - \newsec_{Bang} \right \|_{C_{Frame}^M} 
	& \leq \left\lbrace \begin{array}{ll} 
		C \epsilon t^{4/3 - c \sqrt{\epsilon}}, & (M \leq N-5), \\
	  C \epsilon t^{2/3 - c \sqrt{\epsilon}}, & (M \leq N-3),
		\end{array} \right.
		\label{E:KLIMIT} \\
	\left \| t^2 p - \newp_{Bang} \right \|_{C^M} & \leq \left\lbrace \begin{array}{ll} 
		C \epsilon t^{4/3 - c \sqrt{\epsilon}}, & (M \leq N-5), \\
	  C \epsilon t^{2/3 - c \sqrt{\epsilon}}, & (M \leq N-3),
		\end{array} \right.
		\label{E:PLIMIT}  \\
	\left \| u \right \|_{C_{Frame}^{N-4}} & \leq C \epsilon t^{1/3 - c \sqrt{\epsilon}},
		\label{E:ULIMIT} \\
	\left \| \gfour(\ufour, \Nml) + 1 \right \|_{C^{N-4}} & \leq C \epsilon t^{4/3 - c \sqrt{\epsilon}}.
		\label{E:UNORMALLIMIT}
\end{align}
\end{subequations}
In the above estimates, the frame norms $\| \cdot \|_{C_{Frame}^M}$ are defined in Def.~\ref{D:CMNORMS}.
Furthermore, the limiting fields are close to the corresponding time-rescaled FLRW fields in the following sense:
\begin{subequations}
\begin{align}
\left \| \upsilon_{Bang} - 1 \right \|_{C^{N-3}} & \leq C \epsilon, 
	\label{E:VOLFORMLIMITCLOSETOONE} \\
\left \| \newsec_{Bang} + \frac{1}{3} \ID  \right\|_{C_{Frame}^{N-3}} & \leq C \epsilon, 
	\label{E:KLIMITCLOSETOFLRW} \\
\left \| \newp_{Bang} - \frac{1}{3} \right\|_{C^{N-3}} & \leq C \epsilon.
	\label{E:PLIMITCLOSETOFLRW}
\end{align}
\end{subequations}

In addition, the limiting fields verify the following relations:
\begin{subequations}
\begin{align} 
	(\newsec_{Bang})_{\ a}^a & = - 1, 
		\label{E:LIMITINGKTRACE} \\
	2 \newp_{Bang} + (\newsec_{Bang})_{\ b}^a (\newsec_{Bang})_{\ a}^b & = 1.
		\label{E:LIMITINGFIELDCONSTRAINT}
\end{align}
\end{subequations}

\ \\

\noindent{\textbf{Behavior of the spatial metric:}}
There exists a type $\binom{0}{2}$ tensorfield $M_{ij}^{Bang} \in C^{N-3}$ on $\mathbb{T}^3$
such that
\begin{align} \label{E:METRICENDSTATE}
	\left \| M^{Bang} - \Euc \right \|_{C_{Frame}^{N-3}} & \leq C \epsilon 
\end{align}
and such that the following convergence estimates hold for $t \in (0,1]:$
\begin{align} \label{E:LIMITINGMETRICBEHAVIOR}
	\left\| g * \mbox{exp} \left(2 \ln t \newsec_{Bang} \right) 
		- M^{Bang} \right\|_{C_{Frame}^M} 
	& \leq \left\lbrace \begin{array}{ll} 
		C \epsilon t^{4/3 - c \sqrt{\epsilon}}, & (M \leq N-5), \\
	  C \epsilon t^{2/3 - c \sqrt{\epsilon}}, & (M \leq N-3),
		\end{array} \right.
\end{align}
where $g * \cdot$ denotes left multiplication of the type $\binom{1}{1}$ matrix $\cdot$
by the type $\binom{0}{2}$ matrix $g_{ij}.$

\ \\
%\begin{subequations}
%\begin{align}
%	\| \newg - \Euc \| & \leq C \epsilon t^{- c \epsilon}, \\	
%	\| \upgamma \| & \leq C \epsilon t^{- c \epsilon}, \\
%	\| \newu \| & \leq C \epsilon + C (t^{- c \epsilon} - 1), \\
%	\| \newlap \| + \| \dlap \| & \leq C \epsilon t^{- c \epsilon}. 
%\end{align}
%\end{subequations}

\noindent{\textbf{Quantities that blow up:}} The $|\cdot|_g$ norm of the 
second fundamental form $\SecondFund$ of $\Sigma_t$ verifies the estimate
\begin{align} \label{E:SECFUNDFORMBLOWUP}
	\left \| t |\SecondFund|_g - |(\newsec_{Bang})_{\ b}^a (\newsec_{Bang})_{\ a}^b|^{1/2} \right \|_{C^0} 
		& \leq C \epsilon t^{4/3 - c \sqrt{\epsilon}},
\end{align}
which shows that $|\SecondFund|_g$ blows up like $t^{-1}$ as $t \downarrow 0.$ 

The $|\cdot|_g$ norm of the Riemann curvature of $g$ verifies the estimate
\begin{align} \label{E:3RIEMANNBLOWUP}
	\left \| |\Riemann|_g \right \|_{C^0} & \leq C \epsilon t^{-2/3 - c \epsilon}.
\end{align}

%The spacetime scalar curvature $\Rfour$ verifies the estimate

%\begin{align} \label{E:SPACETIMESCALARCURVATUREBLOWUP}
%	\| t^2 \Rfour + 2 \newp_{Bang} \|_{C^0} & \leq C \epsilon t^{2/3 - c \sqrt{\epsilon}},
%\end{align}
%which shows that $\Rfour$ blows up like $t^{-2}$ as $t \downarrow 0.$

The spacetime Ricci curvature invariant $|\Ricfour|_{\gfour}^2$
verifies the estimate
\begin{align} \label{E:SPACETIMERICCIINVARIANTBLOWUP}
	\left \| t^4 |\Ricfour|_{\gfour}^2
		- 4 \newp_{Bang}^2 \right \|_{C^0} & \leq C \epsilon t^{4/3 - c \sqrt{\epsilon}},
\end{align}
which shows that $|\Ricfour|_{\gfour}^2$
blows up like $t^{-4}$ as $t \downarrow 0.$

The $|\cdot|_{\gfour}^2$ norm of the spacetime Riemann curvature tensor verifies the estimates
\begin{subequations}
\begin{align} \label{E:SPACETIMEKRETSCHMANNBLOWUP}
	\left \|t^4 |\Riemfour|_{\gfour}^2 - F_{Bang}  \right \|_{C^0} & \leq C \epsilon t^{4/3 - c \sqrt{\epsilon}}, \\
	F_{Bang} & := \Big\lbrace 2 (\newsec_{\ b}^a \newsec_{\ a}^b)^2 + 4 \newsec_{\ b}^a \newsec_{\ a}^b
		+ 2 \newsec_{\ b}^a \newsec_{\ c}^b \newsec_{\ d}^c \newsec_{\ a}^d
		+ 8 \newsec_{\ b}^a \newsec_{\ c}^b \newsec_{\ a}^c \Big\rbrace|_{\newsec = \newsec_{Bang}}, \\
	\left\| F_{Bang} - \frac{20}{27} \right\|_{C^0} & \leq C \epsilon,
	\label{E:SPACETIMEKRETSCHMANNLIMITNEARLYCONSTANT}
\end{align}
\end{subequations}
which shows that $|\Riemfour|_{\gfour}^2$ blows up like $t^{-4}$ as $t \downarrow 0.$

The $|\cdot|_{\gfour}^2$ norm of the spacetime tensor 
$\mathbf{P}_{\alpha \beta \mu \nu} = \Riemfour_{\alpha \beta \mu \nu} - \Wfour_{\alpha \beta \mu \nu}$ 
verifies the estimate
\begin{align} \label{E:SPACETIMEPTENSORBLOWUP}
	\left\|t^4 |\mathbf{P}|_{\gfour}^2 - \frac{20}{3} \newp_{Bang}^2 \right\|_{C^0} 
		& \leq C \epsilon t^{4/3 - c \sqrt{\epsilon}}, 
\end{align}
which shows that $|\mathbf{P}|_{\gfour}^2$ blows up like $t^{-4}$ as $t \downarrow 0.$
Above, $\Wfour_{\alpha \beta \mu \nu}$ is the spacetime Weyl curvature tensor.

The $|\cdot|_{\gfour}^2$ norm of $\Wfour_{\alpha \beta \mu \nu}$ verifies the estimates
\begin{subequations}
\begin{align} \label{E:SPACETIMEWEYLBLOWUP}
	\left\|t^4 |\Wfour|_{\gfour}^2 - (F_{Bang} - \frac{20}{3} \newp_{Bang}^2) \right\|_{C^0} 
		& \leq C \epsilon t^{4/3 - c \sqrt{\epsilon}}, \\
	\left\| F_{Bang} - \frac{20}{3} \newp_{Bang}^2 \right\|_{C^0} & \leq C \epsilon,
	\label{E:FBANGVSBANG}
\end{align}
\end{subequations}
which, when combined with the relation $|\Riemfour|_{\gfour}^2 = |\mathbf{P}|_{\gfour}^2 + |\Wfour|_{\gfour}^2$
and the estimates \eqref{E:SPACETIMEKRETSCHMANNBLOWUP} and \eqref{E:SPACETIMEPTENSORBLOWUP}, 
shows that the dominant contribution to the $|\Riemfour|_{\gfour}^2$ singularity at $t = 0$
comes from the $\Ricfour$ components of $\Riemfour.$
\end{theorem}

\begin{remark}
 Some of the powers of $\epsilon$ stated in the above estimates are non-optimal and could be improved with
 additional effort.
\end{remark}

\begin{proof}[Proof of Theorem~\ref{S:STABLESINGULARITY}]
	Let $\upsigma_*$ be the positive constant from Prop.~\ref{P:SOBFORG},
	Prop.~\ref{P:ENERGYINTEGRALINEQUALITIES},
	and Corollary \ref{C:ENERGYINTEGRALINEQUALITIES}. By local well-posedness (Theorem~\ref{T:LOCAL}), if \eqref{E:GEOMETRICDATAARESMALL} holds and 
	$\epsilon_*$ is sufficiently small, then the initial data launch a unique classical renormalized solution\footnote{Recall that the renormalized equations
	are equivalent to the original equations.} 
	to the renormalized constraints \eqref{E:RINTERMSOFKPANDU}-\eqref{E:MOMENTUMCONSTRAINT}, the renormalized
	lapse equations \eqref{E:LAPSERESCALEDELLIPTIC}-\eqref{E:LAPSELOWERDERIVATIVES}, and the 
	renormalized evolution equations 
	\eqref{E:LOGVOLFORMEVOLUTION}, \eqref{E:GEVOLUTION}-\eqref{E:GINVERSEEVOLUTION}, 
	\eqref{E:LITTLEGAMMAEVOLUTIONDETGFORM}-\eqref{E:RESCALEDKEVOLUTION},
	and \eqref{E:RESCALEDPEVOLUTION}-\eqref{E:RESCALEDUEVOLUTION}.
	The renormalized solution exists on a maximal slab $(T,1] \times \mathbb{T}^3$ upon which the 
	following bootstrap assumptions hold for the renormalized solution variable norms:
	 \begin{align} 
		\highnorm{N}(t) & \leq \epsilon t^{-\upsigma_*}, 
			\label{E:HIGHBOOTPROOF} \\
		\lowkinnorm{N-3}(t) & \leq 1, 
			\label{E:LOWKINBOOTPROOF} \\
		\lowpotnorm{N-4}(t) & \leq t^{-\upsigma_*}. 
		\label{E:LOWPOTBOOTPROOF}
	\end{align}
	In deriving \eqref{E:HIGHBOOTPROOF}-\eqref{E:LOWPOTBOOTPROOF}, 
	we have used the standard Sobolev calculus to deduce the smallness of the renormalized variable norms
	\eqref{E:HIGHNORM}-\eqref{E:LOWPOTNORM} from the conclusions 
	of Theorem~\ref{T:LOCAL}, which are stated in terms of the original variables. 
	In fact, by combining the estimates of Theorem~\ref{T:LOCAL} with standard
	elliptic estimates for $n,$ we deduce 
	that the perturbed renormalized solution remains $C \epsilon^2$ close to the 
	renormalized FLRW solution for times $t$ near $1$.
	
	It follows that if $\epsilon_*$ is sufficiently small, then all of the assumptions of
	Corollary \ref{C:ENERGYINTEGRALINEQUALITIES} hold.
	Thus, we conclude from the corollary that there exists a constant $c > 0$ such that 
	the following estimates necessarily hold for $t \in (T,1]:$
	\begin{align} 
		\highnorm{N}(t) & \leq \frac{1}{2} \epsilon t^{- c \sqrt{\epsilon}}, 
			\label{E:HIGHBOOTIMPROVED} \\
		\lowkinnorm{N-3}(t) & \leq \frac{1}{2} \epsilon, 
			\label{E:LOWKINBOOTIMPROVED} \\
		\lowpotnorm{N-4}(t) & \leq \frac{1}{2} \epsilon t^{c \sqrt{\epsilon}}. 
		\label{E:LOWPOTBOOTIMPROVED}
	\end{align} 
	By further shrinking $\epsilon_*$ if necessary, we may assume that
	$c \sqrt{\epsilon} < \upsigma^*.$ 
	We now show that $(T,1] \times \mathbb{T}^3$ must be equal to $(0,1] \times \mathbb{T}^3$
	and that \eqref{E:MAINTHEOREMNORMESTIMATE}-\eqref{E:MAINTHEOREMPOTNORMESTIMATE} hold for $t \in (0,1].$
	Suppose for the sake of contradiction that $T \in (0,1).$ 
	We then observe that none of the four breakdown scenarios from Theorem~\ref{T:LOCAL} can occur on $(T,1] \times \mathbb{T}^3:$
	 \begin{itemize}
	 	\item (1) is ruled out by \eqref{E:GINVERSECOMMUTEDSTRONGPOINTWISE}, which shows that the eigenvalues of the $3 \times 3$ 
	 		matrix 
	 		$t^{2/3} g^{ij} := (\newg^{-1})^{ij}$ are bounded from above in magnitude by $1 + C \epsilon t^{- c \epsilon};$
			hence the eigenvalues of $t^{-2/3} g_{ij} := \newg_{ij}$ are bounded from below in magnitude by 
			$(1 + C \epsilon t^{- c \epsilon})^{-1}$ and can never turn negative since they were positive at $t = 1.$
		\item (2) is ruled out by \eqref{E:LAPSECOMMUTEDSTRONGPOINTWISE} and the relation $n := 1 + t^{4/3} \newlap.$
		\item (3) is ruled out by \eqref{E:PCOMMUTEDSTRONGPOINTWISE} and the relation $t^2 p := \adjustednewp + \frac{1}{3}.$
	 	\item (4) is ruled out by the estimates \eqref{E:HIGHBOOTIMPROVED}-\eqref{E:LOWPOTBOOTIMPROVED} and Sobolev embedding.
	 \end{itemize}
	 
		It therefore follows from the time-continuity of the solution norms 
		that the solution can be extended to exist on a strictly larger slab $(T - \Delta,1] \times \mathbb{T}^3$ 
		(for some number $\Delta > 0$)
		such that on $(T - \Delta,1],$ the estimates \eqref{E:HIGHBOOTIMPROVED}-\eqref{E:LOWPOTBOOTIMPROVED}
		hold with $\frac{1}{2}$ replaced by $\frac{3}{4}.$ We have therefore derived 
		strict improvements of \eqref{E:HIGHBOOTPROOF}-\eqref{E:LOWPOTBOOTPROOF}] on a time interval strictly larger than $(T,1].$ 
		In total, we have contradicted the maximality of the slab $(T,1] \times \mathbb{T}^3.$ 
		We conclude that the solution exists on $(0,1] \times \mathbb{T}^3$
		and that the estimates \eqref{E:MAINTHEOREMNORMESTIMATE}-\eqref{E:MAINTHEOREMPOTNORMESTIMATE} hold for $t \in (0,1].$
	
\ \\
	
\noindent \textbf{Proof of \eqref{E:LENGTHEST}:} Let $\pmb{\zeta} = \pmb{\zeta}(s)$ be a curve verifying the hypotheses of the theorem. Here we use the notation $\dot{\pmb{\zeta}} := \frac{d}{ds}\pmb{\zeta}(s).$ Note that the component
$\pmb{\zeta}^0$ can be identified with the CMC time coordinate. For any causal curve,
we have $\gfour(\dot{\pmb{\zeta}},\dot{\pmb{\zeta}}) \leq 0,$ which implies that
\begin{align} \label{E:GEODESICSPATIALVELOCITYBOUNDEDBY0VELOCITY}
	g_{ab}\dot{\zeta}^a\dot{\zeta}^b \leq n^2 (\dot{\pmb{\zeta}}^0)^2. 
\end{align}
Using the previously established fact that the eigenvalues of $\newg_{ij}(t,x)$ are bounded from below by
$(1 + C \epsilon t^{- c \epsilon})^{-1},$ the estimate \eqref{E:LAPSECOMMUTEDSTRONGPOINTWISE}, 
and the fact that $- \dot{\pmb{\zeta}}^0 > 0$ for past-directed causal curves, we deduce 
from \eqref{E:GEODESICSPATIALVELOCITYBOUNDEDBY0VELOCITY} that
\begin{align} \label{E:EUCLIDEANVELOCITYESTIMATE}
	(\Euc_{ab} \circ \pmb{\zeta}) \dot{\zeta}^a \dot{\zeta}^b 
	& \leq ([(1 + C \epsilon t^{- c \epsilon}) \newg_{ab}] \circ \pmb{\zeta}) 
		\dot{\zeta}^a \dot{\zeta}^b \\
	& \leq (1 + C \epsilon) (\pmb{\zeta}^0)^{-2/3 - c \epsilon} 
	(\dot{\pmb{\zeta}}^0)^2. \notag
\end{align} 
Integrating the square root of \eqref{E:EUCLIDEANVELOCITYESTIMATE}
and using the assumption $\pmb{\zeta}^0(s_1) = t,$ we arrive at the desired
estimate \eqref{E:LENGTHEST}:
\begin{align}
	\int_{s_1}^{s_{Max}} \sqrt{(\Euc_{ab} \circ \pmb{\zeta}) \dot{\zeta}^a \dot{\zeta}^b} \, ds 
		& \leq (1 + C \epsilon) \int_{s_1}^{s_{Max}} (\pmb{\zeta}^0)^{-1/3 - c \epsilon} (-\dot{\pmb{\zeta}}^0) \, ds
		\\
		& = \frac{-(1 + C \epsilon)}{2/3 - c \epsilon} 
			\int_{s_1}^{s_{Max}} \frac{d}{ds} \big( (\pmb{\zeta}^0)^{2/3 - c \epsilon} \big) \, ds
		\notag
		\\
		& = \frac{(1 + C \epsilon)}{2/3 - c \epsilon} 
				\Big\lbrace t^{2/3 - c \epsilon} - s_{Max}^{2/3 - c \epsilon} \Big\rbrace
			\notag \\
		& \leq \left(\frac{3}{2} + C \epsilon \right) t^{2/3 - c \epsilon}.
		\notag
\end{align}
\ \\

\noindent \textbf{Proof of \eqref{E:AFFINTEBLOWUPTIME}:} Let $\pmb{\zeta}(\mathscr{A})$ be a past-directed affinely parametrized geodesic verifying the hypotheses of the theorem. We can view the affine parameter $\mathscr{A}$ as a function of $t = \pmb{\zeta}^0$ along $\pmb{\zeta}.$ We normalize $\mathscr{A}(t)$ by setting $\mathscr{A}(1) = 0.$ We also define $\dot{\pmb{\zeta}}^{\mu} := \frac{d}{d \mathscr{A}} \pmb{\zeta}^{\mu}$ and $\mathscr{A}' = \frac{d}{dt} \mathscr{A}.$ 
By the chain rule, we have 
\begin{align} 
	\mathscr{A}^{'} 
	& = \frac{1}{\dot{\pmb{\zeta}}^0}, 
	&&
	\ddot{\pmb{\zeta}}^0 
	= \dot{\pmb{\zeta}}^0 \frac{d}{dt} \dot{\pmb{\zeta}}^0 = - (\mathscr{A}^{'})^{-3} \mathscr{A}^{''}.
		\label{E:GEODESICDOTANDDDOTCHAINRULE}
\end{align}
Using the geodesic equation \eqref{E:GEO0}, the $g-$Cauchy-Schwarz inequality,
\eqref{E:GEODESICSPATIALVELOCITYBOUNDEDBY0VELOCITY},
and
\eqref{E:GEODESICDOTANDDDOTCHAINRULE}, we deduce
\begin{align} \label{E:ADOUBLEPRIMEFIRSTEQUATION}
	\left|\mathscr{A}^{''} \right| & \leq 
			\big(n^{-1} |\partial_t n| + 2 |\partial n|_g \big) |\mathscr{A}^{'}|
			+ n |\SecondFund|_g |\mathscr{A}^{'}|.
\end{align}
Inserting the estimates \eqref{E:SECONDFUNDUPGRADEPOINTWISEGNORM} (in the case $M=0$),
\eqref{E:LAPSECOMMUTEDSTRONGPOINTWISE}, \eqref{E:GNORMLAPSECOMMUTEDSTRONGPOINTWISE} and \eqref{E:PARTIALTLAPSESTRONGPOINTWISE} into \eqref{E:ADOUBLEPRIMEFIRSTEQUATION}, we deduce 
\begin{align} \label{E:DDTAFFINEGRONWALLREADY}
	\left|\mathscr{A}^{''} \right| & \leq t^{-1} \left(\frac{1}{3} + c \epsilon \right) |\mathscr{A}^{'}|,
	& t & \in (0,1]. 
\end{align}
Applying Gronwall's inequality to \eqref{E:DDTAFFINEGRONWALLREADY}, we deduce
\begin{align} \label{E:DDTAFFINEESTIMATED}
	|\mathscr{A}^{'}(t)| & \leq |\mathscr{A}^{'}(1)| t^{-(1/3 + c \epsilon)}, & t & \in (0,1].
\end{align}
Integrating \eqref{E:DDTAFFINEESTIMATED} from $t = 1$ and using $\mathscr{A}(1)=0,$ we deduce
\begin{align} \label{E:AFFINEESTIMATE}
	\mathscr{A}(t) & \leq \frac{|\mathscr{A}^{'}(1)|}{2/3 - c \epsilon}(1 - t^{2/3 - c \epsilon}), & t  & \in (0,1],
\end{align}
from which the desired estimate \eqref{E:AFFINTEBLOWUPTIME} follows.

\ \\

\noindent \textbf{Proof of \eqref{E:LAPSELIMIT}-\eqref{E:UNORMALLIMIT}, \eqref{E:VOLFORMLIMITCLOSETOONE}-\eqref{E:PLIMITCLOSETOFLRW},
and \eqref{E:LIMITINGKTRACE}-\eqref{E:LIMITINGFIELDCONSTRAINT}:} 
The estimate \eqref{E:ULIMIT} follows directly from \eqref{E:MAINTHEOREMPOTNORMESTIMATE}. 
The estimate \eqref{E:UNORMALLIMIT} then easily follows from inequality \eqref{E:GCOMMUTEDSTRONGPOINTWISE}
and the estimate \eqref{E:ULIMIT}.

Inequality \eqref{E:LAPSELIMIT} follows from \eqref{E:LAPSECOMMUTEDSTRONGPOINTWISE}
(for $M \leq N - 5$) and \eqref{E:MAINTHEOREMNORMESTIMATE} plus the 
the Sobolev embedding estimate $\| n - 1 \|_{C^{N-3}} \lesssim t^{2/3} \highnorm{N}$ 
(for $M = N - 4, N - 3$).

To prove \eqref{E:VOLFORMLIMIT}, we 
use equation \eqref{AE:LOGVOLFORMEVOLUTIONCOMMUTED} and inequality \eqref{E:LAPSELIMIT}
to deduce that
\begin{align}
		\left\| \partial_t \ln \left(\sqrt{\gbigdet} \right) \right\|_{C^M} 
		& \lesssim \left\lbrace \begin{array}{ll} 
		\epsilon t^{1/3 - c \sqrt{\epsilon}}, & (M \leq N-5), \\
	  \epsilon t^{-1/3 - c \sqrt{\epsilon}}, & (M \leq N-3).
		\end{array} \right.
		\label{E:LIMITPROOFVOLFORMLOWERORDERTIMEDERIVATIVES}
	\end{align}
	If $0 < t_- \leq t_+ \leq 1,$ then it follows from integrating \eqref{E:LIMITPROOFVOLFORMLOWERORDERTIMEDERIVATIVES}
	in time that
	\begin{align} \label{E:VOLFORMCAUCHYSEQUENCEESTIMATE}
			\left\| \ln \left(\sqrt{\gbigdet} \right)(t_-) - \ln \left(\sqrt{\gbigdet} \right)(t_+) \right \|_{C^M} 
			& \lesssim \left\lbrace \begin{array}{ll} 
			\epsilon  \int_{t = 0}^{t_+} t^{1/3 - c \sqrt{\epsilon}} \, dt, & (M \leq N-5), \\
	  	\epsilon  \int_{t = 0}^{t_+} t^{-1/3 - c \sqrt{\epsilon}} \, dt, & (M \leq N-3)
			\end{array} \right. \\
			& \lesssim \left\lbrace \begin{array}{ll} 
				\epsilon t_+^{4/3 - c \sqrt{\epsilon}}, & (M \leq N-5), \\
	  		\epsilon t_+^{2/3 - c \sqrt{\epsilon}}, & (M \leq N-3).
				\end{array} \right. 
				\notag
	\end{align}
	Hence, if $\lbrace t_m \rbrace_{m=1}^{\infty}$ is a sequence of positive times 
	such that $t_m \downarrow 0,$ then \eqref{E:VOLFORMCAUCHYSEQUENCEESTIMATE} implies that 
	the sequence $\ln \left(\sqrt{\gbigdet} \right)(t_m,\cdot) \in C^M$ is Cauchy in the norm 
	$\| \cdot \|_{C^M}.$ The existence of a limiting function $\upsilon_{Bang} \in C^{N-3}$ and the desired estimate 
	\eqref{E:VOLFORMLIMIT} thus follow. The estimate \eqref{E:VOLFORMLIMITCLOSETOONE} then follows from the small-data estimate
	$\left\| \sqrt{\gbigdet}(1) - 1 \right \|_{C^{N-3}} \lesssim \epsilon.$  The estimates 
	\eqref{E:KLIMIT}-\eqref{E:PLIMIT} and \eqref{E:KLIMITCLOSETOFLRW}-\eqref{E:PLIMITCLOSETOFLRW} follow similarly from
	the evolution equations \eqref{AE:SECONDFUNDCOMMUTED} and \eqref{AE:PARTIALTPCOMMUTED},
	the strong estimates of Prop.~\ref{P:STRONGPOINTWISE}, and the Sobolev norm bound 
	\eqref{E:MAINTHEOREMNORMESTIMATE}; we omit the tedious but straightforward details.
	
	The relation \eqref{E:LIMITINGKTRACE} is a trivial consequence of the CMC condition $\SecondFund_{\ a}^a = - t^{-1}$
	and \eqref{E:KLIMIT}.
	
	To prove \eqref{E:LIMITINGFIELDCONSTRAINT}, we multiply the
	renormalized Hamiltonian constraint \eqref{E:RINTERMSOFKPANDU} by $t^2$
	and let $t \downarrow 0.$ Lemma~\ref{L:RICCIDECOMP} and the estimates
	\eqref{E:GCOMMUTEDSTRONGPOINTWISE}, \eqref{E:GINVERSECOMMUTEDSTRONGPOINTWISE}, and \eqref{E:LITTLEGAMMACOMMUTEDSTRONGPOINTWISE} 
	imply that $t^2|R| \to 0,$ where $R$ denotes the scalar curvature of $g_{ij}.$ 
	Similarly, the estimates \eqref{E:GCOMMUTEDSTRONGPOINTWISE},
	\eqref{E:PCOMMUTEDSTRONGPOINTWISE}, and \eqref{E:UCOMMUTEDSTRONGPOINTWISE} imply that
	$t^{4/3} |\leftexp{(Junk)}{\mathfrak{H}}| \to 0.$ Also using the 
	estimates \eqref{E:KLIMIT} and \eqref{E:PLIMIT}, we arrive at the desired result \eqref{E:LIMITINGFIELDCONSTRAINT}.
	
	%We remark that the first estimates in \eqref{E:KLIMIT} and \eqref{E:KLIMITCLOSETOFLRW} require the symmetry
	%property $\SecondFund_{ij} = \SecondFund_{ji}.$
	
\ \\

\noindent \textbf{Proof of \eqref{E:METRICENDSTATE} and \eqref{E:LIMITINGMETRICBEHAVIOR}:}	
From equation \eqref{E:PARTIALTGCMC} and the estimates \eqref{E:LAPSELIMIT}, \eqref{E:KLIMIT}, and \eqref{E:KLIMITCLOSETOFLRW}, 
it follows that	
\begin{align} \label{E:LITTLEGEVOUSEDINPROOF}
	\partial_t g_{ij} & = - 2 t^{-1} g_{ia} (\newsec_{Bang})_{\ j}^a + g_{ia} \Delta_{\ j}^a, 
\end{align}	 
where the following estimate for the components $\Delta_{\ j}^i$ holds:
\begin{align}	\label{E:GERRORTERM}
	\| \Delta_{\ j}^i \|_{C^M} 
	& \lesssim \left\lbrace \begin{array}{ll} 
		\epsilon t^{1/3 - c \sqrt{\epsilon}}, & (M \leq N-5), \\
	  \epsilon t^{-1/3 - c \sqrt{\epsilon}}, & (M \leq N-3).
		\end{array} \right. 
\end{align}
Introducing the matrix integrating factor 
$\mbox{exp} \left(2 \ln t \newsec_{Bang}\right) = \mbox{exp} \left(2 \ln t \newsec_{Bang}(x) \right)$
and using the fact that the type $\binom{1}{1}$ matrix 
$\partial_t \left(2 \ln t \newsec_{Bang}\right)$
commutes with the type $\binom{1}{1}$ matrix $\mbox{exp} \left(2 \ln t \newsec_{Bang}\right),$
we deduce the following consequence of \eqref{E:LITTLEGEVOUSEDINPROOF}:
\begin{align} \label{E:PARTIALTGTIMESINTEGRATINGFACTOR}
	\partial_t \left\lbrace g_{ia} \left[\mbox{exp} \left(2 \ln t \newsec_{Bang} \right) \right]_{\ j}^a \right\rbrace
	& = g_{ia} \Delta_{\ b}^a \left[\mbox{exp} \left(2 \ln t \newsec_{Bang} \right) \right]_{\ j}^b.
\end{align}
Using the estimates \eqref{E:GCOMMUTEDSTRONGPOINTWISE}, \eqref{E:KLIMITCLOSETOFLRW},
and \eqref{E:GERRORTERM}, we bound the right-hand side of \eqref{E:PARTIALTGTIMESINTEGRATINGFACTOR} as follows:
\begin{align} \label{E:GLIMITERRORTERM}
	\left\| g_{ia} \left[\mbox{exp} \left(2 \ln t \newsec_{Bang} \right) \right]_{\ b}^a \Delta_{\ j}^b \right \|_{C^M}
	 & \lesssim \left\lbrace \begin{array}{ll} 
		\epsilon t^{1/3 - c \sqrt{\epsilon}}, & (M \leq N-5), \\
	  \epsilon t^{-1/3 - c \sqrt{\epsilon}}, & (M \leq N-3).
		\end{array} \right. 
\end{align} 
From the differential equation \eqref{E:PARTIALTGTIMESINTEGRATINGFACTOR}, the estimate \eqref{E:GLIMITERRORTERM},
and the small-data estimate $\| g_{ij}(1,\cdot) - \Euc_{ij} \|_{C^{N-3}} \lesssim \epsilon,$ 
we argue as in our proof of \eqref{E:VOLFORMLIMIT} to deduce that there exists a field
$M_{ij}^{Bang}(x)$ on $\mathbb{T}^3$ such that the desired estimates 
\eqref{E:METRICENDSTATE} and \eqref{E:LIMITINGMETRICBEHAVIOR} hold.

\ \\

\noindent \textbf{Proof of \eqref{E:SECFUNDFORMBLOWUP}:} To prove \eqref{E:SECFUNDFORMBLOWUP}, we use the 
identity $t|\SecondFund|_g = t |\SecondFund_{\ b}^a \SecondFund_{\ a}^b|^{1/2}$
(which follows from the symmetry property $\SecondFund_{ij} = \SecondFund_{ji}$)
and the estimates \eqref{E:KLIMIT} and \eqref{E:KLIMITCLOSETOFLRW} to deduce the desired inequality:
\begin{align}
		\Big| t |\SecondFund|_g - |(\newsec_{Bang})_{\ b}^a (\newsec_{Bang})_{\ a}^b|^{1/2} \Big|
		& \lesssim \Big| t \SecondFund - \newsec_{Bang} \Big|_{Frame}
		\lesssim \epsilon t^{4/3 - c \sqrt{\epsilon}}. 
	\end{align}
	
\ \\

\noindent \textbf{Proof of \eqref{E:3RIEMANNBLOWUP}:} To prove \eqref{E:3RIEMANNBLOWUP}, we first note that
	\begin{align} \label{E:RIEMAN3NORMINTERMSOFCOMPONENTS}
		|\Riemann|_g^2 = \Riemann_{ab}^{\ \ cd} \Riemann_{cd}^{\ \ ab}.
	\end{align}
	We now claim that the following estimate for the components $\Riemann_{ab}^{\ \ cd}$ holds:
	\begin{align} \label{E:RIEMANN3COMPONENTSESTIMATE}
		\| \Riemann_{ab}^{\ \ cd} \|_{C^0} \lesssim \epsilon t^{- 2/3 - c \epsilon}.
	\end{align}
	The desired estimate \eqref{E:3RIEMANNBLOWUP} will then follow from 
	\eqref{E:RIEMAN3NORMINTERMSOFCOMPONENTS} and \eqref{E:RIEMANN3COMPONENTSESTIMATE}.
	To prove \eqref{E:RIEMANN3COMPONENTSESTIMATE}, we first use 
	the strong estimates \eqref{E:GCOMMUTEDSTRONGPOINTWISE}, \eqref{E:GINVERSECOMMUTEDSTRONGPOINTWISE},
	and \eqref{E:LITTLEGAMMACOMMUTEDSTRONGPOINTWISE} and
	the relation \eqref{E:RICCIDECOMP} to deduce the following
	estimate for the components $\Ric_{\ j}^i:$ 
	\begin{align} \label{E:RIC3COMPONENTSESTIMATE}
		\| \Ric_{\ j}^i \|_{C^0} & \lesssim \epsilon t^{-2/3 - c \epsilon}.
	\end{align}
	The estimate \eqref{E:RIEMANN3COMPONENTSESTIMATE} now follows from
	\eqref{E:RIC3COMPONENTSESTIMATE} and the identities 
	\eqref{E:THREESCHOUTENINTERMSOFTHREERICCI} and \eqref{E:RIEM3INTERMSOFSCHOUT}.

%	To prove \eqref{E:SPACETIMESCALARCURVATUREBLOWUP}, we first 
%	use the relation \eqref{E:SPACETIMESCALARCURVATUREINTERMSOFMATTER} to deduce
	
%	\begin{align} \label{E:SPACETIMESCALARCURVATUREBLOWUPRELATION}
%		t^2 \Rfour & = - 2 \newp.
%	\end{align}
%	The desired estimate \eqref{E:SPACETIMESCALARCURVATUREBLOWUP} now follows from
%	\eqref{E:SPACETIMESCALARCURVATUREBLOWUPRELATION} and \eqref{E:PLIMIT}, \eqref{E:PLIMITCLOSETOFLRW}.
%	The estimate \eqref{E:SPACETIMERICCIINVARIANTBLOWUP} follows similarly with the help of
%	\eqref{E:SPACETIMERICCIGFOURNROMINTERMSOFMATTER}.

\ \\

\noindent \textbf{Proof of \eqref{E:SPACETIMERICCIINVARIANTBLOWUP}-\eqref{E:FBANGVSBANG}:}	
	To derive \eqref{E:SPACETIMERICCIINVARIANTBLOWUP}, we simply insert the estimates \eqref{E:PLIMIT} and 
	\eqref{E:PLIMITCLOSETOFLRW} into the identity \eqref{E:SPACETIMERICCIGFOURNROMINTERMSOFMATTER}.
	
	To derive \eqref{E:SPACETIMEKRETSCHMANNBLOWUP}-\eqref{E:SPACETIMEKRETSCHMANNLIMITNEARLYCONSTANT},
	we first observe the relation 
	\begin{align} \label{E:KRETSCHMANNNORMDECOMP}
		t^4 |\Riemfour|_{\gfour}^2 
			& = t^4 \Riemfour_{ab}^{\ \ cd} \Riemfour_{cd}^{\ \ ab}
				+ 4 t^4 \Riemfour_{a0}^{\ \ c 0} \Riemfour_{c0}^{\ \ a0} 
				\\
		& \ \ - 4 n^{-2} t^{4} g_{cc'} g_{dd'} g^{bb'} \Riemfour_{0b}^{\ \ cd} \Riemfour_{0b'}^{\ \ c'd'}.
			\notag
\end{align}
	We then claim that the following estimates for components hold:
	\begin{align}
		\| t^2 \Riemfour_{ab}^{\ \ cd} 
			- (t^2 \SecondFund_{\ a}^c \SecondFund_{\ b}^d
			- t^2 \SecondFund_{\ a}^d \SecondFund_{\ b}^c) \|_{C^0} 
			& \lesssim \epsilon t^{2/3 - c \sqrt{\epsilon}}, 
			\label{E:4RIEMANNSPATIALBLOWUP} \\
		\| t^2 \Riemfour_{a0}^{\ \ c 0} 
			- (t \SecondFund_{\ a}^c
			+ t^2 \SecondFund_{\ e}^c \SecondFund_{\ a}^e) \|_{C^0} 
			& \lesssim \epsilon t^{2/3 - c \sqrt{\epsilon}}, 
			\label{E:4RIEMANN0SPATIAL0SPATIALBLOWUP} \\
		\| t^2 \Riemfour_{0b}^{\ \ cd} \|_{C^0} 
			& \lesssim t^{1/3 - c \sqrt{\epsilon}}.
			\label{E:4RIEMANN0SPATIALSPATIALSPATIALBLOWUP}
	\end{align}
	Let us accept \eqref{E:4RIEMANNSPATIALBLOWUP}-\eqref{E:4RIEMANN0SPATIALSPATIALSPATIALBLOWUP} for the moment.
	The desired estimates \eqref{E:SPACETIMEKRETSCHMANNBLOWUP}-\eqref{E:SPACETIMEKRETSCHMANNLIMITNEARLYCONSTANT}
	then follow from \eqref{E:KRETSCHMANNNORMDECOMP}, \eqref{E:4RIEMANNSPATIALBLOWUP},
	\eqref{E:4RIEMANN0SPATIAL0SPATIALBLOWUP}, 
	\eqref{E:4RIEMANN0SPATIALSPATIALSPATIALBLOWUP},
	\eqref{E:KLIMIT}, and \eqref{E:KLIMITCLOSETOFLRW}.
	
	To derive \eqref{E:4RIEMANNSPATIALBLOWUP}-\eqref{E:4RIEMANN0SPATIALSPATIALSPATIALBLOWUP},
	we insert the previously derived estimates into the 
	curvature expressions \eqref{E:RFOURALLSPATIALDECOMP}-\eqref{E:RFOURONE0DECOMP}.
	For example, to derive \eqref{E:4RIEMANN0SPATIAL0SPATIALBLOWUP}, 
	we bound the error term $\pmb{\triangle}_{a0}^{\ \ c 0}$
	defined in \eqref{E:RFOURTWOERROR} by
	\begin{align} \label{E:RFOURTWOERRORBOUND}
		\| \pmb{\triangle}_{a0}^{\ \ c 0} \|_{C^0} & \lesssim \epsilon t^{-4/3 - c \sqrt{\epsilon}}.
	\end{align}
	More precisely, the estimate \eqref{E:RFOURTWOERRORBOUND} follows from inserting the strong estimates
	of Prop.~\ref{P:STRONGPOINTWISE} into the expression \eqref{E:RFOURTWOERROR}.
	The time derivative term $\partial_t(t \SecondFund_{\ a}^c)$ is estimated by using 
	the evolution equation \eqref{E:PARTIALTKCMC} to express it in terms of spatial
	derivatives. We then multiply the expression \eqref{E:RFOURTWOERROR} 
	by $t^2$ and use the estimate \eqref{E:RFOURTWOERRORBOUND},
	thereby arriving at \eqref{E:4RIEMANN0SPATIAL0SPATIALBLOWUP}. The estimates
	\eqref{E:4RIEMANNSPATIALBLOWUP} and \eqref{E:4RIEMANN0SPATIALSPATIALSPATIALBLOWUP}
	can be derived similarly.
	 
	To derive \eqref{E:SPACETIMEPTENSORBLOWUP}, we simply insert 
	the estimates \eqref{E:PLIMIT} and \eqref{E:PLIMITCLOSETOFLRW} 
	into the identity \eqref{E:PTENSORNORMINTERMSOFPRESSURE}.
		
	Finally, the estimates \eqref{E:SPACETIMEWEYLBLOWUP}-\eqref{E:FBANGVSBANG} follow from	
	inserting the estimates \eqref{E:PLIMITCLOSETOFLRW},
	\eqref{E:SPACETIMEKRETSCHMANNBLOWUP}, \eqref{E:SPACETIMEKRETSCHMANNLIMITNEARLYCONSTANT}
	and \eqref{E:SPACETIMEPTENSORBLOWUP} into the identity \eqref{E:KRETSCHMANNINTERMSOFWEYLANDPTENSOR}.
\end{proof}

\section*{Acknowledgments}
The authors thank Mihalis Dafermos for 
offering enlightening comments that helped them improve the exposition.
They also thank David Jerison for providing insights that 
aided their proof of Prop.~\ref{P:CMCEXISTS}. 
IR gratefully acknowledges support from NSF grant \# DMS-1001500.
JS gratefully acknowledges support from NSF grant \# DMS-1162211
from a Solomon Buchsbaum grant administered by the Massachusetts Institute of Technology.

\appendix

\section{Calculations and Identities for the Metrics and Curvatures} \label{A:METRICANDCURVATUREID}
   
For convenience, we have gathered some standard metric and curvature relations in this appendix. 
We state many of the relations without proof. For additional background, 
readers can consult \cite{rW1984}. Throughout, 
$\gfour = - n^2 dt^2 + g_{ab} dx^a dx^b$ denotes a Lorentzian metric defined on a four
dimensional manifold.

We adopt the following sign convention for the Riemann curvature $\Riemfour_{\alpha \beta \mu \nu}:$
\begin{align}
	\Dfour_{\alpha} \Dfour_{\beta} X_{\mu} 
		- \Dfour_{\beta} \Dfour_{\alpha} X_{\mu} 	
	& = (\gfour^{-1})^{\nu \nu'} \Riemfour_{\alpha \beta \mu \nu'} X_{\nu}.
\end{align}

We have the relation (see e.g. \cite[pg. 48]{rW1984})
\begin{align} \label{E:RFOURCHRISTOFFELIDENTITY}
	(\gfour^{-1})^{\nu \nu'} \Riemfour_{\alpha \beta \mu \nu'} 
	= \partial_{\beta }\Chfour_{\alpha \ \mu}^{\ \nu} 
		-  \partial_{\alpha} \Chfour_{\beta \ \mu}^{\ \nu} 
		+ \Chfour_{\beta \ \lambda}^{\ \nu} \Chfour_{\alpha \ \mu}^{\ \lambda}
		- \Chfour_{\alpha \ \lambda}^{\ \nu} \Chfour_{\beta \ \mu}^{\ \lambda},
\end{align}
where the Christoffel symbols $\Chfour_{\mu \ \nu}^{\ \lambda}$ of $\gfour_{\mu \nu}$ are defined in \eqref{E:FOURCHRISTOFFEL}.

The spacetime Schouten tensor $\Schfour_{\mu \nu},$ its trace $\Sfour,$  
and the tensor $\mathbf{P}_{\alpha \beta \mu \nu}$ corresponding to $\gfour_{\mu \nu}$
are defined by
\begin{align}
	\Schfour_{\mu \nu} & := \frac{1}{2} \Ricfour_{\mu \nu} - \frac{1}{12} \Rfour \gfour_{\mu \nu}, 
		\label{E:SCHOUTFOUR} \\
	\Sfour & := (\gfour^{-1})^{\alpha \beta} \Schfour_{\alpha \beta},
		\label{E:SCHOUTENTRACE} \\
	\mathbf{P}_{\alpha \beta \mu \nu} 
		& := \Schfour_{\mu \alpha} \gfour_{\nu \beta}
			- \Schfour_{\nu \alpha} \gfour_{\mu \beta}
			+ \Schfour_{\nu \beta} \gfour_{\mu \alpha}
			- \Schfour_{\mu \beta} \gfour_{\nu \alpha}, 
			\label{E:PTENSOR}
\end{align}	
where 
\begin{align}
	\Ricfour_{\mu \nu} &:= (\gfour^{-1})^{\alpha \beta} \Riemfour_{\alpha \mu \beta \nu}, 
	\label{E:RICFOURDEF} \\
	\Rfour &:= (\gfour^{-1})^{\alpha \beta} (\gfour^{-1})^{\mu \nu} \Riemfour_{\alpha \mu \beta \nu}
	\label{E:RFOURDEF}
\end{align}
are the Ricci curvature and scalar curvature of $\gfour_{\mu \nu}.$

The Schouten tensor $\Sch_{\ j}^i$ corresponding to $g_{ij}$ is defined to be
\begin{align} \label{E:THREESCHOUTENINTERMSOFTHREERICCI}
	\Sch_{\ j}^i & := \Ric_{\ j}^i - \frac{1}{4} R \ID_{\ j}^i,
\end{align}
where 
\begin{align}
	\Ric_{\ j}^i & := \Riemann_{a j}^{\ \ a i}, \\
		R & := \Ric_{\ a}^a
\end{align}
are the Ricci curvature and scalar curvature of $g_{ij}.$

The Weyl tensor $\Wfour_{\alpha \beta \mu \nu}$ of $\gfour_{\mu \nu}$ can be expressed
in terms of $\Riemfour_{\alpha \beta \mu \nu}$ and $\mathbf{P}_{\alpha \beta \mu \nu}$
as follows:
\begin{align} \label{E:WEYLDEF}
	\Wfour_{\alpha \beta \mu \nu} 
		& = \Riemfour_{\alpha \beta \mu \nu}
			- \mathbf{P}_{\alpha \beta \mu \nu} \\
		& = \Riemfour_{\alpha \beta \mu \nu}
				- \frac{1}{2} (\gfour_{\alpha \mu} \Ricfour_{\nu \beta} 
				- \gfour_{\alpha \nu} \Ricfour_{\mu \beta}
				+ \gfour_{\beta \nu} \Ricfour_{\mu \alpha}
				- \gfour_{\beta \mu} \Ricfour_{\nu \alpha}) 
				\notag \\
		& \ \ + \frac{1}{6} \Rfour (\gfour_{\alpha \mu} \gfour_{\nu \beta} 
				- \gfour_{\alpha \nu} \gfour_{\mu \beta}).
				\notag
\end{align}

Because the Weyl tensor of the $3-$metric $g_{ij}$ vanishes, the Riemann tensor $\Riemann_{ij}^{\ \ kl}$ of $g_{ij}$ can be expressed as follows
in terms of its Schouten tensor:
\begin{align} \label{E:RIEM3INTERMSOFSCHOUT}
	\Riemann_{ij}^{\ \ kl} & = \Sch_{\ i}^{k} \ID_{\ j}^{l} 
		- \Sch_{\ j}^{k} \ID_{\ i}^{l} 
		+ \Sch_{\ j}^{l} \ID_{\ i}^{k} 
		- \Sch_{\ i}^{l} \ID_{\ j}^{k} .
\end{align}
	
The following properties are verified by $\Riemfour_{\alpha \beta \mu \nu},$ 
$\Wfour_{\alpha \beta \mu \nu},$ and $\mathbf{P}_{\alpha \beta \mu \nu}:$
\begin{align}
	\Riemfour_{\alpha \beta \mu \nu} & 
		= - \Riemfour_{\beta \alpha \mu \nu} 
		= - \Riemfour_{\alpha \beta \nu \mu}
		= \Riemfour_{\mu \nu \alpha \beta}, \\
	\Wfour_{\alpha \beta \mu \nu} & 
		= - \Wfour_{\beta \alpha \mu \nu} 
		= - \Wfour_{\alpha \beta \nu \mu}
		= \Wfour_{\mu \nu \alpha \beta}, \\
	(\gfour^{-1})^{\alpha \beta }\Wfour_{\alpha \beta \mu \nu}	&
		= (\gfour^{-1})^{\alpha \mu }\Wfour_{\alpha \beta \mu \nu}
		= 0, \\
	|\mathbf{P}|_{\gfour}^2 & = 8 |\Schfour|_{\gfour}^2 + 4 \Sfour^2, 
		\label{E:PGNORMINTERMSOFSCHOUTENGNORM} \\
	|\Riemfour|_{\gfour}^2 & = |\Wfour|_{\gfour}^2 + |\mathbf{P}|_{\gfour}^2.
	\label{E:KRETSCHMANNINTERMSOFWEYLANDPTENSOR}
\end{align}

\begin{lemma} [\textbf{Christoffel symbol calculations}] \label{L:CHRIST}
Let
\begin{align} \label{E:GFOURTRANSPORTED}
	\gfour & = - n^2 dt^2 + g_{ab} dx^a dx^b
\end{align}
be a Lorentzian metric on $(T,1] \times \mathbb{T}^3.$ Let  
\begin{subequations}
\begin{align}
	\Chfour_{\mu \ \nu}^{\ \lambda} & := 
		\frac{1}{2} (\gfour^{-1})^{\lambda \sigma} \big(\partial_{\mu} \gfour_{\sigma \nu} 
			+ \partial_{\nu} \gfour_{\mu \sigma} 
			- \partial_{\sigma} \gfour_{\mu \nu} \big), 
			\label{E:FOURCHRISTOFFEL} \\
	\Gamma_{j \ k}^{\ i} & := 
		\frac{1}{2}g^{ia}(\partial_j g_{ak} + \partial_k g_{ja} - \partial_a g_{jk})
		\label{E:THREECHRISTOFFEL}
\end{align}
\end{subequations}
respectively denote the Christoffel symbols of $\gfour_{\mu \nu}$ and $g_{ij}$ 
relative to the coordinates in \eqref{E:GFOURTRANSPORTED}.
Then for $i,j,k = 1,2,3,$ we have (recall that $\partial_t g_{ij} = - 2n g_{ia} \SecondFund_{\ j}^a$)
\begin{subequations}
\begin{align}
	\Chfour_{0 \ 0}^{\ 0} & = \partial_t \ln n, 
	&& 
	\Chfour_{j \ 0}^{\ 0} = \partial_j \ln n, 
	&& \Chfour_{i \ j}^{\ 0} = - n^{-1} g_{ia} \SecondFund_{\ j}^a, 
		\\	
	\Chfour_{0 \ 0}^{\ j} & =  n g^{ja} \partial_a n, 
	&& \Chfour_{j \ 0}^{\ i} = - n \SecondFund_{\ j}^i, 
	&& \Chfour_{j \ k}^{\ i} = \Gamma_{j \ k}^{\ i}. 
\end{align}
\end{subequations}

\end{lemma}

%= \frac{1}{2} g^{ia}(\partial_j g_{ak} + \partial_k g_{aj} - \partial_a g_{jk})
	 
\begin{proof}
	The lemma follows from straightforward computations.
\end{proof}

\begin{lemma}[\textbf{Christoffel symbol identities}] \label{L:CHRISTOFFELIDENTITIES}
The Christoffel symbols $	\Gamma_{j \ k}^{\ i}$ and the contracted Christoffel symbols
$\Gamma_i$ of $g_{ij},$ which are defined by \eqref{E:THREECHRISTOFFEL} and the equations
\begin{align}
	\Gamma^i & := g^{ab} \Gamma_{a \ b}^{\ i}, 
	&& 
	\Gamma_i := g_{ij} g^{ab} \Gamma_{a \ b}^{\ j},
\end{align}
satisfy the following identities (recall that $\upgamma_{j \ k}^{\ i} = g^{ia} \partial_j g_{ak}$):
\begin{subequations}
\begin{align}
	\Gamma_{j \ k}^{\ i} & = \frac{1}{2}(\upgamma_{j \ k}^{\ i} + \upgamma_{k \ j}^{\ i} 
		- g^{ia} g_{b j} \upgamma_{a \ k}^{\ b}), 
		\label{E:BIGGAMMALITTLEGAMMARELATION} \\
	\Gamma_{a \ i}^{\ a} & = \Gamma_{i \ a}^{\ a} = \frac{1}{2} g^{ab} \partial_i g_{ab} = \partial_i \ln \sqrt{\gdet}, 
		\label{E:GAMMACONTRACTEDDETRELATILON} \\
	\Gamma^i & = g^{jk} \upgamma_{j \ k}^{\ i} - \frac{1}{2} g^{ia} \upgamma_{a \ b}^{\ b}, \\
	\Gamma_i & = g_{ia} g^{jk} \upgamma_{j \ k}^{\ a} - \frac{1}{2} \upgamma_{i \ a}^{\ a} 
	= g^{ab} \partial_a g_{bi} - \frac{1}{2} g^{ab} \partial_i g_{ab}.
		\label{E:BIGGAMMACONTRACTEDLOWEREDLITTLEGAMMARELATION} 
\end{align}
\end{subequations}

Furthermore, we have
\begin{subequations}
\begin{align}
	\upgamma_{i \ a}^{\ a} & = 2 \partial_i \ln \sqrt{\gdet}, \\
	\upgamma_{a \ i}^{\ a} & = \Gamma_i + \partial_i 	\ln \sqrt{\gdet},
\end{align}
\end{subequations}
and
\begin{align}
	\partial_a g^{ia} & = - g^{i c} g^{a b} \partial_a g_{c b}
		= - g^{ab} \upgamma_{a \ b}^{\ i} = - g^{i b} \upgamma_{a \ b}^{\ a} = 
		- \Gamma^i - g^{ia} \partial_a \ln \sqrt{\gdet}.
		\label{E:PARTIALGINVERSEFORMULA}
\end{align}
\end{lemma}

\begin{proof}
	The relation \eqref{E:GAMMACONTRACTEDDETRELATILON} is proved in \cite[Equation (3.4.9)]{rW1984}.
	The remaining identities in the lemma follow from straightforward computations.
\end{proof}

\begin{lemma}[\textbf{Geodesic equation}]
Let
\begin{align} \label{E:GFOURGEO}
	\gfour & = - n^2 dt^2 + g_{ab} dx^a dx^b
\end{align}
be a Lorentzian metric on $(T,1] \times \mathbb{T}^3,$ and let $\pmb{\zeta}:[0,A) \rightarrow (T,1] \times \mathbb{T}^3$ 
be an affinely parameterized geodesic. Then relative to the coordinates in \eqref{E:GFOURGEO},
the components of $\pmb{\zeta}$ verify the following system of ODEs:
\begin{subequations}	
	\begin{align}
		\ddot{\pmb{\zeta}}^0 
			+ (\partial_t \ln n)|_{\pmb{\zeta}} (\dot{\pmb{\zeta}}^0)^2 
			+ 2 (\partial_a \ln n)|_{\pmb{\zeta}} \dot{\pmb{\zeta}}^a \dot{\pmb{\zeta}}^0 
			- (n^{-1} g_{ac} \SecondFund_{\ b}^c)|_{\pmb{\zeta}} \dot{\pmb{\zeta}}^a \dot{\pmb{\zeta}}^b 
			& = 0, \label{E:GEO0} \\
		\ddot{\pmb{\zeta}}^j + (n g^{ja} \partial_a n)|_{\pmb{\zeta}} (\dot{\pmb{\zeta}}^0)^2
			- 2 (n \SecondFund_{\ a}^j)|_{\pmb{\zeta}} \dot{\pmb{\zeta}}^a \dot{\pmb{\zeta}}^0
			+ \Gamma_{a \ b}^{\ j}|_{\pmb{\zeta}} \dot{\pmb{\zeta}}^a \dot{\pmb{\zeta}}^b
		& = 0, \label{E:GEOJ}
	\end{align}
\end{subequations}
where the $\Gamma_{j \ k}^{\ i}$ are the Christoffel symbols of $g_{ij},$
$\dot{\pmb{\zeta}}^{\mu} : = \frac{d}{d \mathscr{A}} \pmb{\zeta}^{\mu}(\mathscr{A}),$ and
$\mathscr{A}$ denotes the affine parameter.
\end{lemma}

\begin{proof}
	The geodesic equation is (see e.g. \cite[Equation (3.3.5)]{rW1984}) 
	$\ddot{\pmb{\zeta}}^{\mu} + \Chfour_{\alpha \ \beta}^{\ \mu}|_{\pmb{\zeta}} \dot{\pmb{\zeta}}^{\alpha} \dot{\pmb{\zeta}}^{\beta} = 
	0.$ Equations \eqref{E:GEO0}-\eqref{E:GEOJ} thus follow from Lemma~\ref{L:CHRIST}.
\end{proof}

\begin{lemma}[\textbf{Curvature tensors in terms of the matter}]
	For a solution to the Einstein-stiff fluid 
	system \eqref{E:EINSTEININTRO}-\eqref{E:EOS}, $\speed = 1,$ the following relations
	hold:
	\begin{subequations}
	\begin{align}
		\Rfour & = - 2 p, 
			\label{E:SPACETIMESCALARCURVATUREINTERMSOFMATTER} \\
		\Ricfour_{\mu \nu} & = 2 p \gfour_{\mu \alpha} \gfour_{\nu \beta} \ufour^{\alpha} \ufour^{\beta}, 	
			\label{E:SPACETIMERICCIINTERMSOFMATTER} \\
		|\Ricfour|_{\gfour}^2 & = 4 p^2, 
			\label{E:SPACETIMERICCIGFOURNROMINTERMSOFMATTER} \\
		\Schfour_{\mu \nu} & = p \gfour_{\mu \alpha} \gfour_{\nu \beta} \ufour^{\alpha} \ufour^{\beta} 
			+ \frac{1}{6} p \gfour_{\mu \nu}, 
			\label{E:SPACETIMESCHOUTENINTERMSOFMATTER} \\
		\Sfour & = - \frac{1}{3}p, 
			\label{E:SCHOUTENFOURTRACEINTERMSOFMATTER} \\
		|\Schfour|_{\gfour}^2 &
			= \frac{7}{9} p^2, \label{E:SCHOUTENFOURNORMINTERMSOFPRESSURE}
			\\
		|\mathbf{P}|_{\gfour}^2 &
			= \frac{20}{3} p^2. \label{E:PTENSORNORMINTERMSOFPRESSURE}
		\end{align}
	\end{subequations}
	
	%\Schfour_{\mu \nu} & = p \ufour_{\mu} \ufour_{\nu} + \frac{1}{6} p \gfour_{\mu \nu}, \\
	%	\mathbf{P}_{\alpha \beta \mu \nu} & = p \left\lbrace \ufour_{\mu} \ufour_{\alpha} \gfour_{\nu \beta} 
	%			- \ufour_{\nu} \ufour_{\alpha} \gfour_{\mu \beta}
	%			+ \ufour_{\nu} \ufour_{\beta} \gfour_{\mu \alpha}
	%			- \ufour_{\mu} \ufour_{\beta} \gfour_{\nu \alpha} \right\rbrace
	%		+ \frac{1}{3} p \left\lbrace \gfour_{\mu \alpha} \gfour_{\nu \beta} 
	%			- \gfour_{\nu \alpha} \gfour_{\mu \beta} \right\rbrace,
	
	%\begin{subequations}
	%\begin{align}
	%	\Sfour & = - \frac{1}{3}p, \\
	%	(\gfour^{-1})^{\alpha \beta} 	(\gfour^{-1})^{\mu \nu} \Schfour_{\alpha \mu} \Schfour_{\beta \nu}
	%		& = \frac{7}{9} p^2, \\
	%	\end{align}
	%\end{subequations}

%\begin{align}
%	\mathbf{P}_{\alpha \beta}^{\ \ \mu \nu} & = 4 \Schfour_{\ [\alpha}^{[\mu} \ID_{\ \beta]}^{\nu]}, \\
%	\Schfour_{\mu \nu} & = \frac{1}{2} \Rfour_{\mu \nu} - \frac{1}{12} \Rfour \gfour_{\mu \nu}
%\end{align}

\end{lemma}

\begin{proof}
	Contracting \eqref{E:EINSTEININTRO} against $(\gfour^{-1})^{\mu \nu},$ we deduce that
	$\Rfour = - \Tfour = -2p.$ This proves \eqref{E:SPACETIMESCALARCURVATUREINTERMSOFMATTER}.
	The relation \eqref{E:SPACETIMERICCIINTERMSOFMATTER} then follows from 
	\eqref{E:EINSTEININTRO} and \eqref{E:SPACETIMESCALARCURVATUREINTERMSOFMATTER}.
	\eqref{E:SPACETIMERICCIGFOURNROMINTERMSOFMATTER} then follows from
	\eqref{E:SPACETIMERICCIINTERMSOFMATTER} and \eqref{E:UNORMALIZED}.
	\eqref{E:SPACETIMESCHOUTENINTERMSOFMATTER} then follows from \eqref{E:SCHOUTFOUR},
	\eqref{E:SPACETIMESCALARCURVATUREINTERMSOFMATTER}, and \eqref{E:SPACETIMERICCIINTERMSOFMATTER}.
	\eqref{E:SCHOUTENFOURTRACEINTERMSOFMATTER} then follows from
	\eqref{E:UNORMALIZED}, \eqref{E:SCHOUTENTRACE}, and \eqref{E:SPACETIMESCHOUTENINTERMSOFMATTER}.
	\eqref{E:SCHOUTENFOURNORMINTERMSOFPRESSURE} follows from 
	\eqref{E:SPACETIMESCHOUTENINTERMSOFMATTER} and \eqref{E:UNORMALIZED}.
	\eqref{E:PTENSORNORMINTERMSOFPRESSURE} then follows from
	\eqref{E:PGNORMINTERMSOFSCHOUTENGNORM}, \eqref{E:SCHOUTENFOURTRACEINTERMSOFMATTER},
	and \eqref{E:SCHOUTENFOURNORMINTERMSOFPRESSURE}.
\end{proof}

%\begin{lemma}
%Let

%\begin{align}
%	\gfour & = - n^2 dt^2 + g_{ab} dx^a dx^b
%\end{align}
%be a Lorentzian metric on $(T,1] \times \mathbb{T}^3,$ and let
%$\Nml = n^{-1} \partial_t$ denote the future-direct unit normal to $\Sigma_t.$

%\begin{align}
%	\Riemfour_{a b c d} 
%		& = R_{a b c d} 
%			- \SecondFund_{b c} \SecondFund_{a d}
%			+ \SecondFund_{b d} \SecondFund_{a c}
%\end{align}

%\begin{align}
%	\Nml^{\alpha} \Riemfour_{\alpha b c d} 
%		& = \nabla_{c} \SecondFund_{d b} 
%			- \nabla_{d} \SecondFund_{c b}
%\end{align}

%\begin{align}
%	\Nml^{\alpha} \Nml^{\gamma} \Riemfour_{\alpha \ \gamma d}^{\ b}
%		& = \Nml (\SecondFund_{\ d}^b) 	
%			- n^{-1} g^{be}\Gamma_{e \ d}^{\ f} \partial_f n
%			+ n^{-1} g^{be}\partial_e \partial_d n
%			+ n^{-2} g^{be} (\partial_e n)(\partial_d n)
%\end{align}

%\end{lemma}

\begin{lemma}[\textbf{Decomposition of $\Riemfour_{\alpha \beta \mu \nu},$ 
$\Ricfour_{\mu \nu},$ $\Rfour,$
and $\Wfour_{\alpha \beta \mu \nu}$ into principal terms and error terms}] \label{L:RIEMANNANDWEYLDECOMPOSITIONS}
Let
\begin{align} \label{E:GFOURRIEMDECOMP}
	\gfour & = - n^2 dt^2 + g_{ab} dx^a dx^b
\end{align}
be a Lorentzian metric on $(T,1] \times \mathbb{T}^3.$ 
Assume that $(\gfour, p, \ufour)$ verify the 
Einstein-stiff fluid equations \eqref{E:EINSTEININTRO}-\eqref{E:EOS}, $\speed = 1$
and that the CMC condition $\SecondFund_{\ a}^a = -t^{-1}$ holds.
Then relative to the coordinates in \eqref{E:GFOURRIEMDECOMP},
the components of the spacetime Riemann tensor 
$\Riemfour_{\alpha \beta}^{\ \ \ \mu \nu} 
	:=(\gfour^{-1})^{\mu \mu'}(\gfour^{-1})^{\nu \nu'}\Riemfour_{\alpha \beta \mu' \nu'}$
can be decomposed into principal terms and error terms as follows:	
	\begin{subequations}
	\begin{align}
		\Riemfour_{ab}^{\ \ cd} & =
			\SecondFund_{\ a}^c \SecondFund_{\ b}^d
			- \SecondFund_{\ a}^d \SecondFund_{\ b}^c
			+ \pmb{\triangle}_{ab}^{\ \ cd},
			\label{E:RFOURALLSPATIALDECOMP} \\
		\Riemfour_{a0}^{\ \ c 0} & = t^{-1} \SecondFund_{\ a}^c
			+ \SecondFund_{\ e}^c \SecondFund_{\ a}^e 
			+ \pmb{\triangle}_{a0}^{\ \ c 0}, 
			\label{E:RFOURTWO0DECOMP} \\
		\Riemfour_{0b}^{\ \ cd} & = \pmb{\triangle}_{0b}^{\ \ cd},
		\label{E:RFOURONE0DECOMP}
	\end{align}
	\end{subequations}
	where the error terms are
	\begin{subequations}
	\begin{align}
		\pmb{\triangle}_{ab}^{\ \ cd} & := \Riemann_{ab}^{\ \ cd},
		 \label{E:RFOURFOURERROR} \\
		\pmb{\triangle}_{a0}^{\ \ c 0} & := - t^{-1} n^{-1} \partial_t (t \SecondFund_{\ a}^c)
			+ t^{-1} (n^{-1} - 1)\SecondFund_{\ a}^c 
			\label{E:RFOURTWOERROR} \\
		& \ \ - n^{-1} g^{ec} \partial_a \partial_e n
			+ n^{-1} g^{ec} \Gamma_{a \ e}^{\ f} \partial_f n, 
			 \notag \\
		\pmb{\triangle}_{0b}^{\ \ cd} & := 
					n g^{ce} \partial_e (\SecondFund_{\ b}^d)
				- n g^{de} \partial_e (\SecondFund_{\ b}^c) 
				\label{E:RFOURTHREEERROR} \\
		& \ \ + n g^{ce} \Gamma_{e \ f}^{\ d} \SecondFund_{\ b}^f
			- n g^{ce} \Gamma_{e \ b}^{\ f} \SecondFund_{\ f}^d
			- n g^{de} \Gamma_{e \ f}^{\ c} \SecondFund_{\ b}^f
			+ n g^{de} \Gamma_{e \ b}^{\ f} \SecondFund_{\ f}^c. 
			\notag
	\end{align}
\end{subequations}
Above, $\Riemann_{ab}^{\ \ cd} := g^{cc'}g^{dd'}\Riemann_{abc'd'}$ denotes a component of the Riemann tensor of $g_{ij}.$
	
The components of the spacetime Ricci tensor 
$\Ricfour_{\ \nu}^{\mu} :=(\gfour^{-1})^{\mu \mu'}\Ricfour_{\mu' \nu}$
can be decomposed into principal terms and error terms as follows:	
	\begin{subequations}
		\begin{align}
			\Ricfour_{\ 0}^0 & = - t^{-2} 
				+ \SecondFund_{\ b}^a \SecondFund_{\ a}^b
				+ \pmb{\triangle}_{\ 0}^0,
					\label{E:RICFOUR00DECOMP} \\
		\Ricfour_{\ 0}^i & = \pmb{\triangle}_{\ 0}^i,
			\\
		\Ricfour_{\ j}^i & = \pmb{\triangle}_{\ j}^i,
		\end{align}
	\end{subequations}
	where the error terms are
	\begin{subequations}
	\begin{align}
		\pmb{\triangle}_{\ 0}^0 & := 
				t^{-2}(1 - n^{-1})
				- n^{-1} g^{ab} \partial_a \partial_b n
				+ n^{-1} g^{ab} \Gamma_{a \ b}^{\ c} \partial_c n, \\
		\pmb{\triangle}_{\ 0}^i & := 
			- n g^{ab} \partial_a(\SecondFund_{\ b}^i)
			- n g^{ab} \Gamma_{a \ c}^{\ i} \SecondFund_{\ b}^c
			+ n g^{ab} \Gamma_{a \ b}^{\ c} \SecondFund_{\ c}^i,
			\\
		\pmb{\triangle}_{\ j}^i & := 
			- n^{-1} t^{-1} \partial_t (t\SecondFund_{\ j}^i)
			+ t^{-1}(n^{-1} - 1) \SecondFund_{\ j}^i
			- n^{-1} g^{ai} \partial_j \partial_a n
			+ n^{-1} g^{ai} \Gamma_{j \ a}^{\ b} \partial_b n 
			+ \Ric_{\ j}^i. 
			\label{E:RICFOURTWOERROR}
	\end{align}
	\end{subequations}
	Above, $R_{\ j}^i := g^{ia}R_{aj}$ denotes a component of the Ricci tensor of $g_{ij}.$
	
	The spacetime scalar curvature $\Rfour$ can be decomposed into principal terms and error terms as follows:	
	\begin{align}
		\Rfour & = 
			- t^{-2}
			+ \SecondFund_{\ b}^a \SecondFund_{\ a}^b
			+ \pmb{\triangle},
			\label{E:SCALARFOURDECOMP}
	\end{align}
where the error term is
\begin{align}
	\pmb{\triangle} & :=  2 t^{-2}(1 - n^{-1}) 
			- 2n^{-1} g^{ab} \partial_a \partial_b n
			+ 2 n^{-1} g^{ab} \Gamma_{a \ b}^{\ c} \partial_c n
			+ R. \label{E:SCALARFOURERROR}
\end{align}	
Above, $R$ denotes the scalar curvature of $g_{ij}.$
	
The components of the spacetime Weyl tensor 
$\Wfour_{\alpha \beta}^{\ \ \ \mu \nu} 
	:=(\gfour^{-1})^{\mu \mu'}(\gfour^{-1})^{\nu \nu'}\Wfour_{\alpha \beta \mu' \nu'}$
can be decomposed into principal terms and error terms as follows:		
	\begin{subequations}
	\begin{align}
		\Wfour_{ab}^{\ \ cd} & = 
				\SecondFund_{\ a}^c \SecondFund_{\ b}^d - \SecondFund_{\ a}^d \SecondFund_{\ b}^c	
				+ \frac{1}{6} \underbrace{\Rfour}_{- 2p} \big(\ID_{\ a}^c \ID_{\ b}^d - \ID_{\ a}^d \ID_{\ b}^c \big)
			\label{E:WFOURALLSPATIALDECOMP} \\
		& \ \	+ \pmb{\triangle}_{ab}^{\ \ cd}
			 - \frac{1}{2} \left(\ID_{\ a}^c \pmb{\triangle}_{\ b}^d 
			- \ID_{\ a}^d \pmb{\triangle}_{\ b}^c
			+ \ID_{\ b}^d \pmb{\triangle}_{\ a}^c
			- \ID_{\ b}^c \pmb{\triangle}_{\ a}^d
			\right), \notag \\
		\Wfour_{0b}^{\ \ 0d} & = 
			t^{-1} \SecondFund_{\ b}^d
			+ \SecondFund_{\ e}^d \SecondFund_{\ b}^e 
			+ \frac{1}{6} \underbrace{\Rfour}_{-2p} \ID_{\ b}^d 
			\\
		& \ \ - \frac{1}{2} \ID_{\ b}^d \underbrace{(-t^{-2} + \SecondFund_{\ f}^e \SecondFund_{\ e}^f)}_{0}
			+ \pmb{\triangle}_{0b}^{\ \ 0d}
			- \frac{1}{2} \pmb{\triangle}_{\ b}^d
			- \frac{1}{2} \ID_{\ b}^d \pmb{\triangle}_{\ 0}^0,
			\notag \\
		\Wfour_{0b}^{\ \ cd} & = \pmb{\triangle}_{0b}^{\ \ c d}
			- \frac{1}{2}\left(\ID_{\ b}^d \pmb{\triangle}_{\ 0}^c 
				- \ID_{\ b}^c \pmb{\triangle}_{\ 0}^d \right).
				\label{E:WFOURTHREEERROR}
			\end{align}
		\end{subequations}
		Above, $\ID_{\ j}^i = \mbox{diag}(1,1,1)$ denotes the identity transformation.
				
\end{lemma}

\begin{proof}
	To derive \eqref{E:RFOURALLSPATIALDECOMP}-\eqref{E:RFOURTHREEERROR}, we use the
	relation \eqref{E:RFOURCHRISTOFFELIDENTITY} and Lemma~\ref{L:CHRIST}. 
	\eqref{E:RICFOUR00DECOMP}-\eqref{E:RICFOURTWOERROR} then follow from definition \eqref{E:RICFOURDEF}.
	\eqref{E:SCALARFOURDECOMP} and \eqref{E:SCALARFOURERROR} then follow from Def.~\ref{E:RFOURDEF}.
	To derive \eqref{E:WFOURALLSPATIALDECOMP}-\eqref{E:WFOURTHREEERROR},
	we simply substitute \eqref{E:RFOURALLSPATIALDECOMP}-\eqref{E:RFOURTHREEERROR}
	into the right-hand side of \eqref{E:WEYLDEF}.
\end{proof}

\section{Derivation of the Einstein-Stiff Fluid Equations in CMC-Transported Spatial Coordinates} \label{A:CMCDERIVATION}
  
   \setcounter{theorem}{0}
   \setcounter{definition}{0}
   \setcounter{remark}{0}
   \setcounter{proposition}{0}

In this appendix, we provide some additional details regarding the derivation of the equations of Sect.~\ref{S:EEINCMC}. 
We begin by recalling that relative to CMC-transported spatial coordinates, 
$\gfour$ can be decomposed into a lapse function $n$ and a Riemannian $3-$metric $g_{ij}$ 
induced on $\Sigma_t$ as follows:
\begin{align}
	\gfour & = - n^2 dt^2 + g_{ab} dx^a dx^b \label{AE:GFOURDECOMP}.
\end{align}
The future-directed unit normal to $\Sigma_t$ is
\begin{align} \label{AE:UNITNORMAL}
	\Nml & = n^{-1} \partial_t.
\end{align}
The second fundamental form $\SecondFund$ of $\Sigma_t$ is defined by
requiring that the following relation hold for all pairs of vectors
$X,Y$ tangent to the $\Sigma_t:$
\begin{align} \label{AE:SECONDFUNDDEF}
	\SecondFund(X,Y) = - \gfour(\Dfour_X \Nml, Y).
\end{align}
For such $X,Y,$ the action of the spacetime connection $\Dfour$ can be decomposed into
the action of $\nabla$ and $\SecondFund$ as follows:
\begin{align} \label{AE:DDECOMP}
	\Dfour_X Y = \nabla_X Y - \SecondFund(X,Y)\Nml.
\end{align}

The energy-momentum tensor of a perfect fluid is
\begin{align} \label{E:TFOURSTIFFFLUID}
	\Tfour^{\mu \nu} & = (\rho + p) \ufour^{\mu} \ufour^{\nu} + p(\gfour^{-1})^{\mu \nu}.
\end{align}
$\ufour$ is future-directed and normalized by 
\begin{align}
	\gfour_{\alpha \beta} \ufour^{\alpha} \ufour^{\beta} & = - 1.
	\label{AE:UNORMAL}
\end{align}
$\ufour$ can be decomposed as
\begin{align} \label{AE:UNORMALGEO}
	\ufour = (1 + u_a u^a)^{1/2} \Nml + u^a \partial_a,
\end{align}
where the factor $(1 + u_a u^a)^{1/2}$ enforces the normalization condition \eqref{AE:UNORMAL}. The energy-momentum tensor
\eqref{E:TFOURFLUID} of the perfect fluid can be decomposed (with the indices ``downstairs'') as
\begin{align}
	\Tfour = \Tfour(\Nml,\Nml) \Nml_{\flat} \otimes \Nml_{\flat} 
		- \Tfour(\Nml,\partial_a) \left(\Nml_{\flat} \otimes dx^a + dx^a \otimes \Nml_{\flat} \right) 
		+ T_{ab} dx^a \otimes dx^b,
\end{align}
where $(\Nml_{\flat})_{\mu} := \gfour_{\mu \alpha} \Nml^{\alpha}$ is the metric dual of $\Nml,$
\begin{subequations}
\begin{align}
	\Tfour(\Nml,\Nml) & = p + (\rho + p) (1 + u_a u^a), 
		\label{AE:TFOURNORMAL} \\
	\Tfour(\Nml,\partial_i) & = - (\rho + p) (1 + u_a u^a)^{1/2} u_i, 
		\label{AE:TFOURNORMALSPATIAL} \\
	T_{ij} & = (\rho + p) u_i u_j + p g_{ij}. \label{AE:TFOURSPATIAL}
\end{align}
\end{subequations}

The deformation tensor $^{(\Nml)}\pmb{\pi}$ of $\Nml$ will play an important role in our derivation.
It is defined by
\begin{align}
	^{(\Nml)}\pmb{\pi} := \mathcal{L}_{\Nml} \gfour,
\end{align}
where $\mathcal{L}$ denotes Lie differentiation. Relative to arbitrary coordinates, we have
\begin{align}
	^{(\Nml)}\pmb{\pi}_{\mu \nu} & = \Dfour_{\mu} \Nml_{\nu} + \Dfour_{\nu} \Nml_{\mu}.
\end{align}

If $X$ and $Y$ are any pair of vectorfields tangent to $\Sigma_t,$ then
\begin{subequations}
\begin{align}
	^{(\Nml)}\pmb{\pi}(\Nml,\Nml) & = 0, \\
	^{(\Nml)}\pmb{\pi}(X,Y) & = \gfour(\Dfour_X \Nml,Y) + \gfour(\Dfour_Y \Nml,X) = - 2 \SecondFund(X,Y), \\
	^{(\Nml)}\pmb{\pi}(\Nml,X) & = X \ln n.
\end{align}
\end{subequations}

If $\mathbf{X}$ is any spacetime vectorfield, then we can decompose $\mathbf{X}$ into its
normal and $\Sigma_t-$tangential components as follows:
\begin{align} \label{AE:XDECOMP}
	\mathbf{X} = - \gfour(\mathbf{X},\Nml) \Nml + X^a \partial_a.
\end{align}

We have the following divergence identity:
\begin{align} \label{AE:DIVXDECOMP}
	\Dfour_{\alpha} \mathbf{X}^{\alpha} 
		& = - \Dfour_{\Nml}[\gfour(\mathbf{X},\Nml)]
			+ \nabla_a X^a
			+ X^a \partial_a \ln n + \SecondFund_{\ a}^a \gfour(\mathbf{X},\Nml).
\end{align}

The following identity holds for any spacetime
vectorfield $\mathbf{X}$ and any vectorfield $Y$ tangent to $\Sigma_t:$
\begin{align} \label{AE:DXXY}
	\gfour(\Dfour_{\mathbf{X}} \mathbf{X},Y)
	& = - 2 \gfour(\mathbf{X},\Nml) \gfour([\Nml,X],Y) 
		+ 2 \gfour(\mathbf{X},\Nml) k_{\ b}^a X_a Y^b \\
	& \ \ + [\gfour(\mathbf{X},\Nml)]^2 Y^a \nabla_a \ln n
		+	X^a Y_b \nabla_a X^b. \notag
\end{align}
Here, $[\mathbf{X},\mathbf{Y}]$ denotes the Lie bracket of the vectorfields $\mathbf{X}$ and $\mathbf{Y}$
(see \eqref{E:LIEBRACKETCOORDINATES}).

%Relative to our CMC coordinates $(t,x^1,x^2,x^3),$ we can decompose

%\begin{align} \label{AE:UFOURDECOMP}
%	\ufour & = u^0 \partial_t + u^a \partial_a.
%\end{align}

%The normalization condition \eqref{AE:UNORMAL}, together with \eqref{AE:GFOURDECOMP} and \eqref{AE:UFOURDECOMP}
%lead to the relations

%\begin{align}
%	\partial_t (n u^0) = (1 + u_a u^a)^{-1/2} (u_a \partial_t u^a - n \SecondFund_{\ b}^a u_a u^b)
%\end{align}

\subsection{The constraint equations}

\begin{lemma} [\textbf{The constraint equations relative to a CMC foliation}] \label{L:CONSTRAINTS}
Consider a solution $(\gfour, p, \ufour)$ to the Einstein-stiff fluid system \eqref{E:EINSTEININTRO}-\eqref{E:EOS}, $\speed = 1.$ 
Assume that the CMC condition $\SecondFund_{\ a}^a = - t^{-1}$ holds.
Then the following constraint equations are verified along $\Sigma_t:$
\begin{subequations}
\begin{align}
	R - \SecondFund_{\ b}^a \SecondFund_{\ a}^b + \underbrace{(\SecondFund_{\ a}^a)^2}_{t^{-2}} & = 
		\overbrace{2 p + 4p u_a u^a}^{2 \Tfour(\Nml,\Nml)}, \label{AE:HAMILTONIAN} \\
	\nabla_a \SecondFund_{\ i}^a - \underbrace{\nabla_i \SecondFund_{\ a}^a}_{= 0} & = 
	\underbrace{2p (1 + u_a u^a)^{1/2} u_i}_{- \Tfour(\Nml,\partial_i)}. \label{AE:MOMENTUM}
\end{align}
\end{subequations}

\end{lemma}

\begin{proof}
	See e.g. \cite[Ch. 10]{rW1984}, and note that our $\SecondFund$ has the opposite sign convention of the one in \cite{rW1984}.
\end{proof}

%\begin{align} \label{AE:LAPSEAGAIN}
%	g^{ab} \partial_a \partial_b (n-1) 
%		& = g^{ab} \Gamma_a \partial_b n
%			+ (n-1) \Big\lbrace R + t^{-2} - 2 p u_a u^a \Big\rbrace 
%			+ R - 2 p u_a u^a \\
%		& = g^{ab} \Gamma_a \partial_b n
%			+ (n-1) \Big\lbrace t^{-2} + 2(p - \frac{1}{3}t^{-2})
%				+ [\SecondFund_{\ b}^a + \frac{1}{3}t^{-1}\ID_{\ b}^a] [\SecondFund_{\ a}^b + \frac{1}{3}t^{-1}\ID_{\ a}^b] 
%				+ 2p g_{ab}u^a u^b \Big\rbrace  \notag \\
%		& \ \ + \Big\lbrace 2 (p - \frac{1}{3}t^{-2}) 
%			+ [\SecondFund_{\ b}^a + \frac{1}{3}t^{-1} \ID_{\ b}^a] [\SecondFund_{\ a}^b + \frac{1}{3}t^{-1}\ID_{\ a}^b] 
%			+ 2 (p - \frac{1}{3}t^{-2}) g_{ab}u^a u^b + \frac{2}{3}t^{-2} g_{ab}u^a u^b\Big\rbrace \notag 
%\end{align}

\subsection{The metric evolution equations}

\begin{lemma} [\textbf{The metric evolution equations in CMC-transported spatial coordinates}]
Consider a solution $(\gfour, p, \ufour)$ to the Euler-Einstein system \eqref{E:EINSTEININTRO}-\eqref{E:DTIS0INTRO}, \eqref{E:UNORMALIZED}. Then for a general perfect fluid matter model in CMC-transported spatial coordinates normalized by $\SecondFund_{\ a}^a = - t^{-1},$ the following evolution equations are verified by $g_{ij}$ and $\SecondFund_{\ j}^i:$
\begin{subequations}
\begin{align}
	\partial_t g_{ij} & = - 2 n g_{ia}\SecondFund_{\ j}^a, \label{AE:PARTIALTGCMC} \\
	\partial_t (\SecondFund_{\ j}^i) & = - g^{ia} \nabla_a \nabla_j n
		+ n \Big\lbrace R_{\ j}^i + \overbrace{\SecondFund_{\ a}^a}^{- t^{-1}} \SecondFund_{\ j}^i 
			\underbrace{- (\rho + p) u^i u_j + \frac{1}{2}\ID_{\ j}^i(p-\rho)}_{- T_{\ j}^i + (1/2)\ID_{\ j}^i{\Tfour}} 
			\Big\rbrace,  \label{AE:PARTIALTKCMC}
\end{align}
\end{subequations}
where $R_{\ j}^i$ denotes the Ricci curvature of $g_{ij}$ and $R$ denotes the scalar curvature of $g_{ij}.$

\end{lemma}

\begin{proof}
	 These calculations are standard; see e.g. \cite[Section 6.2]{aS2010} or \cite[Section 10 of Chapter 18]{mT1997III}.
\end{proof}

\subsection{The lapse equation}

\begin{lemma} [\textbf{The lapse equation in CMC-transported spatial coordinates}] 
Consider a solution to the Euler-Einstein system \eqref{E:EINSTEININTRO}-\eqref{E:DTIS0INTRO}, \eqref{E:UNORMALIZED}. Then for a general perfect fluid matter model in CMC-transported spatial coordinates normalized by $\SecondFund_{\ a}^a = - t^{-1},$ the following elliptic PDE is verified by the lapse $n:$
\begin{align} \label{AE:LAPSE}
	g^{ab} \nabla_a \nabla_b (n - 1) 
		& = (n - 1) \Big\lbrace R - (\rho + p)u_a u^a + \underbrace{(\SecondFund_{\ a}^a)^2}_{t^{-2}} + \frac{3}{2} (p - \rho) \Big\rbrace \\
		& \ \ + R - (\rho + p)u_a u^a 
			+ \underbrace{(\SecondFund_{\ a}^a)^2 - \partial_t (\SecondFund_{\ a}^a)}_{= 0} + \frac{3}{2} (p - \rho). \notag
\end{align}
\end{lemma}

\begin{proof}
To derive \eqref{AE:LAPSE}, we take the trace of \eqref{AE:PARTIALTKCMC} and use the CMC condition $\SecondFund_{\ a}^a = - t^{-1}$. 
%To arrive at the final expression on the right-hand side of \eqref{AE:LAPSE}, the Hamiltonian constraint equation \eqref{AE:HAMILTONIAN} was used.
\end{proof}

\subsection{The stiff fluid equations}
We first discuss the stiff fluid equations relative to an arbitrary spacetime coordinate system.
The equations are \eqref{AE:UNORMAL}, 
the equation of state $p = \rho,$ and the divergence relation
\begin{align} 
	\Dfour_{\alpha} \Tfour^{\alpha \nu} & = 0.
		\label{AE:DTIS0} 
\end{align}

By projecting \eqref{AE:DTIS0} in the direction $\ufour^{\nu}$ and then onto the 
directions $\gfour_{\alpha \nu}\Pi^{\mu \alpha}$ that are $\gfour-$orthogonal to $\ufour,$  
it is straightforward to verify that the system \eqref{AE:UNORMAL}, $p = \rho,$ \eqref{AE:DTIS0} is equivalent 
to \eqref{AE:UNORMAL} together with the following system:
\begin{subequations}
\begin{align}
	\ufour^{\alpha} \Dfour_{\alpha} p + 2p \Dfour_{\alpha} \ufour^{\alpha} & = 0, \label{AE:EULERP} \\
	2p \ufour^{\alpha} \Dfour_{\alpha} \ufour^{\mu} + \Pi^{\mu \alpha} \Dfour_{\alpha} p & = 0, \label{AE:EULERU} \\
	\Pi^{\mu \nu} & := (\gfour^{-1})^{\mu \nu} + \ufour^{\mu} \ufour^{\nu}.  \label{AE:PIDEF}
\end{align}
\end{subequations}

We now derive the stiff fluid equations in CMC-transported spatial coordinates.

\begin{lemma} [\textbf{The stiff fluid equations in CMC-transported spatial coordinates}]
In CMC-transported spatial coordinates normalized by $\SecondFund_{\ a}^a = - t^{-1},$ the stiff fluid equations \eqref{AE:UNORMAL} + \eqref{AE:EULERP}-\eqref{AE:PIDEF} can be expressed as follows:
\begin{subequations}
\begin{align}
	(1 + u_a u^a)^{1/2} \partial_t p + n u^a \nabla_a p
		& + 2p \Big\lbrace (1 + u_a u^a)^{-1/2} u_b \partial_t u^b + n \nabla_a u^a \Big\rbrace  
			\label{AE:EULERPCMC} \\
		& = 2p \Big\lbrace -\frac{n}{t} (1 + u_a u^a)^{1/2}
			+ n (1 + u_a u^a)^{-1/2} \SecondFund_{\ f}^e u_e u^f
			- u^a \nabla_a n \Big\rbrace \notag, \\
		2p \Big\lbrace (1 + u_a u^a)^{1/2} \partial_t u^j + n u^a \nabla_a u^j \Big\rbrace
		& + (1 + u_a u^a)^{1/2} u^j \partial_t p + n(g^{ja} + u^j u^a)\nabla_a p 
			\label{AE:EULERUCMC} \\
		& = 4 n p (1 + u_a u^a)^{1/2} \SecondFund_{\ b}^j  u^b - 2p (1 + u_a u^a) g^{jb} \nabla_b n.
		\notag
\end{align}
\end{subequations}

\end{lemma}

\begin{proof}
The normalization condition \eqref{AE:UNORMAL} implies that the following 
relation holds relative to CMC-transported spatial coordinates:
\begin{align}
	n u^0 & = (1 + u_a u^a)^{1/2}. \label{AE:U0INTERMSOFUA} 
\end{align}
The relation \eqref{AE:U0INTERMSOFUA} implies the identity
\begin{align} \label{AE:EULERPFIRSTTERM}
	\ufour^{\alpha} \Dfour_{\alpha} p = n^{-1}(1 + u_a u^a)^{1/2} \partial_t p + u^a \nabla_a p.
\end{align}

We also claim that the following identity holds: 
\begin{align} \label{AE:DIVUID}
	\Dfour_{\alpha} \ufour^{\alpha} & = (1 + u_a u^a)^{-1/2} \left\lbrace - \SecondFund_{\ f}^e u_e u^f 
			+ n^{-1} u_b \partial_t u^b \right\rbrace 
			+ \nabla_a u^a \\
		& \ \ + n^{-1} u^a \nabla_a n - \SecondFund_{\ a}^a (1 + u_b u^b)^{1/2}. \notag
\end{align}	
Let us momentarily take \eqref{AE:DIVUID} for granted. Then multiplying both sides of \eqref{AE:EULERP} by $n$ and using \eqref{AE:EULERPFIRSTTERM} and \eqref{AE:DIVUID}, we deduce \eqref{AE:EULERPCMC}.

To derive \eqref{AE:DIVUID}, we will apply 
\eqref{AE:DIVXDECOMP} with $\mathbf{X} = \ufour.$ We first note that \eqref{AE:UNORMALGEO} implies that
\begin{align} \label{AE:UN}
	\gfour(\ufour,\Nml) & = - (1 + u_a u^a)^{1/2}.
\end{align}
Then using \eqref{AE:UNITNORMAL}, \eqref{AE:PARTIALTGCMC}, and \eqref{AE:UN}, we deduce
\begin{align} \label{AE:DNUN}
	-  \Dfour_{\Nml}[\gfour(\ufour,\Nml)] 
		& = (1 + u_a u^a)^{-1/2} \left\lbrace \frac{1}{2} n^{-1}(\partial_t g_{ef}) u^e u^f 
			+ n^{-1} u_b \partial_t u^b \right\rbrace \\
		& = (1 + u_a u^a)^{-1/2} \left\lbrace - \SecondFund_{\ f}^e u_e u^f 
			+ n^{-1} u_b \partial_t u^b \right\rbrace. \notag
\end{align}
Finally, inserting \eqref{AE:UN} and \eqref{AE:DNUN} into \eqref{AE:DIVXDECOMP}, we arrive at \eqref{AE:DIVUID}.

To derive \eqref{AE:EULERUCMC}, we project \eqref{AE:EULERU} onto the spatial $j$ component 
[i.e., we set $\mu = j$ in \eqref{AE:EULERU}]
and make use of the following identities:
\begin{align} \label{AE:PIJPROJECTED}
	\Pi^{j \alpha} \Dfour_{\alpha} p	& = n^{-1} (1 + u_a u^a)^{1/2} u^j \partial_t p 
		+ g^{ja} \nabla_a p
		+ u^j u^a \nabla_a p,
\end{align}
\begin{align} \label{AE:DUUPROJECTONTOJ}
	\mathbf{u}^{\alpha} \Dfour_{\alpha} \ufour^j 
	& = n^{-1} (1 + u_a u^a)^{1/2} \partial_t u^j
		- 2 (1 + u_a u^a)^{1/2} u^b \SecondFund_{\ b}^j \\
	& \ \ + (1 + u_a u^a)  g^{jb} \nabla_b \ln n
		+ u^a \nabla_a u^j. \notag
\end{align}
\eqref{AE:PIJPROJECTED} follows easily from \eqref{AE:U0INTERMSOFUA}, while we
momentarily take \eqref{AE:DUUPROJECTONTOJ} for granted. Then multiplying both sides of
\eqref{AE:EULERU} by $n$ and using \eqref{AE:PIJPROJECTED}-\eqref{AE:DUUPROJECTONTOJ}, we deduce \eqref{AE:EULERUCMC}.

To derive \eqref{AE:DUUPROJECTONTOJ}, we will make use of the identity \eqref{AE:DXXY}.
Setting $Y$ to be the $g-$ dual of $dx^j$ (i.e., $\gfour(\mathbf{X},Y) = X^j$), $\mathbf{X} = \mathbf{u},$ 
and using \eqref{AE:UN} plus the relation $\gfour([\Nml,X],Y) = \Nml (X^j),$ 
we see that the identity \eqref{AE:DUUPROJECTONTOJ} follows
from \eqref{AE:DXXY}.

\end{proof}

\section{Derivation of The Renormalized Equations} \label{A:RESCALEDEQNS}

 \setcounter{theorem}{0}
   \setcounter{definition}{0}
   \setcounter{remark}{0}
   \setcounter{proposition}{0}

In this appendix, we derive the reformulation of the Einstein-stiff fluid equations 
that was presented in Sect.~\ref{S:RESCALEDVARIABLES}. We use the conventions
for lowering and raising indices that are described in Sect.~\ref{SS:INDICES}.

\begin{lemma} [\textbf{An expression for $R_{\ j}^i(g)$}] \label{L:RICCIDECOMP}
In terms of the renormalized variables of Def.~\ref{D:RESCALEDVAR}, 
the Ricci curvature of the Riemannian $3-$metric $g,$ which we denote by 
$R_{\ j}^i(g) := g^{ia}R_{aj}(g),$ can be expressed as follows:
\begin{align} \label{E:RICCIDECOMP}
	R_{\ j}^i(g) & = - \frac{1}{2} t^{-2/3} (\newg^{-1})^{ef} \partial_e \upgamma_{f \ j}^{\ i}  
		+ \frac{1}{2} t^{-2/3} (\newg^{-1})^{ic} (\partial_c \Gamma_j + \partial_j \Gamma_c) \\
	& \ \ + t^{-2/3} \leftexp{(Ricci)}{\triangle}_{\ j}^i, \notag
\end{align}
where 
\begin{align} \label{E:CONTRACTEDGAMMALOWEEXPRESSION}
	\Gamma_i & = \newg_{ia} (\newg^{-1})^{ef} \upgamma_{e \ f}^{\ a} - \frac{1}{2} \upgamma_{i \ a}^{\ a}, \\
	\partial_i \Gamma_j & = \newg_{ja} (\newg^{-1})^{ef} \partial_i \upgamma_{e \ f}^{\ a} 
		- \frac{1}{2} \partial_i \upgamma_{j \ a}^{\ a}
		+ \newg_{jc} (\newg^{-1})^{ef} \upgamma_{i \ a}^{\ c} \upgamma_{e \ f}^{\ a} 
		- \newg_{jc} (\newg^{-1})^{ef} \upgamma_{i \ e}^{\ a} \upgamma_{a \ f}^{\ c}, 
		\label{E:PARTIALCONTRACTEDGAMMALOWEEXPRESSION} \\
	\leftexp{(Ricci)}{\triangle}_{\ j}^i 
	& := - \frac{1}{2} (\newg^{-1})^{ic} \upgamma_{a \ b}^{\ a} \upgamma_{c \ j}^{\ b} 
		- \frac{1}{2} (\newg^{-1})^{ic} \upgamma_{a \ b}^{\ a} \upgamma_{j \ c}^{\ b}
		+ \frac{1}{2} (\newg^{-1})^{ef} \upgamma_{a \ e}^{\ a} \upgamma_{f \ j}^{\ i}
		\label{E:RICCIERROR} \\
	& \ \ + \frac{1}{4} (\newg^{-1})^{ic} \upgamma_{b \ a}^{\ a} \upgamma_{c \ j}^{\ b}
		+ \frac{1}{4} (\newg^{-1})^{ic} \upgamma_{b \ a}^{\ a} \upgamma_{j \ c}^{\ b}
		- \frac{1}{4} (\newg^{-1})^{ef} \upgamma_{e \ a}^{\ a} \upgamma_{f \ j}^{\ i}
		\notag \\
	& \ \ - \frac{1}{4} (\newg^{-1})^{ic} \upgamma_{c \ b}^{\ a} \upgamma_{a \ j}^{\ b}
		- \frac{1}{4} (\newg^{-1})^{ic} \upgamma_{a \ c}^{\ b} \upgamma_{j \ b}^{\ a} 
		- \frac{1}{4} (\newg^{-1})^{ic} \upgamma_{c \ b}^{\ a} \upgamma_{j \ a}^{\ b}
		\notag \\
	& \ \ + \frac{1}{4} (\newg^{-1})^{ic} \newg_{ab} (\newg^{-1})^{ef} \upgamma_{c \ e}^{\ a} \upgamma_{f \ j}^{\ b}
		+ \frac{1}{4} (\newg^{-1})^{ic} \newg_{ab} (\newg^{-1})^{ef} \upgamma_{e \ c}^{\ b} \upgamma_{f \ e}^{\ a} 
		\notag \\
	& \ \ - \frac{1}{2} (\newg^{-1})^{ic} \upgamma_{b \ c}^{\ a} \upgamma_{a \ j}^{\ b}.
		\notag 
\end{align}

%\begin{align} \label{E:CONTRACTEDGAMMALOWEEXPRESSION}
%	\Gamma_i & = \newg_{ia} (\newg^{-1})^{ef} \upgamma_{e \ f}^{\ a} - \frac{1}{2} \upgamma_{i \ a}^{\ a}, \\
%	\partial_i \Gamma_j & = \newg_{jb} (\newg^{-1})^{ef} \partial_i \upgamma_{e \ f}^{\ b} 
%		- \frac{1}{2} \partial_i \upgamma_{j \ b}^{\ b}
%		+ \newg_{jc} (\newg^{-1})^{ef} \upgamma_{i \ b}^{\ c} \upgamma_{e \ f}^{\ b} 
%		- \newg_{jb} (\newg^{-1})^{fc} \upgamma_{i \ c}^{\ e} \upgamma_{e \ f}^{\ b}, \\
%	\triangle_{\ j}^{i;(Ricci)} & = 
%		 \frac{1}{2} (\newg^{-1})^{ef} \upgamma_{e \ a}^{\ i} \upgamma_{f \ j}^{\ a} \label{E:RICCIERROR} \\
%	& \ \ - \frac{1}{2} (\newg^{-1})^{ia} \upgamma_{a \ j}^{\ b} \Gamma_b  
%		- \frac{1}{2} (\newg^{-1})^{ab} \upgamma_{j \ a}^{\ i} \Gamma_b	
%		- \frac{1}{2} (\newg^{-1})^{ab} \newg_{jf} (\newg^{-1})^{ie} \upgamma_{a \ e}^{\ f} \Gamma_b	
%			\notag \\
%	& \ \ + \frac{1}{2} \newg_{ab} (\newg^{-1})^{ef} (\newg^{-1})^{ic}  \upgamma_{e \ c}^{\ a} \upgamma_{f \ j}^{\ b}
%		- \frac{1}{2} (\newg^{-1})^{ic} \upgamma_{a \ c}^{\ b} \upgamma_{b \ j}^{\ a} \notag \\
%	& \ \ + \frac{1}{2} (\newg^{-1})^{ic}  \upgamma_{c \ b}^{\ a} \upgamma_{a \ j}^{\ b} 
%		+ \frac{1}{2} \newg_{ab} (\newg^{-1})^{ef} (\newg^{-1})^{ic} \upgamma_{c \ e}^{\ a} \upgamma_{f \ j}^{\ b}
%		\notag \\
%	& \ \ - \frac{1}{4} (\newg^{-1})^{ic} \newg_{ab} (\newg^{-1})^{ef} \upgamma_{c \ e}^{\ a} \upgamma_{j \ f}^{\ b}. \notag
%\end{align}

\end{lemma}

\begin{proof}
The lemma follows from the identity 
\begin{align}
	\Ric_{\ j}^i(g) =
		g^{ic} \partial_a \Gamma_{c \ j}^{\ a}
	- g^{ic} \partial_c \Gamma_{j \ a}^{\ a}
  + g^{ic} \Gamma_{a \ b}^{\ a} \Gamma_{c \ j}^{\ b}
	- g^{ic} \Gamma_{c \ b}^{\ a} \Gamma_{a \ j}^{\ b}
\end{align}
(see e.g. \cite[Equation (3.4.4)]{rW1984}), the identities of Lemma~\ref{L:CHRISTOFFELIDENTITIES}, 
and tedious but straightforward computations.

%\begin{align}
%	\Ric_{\ j}^i 
%	= g^{ic} \partial_a \Gamma_{c \ j}^{\ a}
%	- g^{ic} \partial_c \Gamma_{a \ j}^{\ a}
 % + g^{ic} \Gamma_{a \ b}^{\ a} \Gamma_{c \ j}^{\ b}
%	- g^{ic} \Gamma_{c \ b}^{\ a} \Gamma_{a \ j}^{\ b}
%\end{align}

\end{proof}

\subsection{The renormalized constraint equations}

In this section, we derive the constraint equations verified by the renormalized variables.

\begin{proposition}[\textbf{The renormalized constraints}] \label{AP:CCONSTRAINTSREFORM}
In terms of the renormalized variables of Def.~\ref{D:RESCALEDVAR}, the constraint equations 
\eqref{AE:HAMILTONIAN}-\eqref{AE:MOMENTUM} can be decomposed as follows:
\begin{subequations}
\begin{align}
	R & = 2 t^{-2} \adjustednewp 
			 + t^{-2} \leftexp{(Border)}{\mathfrak{H}} + t^{-2/3} \leftexp{(Junk)}{\mathfrak{H}}, 
		\label{AE:RINTERMSOFKPANDU}  \\
	\partial_a \freenewsec_{\ i}^a & = 
			\frac{2}{3} \newg_{ia} \newu^a 
			+ \leftexp{(Border)}{\mathfrak{M}}_i +  t^{4/3} \leftexp{(Junk)}{\mathfrak{M}}_i,
			\label{AE:MOMENTUMCONSTRAINT} 
\end{align}
\end{subequations}
where the error terms $\leftexp{(Border)}{\mathfrak{H}},$ 
$\leftexp{(Junk)}{\mathfrak{H}},$ 
$\leftexp{(Border)}{\mathfrak{M}}_i,$ 
and $\leftexp{(Junk)}{\mathfrak{M}}_i$
are defined by
\begin{subequations}
\begin{align}
		\leftexp{(Border)}{\mathfrak{H}} & := \freenewsec_{\ b}^a \freenewsec_{\ a}^b, \label{AE:HAMILTONIANERRORBORDER} \\
		\leftexp{(Junk)}{\mathfrak{H}} & := \frac{4}{3} \newg_{ab} \newu^a \newu^b 
			+ 4 \adjustednewp \newg_{ab} \newu^a \newu^b,
			 \label{AE:HAMILTONIANERRORJUNK} \\	
		\leftexp{(Border)}{\mathfrak{M}}_i & := - \frac{1}{2} \upgamma_{b \ a}^{\ a} \freenewsec_{\ i}^b
			+ \frac{1}{2} \left(\upgamma_{a \ i}^{\ b} + \upgamma_{i \ a}^{\ b} - (\newg^{-1})^{bl} \newg_{ma} \upgamma_{l \ i}^{\ m} \right) 
				\freenewsec_{\ b}^a  
			\label{AE:MOMENTUMERRORBORDERLINE}  \\
		& \ \ + 2 \adjustednewp \newg_{ia} \newu^a,
			\notag \\
		\leftexp{(Junk)}{\mathfrak{M}}_i & := 2 t^{-4/3} \adjustednewp \Big[\big(1 + t^{4/3} \newg_{ab} \newu^a \newu^b\big)^{1/2} - 1 \Big] 
				\newg_{ic}\newu^c \label{AE:MOMENTUMERRORJUNK} \\
		& \ \ + \frac{2}{3} t^{-4/3} \Big[\big(1 + t^{4/3} \newg_{ab} \newu^a \newu^b\big)^{1/2} - 1 \Big] 
			\newg_{ic} \newu^c.
			\notag 
\end{align}
\end{subequations}

Furthermore, the following equivalent version of \eqref{AE:MOMENTUMCONSTRAINT} holds:
\begin{subequations}
\begin{align} 
	(\newg^{-1})^{ia} \partial_a \freenewsec_{\ i}^j 
	& = \frac{2}{3} \newu^j 
		+ \leftexp{(Border)}{\widetilde{\mathfrak{M}}}^j + t^{4/3} \leftexp{(Junk)}{\widetilde{\mathfrak{M}}}^j, 
		\label{AE:MOMENTUMCONSTRAINTRAISED}
\end{align}
where the error terms and $\leftexp{(Border)}{\widetilde{\mathfrak{M}}}^j$ and $\leftexp{(Junk)}{\widetilde{\mathfrak{M}}}^j$
are defined by
\begin{align}
	\leftexp{(Border)}{\widetilde{\mathfrak{M}}}^j & := 
			- \frac{1}{2} (\newg^{-1})^{ia} \left(\upgamma_{a \ b}^{\ j} 
			+ \upgamma_{b \ a}^{\ j} 
			- (\newg^{-1})^{jl} \newg_{ma} \upgamma_{l \ b}^{\ m} \right) 
			\freenewsec_{\ i}^b 
		\label{AE:MOMENTUMALTERRORBORDERLINE} \\
	& \ \ + \frac{1}{2} (\newg^{-1})^{ia} \left(\upgamma_{a \ i}^{\ b} 
			+ \upgamma_{i \ a}^{\ b} 
			- (\newg^{-1})^{bl} \newg_{ma} \upgamma_{l \ i}^{\ m} \right) 
			\freenewsec_{\ b}^j
			\notag \\
	& \ \ + 2 \adjustednewp \newu^j, \notag \\
	\leftexp{(Junk)}{\widetilde{\mathfrak{M}}}^j & := 2 t^{-4/3} \adjustednewp 
		\Big[\big(1 + t^{4/3} \newg_{ab} \newu^a \newu^b \big)^{1/2} - 1 \Big] \newu^j \label{AE:MOMENTUMALTERRORJUNK} \\
	& \ \ + \frac{2}{3} t^{-4/3} \Big[\big(1 + t^{4/3} \newg_{ab} \newu^a \newu^b \big)^{1/2} - 1 \Big] \newu^j. \notag
\end{align}
\end{subequations}

%\begin{align}
%	2 \big[(\sqrt{\gbigdet})^2(\newp - \frac{1}{3})\big] 
%		& = t^2(\sqrt{\gbigdet})^2 R - \frac{4}{3} t^{4/3} (\sqrt{\gbigdet})^2 \newg_{ab} \newu^a \newu^b
%		- 4 t^{4/3} \big[(\sqrt{\gbigdet})^2(\newp - \frac{1}{3}) \big] \newg_{ab} \newu^a \newu^b \\
%	& \ \ - \big[\sqrt{\gbigdet} (\newsec_{\ b}^a + \frac{1}{3} \ID_{\ b}^a) \big] 
%		\big[\sqrt{\gbigdet} (\newsec_{\ a}^b + \frac{1}{3}\ID_{\ a}^b) \big]. \notag 
%\end{align}

\end{proposition}

\begin{proof}
To derive \eqref{AE:RINTERMSOFKPANDU}, we substitute the renormalized variables of Def.~\ref{D:RESCALEDVAR}
into equation \eqref{AE:HAMILTONIAN}, make use of the CMC condition $\SecondFund_{\ a}^a = -\frac{1}{t},$
and perform tedious algebraic computations. The only slightly subtle computation is the identity
\begin{align}
	\SecondFund_{\ b}^a \SecondFund_{\ a}^b - \overbrace{(\SecondFund_{\ a}^a)^2}^{t^{-2}} + 2 p 
	& = 2t^{-2} \adjustednewp
		+ t^{-2} \freenewsec_{\ b}^a \freenewsec_{\ a}^b,
\end{align}
the proof of which relies on the relation $\freenewsec_{\ a}^a = 0;$
this relation is a consequence of $\SecondFund_{\ a}^a = -\frac{1}{t}.$

To derive \eqref{AE:MOMENTUMCONSTRAINT}, we first note the identity
\begin{align} \label{E:MOMENTUMCONSTRAINTDIVID}
	t \nabla_a \SecondFund_{\ i}^a 
	& = \partial_a \freenewsec_{\ i}^a
		+ \frac{1}{2} \upgamma_{b \ a}^{\ a} \freenewsec_{\ i}^b
		- \frac{1}{2} \left(\upgamma_{a \ i}^{\ b} + \upgamma_{i \ a}^{\ b} - (\newg^{-1})^{bl} \newg_{ma} \upgamma_{l \ i}^{\ m} \right) 
			\freenewsec_{\ b}^a, 
\end{align}
which follows from \eqref{E:BIGGAMMALITTLEGAMMARELATION} and the fact that  $\nabla_a \ID_{\ i}^j = 0.$
We then multiply the momentum constraint equation \eqref{AE:MOMENTUM} by $t,$ use the identity
\eqref{E:MOMENTUMCONSTRAINTDIVID}, substitute the renormalized variables of Def.~\ref{D:RESCALEDVAR}, 
and perform tedious algebraic computations.

The proof of \eqref{AE:MOMENTUMCONSTRAINTRAISED} follows similarly with the help of the identity
\begin{align} \label{E:MOMENTUMCONSTRAINTDIVIDII}
	t (\newg^{-1})^{ia} \nabla_a \SecondFund_{\ i}^j
	& = (\newg^{-1})^{ia} \partial_a \freenewsec_{\ i}^j
		+ \frac{1}{2} (\newg^{-1})^{ia} \left(\upgamma_{a \ b}^{\ j} 
			+ \upgamma_{b \ a}^{\ j} 
			- (\newg^{-1})^{jl} \newg_{ma} \upgamma_{l \ b}^{\ m} \right) 
			\freenewsec_{\ i}^b 
			\\
	& \ \ - \frac{1}{2} (\newg^{-1})^{ia} \left(\upgamma_{a \ i}^{\ b} 
			+ \upgamma_{i \ a}^{\ b} 
			- (\newg^{-1})^{bl} \newg_{ma} \upgamma_{l \ i}^{\ m} \right) 
			\freenewsec_{\ b}^j.
			\notag
\end{align}

\end{proof}

\subsection{The renormalized lapse equations}

In this section, we derive the elliptic PDEs verified by the renormalized lapse variables.

\begin{proposition} [\textbf{The renormalized lapse equations}] \label{AP:LAPSERESCALED}
Let $\mathcal{L} = t^{4/3} (\newg^{-1})^{ab} \partial_a \partial_b - (1 + f) 
= t^{4/3} (\newg^{-1})^{ab} \partial_a \partial_b - (1 + \widetilde{f})$ be the elliptic operator from Definition \eqref{D:LAPSEOP}. 
Assume the stiff fluid equation of state $p = \rho.$
Then in terms of the renormalized variables of Def.~\ref{D:RESCALEDVAR}, the lapse equation \eqref{AE:LAPSE} can be expressed in the following two forms:
\begin{subequations}
\begin{align}
	\mathcal{L} \newlap & = 2 t^{-4/3} \adjustednewp 
		+ t^{-4/3} \leftexp{(Border)}{\mathfrak{N}} + \leftexp{(Junk)}{\mathfrak{N}}, 
		\label{AE:LAPSERESCALED} \\
	\mathcal{L} \newlap & = \leftexp{(Border)}{\widetilde{\mathfrak{N}}}  + t^{2/3} \leftexp{(Junk)}{\widetilde{\mathfrak{N}}},
	\label{AE:LAPSERESCALEDLOWER} 
\end{align}
\end{subequations}
where the error terms
$\leftexp{(Border)}{\mathfrak{N}}$
and $\leftexp{(Junk)}{\mathfrak{N}}$
are defined by
\begin{subequations}
\begin{align}
	\leftexp{(Border)}{\mathfrak{N}} & := \freenewsec_{\ b}^a \freenewsec_{\ a}^b, \label{AE:LAPSENBORDERLINE} \\
	\leftexp{(Junk)}{\mathfrak{N}} & := t^{2/3} \left((\newg^{-1})^{ef} \upgamma_{e \ f}^{\ b} 
			- \frac{1}{2} (\newg^{-1})^{ab} \upgamma_{a \ c}^{\ c} \right) \dlap_b  \label{AE:LAPSENJUNK} \\
	& \ \ + 2 \adjustednewp \newg_{ab} \newu^a \newu^b 
		+ \frac{2}{3} \newg_{ab} \newu^a \newu^b, \notag  
\end{align}
\end{subequations}
and the error terms
$\leftexp{(Border)}{\widetilde{\mathfrak{N}}}$
and $\leftexp{(Junk)}{\widetilde{\mathfrak{N}}}$
are defined by
\begin{subequations}
\begin{align}	
	\leftexp{(Border)}{\widetilde{\mathfrak{N}}} 
	& :=   
		- \frac{1}{2} (\newg^{-1})^{ef} \partial_e \upgamma_{f \ a}^{\ a}  
		+ (\newg^{-1})^{ab} \partial_a \Gamma_b
		+	\leftexp{(Ricci)}{\triangle}_{\ a}^a \label{AE:LAPSENTILDEBORDERLINE} \\
	& \ \ - 2 \adjustednewp \newg_{ab} \newu^a \newu^b 
			- \frac{2}{3} \newg_{ab} \newu^a \newu^b, \notag   \\
	\leftexp{(Junk)}{\widetilde{\mathfrak{N}}} 
	& := \left((\newg^{-1})^{ef} \upgamma_{e \ f}^{\ b} 
			- \frac{1}{2} (\newg^{-1})^{ab} \upgamma_{a \ c}^{\ c} \right) \dlap_b. \label{AE:LAPSENTILDEJUNK} 
\end{align}
\end{subequations}
The quantities 
$\partial_a \Gamma_b$
and
$\leftexp{(Ricci)}{\triangle}_{\ a}^a$ 
appearing in \eqref{AE:LAPSENTILDEBORDERLINE}
are respectively defined in 
\eqref{E:PARTIALCONTRACTEDGAMMALOWEEXPRESSION}
and
\eqref{E:RICCIERROR}.

\end{proposition}

\begin{proof}
	To derive \eqref{AE:LAPSERESCALEDLOWER}, we multiply \eqref{AE:LAPSE} by $t^{2/3},$
	use the relation $\SecondFund_{\ a}^a = -\frac{1}{t},$ the identity
	$g^{ab} \nabla_a \nabla_b n = g^{ab} \partial_a \partial_b (n-1) - g^{ab} \Gamma_a \partial_b(n-1),$
	the relation \eqref{E:BIGGAMMACONTRACTEDLOWEREDLITTLEGAMMARELATION}, Lemma~\ref{L:RICCIDECOMP},
	Def.~\ref{D:RESCALEDVAR}, and perform straightforward algebraic computations.
	
	The proof of \eqref{AE:LAPSERESCALED} is similar. The only difference is that 
	we replace the two occurrences of the scalar curvature $R$ in \eqref{AE:LAPSE} 
	with the right-hand side of the Hamiltonian constraint equation \eqref{AE:RINTERMSOFKPANDU}.
	
\end{proof}

\subsection{The renormalized evolution equations}

In this section, we derive the evolution equations verified by the renormalized variables.

\begin{proposition} [\textbf{The renormalized volume form factor evolution equation}]  \label{AP:RESCALEDVOLFORMEVOLUTION}
In terms of the renormalized variables of Def.~\ref{D:RESCALEDVAR},
the renormalized volume form factor $\sqrt{\gbigdet}$ verifies the following evolution equation:
\begin{align} \label{AE:LOGVOLFORMEVOLUTION}
	\partial_t \ln \sqrt{\gbigdet} & = t^{1/3} \newlap.
\end{align}

\end{proposition}

\begin{proof}
We first note the standard matrix identity
\begin{align} \label{AE:DERIVATIVEMATRIXDET}
	\partial_t \sqrt{\gdet} & = \frac{1}{2} \sqrt{\gdet} g^{ab}\partial_t g_{ab}.
\end{align}
Using \eqref{AE:PARTIALTGCMC} to substitute for $\partial_t g_{ab}$ in \eqref{AE:DERIVATIVEMATRIXDET}, 
and making use of the CMC condition $\SecondFund_{\ a}^a = - \frac{1}{t},$ we deduce
\begin{align} \label{AE:DERIVATIVEMATRIXDETSUBBED}
	\partial_t \sqrt{\gdet} & = \sqrt{\gdet} \frac{n}{t}.
\end{align}
The identity \eqref{AE:LOGVOLFORMEVOLUTION} now follows from substituting the renormalized variables of Def.~\ref{D:RESCALEDVAR}
into \eqref{AE:DERIVATIVEMATRIXDETSUBBED}.

\end{proof}

\begin{proposition} [\textbf{The renormalized metric evolution equations}] \label{AP:METRICEVOLUTIONREFORMULATED}
In terms of the renormalized variables of Def.~\ref{D:RESCALEDVAR},
the renormalized metric $\newg_{ij}$ and its inverse $(\newg^{-1})^{ij}$ verify the following evolution equations:
\begin{subequations}
\begin{align} \label{AE:GEVOLUTION}
	\partial_t \newg_{ij}
	& = - 2t^{-1} \newg_{ia} \freenewsec_{\ j}^a
		+ t^{1/3} \leftexp{(Junk)}{\mathfrak{G}}_{ij},
		\\
 \partial_t (\newg^{-1})^{ij} & = 2 t^{-1} (\newg^{-1})^{ia} \freenewsec_{\ a}^j 
 		+ t^{1/3} \leftexp{(Junk)}{\widetilde{\mathfrak{G}}}^{ij},
		\label{AE:GINVERSEEVOLUTION}
\end{align}
\end{subequations}
where the error terms $\leftexp{(Junk)}{\mathfrak{G}}_{ij}$ and 
$\leftexp{(Junk)}{\widetilde{\mathfrak{G}}}^{ij}$ are defined by
\begin{subequations}
\begin{align}
	\leftexp{(Junk)}{\mathfrak{G}}_{ij} & := - 2 \newlap \newg_{ia} \freenewsec_{\ j}^a
		+ \frac{2}{3} \newlap \newg_{ij}, 
			\label{E:METRICEVOLUTIONJUNK} \\
	\leftexp{(Junk)}{\widetilde{\mathfrak{G}}}^{ij} & := 2 \newlap (\newg^{-1})^{ia} \freenewsec_{\ a}^j
		- \frac{2}{3} \newlap (\newg^{-1})^{ij}.
		\label{E:INVERSEMETRICEVOLUTIONJUNK}
\end{align}
\end{subequations}

Furthermore, the quantities $\upgamma_{e \ i}^{\ b}$ and $\freenewsec_{\ j}^i$
verify the following evolution equations:
\begin{subequations}
\begin{align}
	\partial_t \upgamma_{e \ i}^{\ b} 
	& = - 2 t^{-1} \big[1 + t^{4/3} \newlap\big] \partial_e \freenewsec_{\ i}^b 
		+ \frac{2}{3}t^{-1/3} \dlap_e \ID_{\ i}^b 	
		\label{AE:LITTLEGAMMAEVOLUTIONDETGFORM} \\
	& \ \ + t^{-1} \leftexp{(Border)}{\mathfrak{g}}_{e \ i}^{\ b}
		+ t^{1/3} \leftexp{(Junk)}{\mathfrak{g}}_{e \ i}^{\ b}, \notag \\
	\partial_t \freenewsec_{\ j}^i 
	& = - \frac{1}{2} t^{1/3} 
		\big[1 + t^{4/3} \newlap\big] (\newg^{-1})^{ef} \partial_e \upgamma_{f \ j}^{\ i}
			\label{AE:KEVOLUTION} \\
	& \ \ + \frac{1}{2} t^{1/3} \big[1 + t^{4/3} \newlap\big] 
		 (\newg^{-1})^{ia} \partial_a \Big( \underbrace{\newg_{jb} (\newg^{-1})^{ef} \upgamma_{e \ f}^{\ b} 
				- \frac{1}{2} \upgamma_{j \ b}^{\ b}}_{\Gamma_j} \Big)  \notag \\
	& \ \  + \frac{1}{2} t^{1/3} \big[1 + t^{4/3} \newlap\big]  
		(\newg^{-1})^{ia}\partial_j \Big( \underbrace{\newg_{ab} (\newg^{-1})^{ef} \upgamma_{e \ f}^{\ b} 
				- \frac{1}{2} \upgamma_{a \ b}^{\ b}}_{\Gamma_a} \Big) \notag \\
	& \ \ - t (\newg^{-1})^{ia} \partial_a \dlap_j 
		+ \frac{1}{3} t^{1/3} \newlap \ID_{\ j}^i \notag \\
	& \ \ + t^{1/3} \leftexp{(Junk)}{\mathfrak{K}}_{\ j}^i, \notag
\end{align}
\end{subequations}
where the error terms 
$\leftexp{(Border)}{\mathfrak{g}}_{e \ i}^{\ b},$ 
$\leftexp{(Junk)}{\mathfrak{g}}_{e \ i}^{\ b},$
and $\leftexp{(Junk)}{\mathfrak{K}}_{\ j}^i$ are defined by
\begin{subequations}
\begin{align}		
	\leftexp{(Border)}{\mathfrak{g}}_{e \ i}^{\ b} 
	& := 2 \freenewsec_{\ a}^b \upgamma_{e \ i}^{\ a}  
			- 2 \freenewsec_{\ i}^a  \upgamma_{e \ a}^{\ b}  
			- 2 t^{2/3} \dlap_e \freenewsec_{\ i}^b,
		\label{AE:GAMMAEVOLUTIONERROR} \\
		\leftexp{(Junk)}{\mathfrak{g}}_{e \ i}^{\ b} 
	& := 2 \newlap \freenewsec_{\ a}^b \upgamma_{e \ i}^{\ a}  
			- 2 \newlap \freenewsec_{\ i}^a  \upgamma_{e \ a}^{\ b}, 
		\label{AE:JUNKGAMMAEVOLUTIONERROR} \\
 	\leftexp{(Junk)}{\mathfrak{K}}_{\ j}^i 
	& := - \Psi \freenewsec_{\ j}^i 
		  \label{AE:KEVOLUTIONERROR} \\
	& \ \ + \frac{1}{2} t^{2/3} (\newg^{-1})^{ia} 
			\left( \upgamma_{a \ j}^{\ b} + \upgamma_{j \ a}^{\ b} - (\newg^{-1})^{bl} \newg_{ma} \upgamma_{l \ j}^{\ m} \right)
			\dlap_b \notag \\
	& \ \ + \big[1 + t^{4/3} \newlap\big]  \leftexp{(Ricci)}{\triangle}_{\ j}^i 
			\notag \\	
	& \ \ - \frac{2}{3} \big[1 + t^{4/3} \newlap\big]  \newg_{ja} \newu^a \newu^i \notag \\
	& \ \ - 2 \big[1 + t^{4/3} \newlap\big] \adjustednewp \newg_{ja} \newu^a \newu^i,  \notag
\end{align}
and $\leftexp{(Ricci)}{\triangle}_{\ j}^i$ is defined in \eqref{E:RICCIERROR}.
\end{subequations}

\end{proposition}

\begin{proof}
To derive \eqref{AE:GEVOLUTION}, we simply substitute the renormalized variables of Def.~\ref{D:RESCALEDVAR}
into \eqref{AE:PARTIALTGCMC}. Equation \eqref{AE:GINVERSEEVOLUTION} follows from \eqref{AE:GEVOLUTION}
and the matrix identity $\partial_t (\newg^{-1})^{ij} = - (\newg^{-1})^{ia} (\newg^{-1})^{jb} \partial_t \newg_{ab}.$

To derive \eqref{AE:LITTLEGAMMAEVOLUTIONDETGFORM}, we first use the definition
$\upgamma_{e \ i}^{\ b} = g^{ab} \partial_e g_{ai},$ the matrix identity $\partial_t g^{ab} = -g^{ae}g^{bf} \partial_t g_{ef},$
and equation \eqref{AE:PARTIALTGCMC} to deduce the identity
\begin{align} \label{AE:PARTIALTGAMMAFIRSTREL}
	\partial_t \upgamma_{e \ i}^{\ b} & =  2n \upgamma_{e \ i}^{\ a} \SecondFund_{\ a}^b
		- 2n \upgamma_{e \ a}^{\ b} \SecondFund_{\ i}^a
		- 2 (\partial_e n) \SecondFund_{\ i}^b
		- 2 n \partial_e (\SecondFund_{\ i}^b).
\end{align}
Equation \eqref{AE:LITTLEGAMMAEVOLUTIONDETGFORM} now follows from substituting the renormalized variables of Def.~\ref{D:RESCALEDVAR}
into \eqref{AE:PARTIALTGAMMAFIRSTREL}.

To derive \eqref{AE:KEVOLUTION}, we first multiply \eqref{AE:PARTIALTKCMC} by $t$ to deduce
\begin{align} \label{E:PARTIALTKFIRSTREL}
	\partial_t \left[t \SecondFund_{\ j}^i + \frac{1}{3} \ID_{\ j}^i \right] 
		& = - t g^{ia} \partial_a \partial_j n 
			+ t g^{ia} \Gamma_{a \ j}^{\ b} \partial_b n
			- \frac{(n-1)}{t} [t \SecondFund_{\ j}^i]
			+ n t \Big\lbrace R_{\ j}^i - 2p u_j u^i \Big\rbrace.
\end{align}
We then substitute the renormalized variables of Def.~\ref{D:RESCALEDVAR} into \eqref{E:PARTIALTKFIRSTREL}, use
the relation \eqref{E:BIGGAMMALITTLEGAMMARELATION}, and use \eqref{E:RICCIDECOMP} to substitute for $R_{\ j}^i.$ 
Tedious algebraic computations then lead to \eqref{AE:KEVOLUTION}.

\end{proof}

\begin{proposition}  [\textbf{The renormalized stiff fluid evolution equations}] \label{AP:EULERALT}
In terms of the renormalized variables of Def.~\ref{D:RESCALEDVAR}, the stiff fluid equations can be decomposed as follows:
\begin{subequations}
\begin{align}
	\partial_t \adjustednewp 
		& + 2 t^{1/3} \big[1 + t^{4/3} \newlap\big]\big[1 + t^{4/3} \newg_{ab} \newu^a \newu^b\big]^{1/2} 
			\big[\adjustednewp + \frac{1}{3} \big] 
			\partial_c \newu^c \label{AE:RESCALEDPEVOLUTION} \\
		& - 2 t^{5/3} \frac{\big[1 + t^{4/3} \newlap\big] \big[\adjustednewp + \frac{1}{3} \big]}{\big[1 + t^{4/3} \newg_{ab} \newu^a 
			\newu^b\big]^{1/2}} \newg_{ef} \newu^e \newu^c \partial_c \newu^f	\notag \\
		& = - \frac{2}{3} t^{1/3} \newlap 
			+ t^{1/3} \leftexp{(Junk)}{\mathfrak{P}}, \notag 
\end{align}
\begin{align}
	\partial_t \newu^j 
		& - t^{1/3} \big[1 + t^{4/3} \newlap\big] \big[1 + t^{4/3} \newg_{ab} \newu^a \newu^b\big]^{1/2} \newu^j \partial_c \newu^c 
			\label{AE:RESCALEDUEVOLUTION} \\
		& + t^{5/3} \frac{\big[1 + t^{4/3} \newlap\big]}{\big[1 + t^{4/3} \newg_{ab} \newu^a \newu^b\big]^{1/2}} 
			\newg_{ef} \newu^j \newu^e \newu^c \partial_c \newu^f \notag \\
	& + t^{1/3} \frac{\big[1 + t^{4/3} \newlap\big] \newu^c \partial_c \newu^j}{\big[1 + t^{4/3} \newg_{ab} \newu^a \newu^b\big]^{1/2}} 
		\notag \\
	& + t^{-1} \frac{\big[1 + t^{4/3} \newlap\big] \left\lbrace (\newg^{-1})^{jc} + t^{4/3} \newu^j \newu^c \right\rbrace \partial_c 
			\adjustednewp}{2\big[1 + t^{4/3} \newg_{ab} \newu^a \newu^b\big]^{1/2} 
			\big[\adjustednewp + \frac{1}{3} \big]} \notag \\
	& = - t^{-1/3} (\newg^{-1})^{j a} \dlap_a 
		+ t^{-1} \leftexp{(Border)}{\mathfrak{U}}^j + t^{1/3} \leftexp{(Junk)}{\mathfrak{U}}^j, \notag 
\end{align}
\end{subequations}
where the error terms $\leftexp{(Junk)}{\mathfrak{P}},$ $\leftexp{(Border)}{\mathfrak{U}}^j,$ and $\leftexp{(Junk)}{\mathfrak{U}}^j$ 
are defined by
\begin{subequations}
\begin{align}
	\leftexp{(Junk)}{\mathfrak{P}} & : = - 2 \Psi \adjustednewp 
		\label{AE:PERROR} \\
	& \ \ - \frac{4}{9} \big[1 + t^{4/3} \newlap\big] \newg_{ab} \newu^a \newu^b 
		 \notag \\
	& \ \ - \frac{4}{3} \big[1 + t^{4/3} \newlap\big] \adjustednewp \newg_{ab}\newu^a \newu^b
		\notag \\
	& \ \ - \frac{2}{3} \big[1 + t^{4/3} \newlap\big]
		 \freenewsec_{\ b}^a  \newg_{ac} \newu^b \newu^c 
		 \notag \\
	& \ \ - 2 \big[1 + t^{4/3} \newlap\big] \adjustednewp 
		\freenewsec_{\ b}^a \newg_{ac} \newu^b \newu^c 
		 \notag \\
	& \ \ - \frac{1}{3}  \big[1 + t^{4/3} \newlap\big] \big[1 + t^{4/3} \newg_{ab} \newu^a \newu^b\big]^{1/2} 
			\upgamma_{f \ e}^{\ e} \newu^f 
			\notag \\
	& \ \ - \big[1 + t^{4/3} \newlap\big] \big[1 + t^{4/3} \newg_{ab} \newu^a \newu^b\big]^{1/2} 
			\adjustednewp \upgamma_{f \ e}^{\ e} \newu^f \notag \\
	& \ \ + \frac{1}{3} t^{4/3} \frac{\big[1 + t^{4/3} \newlap\big]  \newg_{ij}
		\left(\upgamma_{e \ f}^{\ i} + \upgamma_{f \ e}^{\ i} - (\newg^{-1})^{ic} \newg_{d e} \upgamma_{c \ f}^{\ d}\right)
			\newu^e \newu^f \newu^j}{\big[1 + t^{4/3} \newg_{ab} \newu^a \newu^b\big]^{1/2}} 
			\notag \\
	& \ \ + t^{4/3} \frac{\big[1 + t^{4/3} \newlap\big] \adjustednewp \newg_{ij}
			\left(\upgamma_{e \ f}^{\ i} + \upgamma_{f \ e}^{\ i} - (\newg^{-1})^{ic} \newg_{d e} \upgamma_{c \ f}^{\ d}\right) 
				\newu^e \newu^f \newu^j}{\big[1 + t^{4/3} \newg_{ab} \newu^a \newu^b \big]^{1/2}},
			\notag  
	\end{align}
	\begin{align}
	\leftexp{(Border)}{\mathfrak{U}}^j 
	& := 2 \freenewsec_{\ a}^j \newu^a,  \label{AE:UJBORDERLINE}
\end{align}
\begin{align}
	\leftexp{(Junk)}{\mathfrak{U}}^j 
	& := \frac{1}{3} \newlap \newu^j 
		\label{AE:UJJUNK}  \\
	& \ \ + \frac{2}{3} \big[1 + t^{4/3} \newlap\big] \newg_{ab} \newu^a \newu^b \newu^j 
		\notag \\
	& \ \ + 2 \newlap \freenewsec_{\ a}^j \newu^a
		\notag \\
	& \ \ + t^{-2/3} \big\lbrace 1 - \big[1 + t^{4/3} \newg_{ab} \newu^a \newu^b\big]^{1/2} \big\rbrace (\newg^{-1})^{jc} \dlap_c
		\notag \\
	& \ \ + \big[1 + t^{4/3} \newlap\big] 
		\freenewsec_{\ a}^b  \newg_{bc} \newu^a \newu^c \newu^j \notag \\
	& \ \ - \frac{1}{2} \frac{\big[1 + t^{4/3} \newlap\big] \left(\upgamma_{e \ f}^{\ j} + \upgamma_{f \ e}^{\ j}
		 	- (\newg^{-1})^{jc} \newg_{d e} \upgamma_{c \ f}^{\ d}\right) \newu^e \newu^f}
		 {\big[1 + t^{4/3} \newg_{ab} \newu^a \newu^b\big]^{1/2}} \notag \\
	& \ \ + \frac{1}{2} \big[1 + t^{4/3} \newlap\big] \big[1 + t^{4/3} \newg_{ab} \newu^a \newu^b\big]^{1/2} 
		\upgamma_{f \ e}^{\ e} \newu^f \newu^j \notag \\
	& \ \ - \frac{1}{2} t^{4/3} \frac{\big[1 + t^{4/3} \newlap\big] \newg_{dc} 
		\left(\upgamma_{e \ f}^{\ d} + \upgamma_{f \ e}^{\ d} - (\newg^{-1})^{dl} \newg_{me} \upgamma_{l \ f}^{\ m} \right) 
			\newu^c \newu^e \newu^f \newu^j}
			{\big[1 + t^{4/3} \newg_{ab} \newu^a \newu^b\big]^{1/2}}. \notag 
\end{align}
\end{subequations}

\end{proposition}

\begin{proof}
We first contract \eqref{AE:EULERUCMC} against $u_j$ and multiply by $p^{-1}(1 + u_a u^a)^{-1}$ to deduce
\begin{align} \label{AE:EULERCONTRACTEDID}
		 (1 + u_a u^a)^{-1/2} u_b \partial_t u^b   
		& = - n (1 + u_au^a)^{-1} u^e u_f \nabla_e u^f
			- \frac{1}{2} (1 + u_a u^a)^{-1/2} u_b u^b \partial_t \ln p
			- \frac{1}{2} n u^a \nabla_a \ln p  \\
		& \ \ + 2 n (1 + u_a u^a)^{-1/2} \SecondFund_{\ f}^e u_e u^f 
			- u^a \nabla_a n. \notag
\end{align}
We then substitute $2p \times$ the right-hand side of \eqref{AE:EULERCONTRACTEDID} for the first product in braces on the left-hand
side of \eqref{AE:EULERPCMC} and multiply the resulting equation by $(1 + u_a u^a)^{1/2}$ to deduce
\begin{align} \label{AE:EULERFIRST}
		\partial_t p 
		& - 2np (1 + u_a u^a)^{-1/2} u^e u_f \nabla_e u^f
			+ 2np (1 + u_a u^a)^{1/2} \nabla_b u^b  \\
		& = - 2 \frac{1}{t} p
			- 2 \frac{(n-1)}{t} p
			- 2 \frac{n}{t} p u_a u^a 
			- 2 n p \SecondFund_{\ b}^a u_a u^b. \notag
\end{align}

To eliminate the first term on the right-hand side of \eqref{AE:EULERFIRST}, we multiply both sides of
\eqref{AE:EULERFIRST} by $t^2$ and use the identity 
$t^2 \partial_t p = \partial_t\Big[t^2 p - \frac{1}{3} \Big] - 2tp,$ 
thereby arriving at the following equation:
\begin{align} \label{AE:EULERSECOND}
	\partial_t \Big[t^2 p - \frac{1}{3} \Big]
	& - 2n [t^2p] (1 + u_a u^a)^{-1/2} u^e u_f \nabla_e u^f
			+ 2n [t^2p] (1 + u_a u^a)^{1/2} \nabla_b u^b  \\
	& = - 2 \frac{(n-1)}{t} [t^2 p]
			- 2 \frac{n}{t} [t^2 p] u_a u^a
			- 2 n [t^2 p] \SecondFund_{\ b}^a u_a u^b. 
			\notag
\end{align}
Equation \eqref{AE:RESCALEDPEVOLUTION} now follows from expanding covariant derivatives in terms of partial derivatives and Christoffel symbols
in \eqref{AE:EULERSECOND}, substituting the renormalized variables of Def.~\ref{D:RESCALEDVAR} into
\eqref{AE:EULERSECOND}, using the relation \eqref{E:BIGGAMMALITTLEGAMMARELATION}, and carrying out tedious algebraic computations.

To derive \eqref{AE:RESCALEDUEVOLUTION}, we first use \eqref{AE:EULERFIRST} to substitute for $\partial_t p$ in \eqref{AE:EULERUCMC} and perform straightforward algebraic computations to deduce
 \begin{align} \label{AE:EULERUCMCFIRST}
	\partial_t u^j 
	& - n (1 + u_a u^a)^{1/2} u^j \nabla_b u^b
	+ n (1 + u_a u^a)^{-1/2} u^b \nabla_b u^j \\
	& \ \ + n (1 + u_a u^a)^{-1/2} u^j u^e u_f \nabla_e u^f
	+ \frac{1}{2} n (1 + u_a u^a)^{-1/2} (g^{ja} + u^j u^a) \nabla_a \ln p \notag \\
	& = \frac{1}{3} \frac{1}{t} u^j 
		+ \frac{1}{3} \frac{(n-1)}{t} u^j 
		+ \frac{2}{3} u^j \frac{n}{t} u_a u^a 
		\notag \\
	& \ \ + 2n \big[\SecondFund_{\ a}^j + \frac{1}{3} t^{-1} \ID_{\ a}^j \big] u^a
		+ n u^j \big[\SecondFund_{\ b}^a + \frac{1}{3} t^{-1} \ID_{\ b}^a \big] u_a u^b
		- (1 + u_a u^a)^{1/2} g^{jb} \nabla_b n. \notag
\end{align}
Multiplying \eqref{AE:EULERUCMCFIRST} by $t^{-1/3},$ using the identity 
$\partial_t \big[t^{-1/3} u^j \big] = t^{-1/3} \partial_t u^j - \frac{1}{3} t^{-1} \big[t^{-1/3} u^j \big]$
to eliminate the first term on the right-hand side of \eqref{AE:EULERUCMCFIRST}, and expanding covariant
derivatives in terms of coordinate derivatives and Christoffel symbols, we deduce
\begin{align} \label{AE:EULERUCMCSECOND}
	\partial_t \big[t^{-1/3} u^j \big] 
	& - n (1 + u_a u^a)^{1/2} \big[t^{-1/3} u^j \big] 
		\left( \partial_b u^b + \Gamma_{b \ c}^{\ b} u^c \right)
		 \\
	& \ \ + n (1 + u_a u^a)^{-1/2} u^b \left(\partial_b \big[t^{-1/3} u^j \big] + \Gamma_{b \ c}^{\ j} \big[t^{-1/3} u^c \big] \right)
		\notag \\
	& \ \ + n (1 + u_a u^a)^{-1/2} \big[t^{-1/3} u^j \big] u^e u_f  \left(\partial_e u^f + \Gamma_{e \ b}^{\ f} u^b \right) \notag \\
	& \ \ + \frac{1}{2} n (1 + u_a u^a)^{-1/2} \left(t^{-1/3} g^{ja} + \big[t^{-1/3} u^j \big] u^a \right) \partial_a \ln p \notag \\
	& = \frac{1}{3} \frac{(n-1)}{t} \big[t^{-1/3} u^j \big] 
		+ \frac{2}{3} \frac{n}{t} u_a u^a \big[t^{-1/3} u^j \big] \notag \\
	& \ \ + 2n \big[\SecondFund_{\ a}^j 
		+ \frac{1}{3} t^{-1} \ID_{\ a}^j \big] \big[t^{-1/3} u^a \big]
		+ n \big[\SecondFund_{\ b}^a + \frac{1}{3} t^{-1} \ID_{\ b}^a \big] u_a u^b \big[t^{-1/3} u^j \big]
		- t^{-1/3} (1 + u_a u^a)^{1/2} g^{jb} \partial_b n. \notag
\end{align}
Finally, we substitute the renormalized variables of Def.~\ref{D:RESCALEDVAR} into \eqref{AE:EULERUCMCSECOND}, use the relation \eqref{E:BIGGAMMALITTLEGAMMARELATION}, and perform tedious algebraic computations, thereby arriving at \eqref{AE:RESCALEDUEVOLUTION}.

\end{proof}

\section{The Commuted Renormalized Equations} \label{A:COMMUTED}

 \setcounter{theorem}{0}
   \setcounter{definition}{0}
   \setcounter{remark}{0}
   \setcounter{proposition}{0}
  
In this appendix we provide the full structure of the $\partial_{\vec{I}}-$commuted equations.
They can be derived in a straightforward
fashion by commuting the equations of Appendix \ref{A:RESCALEDEQNS} with the 
operator $\partial_{\vec{I}},$ and we therefore omit the proofs. Throughout we use
the commutator notation $[\cdot, \cdot]$ of Sect.~\ref{SS:COMMUTATORS}.

\subsection{The $\partial_{\vec{I}}-$commuted renormalized momentum constraint equations}
\label{SS:COMMUTEDMOMENTUMRENORM}
\begin{proposition} [\textbf{The $\partial_{\vec{I}}-$commuted renormalized momentum constraint equations}]
Given a solution to \eqref{AE:MOMENTUMCONSTRAINT} and \eqref{AE:MOMENTUMCONSTRAINTRAISED},
the corresponding $\partial_{\vec{I}}-$differentiated quantities
verify the following constraint equations:
\begin{subequations}
\begin{align}
	\partial_a \partial_{\vec{I}} \freenewsec_{\ i}^a 
	& = \frac{2}{3} \newg_{ia}\partial_{\vec{I}} \newu^a  
		+ \leftexp{(\vec{I});(Border)}{\mathfrak{M}}_i
		+ t^{4/3} \leftexp{(\vec{I});(Junk)}{\mathfrak{M}}_i, 
		\label{AE:MOMENTUMCONSTRAINTCOMMUTED} \\
	(\newg^{-1})^{ia} \partial_a \partial_{\vec{I}} \freenewsec_{\ i}^j 
	& = \frac{2}{3} \partial_{\vec{I}} \newu^j  
		+ \leftexp{(\vec{I});(Border)}{\widetilde{\mathfrak{M}}}^j
		+ t^{4/3} \leftexp{(\vec{I});(Junk)}{\widetilde{\mathfrak{M}}}^j, 
		\label{AE:RAISEDMOMENTUMCONSTRAINTCOMMUTED}
\end{align}
\end{subequations}
where the error terms
$\leftexp{(\vec{I});(Border)}{\mathfrak{M}}_i,$
$\leftexp{(\vec{I});(Junk)}{\mathfrak{M}}_i,$
$\leftexp{(\vec{I});(Border)}{\widetilde{\mathfrak{M}}}^j,$
and $\leftexp{(\vec{I});(Junk)}{\widetilde{\mathfrak{M}}}^j$
are defined by
\begin{subequations}
\begin{align}
	\leftexp{(\vec{I});(Border)}{\mathfrak{M}}_i & := \partial_{\vec{I}} \leftexp{(Border)}{\mathfrak{M}}_i
		+ \frac{2}{3} \Big[\partial_{\vec{I}} , \newg_{ia} \Big] \newu^a, 
		\label{AE:MOMENTUMCONSTRAINTCOMMUTEDERRORBORDERLINE} \\
	\leftexp{(\vec{I});(Junk)}{\mathfrak{M}}_i & := \partial_{\vec{I}} \leftexp{(Junk)}{\mathfrak{M}}_i,
		\label{AE:MOMENTUMCONSTRAINTCOMMUTEDERRORJUNK} \\
	\leftexp{(\vec{I});(Border)}{\widetilde{\mathfrak{M}}}^j
		& := \partial_{\vec{I}} \leftexp{(Border)}{\widetilde{\mathfrak{M}}}^j 
			- \Big[\partial_{\vec{I}}, (\newg^{-1})^{ia} \Big] 
			\partial_a \freenewsec_{\ i}^j, 	
		\label{AE:RAISEDMOMENTUMCONSTRAINTCOMMUTEDERRORBORDERLINE} \\
	\leftexp{(\vec{I});(Junk)}{\widetilde{\mathfrak{M}}}^j
	& := \partial_{\vec{I}} \leftexp{(Junk)}{\widetilde{\mathfrak{M}}}^j,
		\label{AE:RAISEDMOMENTUMCONSTRAINTCOMMUTEDERRORJUNK}
\end{align}
\end{subequations}
and $\leftexp{(Border)}{\mathfrak{M}}_i,$ $\leftexp{(Junk)}{\mathfrak{M}}_i,$ 
$\leftexp{(Border)}{\widetilde{\mathfrak{M}}}^j,$ and $\leftexp{(Junk)}{\widetilde{\mathfrak{M}}}^j$ are respectively defined in
\eqref{AE:MOMENTUMERRORBORDERLINE}, \eqref{AE:MOMENTUMERRORJUNK}, \eqref{AE:MOMENTUMALTERRORBORDERLINE},
and \eqref{AE:MOMENTUMALTERRORJUNK}.

\end{proposition}

\hfill $\qed$

\subsection{The $\partial_{\vec{I}}-$commuted renormalized lapse equations}
\label{SS:COMMUTEDLAPSERENORM}
\begin{proposition} [\textbf{The $\partial_{\vec{I}}-$commuted renormalized lapse equations}]
Given a solution to \eqref{AE:LAPSERESCALED} and \eqref{AE:LAPSERESCALEDLOWER},
the corresponding differentiated quantity $\partial_{\vec{I}} \newlap$ verifies both of the following elliptic PDEs:
\begin{subequations}
\begin{align} 
	\mathcal{L} \partial_{\vec{I}} \newlap
	& = 2 t^{-4/3} \partial_{\vec{I}} \adjustednewp 
		+ t^{-4/3} \leftexp{(\vec{I});(Border)}{\mathfrak{N}} 
		+ \leftexp{(\vec{I});(Junk)}{\mathfrak{N}}, 
		\label{AE:LAPSEICOMMUTED} \\
	\mathcal{L} \partial_{\vec{I}} \newlap
		& = \leftexp{(\vec{I});(Border)}{\widetilde{\mathfrak{N}}} 
			+ t^{2/3} \leftexp{(\vec{I});(Junk)}{\widetilde{\mathfrak{N}}}, 
		\label{AE:ALTERNATELAPSEICOMMUTEDLOWER}
\end{align}
\end{subequations}
where $\mathcal{L}$ is the elliptic operator from Def.~\ref{D:LAPSEOP},
the error terms
$\leftexp{(\vec{I});(Border)}{\mathfrak{N}}$
and $\leftexp{(\vec{I});(Junk)}{\mathfrak{N}}$
are defined by
\begin{subequations}
\begin{align} 
	\leftexp{(\vec{I});(Border)}{\mathfrak{N}} 
	& := \partial_{\vec{I}} \leftexp{(Border)}{\mathfrak{N}},
		\label{AE:LAPSEICOMMUTEDINHOMOGENEOUSTERMBORDERLINE} \\
	\leftexp{(\vec{I});(Junk)}{\mathfrak{N}} 
	& := \partial_{\vec{I}} \leftexp{(Junk)}{\mathfrak{N}}
		- t^{2/3} \left[ \partial_{\vec{I}} , (\newg^{-1})^{ab} \right] \partial_a \dlap_b
		\label{AE:LAPSEICOMMUTEDINHOMOGENEOUSTERMJUNK}  \\
	& \ \ + 2 \left[ \partial_{\vec{I}} , \adjustednewp  \right] \newlap \notag \\
	& \ \ + \left[ \partial_{\vec{I}} , \freenewsec_{\ b}^a \freenewsec_{\ a}^b \right] \newlap 
		\notag \\
	& \ \ + 2 t^{4/3} \left[ \partial_{\vec{I}} , 
			\adjustednewp \newg_{ab} \newu^a \newu^b \right] \newlap
		+ \frac{2}{3} t^{4/3} \left[ \partial_{\vec{I}} , \newg_{ab} \newu^a \newu^b \right] \newlap,
		\notag
\end{align}
the error terms $\leftexp{(\vec{I});(Border)}{\widetilde{\mathfrak{N}}}$
and $\leftexp{(\vec{I});(Junk)}{\widetilde{\mathfrak{N}}}$
are defined by
\begin{align} 
	\leftexp{(\vec{I});(Border)}{\widetilde{\mathfrak{N}}} 
	& := \partial_{\vec{I}} \leftexp{(Border)}{\widetilde{\mathfrak{N}}},
		\label{AE:LAPSETILDEICOMMUTEDINHOMOGENEOUSTERMBORDERLINE} \\
	\leftexp{(\vec{I});(Junk)}{\widetilde{\mathfrak{N}}} 
	& := \partial_{\vec{I}} \leftexp{(Junk)}{\widetilde{\mathfrak{N}}}
			- \left[ \partial_{\vec{I}} , (\newg^{-1})^{ab} \right] \partial_a \dlap_b
		 	\label{AE:LAPSETILDEICOMMUTEDINHOMOGENEOUSTERMJUNK}  \\
		& \ \ - \frac{1}{2} t^{2/3} \left[ \partial_{\vec{I}} , 
				(\newg^{-1})^{ef} \partial_e \upgamma_{f \ a}^{\ a} \right] \newlap
		+ t^{2/3} \left[ \partial_{\vec{I}} , (\newg^{-1})^{ab} \partial_a \Gamma_b \right] \newlap
			\notag \\
		& \ \ + t^{2/3} \left[ \partial_{\vec{I}} , \leftexp{(Ricci)}{\triangle}_{\ a}^a \right] \newlap 
			- 2 t^{2/3} \left[ \partial_{\vec{I}} , \adjustednewp \newg_{ab} \newu^a \newu^b  \right] 
			\newlap \notag \\
		& \ \ - \frac{2}{3} t^{2/3} \left[ \partial_{\vec{I}}, \newg_{ab} \newu^a \newu^b \right] \newlap, \notag
\end{align}
\end{subequations}
$\leftexp{(Border)}{\mathfrak{N}},$ $\leftexp{(Junk)}{\mathfrak{N}},$ $\leftexp{(Border)}{\widetilde{\mathfrak{N}}},$
and $\leftexp{(Junk)}{\widetilde{\mathfrak{N}}}$ are
respectively defined in \eqref{AE:LAPSENBORDERLINE}, \eqref{AE:LAPSENJUNK},
\eqref{AE:LAPSENTILDEBORDERLINE}, and \eqref{AE:LAPSENTILDEJUNK}, and the Ricci error term
$\leftexp{(Ricci)}{\triangle}_{\ a}^a$ is defined in \eqref{E:RICCIERROR}.

\end{proposition}

\hfill $\qed$

\subsection{The $\partial_{\vec{I}}-$commuted renormalized evolution equations}
\label{SS:COMMUTEDEVOLUTIONRENORM}
\subsubsection{The $\partial_{\vec{I}}-$commuted volume form factor evolution equation}
\label{SSS:COMMUTEDVOLFORMRENORM}
\begin{lemma} [\textbf{The $\partial_{\vec{I}}-$commuted renormalized volume form factor equation}]
Given a solution to \eqref{AE:LOGVOLFORMEVOLUTION}, the corresponding differentiated quantity $\partial_{\vec{I}} \ln \left(\sqrt{\gbigdet}\right)$ verifies the following evolution equation:
\begin{align} \label{AE:LOGVOLFORMEVOLUTIONCOMMUTED}
	\partial_t \partial_{\vec{I}} \ln \left(\sqrt{\gbigdet}\right)  
		& =	t^{1/3} \partial_{\vec{I}} \newlap.
\end{align}

\end{lemma}

\hfill $\qed$

\subsubsection{The $\partial_{\vec{I}}-$commuted renormalized metric evolution equations}
\label{SSS:COMMUTEDMETRICRENORM}
\begin{proposition} [\textbf{The $\partial_{\vec{I}}-$commuted renormalized metric evolution equations}]
Given a solution to \eqref{AE:GEVOLUTION}-\eqref{AE:GINVERSEEVOLUTION},
the corresponding differentiated quantities $\partial_{\vec{I}} \newg_{ij}$ and $\partial_{\vec{I}} (\newg^{-1})^{ij}$ 
verify the following evolution equations:
\begin{subequations}
\begin{align} 
	\partial_t \partial_{\vec{I}} \newg_{ij}
	& =  t^{-1} \leftexp{(\vec{I});(Border)}{\mathfrak{G}}_{ij}
		+ t^{1/3} \leftexp{(\vec{I});(Junk)}{\mathfrak{G}}_{ij},
		\label{AE:GEVOLUTIONCOMMUTED} \\
	\partial_t \partial_{\vec{I}} (\newg^{-1})^{ij}
	& = t^{-1} \leftexp{(\vec{I});(Border)}{\widetilde{\mathfrak{G}}}^{ij}
		+ t^{1/3} \leftexp{(\vec{I});(Junk)}{\widetilde{\mathfrak{G}}}^{ij},
		\label{AE:GINVERSEEVOLUTIONCOMMUTED}
\end{align}
\end{subequations}
where the error terms $\leftexp{(\vec{I});(Border)}{\mathfrak{G}}_{ij},$ $\leftexp{(\vec{I});(Junk)}{\mathfrak{G}}_{ij}$
$\leftexp{(\vec{I});(Border)}{\widetilde{\mathfrak{G}}}^{ij},$ and $\leftexp{(\vec{I});(Junk)}{\widetilde{\mathfrak{G}}}^{ij}$
are defined by
\begin{subequations}
\begin{align}
	\leftexp{(\vec{I});(Border)}{\mathfrak{G}}_{ij} 
	& := - 2 \partial_{\vec{I}} \left\lbrace \newg_{ia} \freenewsec_{\ j}^a \right\rbrace, 
		\label{E:METRICEVOLUTIONBORDERCOMMUTED} \\
	\leftexp{(\vec{I});(Junk)}{\mathfrak{G}}_{ij} 
	& :=  \partial_{\vec{I}} \leftexp{(Junk)}{\mathfrak{G}}_{ij},
		\label{E:METRICEVOLUTIONJUNKCOMMUTED} \\
	\leftexp{(\vec{I});(Border)}{\widetilde{\mathfrak{G}}}^{ij} 
	& := 2 \partial_{\vec{I}} \left\lbrace(\newg^{-1})^{ia} \freenewsec_{\ a}^j \right\rbrace,
		\label{E:INVERSEMETRICEVOLUTIONBORDERCOMMUTED} \\
	\leftexp{(\vec{I});(Junk)}{\widetilde{\mathfrak{G}}}^{ij} 
	& := \partial_{\vec{I}}\leftexp{(Junk)}{\widetilde{\mathfrak{G}}}^{ij},
		\label{E:INVERSEMETRICEVOLUTIONJUNKCOMMUTED}
\end{align}
\end{subequations}
and the error terms $\leftexp{(Junk)}{\mathfrak{G}}_{ij}$ and $\leftexp{(Junk)}{\widetilde{\mathfrak{G}}}^{ij}$ are 
respectively defined in \eqref{E:METRICEVOLUTIONJUNK} and \eqref{E:INVERSEMETRICEVOLUTIONJUNK}.

Furthermore, given a solution to \eqref{AE:LITTLEGAMMAEVOLUTIONDETGFORM}-\eqref{AE:KEVOLUTION},
the corresponding differentiated quantities $\partial_{\vec{I}} \upgamma_{e \ i}^{\ b}$ and 
$\partial_{\vec{I}} \freenewsec_{\ j}^i$ verify the following evolution equations:
\begin{subequations}
\begin{align}
	\partial_t \partial_{\vec{I}} \upgamma_{e \ i}^{\ b} & = 
		- 2 t^{-1} \big[1 + t^{4/3} \newlap\big]  
			\partial_e \partial_{\vec{I}} \freenewsec_{\ i}^b 
			\label{AE:METRICGAMMACOMMUTED} \\
	& \ \ + \frac{2}{3} t^{-1/3} \partial_{\vec{I}} \dlap_e \ID_{\ i}^b 
		\notag \\
	& \ \ +  t^{-1} \leftexp{(\vec{I});(Border)}{\mathfrak{g}_{e \ i}^{\ b}}
		+ t^{1/3} \leftexp{(\vec{I});(Junk)}{\mathfrak{g}_{e \ i}^{\ b}}, 
		\notag \\
	\partial_t \partial_{\vec{I}} \freenewsec_{\ j}^i
	& =  - \frac{1}{2} t^{1/3} \big[1 + t^{4/3} \newlap\big] (\newg^{-1})^{ef} \partial_e 
		\partial_{\vec{I}} \upgamma_{f \ j}^{\ i}
			\label{AE:SECONDFUNDCOMMUTED} \\
	& \ \ + \frac{1}{2} t^{1/3} \big[1 + t^{4/3} \newlap\big] 
		 (\newg^{-1})^{ia} \partial_a \leftexp{(\vec{I})}{\Gamma_j}  \notag \\
	& \ \  + \frac{1}{2} t^{1/3} \big[1 + t^{4/3} \newlap \big]  
		(\newg^{-1})^{ia}\partial_j \leftexp{(\vec{I})}{\Gamma_a} \notag \\
	& \ \ - t (\newg^{-1})^{ia} \partial_a \partial_{\vec{I}} \dlap_j
		+ \frac{1}{3}t^{1/3}  \partial_{\vec{I}} \newlap \ID_{\ j}^i 
		\notag \\
	& \ \ + t^{1/3} \leftexp{(\vec{I});(Junk)}{\mathfrak{K}}_{\ j}^i,
		\notag
\end{align}
\end{subequations}
where the error terms
$\leftexp{(\vec{I});(Border)}{\mathfrak{g}_{e \ i}^{\ b}},$
$\leftexp{(\vec{I});(Junk)}{\mathfrak{g}_{e \ i}^{\ b}},$
and $\leftexp{(\vec{I});(Junk)}{\mathfrak{K}}_{\ j}^i$
are defined by
\begin{subequations}
\begin{align}
	\leftexp{(\vec{I});(Border)}{\mathfrak{g}_{e \ i}^{\ b}} 
	& := \partial_{\vec{I}} \leftexp{(Border)}{\mathfrak{g}}_{e \ i}^{\ b},
		\label{AE:METRICGAMMACOMMUTEDERRORTERMS} \\
	\leftexp{(\vec{I});(Junk)}{\mathfrak{g}_{e \ i}^{\ b}} 
	& := \partial_{\vec{I}} \leftexp{(Junk)}{\mathfrak{g}}_{e \ i}^{\ b}
		- 2 \left[\partial_{\vec{I}}, \newlap \right] \partial_e \freenewsec_{\ i}^b,
		\label{AE:JUNKMETRICGAMMACOMMUTEDERRORTERMS} \\
	\leftexp{(\vec{I});(Junk)}{\mathfrak{K}}_{\ j}^i 
	& :=  \partial_{\vec{I}} \leftexp{(Junk)}{\mathfrak{K}}_{\ j}^i
		\label{AE:SECONDFUNDCOMMUTEDERRORTERMS} \\
	& \ \ - \frac{1}{2} \left[\partial_{\vec{I}}, \big[1 + t^{4/3} \newlap \big] (\newg^{-1})^{ef}  \right]  
		 \partial_e \upgamma_{f \ j}^{\ i} 
		  \notag \\
		& \ \ + \frac{1}{2} \big[1 + t^{4/3} \newlap \big] (\newg^{-1})^{ia}
			\left[\partial_{\vec{I}}, \newg_{jc} (\newg^{-1})^{ef} \upgamma_{a \ b}^{\ c} 
			- \newg_{jb} (\newg^{-1})^{fc} \upgamma_{a \ c}^{\ e} \right] \upgamma_{e \ f}^{\ b}
			\notag \\
		& \ \ + \frac{1}{2} \big[1 + t^{4/3} \newlap \big] (\newg^{-1})^{ia}
			\left[\partial_{\vec{I}}, \newg_{jb} (\newg^{-1})^{ef} \right] \partial_a \upgamma_{e \ f}^{\ b}
			\notag \\
		& \ \ + \frac{1}{2} \left[\partial_{\vec{I}}, \big[1 + t^{4/3} \newlap \big](\newg^{-1})^{ia} \right]
			\partial_a \Gamma_j \notag \\
		& \ \ + \frac{1}{2} \big[1 + t^{4/3} \newlap \big] (\newg^{-1})^{ia}
			\left[\partial_{\vec{I}}, \newg_{ac} (\newg^{-1})^{ef} \upgamma_{j \ b}^{\ c} 
			- \newg_{ab} (\newg^{-1})^{fc} \upgamma_{j \ c}^{\ e} \right] \upgamma_{e \ f}^{\ b}
			\notag \\
		& \ \ + \frac{1}{2} \big[1 + t^{4/3} \newlap \big] (\newg^{-1})^{ia}
			\left[\partial_{\vec{I}}, \newg_{ab} (\newg^{-1})^{ef} \right] \partial_j \upgamma_{e \ f}^{\ b}
			\notag \\
		& \ \ + \frac{1}{2} \left[\partial_{\vec{I}}, \big[1 + t^{4/3} \newlap \big] (\newg^{-1})^{ia} \right]
			\partial_j \Gamma_a \notag \\
		& \ \ - t^{2/3} \left[\partial_{\vec{I}}, (\newg^{-1})^{ia} \right] \partial_a \dlap_j, 
		\notag
\end{align}
\end{subequations}
\begin{align}
		\leftexp{(\vec{I})}{\Gamma_j} & := \newg_{jb} (\newg^{-1})^{ef} \partial_{\vec{I}} \upgamma_{e \ f}^{\ b} 
	- \frac{1}{2} \partial_{\vec{I}} \upgamma_{j \ b}^{\ b}, 
	\label{E:ICONTRACTEDGAMMA} \\
	\partial_a \Gamma_j & = \newg_{jb} (\newg^{-1})^{ef} \partial_a \upgamma_{e \ f}^{\ b} 
			- \frac{1}{2} \partial_a \upgamma_{j \ b}^{\ b}
			+ \newg_{jc} (\newg^{-1})^{ef} \upgamma_{a \ b}^{\ c} \upgamma_{e \ f}^{\ b} 
			- \newg_{jb} (\newg^{-1})^{fc} \upgamma_{a \ c}^{\ e} \upgamma_{e \ f}^{\ b},
\end{align}
and $\leftexp{(Border)}{\mathfrak{g}}_{e \ i}^{\ b},$
$\leftexp{(Junk)}{\mathfrak{g}}_{e \ i}^{\ b},$ 
and $\leftexp{(Junk)}{\mathfrak{K}}_{\ j}^i$ are respectively defined in 
\eqref{AE:GAMMAEVOLUTIONERROR}, \eqref{AE:JUNKGAMMAEVOLUTIONERROR}, and \eqref{AE:KEVOLUTIONERROR}.

\end{proposition}

\hfill $\qed$

\subsubsection{The $\partial_{\vec{I}}-$commuted renormalized stiff fluid evolution equations}
\label{SSS:COMMUTEDFLUIDRENORM}
\begin{proposition} [\textbf{The $\partial_{\vec{I}}-$commuted renormalized stiff fluid evolution equations}]
Given a solution to \eqref{AE:RESCALEDPEVOLUTION}-\eqref{AE:RESCALEDUEVOLUTION},
the corresponding differentiated quantities  
$\partial_{\vec{I}} \adjustednewp$ and $\partial_{\vec{I}} \newu^j$
verify the following evolution equations:
\begin{subequations}
\begin{align}
	\partial_t \partial_{\vec{I}} \adjustednewp 
		& + 2t^{1/3} \big[1 + t^{4/3} \newlap\big]\big[1 + t^{4/3} \newg_{ab} \newu^a \newu^b\big]^{1/2} 
			\big[\adjustednewp + \frac{1}{3} \big] 
		\partial_c \partial_{\vec{I}} \newu^c  \label{AE:PARTIALTPCOMMUTED} \\
		& - 2 t^{5/3} \frac{\big[1 + t^{4/3} \newlap\big] \big[\adjustednewp + \frac{1}{3} \big]}
			{\big[1 + t^{4/3} \newg_{ab} \newu^a \newu^b \big]^{1/2}} 
			\newg_{ef} \newu^e \newu^c \partial_c \partial_{\vec{I}} \newu^f	\notag \\
		& =  - \frac{2}{3} t^{1/3} \partial_{\vec{I}} \newlap 
			+ t^{1/3} \leftexp{(\vec{I});(Junk)}{\mathfrak{P}}, \notag
	\end{align}
	\begin{align}  \label{AE:PARTIALTUJCOMMUTED}
	\partial_t \partial_{\vec{I}} \newu^j 	
	& - t^{1/3} \big[1 + t^{4/3} \newlap\big] \big[1 + t^{4/3} \newg_{ab} \newu^a \newu^b\big]^{1/2} \newu^j 
		\partial_c \partial_{\vec{I}} \newu^c  \\
	& + t^{5/3} \frac{\big[1 + t^{4/3} \newlap\big]}{\big[1 + t^{4/3} \newg_{ab} \newu^a \newu^b\big]^{1/2}} 
		\newg_{ef}\newu^e \newu^j \newu^c  \partial_c \partial_{\vec{I}} \newu^f \notag \\
	& \ \ + t^{1/3} \frac{\big[1 + t^{4/3} \newlap \big] \newu^c \partial_c \partial_{\vec{I}} \newu^j}
		{\big[1 + t^{4/3} \newg_{ab} \newu^a \newu^b\big]^{1/2}} \notag \\
	& \ \ + t^{-1} \frac{\big[1 + t^{4/3} \newlap\big] 
			\left\lbrace (\newg^{-1})^{jc} + t^{4/3} \newu^j \newu^c \right\rbrace \partial_c \partial_{\vec{I}}
			\adjustednewp}{2\big[1 + t^{4/3} \newg_{ab} \newu^a \newu^b\big]^{1/2} 
			\big[\adjustednewp + \frac{1}{3} \big]} \notag \\
		& = - t^{-1/3}(\newg^{-1})^{j a} \partial_{\vec{I}} \dlap_a 
			+ t^{-1} \leftexp{(\vec{I});(Border)}{\mathfrak{U}}^j 
			+ t^{1/3} \leftexp{(\vec{I});(Junk)}{\mathfrak{U}}^j, 
		\notag 
\end{align}
\end{subequations}
where the error terms 
$\leftexp{(\vec{I});(Junk)}{\mathfrak{P}},$
$\leftexp{(\vec{I});(Border)}{\mathfrak{U}}^j,$ 
and $\leftexp{(\vec{I});(Junk)}{\mathfrak{U}}^j,$
are defined by
\begin{subequations}
\begin{align} 
\leftexp{(\vec{I});(Junk)}{\mathfrak{P}} 
& := 
		\partial_{\vec{I}} \leftexp{(Junk)}{\mathfrak{P}} 
		\label{AE:PICOMMUTEDJUNK} \\
	& \ \ - 2 \left[ \partial_{\vec{I}} , \big[1 + t^{4/3} \newlap\big] \big[1 + t^{4/3} \newg_{ab} \newu^a \newu^b\big]^{1/2} 
		\big[\adjustednewp + \frac{1}{3} \big] \right] \partial_c \newu^c 
		\notag \\
	& \ \ + 2 t^{4/3} \left[ \partial_{\vec{I}} , \frac{\big[1 + t^{4/3} \newlap \big] 
		\big[\adjustednewp + \frac{1}{3} \big]}
			{\big[1 + t^{4/3} \newg_{ab} \newu^a \newu^b\big]^{1/2}} \newg_{ef} \newu^e \newu^c \right] \partial_c \newu^f, \notag
\end{align}
\begin{align}
	\leftexp{(\vec{I});(Border)}{\mathfrak{U}}^j 
	& := \partial_{\vec{I}} \leftexp{(Border)}{\mathfrak{U}}^j
		\label{AE:UJICOMMUTEDERRORBORDERLINE} \\
	 & \ \ - \left[ \partial_{\vec{I}} , 
	 		\frac{(\newg^{-1})^{jc} }{2\big[1 + t^{4/3} \newg_{ab} \newu^a 
			\newu^b\big]^{1/2} \big[\adjustednewp + \frac{1}{3} \big]} \right] 
			\partial_c \adjustednewp, \notag 
\end{align}
\begin{align}
	\leftexp{(\vec{I});(Junk)}{\mathfrak{U}}^j 
	& := \partial_{\vec{I}} \leftexp{(Junk)}{\mathfrak{U}}^j
		\label{AE:UJICOMMUTEDERRORJUNK}  \\
	& \ \ - \left[\partial_{\vec{I}} ,(\newg^{-1})^{ja} \right] \partial_a \newlap 
			\notag \\
	& \ \ + \left[ \partial_{\vec{I}} ,
			\big[1 + t^{4/3} \newlap\big]\big[1 + t^{4/3} \newg_{ab} \newu^a \newu^b\big]^{1/2} \newu^j 
			\right] \partial_c \newu^c \notag \\
	& \ \ - t^{4/3} \left[ \partial_{\vec{I}} ,
			\frac{\big[1 + t^{4/3} \newlap\big]}{\big[1 + t^{4/3} \newg_{ab} \newu^a \newu^b\big]^{1/2}}
			\newg_{ef} \newu^j \newu^e \newu^c \right] \partial_c \newu^f \notag \\
	& \ \ - \left[ \partial_{\vec{I}} ,
			\frac{\big[1 + t^{4/3} \newlap\big] \newu^c }{\big[1 + t^{4/3} \newg_{ab} \newu^a \newu^b\big]^{1/2}} \right]
			\partial_c \newu^j 
			\notag \\
	& \ \ - \left[ \partial_{\vec{I}} ,
			\frac{\big[1 + t^{4/3} \newlap\big] \newu^j \newu^c}{2\big[1 + t^{4/3} \newg_{ab} 
			\newu^a \newu^b\big]^{1/2} \big[\adjustednewp + \frac{1}{3} \big]} \right] 
			\partial_c \adjustednewp 
			\notag \\
	& \ \ - \left[ \partial_{\vec{I}} , \frac{\newlap (\newg^{-1})^{jc} }{2\big[1 + t^{4/3} \newg_{ab} \newu^a 
			\newu^b\big]^{1/2} \big[\adjustednewp + \frac{1}{3} \big]} \right] 
			\partial_c \adjustednewp, \notag
\end{align}
\end{subequations}
and $\leftexp{(Junk)}{\mathfrak{P}},$ $\leftexp{(Border)}{\mathfrak{U}}^j,$ and $\leftexp{(Junk)}{\mathfrak{U}}^j$
are respectively defined in \eqref{AE:PERROR}, \eqref{AE:UJBORDERLINE}, and \eqref{AE:UJJUNK}.

\end{proposition}

\hfill $\qed$

\bibliographystyle{amsalpha}
\bibliography{JBib}

\end{document}